\let\ORIlabel\label
\let\ORIrefstepcounter\refstepcounter
\AddToHook{package/hyperref/before}{
   \let\label\ORIlabel 
   \let\refstepcounter\ORIrefstepcounter}

\documentclass{siamart190516}
\usepackage{amsmath}
\usepackage{float}

\usepackage[font=small,labelfont=bf]{caption}
\usepackage{afterpage}
\usepackage{amsfonts}
\usepackage{graphicx}
\usepackage{epstopdf}
\usepackage{color}
\usepackage{tikz-cd}
\usepackage{bm}
\usepackage{multirow}
\usepackage{rotating}
\usepackage{amssymb}

\usepackage[labelformat=simple]{subcaption}

\def\nmax{n_{\rm max}}

\def\T{{\mathcal T}}

\newcommand{\bsub}{\begin{subequations}}
\newcommand{\esub}{\end{subequations}$\!$}

\newcommand{\eps}{{\varepsilon}}
\newcommand{\X}{{\bf X}}
\newcommand{\K}{{\mathcal K}}
\newcommand{\D}{{\mathcal D}}

\newcommand{\pa}{{\partial\Omega}}

\newcommand{\x}{{\bf{x}}}

\newcommand{\vv}{{\bf{v}}}
\newcommand{\vc}{{\bf{C}}}
\newcommand{\vac}{{\bf{A}}}

\newcommand{\evec}{{\bf e}}
\newcommand{\vq}{{\bf q}}

\newcommand{\vb}{{\bf{b}}}
\newcommand{\y}{{\bf{y}}}
\newcommand{\n}{{\bf{n}}}
\newcommand{\bxi}{{\bf{\xi}}}
\newcommand{\G}{{\mathcal G}_{s}}
\newcommand{\area}{|\Omega|}
\newcommand{\R}{{\mathbb{R}}}
\newcommand{\Cmu}{{\mathcal C}}

\newcommand{\p}{\prime}
\newcommand{\PT}{{\Gamma}}
\newcommand{\DT}{{\bf p}}
\newcommand{\W}{{\mathcal W}}
\newcommand{\OO}{{\mathcal O}}

\newcommand{\keff}{k_{\rm eff}}
\newcommand{\ueff}{u_{\rm eff}}
\newcommand{\Keff}{{\mathcal K}_{\rm eff}}

\def\hmax{h_{\rm max}}

\newcommand{\0}{{\mathcal O}}

\newtheorem{prop}{Proposition}

\newtheorem{remark}{Remark}

\renewcommand\thesection{\arabic{section}}
\renewcommand\thesubsection{\thesection.\arabic{subsection}}
\renewcommand\thesubsubsection{\thesubsection.\arabic{subsubsection}}
\renewcommand\theequation{\thesection.\arabic{equation}}

\graphicspath{{figures/}}

\title{The Asymptotic Analysis of Some PDE and Steklov Eigenvalue Problems
  with Partially Reactive Patches in 3-D}

\author{Denis S. Grebenkov \thanks{Laboratoire de Physique de la
    Mati\`{e}re Condens\'{e}e, CNRS -- Ecole Polytechnique,
    Institut Polytechnique de Paris, 91120 Palaiseau, France.  Email:
    denis.grebenkov@polytechnique.edu (corresponding author)}, \and Michael
  J. Ward \thanks{Department of Mathematics, University of British
    Columbia, Vancouver, B.C., Canada, V6T 1Z2. Email:
    ward@math.ubc.ca}}

\date{\today}
\begin{document}

\label{firstpage}
\maketitle

\baselineskip=12pt

\begin{abstract}
We consider steady-state diffusion in a three-dimensional bounded
domain with a smooth reflecting boundary that is partially covered by
small partially reactive patches.  By using the method of matched
asymptotic expansions, we investigate the competition of these patches
for a diffusing particle and the crucial role of surface reactions on
these targets.  After a brief overview of former contributions to this
field, we first illustrate our approach by considering the classical
problems of the mean first-reaction time (MFRT) and the splitting
probability for partially reactive patches characterized by a Robin
boundary condition.  For a spherical domain, we derive a three-term
asymptotic expansion for the MFRT and splitting probabilities in the
small-patch limit.  This expansion is valid for arbitrary
reactivities, and also accounts for the effect of the spatial
configuration of patches on the boundary.  Secondly, we consider more
intricate surface reactions modeled by mixed Steklov-Neumann or
Steklov-Neumann-Dirichlet problems.  We provide the first derivation
of the asymptotic behavior of the eigenvalues and eigenfunctions for
these spectral problems in the small-patch limit for a spherical
domain.  Extensions of these asymptotic results to arbitrary domains
and their physical applications are discussed.
\end{abstract}



%  35B25 (sing pert. of PDE)
%  35J25 (mixed BVP of elliptic type... Laplace, Poisson)
%  35P15 (eigenvalue problems... including Steklov)
%  60J65 (diffusion problems )
%  secondary:   35J08 (Green's functions)

% From my SIAM
% 35J05 - Laplacian operator, reduced wave equation (Helmholtz equation), Poisson equation
% 35Pxx - Spectral theory and eigenvalue problems
% 49Rxx - Variational methods for eigenvalues of operators
% 51Pxx - Geometry and physics

\section{Introduction}\label{all:intro}

Many vital processes in microbiology rely on diffusive search
for small targets, such as proteins searching for their partners or
specific sites on a DNA chain, ions searching for channels on the
plasma membrane of the cell, viruses searching for nuclear pores, etc.
\cite{Alberts,Lauffenburger,Bressloff13,Holcman13,Grebenkovbook}.  In
heterogeneous catalysis, reactants move toward active sites on a
solid catalytic surface to be chemically transformed
\cite{House,Lindenberg,Grebenkov23b}.  In nuclear magnetic resonance
experiments, spin-bearing molecules diffuse in tissues or mineral
samples and may relax their magnetization on magnetic impurities that
are located on confining walls
\cite{Callaghan,Grebenkov07,Kiselev17}.  Various search problems on a
macroscopic scale such as animal or human behavior are inspired by
ecology \cite{Kurella15,Grebenkovbook}.  These and many other natural
phenomena are often modeled by reflected Brownian motion that is
confined inside a bounded domain by an impenetrable boundary.  This
stochastic process is stopped (or killed) under certain conditions
that represent interactions of the diffusing particle with prescribed
regions (called ``targets'') that are located either inside the domain
or on its boundary.  Depending on the context and application, such
interactions may represent a chemical reaction, a binding, a
conformational change to a different state, a relaxation of
magnetization or fluorescence, a passage through an ion channel or a
pore, an escape from the domain, etc.

One of the simplest and most studied stopping condition is the first
arrival onto the targets.  In this scenario, the efficiency of the
diffusive search is usually characterized by the distribution of the
first-passage time (FPT) to single or multiple targets
\cite{Redner,Schuss,Metzler,Masoliver,Dagdug,Grebenkovbook}.
Many former works have been dedicated to the analysis of the {\em
mean} first-passage time (MFPT) and its dependence on the shape, size,
and spatial arrangement of the targets in the confining domain.  If
$\X$ denotes the starting point of reflected Brownian motion in a
bounded domain $\Omega \subset \R^d$ with a smooth boundary $\pa$, the
MFPT $T(\X)$ to a target region (also called a reactive patch) $\pa_a
\subset \pa$ satisfies the Poisson equation with mixed
Dirichlet-Neumann boundary condition:
\begin{subequations}  \label{eq:MFPT}
\begin{align}   \label{eq:MFPT_eq}
-D \Delta T(\X) & = 1 \,, \quad \X\in\Omega\,, \\
 T(\X) & = 0 \,, \quad \X\in \pa_a\,, \\  
\partial_n T(\X) & = 0 \,, \quad \X\in\pa_r = \pa \backslash \pa_a\,,
\end{align}
\end{subequations}
where $D > 0$ is a constant diffusion coefficient, $\Delta$ is the
Laplace operator, and $\partial_n$ is the outward normal derivative.
The Dirichlet condition on $\pa_a$ characterizes the first-arrival
stopping condition (if the particle starts from $\pa_a$, the process
is immediately stopped, yielding $T = 0$), whereas the Neumann
condition on $\pa_r$ ensures zero flux across the reflecting part of
the boundary.  From a mathematical viewpoint, the mixed boundary
condition presents the major challenge in solving this boundary value
problem (BVP) in general domains, in which a trivial separation of
variables, such as in rectangles or concentric spheres, is not
applicable.  Despite the existence of advanced techniques such as dual
integral or series equations \cite{Sneddon}, or the generalized method
of separation of variables \cite{Grebenkov19f,Grebenkov20f}, the disk
with an arc-shaped target is presumably the only nontrivial example,
for which an exact explicit solution of (\ref{eq:MFPT}) has been found
\cite{Singer06b,Rupprecht15,Grebenkov16,Marshall16}.

For this reason, various approximate, numerical and asymptotic
techniques have been developed over the past two decades.  When the
patch $\pa_a$ has small area, determining the asymptotic behavior of
the MFPT is usually referred to as {\em the narrow escape problem}
(see an overview in \cite{Holcman14}).  The presence of a small
Dirichlet patch $\pa_a$ on an otherwise reflecting boundary is a {\em
singular} perturbation \cite{Ozawa81,Mazya85,Ward93}.  In fact, if
$\pa_a$ was absent, the MFPT would be infinite. As a consequence, the
solution of (\ref{eq:MFPT}) diverges in the small-patch limit $\pa_a
\to \emptyset$.  The very first asymptotic result for the MFPT to a
small circular patch of radius $\epsilon$ on a spherical boundary of
radius $R$ dates back to Lord Rayleigh, who found (in the context of
acoustics) that, to leading-order in $\epsilon$, $T \sim
|\Omega|/(4D\epsilon)$, where $|\Omega| = 4\pi R^3/3$ is the volume of
the spherical domain \cite{Rayleigh}.  Even though the particle
escapes through the two-dimensional boundary patch, the MFPT scales
inversely with the radius of the patch ($\sim \epsilon$) and not with
its area ($\sim \epsilon^2$).  This is a characteristic feature of a
perfectly reactive circular patch with Dirichlet condition, also known
as the diffusion-limited regime in diffusion-controlled reactions
\cite{North66,Wilemski73,Calef83,Berg85,Rice85,Grebenkov23b};
this behavior is typical for Dirichlet patches of arbitrary
shape. However, as we will discuss below, an additional scaling regime
emerges for more realistic reaction mechanisms.  Rayleigh's
leading-order result has been generalized to other domains and patch
shapes, both in two and three dimensions, and the asymptotic results
have been applied to many specific problems in microbiology (see
\cite{Holcman04,Singer06a,Singer06b,Singer06c,Schuss07,Holcman08,Singer09}
and the overviews in \cite{Holcman13,Holcman14}).

The leading $\OO(1/\epsilon)$ behavior clearly indicates that an
accurate characterization of the MFPT requires the knowledge of {\it
subleading} terms, at least up to $\OO(1)$.
Such refined asymptotic results can be obtained by the method of {\em
matched asymptotic expansions} \cite{Lagerstrom88}, which was tailored
in \cite{Ward93} and \cite{Ward93b} to analyze PDE problems with
strong localized perturbations.  This method has then been implemented
to systematically calculate higher-order asymptotic approximations of
the MFPT, the splitting probability, and related first-passage
quantities in both two- and three-dimensional settings with either
boundary patches or interior targets
\cite{Coombs09,Cheviakov10,Pillay2010,Cheviakov11,Chen11,Cheviakov13,Delgado15,Bressloff21a,Bressloff21b,Iyaniwura21}.
Some asymptotic results for related problems on unbounded domains are
given in \cite{Lindsay17,Lindsay18a,Lawley20} (see also the references
therein).  An alternative approach to the method of matched asymptotics
was described in \cite{Chevalier11}.  Even though our discussion is
focused on ordinary diffusion, extensions to more sophisticated
diffusion processes, and to problems with stochastic resetting, have
been explored (see
\cite{Condamin07,Benichou08,Bressloff15,Guerin16,Bressloff20} and the
references therein).  We also emphasize that the knowledge of the {\em
mean} FPT does not fully characterize the distribution of this random
variable \cite{Godec16,Grebenkov18a}.  Analytical and numerical
studies dealing with the full probability distribution include
\cite{Rupprecht15,Grebenkov19a,Reva21,Lindsay18b,Cherry22} (see the
references therein).

The above stopping condition corresponds to the simplest reaction
kinetics when the reaction event occurs certainly and instantly upon
the first arrival onto the target.  This perfect reaction scenario
assumes that the diffusive search is the only limiting factor and thus
greatly oversimplifies reaction kinetics that is relevant in many
applications \cite{Piazza22,Grebenkov23b}.  Indeed, a particle that
arrives onto the target may not react instantly due to various
reasons: (i) a reaction event often requires overcoming an activation
energy barrier \cite{Weiss86,Hanggi90}; (ii) an escape event may
involve overcoming an entropic barrier
\cite{Zhou91,Reguera06,Chapman16}; (iii) a macromolecule may need to
be in the proper conformational state to bind its partner
\cite{Cortes10,Galanti16b,Luking22}; (iv) the target may switch
between active and passive states (e.g., an ion channel can be open or
closed) \cite{Benichou00,Reingruber09,Lawley15}; (v) the target may be
microscopically inhomogeneous so that the arrival point may be inert
\cite{BergPurcell1977,Berezhkovskii04,Berezhkovskii06,Muratov08,Lindsay17,Bernoff18,Punia21}.
Whatever the microscopic origin is, such {\em imperfect} targets,
generally referred to as {\em partially reactive}, are often modeled
by a Robin boundary condition \cite{Collins49}.  For instance, the BVP
(\ref{eq:MFPT}) is replaced by
\begin{subequations}  \label{eq:MFRT}
\begin{align}   
-D \Delta \T(\X) & = 1 \,, \quad \X\in\Omega\,, \\  \label{eq:MFRT_Robin}
D \partial_n \T + \K \T & = 0\,, \quad \X\in\pa_a\,, \\
\partial_n \T(\X) & = 0 \,, \quad \X\in\pa_r = \pa \backslash \pa_a\,,
\end{align}
\end{subequations}
with the mixed Robin-Neumann boundary condition (\ref{eq:MFRT_Robin}).
The reactivity $\K$, which has units of length per time, characterizes
the facility of the reaction event, by ranging from $0$ (inert passive
target, no reaction) to $+\infty$ (perfect reaction upon the first
arrival).  If the particle that arrives onto the target fails to
react, it is reflected from the target and resumes its diffusion in
the domain until the next arrival, and so on.  As a consequence, the
successful reaction event is generally preceded by a sequence of
diffusive excursions in the bulk after each failed reaction attempt,
and the mean {\em first-reaction time} (MFRT), $\T(\X)$, satisfying
(\ref{eq:MFRT}), can significantly exceed the MFPT $T(\X)$ satisfying
(\ref{eq:MFPT}).

More generally, the Robin boundary condition is usually
understood as a mass conservation law \cite{Collins49}: if $c(\X)$ is
a steady-state concentration of particles, their diffusive flux
density $-D\partial_n c$ from the bulk is equal to the ``reactive
flux density'' on the patch, $\K c$, which is proportional to their
concentration $c$.  This macroscopic relation admits various microscopic
interpretations.  \\
(i) A perfectly absorbing patch can be separated from the bulk by a
thin permeable membrane of thickness $a$ and of diffusivity $D_m$
(e.g., a thin membrane separating the interior of $\Omega$ from an
absorbing exterior reservoir).  The continuity conditions for the
diffusion equation in such a biphasic medium allow one to represent
the effect of the membrane by imposing the Robin boundary condition
with permeability $\K = D_m/a$ on the frontier between the interior
and the membrane \cite{Sapoval94,Felici03}.  When the membrane
thickness $a$ vanishes, one has $\K\to \infty$, so that the Dirichlet
boundary condition is recovered, as expected.  A similar
interpretation appears in an electrochemical setting, in which $\K$ is
proportional to the Faraday resistance (or surface impedance) of a
metal electrode
\cite{Halsey92,Sapoval94}.  \\
(ii) In the framework of chemical reactions or surface relaxation, the
patch mimics a thin reactive layer of thickness $a$ that represents
the effect of short-range interactions between the diffusing particle
and reactive sites on the boundary.  If this layer is characterized by
a spatially constant bulk reaction rate $\nu$, the overall effect of
this layer onto the concentration of particles (and related
quantities) can be modeled via the Robin boundary condition with the
reactivity $\K = a \nu$ \cite{Grebenkov20}. \\
(iii) The reactivity $\K$ can also be related to the probability of
absorption upon encounter with the boundary, $p_a$, which is highly
relevant in computational chemistry, molecular simulations, and
intracellular transport models.  For instance, if the reflected
Brownian motion is approximated by a random walk on a regular lattice
$(a\mathbb{Z})^d$ with lattice spacing $a$, the effect of a partially
reactive patch is modeled as follows: at each arrival onto any
reactive site of the patch, the particle can either react with the
absorption probability $p_a = 1/(1 + D/(\K a))$, or jump to the
neighboring bulk site with the probability $1-p_a$
\cite{Grebenkov03,Grebenkov23a}.  As the reactivity $\K$ varies from
$0$ to $+\infty$, the absorption probability ranges from $0$ to $1$,
as expected.  This expression for $p_a$ can alternatively be seen as a
finite-difference discretization of the Robin boundary condition.  The
same absorption probability appears in the construction and Monte
Carlo simulations of partially reflected Brownian motion
\cite{Grebenkov06,Grebenkov07b} and other partially reflected
diffusions (see \cite{Andrews04,Erban07,Singer08} and references
therein; note that an approximate form $p_a \approx a \K/D\ll 1$ is
sometimes used, given that $a$ is a small parameter). \\
(iv) A more general probabilistic interpretation of the Robin boundary
condition in terms of the boundary local time $\ell_t$ was provided in
\cite{Grebenkov20}.  If each encounter with a thin reactive layer of
thickness $a$ can result in a reaction event with the probability
$p_a$, the number of failed attempts until the successful reaction
obeys the geometric distribution.  Understanding an encounter as
downcrossing of the layer, one can approximate the number of such
downcrossings up to time $t$ as $\ell_t/a$.  In the limit $a\to 0$, an
appropriate rescaling allows one to define the first-reaction time
\begin{equation} \label{eq:tau_def}
\tau = \inf\{ t>0 ~:~ \ell_t > \hat{\ell}\} 
\end{equation}
as the first instance when the boundary local time $\ell_t$ exceeds a
random threshold $\hat{\ell}$ obeying the exponential distribution:
$\mathbb{P}\{\hat{\ell} > \ell \} = e^{-\ell \K/D}$ (see details in
\cite{Grebenkov20}).  In this interpretation, the patch shape and the
diffusive dynamics in $\Omega$ determine the boundary local time
$\ell_t$, whereas the reactivity $\K$ controls the independent
threshold $\hat{\ell}$.  As a consequence, the geometric and diffusive
properties of the system are disentangled from the surface reaction.

In summary, various physical mechanisms, microscopic models, and
probabilistic constructions lead to the Robin boundary condition,
which appears as a natural extension of and a more realistic
alternative to the Dirichlet patch.

In this paper, we aim at adjusting the method of matched asymptotic
expansions to deal with Robin patches.  We emphasize that the effect
of partial reactivity on the efficiency of the diffusive search, which
was ignored in most former studies, has recently attracted
considerable attention.  For instance, the small-patch asymptotic
behavior of the MFRT was deduced in \cite{Grebenkov17a} by using a
constant-flux approximation.  In particular, for a circular patch on a
spherical boundary, the MFRT was shown to behave as $|\Omega|/(\K
|\pa_a|)$ in the leading order in $\epsilon$, where $|\pa_a| = \pi
\epsilon^2$ is the surface area of the patch, and $|\Omega|$ is
the volume of the domain.  The faster divergence $\0(\epsilon^{-2})$
is reminiscent of the reaction-limited regime, which becomes dominant
in the small-patch limit for a finite reactivity $\K$.  The change of
scaling between diffusion-limited and reaction-limited regimes
suggests a nontrivial dependence on the reactivity, especially in the
limit $\K \to \infty$.  In fact, the limits $\K\to \infty$ and
$\epsilon \to 0$ are not interchangeable.  This point was further
investigated in \cite{Guerin23,Cengiz24,Grebenkov25}, where the
behavior of the MFRT on a small circular patch was studied for small,
intermediate, and large reactivities by different methods.  The effect
of spatially heterogeneous reactivity was analyzed via a spectral
approach \cite{Grebenkov19b}, whereas the role of target anisotropy
onto the MFRT was studied in \cite{Chaigneau22}.  The trapping rate of
a reflecting plane covered by partially reactive circular patches was
estimated by means of boundary homogenization technique
\cite{Plunkett24} (see also \cite{Taflia11,Freche11} for other
homogenization results with partially reactive patches and their
applications to neuron signaling).  An elaborate asymptotic analysis
for partially reactive targets in planar domains was developed in
\cite{Grebenkov-Ward25b}.

To consolidate and extend this recent progress, we aim to establish a
streamlined and systematic asymptotic framework for analyzing the MFRT
and the splitting probability in a 3-D domain with partially reactive
boundary patches. Previous asymptotic analyses of the MFRT (see
\cite{Cheviakov10,Lindsay17,Cheviakov15,Tzou21}) and related problems
with {\em perfectly} reactive patches in 3-D have shown that a
higher-order asymptotic expansion is required to study the effect of
inter-patch interactions, and that some gauge terms in the asymptotic
expansion can have an intricate dependence on $\epsilon$.  Our
framework is then extended to treat certain mixed Steklov eigenvalue
problems, which are also relevant to diffusive capture.

Some key common features in all of these problems, and open questions
regarding the effect of partially reactive patches in 3-D include the
following:

(i) Many aforementioned works have dealt only with a single patch.  In
turn, multiple patches compete with each other for capturing a diffusing
particle (the so-called diffusional screening \cite{Felici03} or
diffusive interactions \cite{Traytak96,Traytak97,Berezhkovskii12}), which
leads to various optimization problems on a suitable spatial
arrangement of patches.  The MFPT to multiple perfectly reactive
circular patches has been studied for a sphere in
\cite{Cheviakov10}.  Can one provide a more streamlined asymptotic
approach to analyze the competition of multiple partially reactive
arbitrarily-shaped patches and its effect on the MFRT and related
quantities?

(ii) Many former works on the MFRT focused on the leading-order term,
which scales as $\epsilon^{-2}$ when ${\mathcal K}$ is considered
fixed on a patch while $\epsilon\to 0$.  However, the expected
``correction'' terms $\0(\epsilon^{-1})$, $\0(\log \epsilon)$ and
$\0(1)$ may provide significant contributions, and their knowledge is
required for an accurate estimation of the MFRT in applications,
especially if $\epsilon$ is not too small.  Can one develop a
systematic approach to capture these higher-order terms in the MFRT
for fixed ${\mathcal K}$ as $\epsilon \to 0$?

(iii) To our knowledge, previous analyses on partially reactive
patches have assumed their {\em circular} shape.  How does the shape
of the partially reactive patch affect the asymptotic behavior of the MFRT?

(iv) What is the impact of the spatial arrangement of patches?  This
question was addressed for perfectly reactive circular patches for the
MFPT in \cite{Cheviakov10}, but remains open for partially reactive
ones.

(v) Even though the use of the MFRT is a common way to characterize
the efficiency of the diffusive search, other quantities may be needed
to reveal versatile facets of this phenomenon.  For instance, one
often employs the splitting probabilities to describe the relative
contributions of individual patches in their competition for capturing
the diffusing particles.  What is the asymptotic behavior of the
splitting probabilities of partially reactive patches?

(vi) Finally, the Robin boundary condition (\ref{eq:MFRT_Robin})
implements the simplest model of a constant reactivity on the patch.
The encounter-based approach \cite{Grebenkov20} allows one to
introduce a much more general class of surface reactions that
describe, e.g., progressive activation or de-activation of the patch
by its interaction with diffusing particles, non-Markovian binding,
surface adsorption,
etc. \cite{Grebenkov23a,Bressloff23d,Bressloff23e}.  In probabilistic
terms, one can still use the stopping condition
(\ref{eq:tau_def}) that involves the boundary local time $\ell_t$ (as
a proxy of the number of encounters of the particle with the patch);
however, the former exponential distribution of the random threshold
$\hat{\ell}$ can now be replaced any suitable distribution
\cite{Grebenkov20c,Grebenkov22}.  One can show that the PDE
formulation of this framework substitutes the Robin boundary condition
by an integral equation on the patch.  A natural framework for solving
and analyzing such PDEs relies on the Steklov-Neumann problem
\cite{Grebenkov25} or the Steklov-Dirichlet-Neumann problem
\cite{Grebenkov23} (see \S \ref{sec:results} for their formulation).
These are basic extensions of the conventional Steklov spectral
problem that has been thoroughly studied in spectral geometry
\cite{Levitin,Girouard17,Colbois24}.  To characterize the
efficiency of multiple small patches with more sophisticated reaction
mechanisms, can one analyze the asymptotic behavior of the eigenvalues
and eigenfunctions of the related Steklov problems in the small-patch
limit?  To our knowledge, this asymptotic problem was not addressed in
the past (except for \cite{Grebenkov25} in the case of a single
circular patch in 3-D, and for \cite{Grebenkov-Ward25b} in two
dimensions).

In this paper, we aim to present a common theoretical framework
that synthesises previous approaches and addresses the open issues
discussed above. In our approach, we combine the method of matched
asymptotic expansions, relying on strong localized perturbation theory
\cite{Ward93}, with spectral expansions based on the local exterior
Steklov problem on each patch.  The use of geodesic normal coordinates
is another key tool that allows for a more streamlined access to the
higher-order terms in the asymptotic expansions that represent
inter-patch interactions, which arise from the the overall spatial
configuration of the patches.  In the next section, we formulate four
asymptotic problems and summarize our main results.

\section{Summary of main results}  \label{sec:results}

We consider reflected Brownian motion in a three-dimensional bounded
domain $\Omega$, with a smooth boundary $\pa$, which consists of the
union $\partial\Omega_a = \cup_{i=1}^{N} \partial\Omega_i$ of $N$
reactive patches $\partial\Omega_i \subset \pa$ and the remaining
reflecting (inert) boundary $\partial\Omega_r = \pa \backslash \pa_a$.
Each reactive boundary patch $\partial\Omega_i$ is assumed to be
simply-connected with a smooth boundary, but with an otherwise
arbitrary shape.  We denote $2L_i$ to be the diameter of the
orthogonal projection of the patch $\partial\Omega_i$ when mapped onto
the tangent plane.  Our asymptotic analysis will exploit an assumed
length-scale separation ${L/R}\ll 1$, where $L\equiv
\max\limits_{i} \{L_{i}\}$, and $2R$ is the diameter of the confining
domain $\Omega$.  The patches are assumed to be well-separated in the
sense that $\mathrm{dist}\{\pa_i, \pa_j\} \gg L$ for all $i\neq j$.
We will study four different problems that can be analyzed with a
common theoretical framework. For this reason, we will employ the same
notations, e.g., $U(\X)$, for formulating and studying these problems.

(I) The mean first-reaction time, $\T(\X) = U(\X)$, on the union
$\pa_a$ of partially reactive patches 
$\partial\Omega_1,\ldots,\partial\Omega_N$ with finite reactivities
$\K_1,\ldots,\K_N$, satisfies a Poisson equation with mixed
Neumann-Robin boundary conditions, re-formulated from (\ref{eq:MFRT})
as
\bsub \label{mfpt:ssp0}
\begin{align}
  \Delta U & = - \frac{1}{D} \,, \quad \X \in \Omega \,,
  \\
   D\partial_{n} U + \K_i U & = 0\,, \quad \X \in
     \partial\Omega_i \,, \quad i=1,\ldots,N \,, \\
  \partial_{n} U & = 0 \,, \quad \X \in \partial\Omega_r\,.
\end{align}
\esub 
We also consider the volume-averaged MFRT,
\begin{equation}   \label{eq:global}
 \overline{U} \equiv \frac{1}{\area} \int_{\Omega} U(\X) \, d\X \,, 
\end{equation}
which corresponds to the average with respect to a uniform
distribution of initial points $\X\in\Omega$, where $|\Omega|$
denotes the volume of $\Omega$.

(II) The splitting probability, $U(\X)$, to react on the first patch
$\partial\Omega_1$, before reacting on the other patches, satisfies
\bsub \label{split:ssp_0}
\begin{align}
  \Delta U &= 0 \,, \quad \X \in \Omega \,, \\
  D\partial_{n} U + \K_i U &= \delta_{i1}\K_i\,, \quad
      \X \in \partial \Omega_i \,, \quad i=1,\ldots,N \,, \\
  \partial_{n} U &= 0 \,, \quad \X \in \partial \Omega_r \,,
\end{align}
\esub
where $\delta_{11}=1$ and $\delta_{i1}=0$ for $i=2,\ldots,N$. 
The volume-averaged splitting probability is also given by
(\ref{eq:global}).

(III) More sophisticated surface reactions can be formulated in terms
of various Steklov spectral problems
(cf.~\cite{Grebenkov20,Grebenkov22,Grebenkov23a,Grebenkov23,Grebenkov24,Grebenkov25}).
One such problem consists of finding the eigenpairs $\{ \Sigma, U\}$
of the mixed Steklov-Dirichlet-Neumann (SDN) problem formulated as
\bsub \label{sdn:eig_0}
\begin{align}
  \Delta U &= 0 \,, \quad \X \in \Omega \,, \\
  \partial_{n} U &=\Sigma U\,, \quad \X \in \partial   \Omega_1 \,, \\ 
  U &=0 \,, \quad \X\in\partial\Omega_i \,,
  \quad i=2,\ldots,N \,,\\
  \partial_{n} U &= 0 \,, \quad \X \in \partial \Omega_r.
\end{align}
\esub 
These eigenpairs allow one to solve the escape problem when the
diffusing particle has to react on the patch $\partial\Omega_{1}$
before escaping from the domain $\Omega$ through any Dirichlet patch
$\partial\Omega_{j}$, for $j=2,\ldots,N$, which may represent holes or
channels on the domain boundary \cite{Grebenkov23}.

(IV) Finally, we will consider the mixed Steklov-Neumann (SN) problem
that consists of finding the eigenpairs $\{\Sigma, U\}$ satisfying
\bsub \label{sn:eig_0}
\begin{align}
  \Delta U & = 0 \,, \quad \X \in \Omega \,, \\
  \partial_{n} U & =\Sigma U\,, \quad \X \in 
  \partial\Omega_i \,, \quad i=1,\ldots,N \,,\\
  \partial_{n} U &= 0 \,, \quad \X \in \partial \Omega_r \,.
\end{align}
\esub These eigenpairs can be used to investigate sophisticated
surface reactions on multiple patches and their competition.  Some
general spectral properties of mixed Steklov-Neumann problems, known
as sloshing or ice-fishing problems in hydrodynamics, were studied
previously in \cite{Henrici70,Fox83,Kozlov04,Levitin22} (see also the
references therein).

For each of these four problems we focus on the spherical domain
$\Omega = \{ \X\in \R^3 \, \vert \, |\X| \leq R\}$ in order to derive
the three-term asymptotic behavior in the limit $\eps={L/R}\ll 1$.  We
stress that the leading-order terms in our derived asymptotic results
do not depend on the geometry of the confining domain and are thus
valid for an arbitrary domain with a smooth boundary.  However, our
emphasis is on calculating the higher-order terms in the asymptotic
expansions, which are often relevant for applications, and are needed
for determining the effect of the location of the patches on the
surface and for deriving a homogenization result for the MFRT.
Although the methodology for deriving the three-term expansions can
potentially be extended to arbitrary 3-D domains (see discussion in \S
\ref{sec:discussion}), we will restrict our analysis to the sphere.
Some preliminary results that are the central building blocks for our
analysis are summarized in \S \ref{mfpt_sec:prelim}, including the
geodesic normal coordinates, the reactive capacitance, and the surface
Neumann Green's function, which is available analytically for a
sphere.

For our analysis, it is convenient to reformulate the four problems in
terms of dimensionless variables defined by
\begin{equation}\label{intro:scalings}
  \x = \frac{\X}{R} \,, \quad L =\max\limits_{i}\{L_i\} \,, \quad \eps = \frac{L}{R}
  \,, \quad a_i=\frac{L_i}{L} \,, \quad \kappa_i = \frac{L \K_i}{D} \,,
\quad   \sigma =  \Sigma L \,.
\end{equation}
Moreover, the dimensionless MFRT $u(\x)$ is expressed as
\begin{equation}\label{intro:mfpt_scale}
  u(\x)=\frac{D}{R^2}U(\x R) \,.
\end{equation}
Such a rescaling of $U$ is not needed for the splitting probability
(which is already dimensionless), nor for the Steklov eigenfunctions,
which are defined up to a suitable normalization.

In terms of the new variables (\ref{intro:scalings}) and
(\ref{intro:mfpt_scale}), the dimensionless MFRT $u(\x)$ in the unit
sphere $\Omega$ with partially reactive patches and dimensionless
reactivities $\kappa_i$ satisfies
\bsub \label{mfpt:ssp}
\begin{align}
  \Delta_{\x} u & = - 1 \,, \quad \x \in \Omega \,,
                  \label{mfpt:ssp_1}\\
  \eps\partial_{n} u + \kappa_i u & = 0\,, \quad \x \in \partial\Omega^{\eps}_i
                                    \,, \quad i=1,\ldots,N \,, \\
  \partial_{n} u & = 0 \,, \quad \x \in \partial \Omega_r
  =\partial\Omega\backslash\partial\Omega_{a}\,, \label{1:ssp_2}
\end{align}
\esub 
where $\Delta_{\x}$ is the Laplacian in $\x$, and $\partial_n$ is
again the outward normal derivative to $\partial\Omega$.  Each
reactive boundary patch $\partial\Omega^{\eps}_i$, of small diameter
${\mathcal O}(\eps)$, is assumed to be simply-connected with a smooth
boundary, but with an otherwise arbitrary shape, and satisfies
$\partial\Omega^{\eps}_i\to\x_i\in \partial\Omega$ as $\eps\to 0$.
The patches are assumed to be well-separated in the sense that
$\left|\x_i-\x_j\right|={\mathcal O}(1)$ for all $i\neq j$.  With
respect to a uniform distribution of initial points $\x\in\Omega$ for
the reflected Brownian motion, the dimensionless volume-averaged MFRT
is
\begin{equation} 
 \overline{u} \equiv \frac{1}{\area} \int_{\Omega} u(\x) \,
 d\x \,, \label{mfpt:ubar}
\end{equation}
where $|\Omega|={4\pi/3}$ is the volume of $\Omega$. The geometry of a
confining sphere with reactive patches on its boundary is depicted in
Fig.~\ref{fig:schem}.  

\begin{figure}[htbp]
  \centering
    \begin{subfigure}[b]{0.42\textwidth}  
    \includegraphics[width =\textwidth]{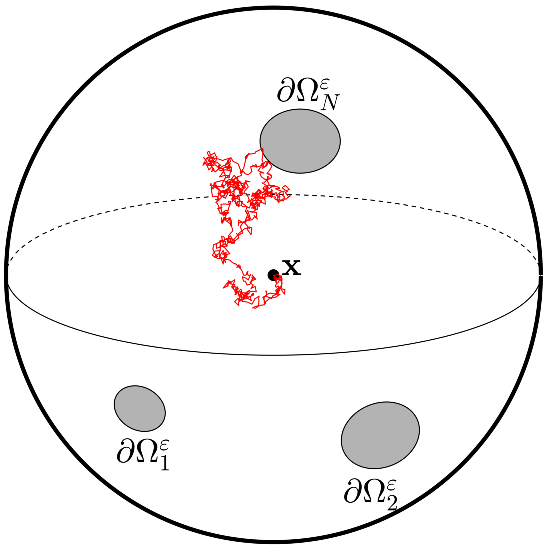} % scheme_sphere.eps}
    \caption{Brownian trajectory inside unit sphere}
    \label{fig:schem}
  \end{subfigure}  \hskip 10mm
  \begin{subfigure}[b]{0.45\textwidth}  
    \includegraphics[width=\textwidth]{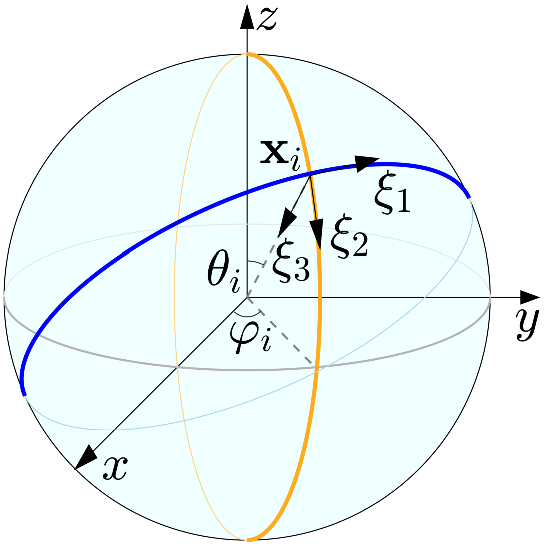} % geodesic_coords.png}
    \caption{Geodesic normal coordinates}
   \label{fig:geodesic}
  \end{subfigure}  
\caption{
(a): Sketch of a Brownian trajectory inside the unit sphere in $\R^{3}$
with partially reactive patches $\pa^{\eps}_1, \ldots, \pa^{\eps}_N$
on the boundary. (b): Geodesic normal coordinates
$(\xi_1,\xi_2,\xi_3)^T$ centered at $\x_i\in \partial\Omega$, with the
geodesics (orange and blue curves) indicated. }
\end{figure}

In \S \ref{mfpt_sec:expan} we will derive an asymptotic expansion for
$u(\x)$ and $\overline{u}$ for arbitrary $\kappa_i>0$ in the limit
$\eps\to 0$ of small patches.  The main result is summarized in
Proposition \ref{mfpt_b:main_res} of \S
\ref{mfpt_sec:expan}.  For the special case where $\kappa_i=\infty$
and when the patches are disks, such an analysis has been performed in
\cite{Cheviakov10} by expanding spherical coordinates near each
patch. In \S \ref{mfpt_sec:expan} we will use a different and simpler
approach than in \cite{Cheviakov10} that relies on geodesic normal
coordinates as introduced in \S
\ref{mfpt_sec:prelim}, which allows us to more readily consider the
case of finite $\kappa_i$ and arbitrary patch shapes. More explicit
results will be obtained when the patches are locally circular with
radii $\eps a_i$ for $i=1,\ldots,N$.  In \S \ref{sec:moderateK} we use
our main result to deduce a four-term asymptotic expansion for the
case where the dimensional reactivities ${\mathcal K}_i$ are
considered fixed but $\eps\to 0$; this corresponds to the case where
$\kappa_i={\mathcal O}(\eps)$.  In \S \ref{mfpt:homog}, we obtain the
effective reactivity of the spherical surface in the homogenized
limit.  As a byproduct of this analysis, we also derive in \S
\ref{sec:mfrt:eig} a three-term expansion for the principal (lowest)
eigenvalue of the Laplace operator with mixed Neumann-Robin boundary
conditions.

In a similar way, in terms of the dimensionless variables
(\ref{intro:scalings}), the splitting probability $u(\x)$ for a
Brownian particle in the unit sphere $\Omega$ to react on a specific
target patch on the domain boundary, labeled below by
$\partial\Omega_1^{\eps}$, before reacting on any of the remaining
$N-1$ boundary patches (with $N\geq 2$) satisfies
\bsub \label{split:ssp}
\begin{align}
  \Delta_{\x} u &= 0 \,, \quad \x \in \Omega \,, \label{split:ssp_1}\\
  \eps\partial_{n} u + \kappa_i u &= \delta_{i1}\kappa_i\,, \quad \x \in \partial
  \Omega_i^{\eps} \,, \quad i=1,\ldots,N \,, \\
  \partial_{n} u &= 0 \,, \quad \x \in \partial \Omega_r
  =\partial\Omega\backslash\partial\Omega_{a}\,. \label{split:ssp_2}
\end{align}
\esub Here $\delta_{11}=1$ and $\delta_{i1}=0$ for $i=2,\ldots,N$ and
we have used the same notation and assumptions on the patches as for
the MFRT problem (\ref{mfpt:ssp}).  The asymptotic analysis of
(\ref{split:ssp}) is done in \S \ref{split:intro},
with our main result being summarized in Proposition
\ref{splitt_b:main_res}.

In terms of (\ref{intro:scalings}) the dimensionless SDN problem in
the unit sphere $\Omega$ consists of finding the eigenpairs
$\{ \sigma, u\}$ satisfying
\bsub \label{sdn:eig}  
\begin{align}
  \Delta_{\x} u &= 0 \,, \quad \x \in \Omega \,,  \label{sdn:eig_1} \\
 \eps \partial_{n} u &=\sigma u\,, \quad \x \in \partial \Omega_1^{\eps} \,,   \label{sdn:eig_2} \\ 
  u &=0 \,, \quad \x\in\partial\Omega_i^{\eps} \,,
  \quad i=2,\ldots,N \,,    \label{sdn:eig_3}\\
  \partial_{n} u &= 0 \,, \quad \x \in \partial \Omega_r.    \label{sdn:eig_4}
\end{align}
\esub In the limit $\eps\to 0$, this problem is analyzed in \S
\ref{stekDN:intro}, with the main result summarized in Proposition
\ref{sdn:main_res}.

In turn, the dimensionless SN problem for (\ref{sn:eig_0})
consists of finding the eigenpairs $\{\sigma, u\}$ satisfying
\bsub \label{sn:eig}  
\begin{align}
  \Delta_{\x} u & = 0 \,, \quad \x \in \Omega \,, \\
  \eps \partial_{n} u & =\sigma u\,, \quad \x \in 
  \partial\Omega_i^{\eps} \,, \quad i=1,\ldots,N \,,\\
  \partial_{n} u &= 0 \,, \quad \x \in \partial \Omega_r \,.
\end{align}
\esub 
This problem is studied in \S \ref{stekN}, with the main results given
in Propositions \ref{sn:main_res} and \ref{sn:degen:main_res} that
distinguish between non-resonant and near-resonant cases.  For a
single circular patch, the leading-order asymptotic behavior of the SN
problem was thoroughly analyzed in \cite{Grebenkov25}.  We will extend
this previous analysis for the Steklov eigenvalues by determining a
three-term asymptotic result that pertains to multiple well-separated,
but arbitrary-shaped, patches.

Finally, in \S \ref{sec:discussion} we discuss applications of the
derived asymptotic results and summarize several open problems.

\section{Preliminaries}\label{mfpt_sec:prelim}

We first derive some preliminary results that are central for our
asymptotic analysis as $\eps\to 0$ of the MFRT, the splitting
probability, and the Steklov eigenvalue problems. Our framework will
use strong localized perturbation theory \cite{Ward93} based on the
method of matched asymptotic expansions. 

\subsection{Geodesic normal coordinates}

To construct the local expansion near each patch on $\partial\Omega$
it is convenient to introduce geodesic normal coordinates
$\bxi=(\xi_1,\xi_2,\xi_3)^T\in \left({-\pi/2},{\pi/2}\right) \times
\left(-\pi,\pi\right)\times [0,1]$ in $\Omega\cup \partial\Omega$ so
that $\bxi=0$ corresponds to $\x_i\in\partial\Omega$, with $\xi_3>0$
corresponding to the interior of $\Omega$.  Here $\xi_2$ can be viewed
as the polar angle of a spherical coordinate system centered at $\x_i$
on the sphere, but defined on the range
$\xi_2\in\left({-\pi/2},{\pi/2}\right)$ that avoids the usual
coordinate singularity of spherical coordinates at the north pole.
The curves obtained by setting $\xi_3=0$ and fixing either $\xi_1=0$
or $\xi_2=0$ are geodesics on $\partial\Omega$ (see
Fig.~\ref{fig:geodesic}).

In terms of the global transformation $\x=\x(\bxi)$ between Cartesian
and geodesic coordinates, in (\ref{app_g:global}) of Appendix
\ref{app_g:geod} we derive an exact expression for the Laplacian of a
generic function ${\mathcal V}(\bxi)\equiv u\left(\x(\bxi)\right)$.
Then, by introducing the inner, or local variables,
$\y=(y_1,y_2,y_3)^T$, defined by
\begin{equation}\label{mfpt:innvar}
  \xi_1=\eps y_1 \,, \qquad \xi_2=\eps y_2 \,, \qquad \xi_3 =\eps y_3\,,
\end{equation}
we derive in Appendix \ref{app_g:geod} that for $\eps\to 0$, and with
$V(\y)={\mathcal V} (\eps \y)$ and
$\Delta_{\y}V \equiv V_{y_1 y_2} + V_{y_2 y_2} + V_{y_3 y_3}$,  we have
\begin{equation}\label{mfpt:local}
  \Delta_{\x} u = \eps^{-2} \Delta_{\y} V + \eps^{-1} \left[2 y_3 \left(
      V_{y_1 y_1} + V_{y_2 y_2}\right) - 2 V_{y_3} \right] +
  {\mathcal O}(1)\,.
\end{equation}
This two-term inner expansion will be central in our asymptotic analysis.

\subsection{Reactive capacitance}

The leading-order term in our local or inner expansion near $\x=\x_i$
relies on the canonical solution $w_i = w_{i}(\y;\kappa_i)$ satisfying
\bsub \label{mfpt:wc}
\begin{align}
    \Delta_{\y} w_{i} &=0 \,, \quad \y \in \R_{+}^{3} \,, \label{mfpt:wc_1}\\
    -\partial_{y_3} w_{i} + \kappa_i w_{i} &=\kappa_i \,, \quad y_3=0 \,,\,
    (y_1,y_2)\in \PT_i\,,  \label{mfpt:wc_2}\\
    \partial_{y_3} w_{i} &=0 \,, \quad y_3=0 \,,\, (y_1,y_2)\notin \PT_i
    \,, \label{mfpt:wc_3}\\
  w_{i}&\sim \frac{C_{i}(\kappa_i)}{|\y|} +
                {  \frac{\DT_i(\kappa_i) {\bf \cdot} \y}{|\y|^3}}
         + \cdots\,,  \quad \mbox{as}\quad
    |\y|\to \infty \,, \label{mfpt:wc_4}
\end{align}
\esub where the neglected term in (\ref{mfpt:wc_4}) is a
quadrupole.  Here 
\begin{equation*}
\R_{+}^{3}\equiv\lbrace{\y=(y_1,y_2,y_3) \, \vert \, \, y_3>0 \,, \, -\infty<y_1,y_2<\infty \rbrace} 
\end{equation*}
is the upper half-space, and $\PT_i \asymp
\eps^{-1}\partial\Omega^{\eps}_i$ is the compact flat Robin
patch on the horizontal plane $y_3=0$, obtained by rescaling and
flattening the small patch $\pa_i^{\eps}$ on the spherical boundary
(e.g., if $\partial\Omega^{\eps}_i$ is a spherical cap of diameter
$2\eps$, then $\Gamma_i$ is the unit disk).  In (\ref{mfpt:wc_4}), the
dipole vector $\DT_i=\DT_i(\kappa_i)$ has the form
$\DT_i=(p_{1i},p_{2i},0)^T$ to ensure that the far-field behavior
(\ref{mfpt:wc_4}) satisfies (\ref{mfpt:wc_3}). When $\PT_i$ is
symmetric in $y_1$ and $y_2$, such as when $\PT_i$ is a disk, we must
have $p_{i1}=p_{i2}=0$ by symmetry, so that the dipole term in the
far-field (\ref{mfpt:wc_4}) vanishes identically.

By using the divergence theorem over a large hemisphere, we readily
obtain 
\begin{equation}\label{mfpt:wc_charge}
  C_i(\kappa_i) = \frac{1}{\pi} \int_{\PT_i} q_{i}(y_1,y_2; \kappa_i) \,
  dy_1 dy_2  \,, \quad \mbox{where} \quad
  q_i(y_1,y_2;\kappa_i)\equiv -\frac{1}{2} \partial_{y_3} w_{i}\vert_{y_3=0}\,.
\end{equation}
In analogy to electrostatics, $C_i(\kappa_i)$ can be interpreted as a
{\em capacitance} of the partially reactive patch $\PT_i$ with
reactivity $\kappa_i$, which we will refer to as the {\em reactive
capacitance}; in turn, we will refer to $q_i$ as the associated {\em
charge density}.  Although there is no explicit analytical solution to
(\ref{mfpt:wc}) for arbitrary $\kappa_i$, in Appendix
\ref{sec:Cmu} we establish a spectral representation
(\ref{eq:wi_spectral}) of $w_i$ in terms of eigenpairs of a suitable
exterior {\em local} Steklov problem (\ref{eq:Psi_def}), from which we
deduce
\begin{equation}  \label{eq:Cmu_def0}  
C_i(\kappa_i) = \frac{\kappa_i}{2\pi} \sum\limits_{k=0}^\infty
\frac{\mu_{ki} d_{ki}^2}{\mu_{ki} + \kappa_i} \,.
\end{equation}
In (\ref{eq:Cmu_def0}), $\mu_{ki} > 0$ are the Steklov eigenvalues
that correspond to nontrivial spectral weights $d_{ki} \ne 0$ defined
in (\ref{eq:dj}).  Both $\mu_{ki}$ and $d_{ki}$ depend on the shape of
the patch $\PT_i$.  Although their numerical computation is required
for a given patch shape (see details in \cite{Grebenkov26}), the
functional form of $C_i(\kappa_i)$ and its dependence on reactivity
$\kappa_i$ is universal.  Moreover, in the important special case
where all the patches are of the same shape (but of variable size),
such as a collection of disks, the rescaling relations
(\ref{eq:rescaling}) imply that
\begin{equation}  \label{eq:Cmu_scaling}
C_i(\kappa_i) = a_i \Cmu(\kappa_i a_i) , \qquad i = 1,\ldots,N,
\end{equation}
where $\Cmu(\mu)$ is the reactive capacitance of the rescaled patches
$\PT_i/a_i$, which needs to be computed only once for a given patch
shape.

For an arbitrary patch shape, we readily calculate from (\ref{eq:Cmu_def0})
that the derivative,
\begin{equation} \label{eq:dCmu} 
  C_{i}^{\prime}(\kappa_i) \equiv \frac{dC_i(\kappa_i)}{d\kappa_i} =
  \frac{1}{2\pi} \sum\limits_{k=0}^\infty
  \frac{\mu_{ki}^2 d_{ki}^2}{(\mu_{ki} + \kappa_i)^2} > 0\,,
\end{equation}
is strictly positive for all $\kappa_i$ (except for the simple
poles $\{-\mu_{ki}\}$), so that $C_{i}(\kappa_i)$ increases
monotonically between consecutive poles,
and on the positive semi-axis $\kappa_i > 0$.  Moreover, in the
small-reactivity limit $\kappa_i\to 0$, one can employ the
Taylor expansion
\begin{subequations}
\begin{equation} \label{eq:Cmu_Taylor}  
C_{i}(\kappa_i) = - a_i \sum\limits_{n=1}^{\infty} c_{ni} \,(-\kappa_i a_i)^n  \,, 
\end{equation}
where the coefficients
\begin{equation}  \label{eq:Cmu_Taylor_cn}
c_{ni} = \frac{1}{2\pi a_i^{n+1}} \sum\limits_{k=0}^\infty
  \frac{d_{ki}^2}{\mu_{ki}^{n-1}}
\end{equation}
\end{subequations}
are defined to be invariant under dilations of the patch.  In
Appendix \ref{sec:Cmu}, we show that
\begin{equation}\label{eq:c1_def} 
  c_{1i} = \frac{|\PT_i|}{2\pi a_i^2}  \,, \qquad  
  c_{2i} = \frac{1}{2\pi a_i^3} \int\limits_{\PT_i} \omega_i(\y) \, d\y \,,
  \qquad
  c_{3i} = \frac{1}{2\pi a_i^4} \int\limits_{\PT_i} \omega_i^2(\y) \,d\y \,, 
\end{equation}
where $\omega_{i}(\y)$ is defined by
\begin{equation}\label{eq:omega_def}
  \omega_i(\y) = \int\limits_{\PT_i} \frac{d\y^{\prime}}{2\pi|\y-\y^{\prime}|} \,,
  \qquad \mbox{for} \quad \y\in\PT_i \,.
\end{equation}
To leading order the expansion (\ref{eq:Cmu_Taylor}) yields
\begin{equation}  \label{eq:Ci_limit0}
  C_i(\kappa_i) \sim \kappa_i \frac{|\PT_i|}{2\pi}  \,, \qquad
  \mbox{as} \quad \kappa_i\to 0 \,.
\end{equation}

In the opposite high-reactivity limit, $C_i(\kappa_i)$ approaches the
(electrostatic) capacitance $C_i(\infty)$ of the patch $\PT_i$.  After
inspecting the spectral representation (\ref{eq:Cmu_def0}), we propose
the following heuristic approximation over the entire range of
reactivities:
\begin{equation}  \label{eq:Cmu_approx}
  C_i(\kappa_i) \approx C_i^{\rm app}(\kappa_i) =
  \frac{\kappa_i C_i(\infty)}{\kappa_i + 2\pi C_i(\infty)/|\PT_i|} \,,
  \qquad \mbox{for} \quad \kappa_i > 0\,.
\end{equation}
This sigmoidal approximation gives the correct limit as $\kappa_i\to
\infty$ and agrees with the leading-order term in the small-reactivity
limit $\kappa_i\to 0$ (it is also close to the lower bound
(\ref{eq:Cmu_bound}) derived in Appendix \ref{sec:Cmu0}).  However,
this approximation fails to recover the higher-order terms for
$\kappa_i\ll 1$, and does not correctly reproduce the asymptotic
approach to $C_i(\infty)$ (see (\ref{mfpt:cj_large_a})). We remark
that a similar sigmoidal formula was developed for approximating the
principal eigenvalue of the Laplace operator \cite{Chaigneau22} and
for studying the boundary local time distribution on small targets
\cite{Grebenkov22}.

\subsection*{Circular patch}

When $\PT_i$ is a disk of radius $a_i$, the limiting problem with
$\kappa_i=\infty$ in (\ref{mfpt:wc}), for which $w_i=1$ on $\PT_i$,
is the classical problem for the capacitance of a flat disk
(cf.~\cite{Jackson}), whose solution, labeled by $w_i(\y;\infty)$, is
(see page 38 of \cite{Fabrikant89}) \bsub \label{mfpt:wcinf}
\begin{equation}
  w_i(\y;\infty) = \frac{2}{\pi}
  \sin^{-1}\left(\frac{a_i}{B(y_3,\rho_0)} \right) \,,
    \label{mfpt:wcinf_1}
\end{equation}
where $\rho_0 \equiv (y_1^2 + y_2^2)^{1/2}$ and 
\begin{equation}
  B(y_3,\rho_0) \equiv  \frac{1}{2} \left(  \left[ (\rho_0 + a_i)^2 + y_3^2
  \right]^{1/2}  +   \left[ (\rho_0 - a_i)^2 + y_3^2 \right]^{1/2} \right) \,.
  \label{mfpt:wcinf_2}
\end{equation}
From this solution, we obtain the far-field behavior
\begin{equation}
    w_i(\y;\infty) \sim C_i(\infty) \left( \frac{1}{|\y|} +
      \frac{a_i^2}{6|\y|^5} 
  \left(y_1^2 + y_2^2 - 2 y_3^2 \right) + \cdots 
\right)\,,  \quad \mbox{as} \quad |\y| \to \infty \,,
\label{mfpt:wcinf_ff}
\end{equation}
\esub 
where $C_i(\infty)={2a_i/\pi}$ is the electrostatic capacitance of the
circular disk of radius $a_i$ (cf.~\cite{Jackson}). Owing to the
symmetry of the disk, (\ref{mfpt:wcinf_ff}) confirms that there is no
dipole term in the far-field.  In addition, from (\ref{mfpt:wcinf_1})
and the radial symmetry, we calculate
\begin{equation}\label{mfpt:wc_q}
 q_i(y_1,y_2;\infty)=q_i(\rho_0;\infty)\equiv
  -\frac{1}{2} \partial_{y_3} w_i(\y;\infty)\vert_{y_3=0}=
  \frac{1}{\pi \sqrt{a_i^2-\rho_0^2}} \,, \quad 0\leq \rho_0\leq a_i \,,
\end{equation}
which is needed in our analysis below. In the large-reactivity limit, one has
\begin{equation}
  C_i(\kappa_i) \to C_i(\infty) = \frac{2a_i}{\pi} \,, \quad \textrm{as}\quad
  \kappa_i\to\infty \,.
\end{equation}
However, owing to the edge singularity in (\ref{mfpt:wc_q}) at
$\rho_0=a_i$, the difference $C_i(\kappa_i)-C_{i}(\infty)$ is not
analytic for $\kappa_i\gg 1$.  This difference has been estimated
analytically from an integral equation formulation in \cite{Guerin23},
and in our notation is given explicitly in (\ref{eq:Cmu_asympt}) of
Appendix \ref{sec:large_mu} (see also Fig.~\ref{fig:Cmu_asympt} and
the order estimate in (\ref{mfpt:cj_large_a})).

\begin{figure}
\begin{center}
\includegraphics[width=88mm]{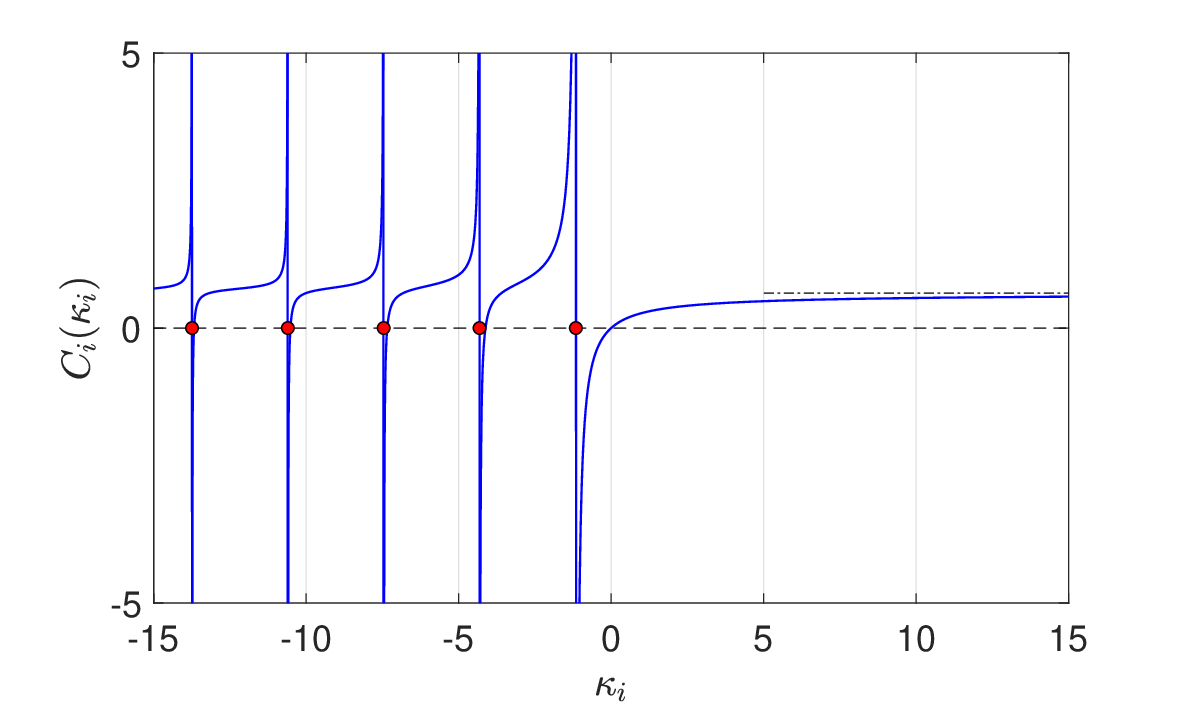}  % {Cmu_disk2.eps}
\end{center}
\caption{ 
The reactive capacitance $C_i(\kappa_i)$ for a circular patch $\PT_i$
of unit radius ($a_i = 1$), as computed from (\ref{eq:Cmu_def0}) by
truncating the series after 100 terms.  Filled circles presents the
poles $\{-\mu_{ki}\}$, all located on the negative axis, at which
$C_i(\kappa_i)$ diverges.  The dash-dotted horizontal line indicates
the asymptotic limit $C_i(\infty)=2/\pi$.}
\label{fig:Cmu}
\end{figure}

For a circular patch, one can use the oblate spheroidal
coordinates to efficiently solve the exterior local Steklov problem as
in \cite{Grebenkov24} to compute $\mu_{ki}$ and $d_{ki}$. Some of
these values are reported in Table \ref{table:muk_disk} of Appendix
\ref{sec:Cmu}, whereas the function $C_i(\kappa_i)$ is shown in
Fig.~\ref{fig:Cmu}.  As expected, this function increases
monotonically from $0$ at $\kappa_i = 0$ to its limit
$C_i(\infty) = 2a_i/\pi$ as $\kappa_i \to \infty$.
The asymptotic behavior of $C_i(\kappa_i)$ for $\kappa_i\ll 1$
is given by (\ref{eq:Cmu_Taylor}), in which the exact values of the
first three coefficients $c_{ni}$ are determined in Appendix
\ref{app_a:Ckappa} (see also Appendix \ref{sec:Cmu}) as
\begin{equation}  \label{eq:cn_exact}
  c_{1i} = \frac12 \,, \quad c_{2i} = \frac{4}{3\pi} \approx 0.4244\,,
  \quad c_{3i} = \frac{4}{\pi^2} \int_{0}^{1} r \left[E(r)\right]^2
\, dr \approx 0.3651\,, 
\end{equation}
where $E(r)$ is the complete elliptic integral of the second kind. In
(\ref{eq:cn_approx}) of Appendix \ref{sec:Cmu} we give a fully
explicit accurate approximation for all coefficients $c_{ni}$ with $n
\geq 2$.

\begin{figure}[htbp]
  \centering
     \begin{subfigure}[b]{0.45\textwidth}  
      \includegraphics[width =\textwidth]{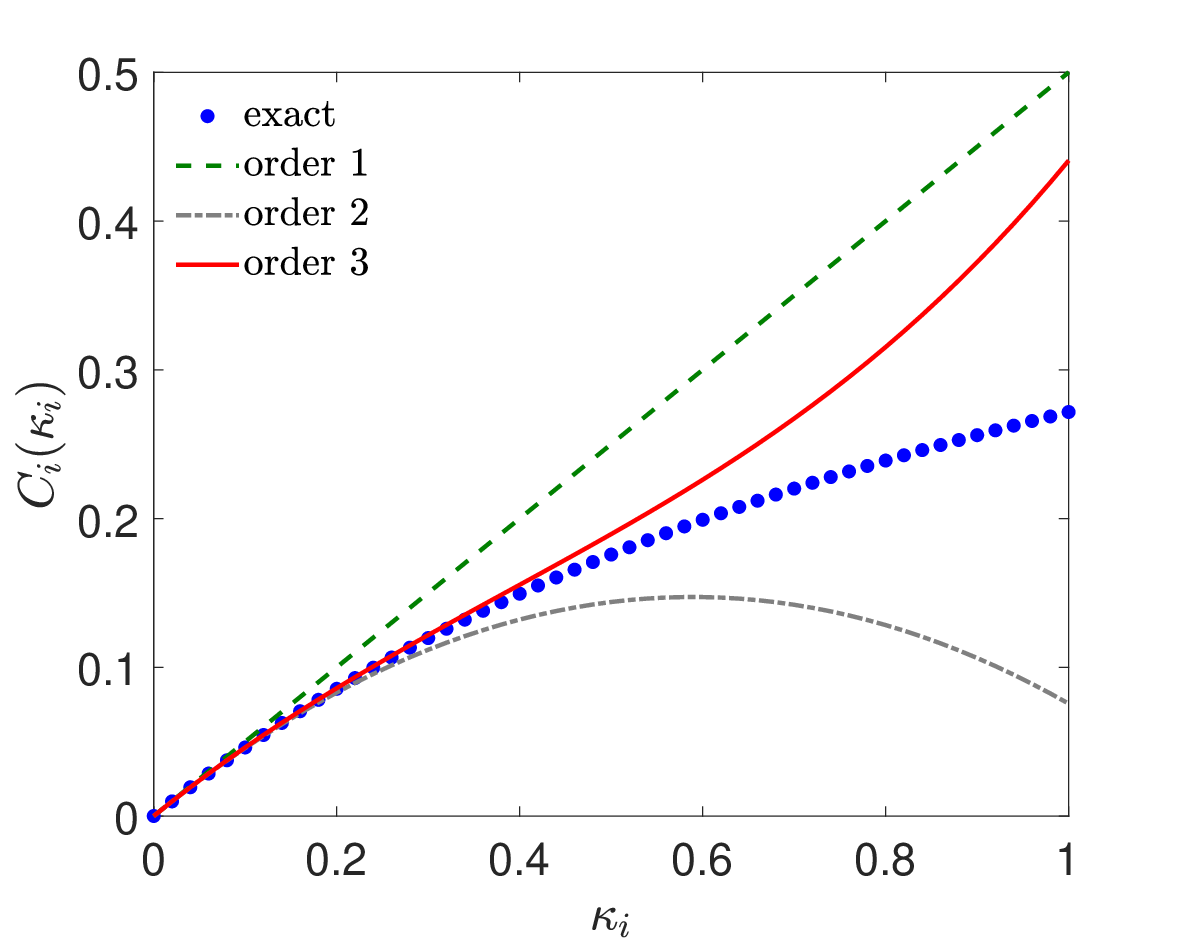} % {Cmu_small2.eps} 
        \caption{Small $\kappa_i$ comparison}
        \label{fig:Cmu_small}
    \end{subfigure}  \hskip 10mm
    \begin{subfigure}[b]{0.45\textwidth}
      \includegraphics[width=\textwidth]{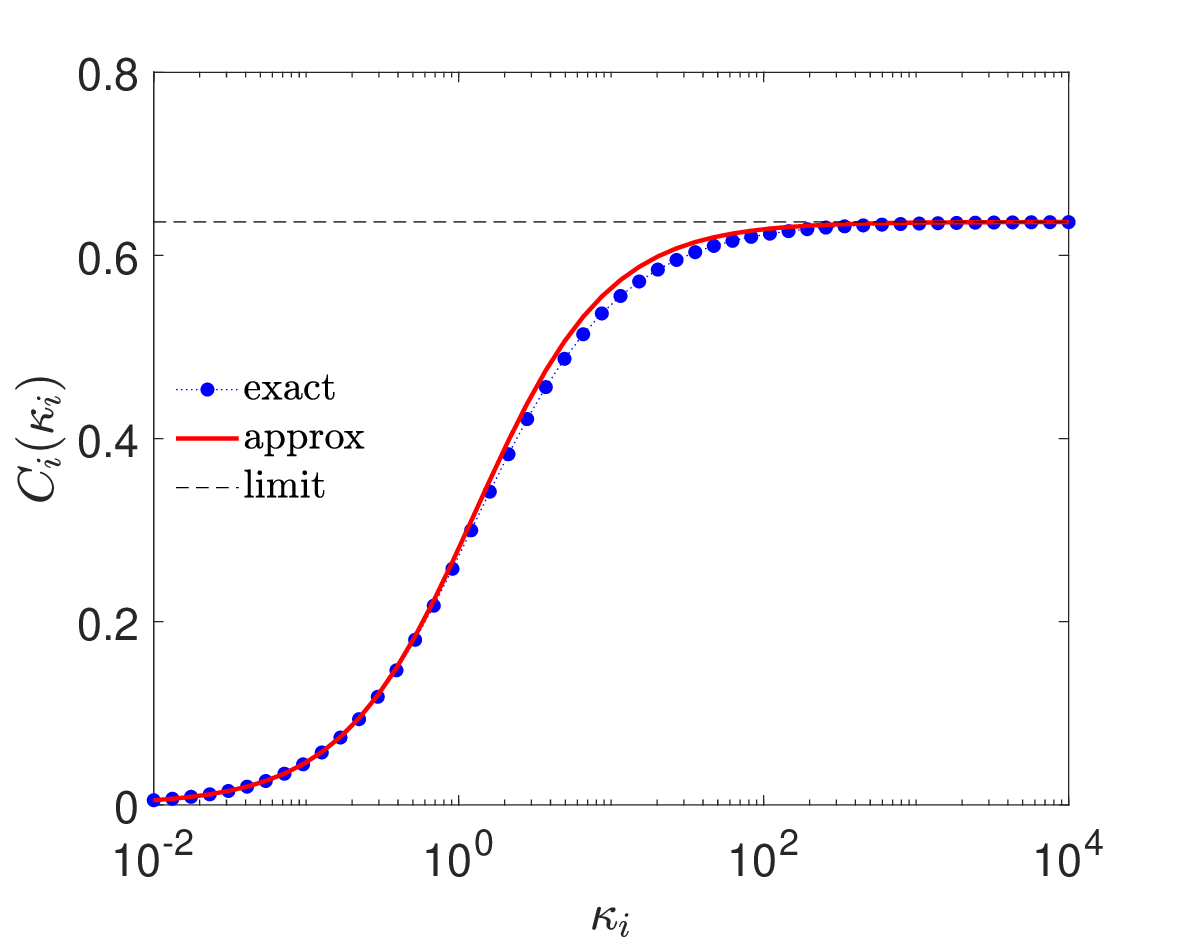}  % {Cmu_approx_new2.eps}
        \caption{Sigmoidal approximation} 
        \label{fig:Cmu_approx_new}
    \end{subfigure}
\caption{ 
The reactive capacitance $C_i(\kappa_i)$ for the circular patch
$\PT_i$ of unit radius ($a_i = 1$).  Filled circles present the exact
spectral expansion (\ref{eq:Cmu_def0}), which was numerically computed
by truncating the series after 100 terms.  {\bf (a):} Comparison with
the one-, two-, and three-term approximations obtained from
(\ref{eq:Cmu_Taylor}) and (\ref{eq:cn_exact}), valid for $\kappa_i \ll
1$.  {\bf (b):} The sigmoidal approximation (\ref{mfpt:sigmoidal_2})
provides a decent approximation of the numerical result for
$C_i(\kappa_i)$ on the full range $\kappa_i >0$. }
\label{fig:capprox}
\end{figure}

Figure~\ref{fig:Cmu_small} shows that a three-term 
small-reactivity series expansion of $C_{i}(\kappa_i)$ in
(\ref{eq:Cmu_Taylor}), with the coefficients from (\ref{eq:cn_exact}),
provides a very close approximation for $C_{i}(\kappa_i)$ on
the range $0<\kappa_i a_i <0.45$.  Finally, the heuristic
formula (\ref{eq:Cmu_approx}), when applied to a disk-shaped patch of
radius $a_i$, reads
\begin{equation}\label{mfpt:sigmoidal_2}
  C_{i}(\kappa_i)\approx C_i^{\rm app}(\kappa_i) =
  a_i {\mathcal C}^{\rm app}(\kappa_i a_i) \,, \quad
  \mbox{where} \quad {\mathcal C}^{\rm app}(\mu)= \frac{2\mu}
  {\pi \mu + 4}   \,.
\end{equation}
We verified numerically that (\ref{mfpt:sigmoidal_2}) provides a good
approximation (see Fig.~\ref{fig:Cmu_approx_new}), with a
maximal relative error of $4\%$, over the entire range of 
$\kappa_i>0$.  We summarize the asymptotic results above as
follows:
\begin{lemma}\label{lemma:Cj_kappa} 
When $\PT_i$ is the disk $y_1^2+y_2^{2}\leq a_i^2$,  its
reactive capacitance is determined by (\ref{eq:Cmu_def0}), as well as
by
\begin{equation}\label{mfpt:Ej_part}
  C_i(\kappa_i) = 2 \int_{0}^{a_i} q_i(\rho_0;\kappa_i) \rho_0 \, d\rho_0 \,, \qquad
  q_i(\rho_0;\kappa_i)=-\frac{1}{2} w_{i, y_3}\vert_{y_3=0} \,,
\end{equation}
where $w_i$ is the solution to (\ref{mfpt:wc}).  It has the
asymptotics
\bsub  \label{mfpt:cj_small}
\begin{align}
      C_i(\kappa_i) &\sim C_i(\infty) + {\mathcal O}\left(
  \frac{\log\kappa_i}{\kappa_i}\right) \,,\quad \mbox{as}\quad
                    \kappa_i\to\infty\,, \quad \mbox{with} \quad
  C_i(\infty)=\frac{2a_i}{\pi} \,, \label{mfpt:cj_large_a}  \\  
  C_i(\kappa_i) &  \sim a_i \biggl[c_{1i} \kappa_i a_i - c_{2i} (\kappa_i a_i)^2 
     + c_{3i} (\kappa_i a_i)^3 + {\mathcal O}((\kappa_ia_i)^4)\biggr] \,,
     \quad \mbox{as}  \quad \kappa_i\to 0 \,,  \label{mfpt:cj_small_b}
\end{align}
\esub 
with the coefficients $c_{ni}$ for $n=1,2,3$ given by
(\ref{eq:cn_exact}).  The error estimate in (\ref{mfpt:cj_large_a})
follows from (\ref{eq:Cmu_asympt}) of Appendix \ref{sec:large_mu}.
\end{lemma}

\subsection{Monopole from a Higher-Order Inner Solution}
\label{prel:high}

Our higher-order asymptotic analysis of each of our four problems
(\ref{mfpt:ssp})--(\ref{sn:eig}) also involves the monopole
coefficient $E_i=E_{i}(\kappa_i)$, which is defined by the solution to
the following inhomogeneous inner problem (see Appendix
\ref{app_h:inn2}):
\bsub\label{mfpt:inn2_probh}
\begin{align}
\Delta_{\y} \Phi_{2hi} &= 0 \,, \quad \y \in \R_{+}^{3} \,,\label{mfpt:inn2_h1}\\
  -\partial_{y_3} \Phi_{2hi} + \kappa_i \Phi_{2hi} &= -\kappa_i {\mathcal F}_{i}
                                                 \,, \quad y_3=0 \,,\,
    (y_1,y_2)\in \PT_i\,,  \label{mfpt:inn2_h2}\\
    \partial_{y_3} \Phi_{2hi} &=0 \,, \quad y_3=0 \,,\, (y_1,y_2)\notin \PT_i
    \,, \label{mfpt:inn2_h3}\\
     \Phi_{2hi} &  \sim \frac{E_i}{|\y|}\,,  \quad
    \mbox{as} \quad |\y|\to \infty \,, \label{mfpt:inn2_h4}
  \end{align}
  \esub where ${\mathcal F}_i={\mathcal F}_{i}(y_1,y_2;\kappa_i)$ is the unique
  solution to
  \bsub\label{mfpt:fprob}
\begin{gather}  
  {\mathcal F}_{i,y_1 y_1} + {\mathcal F}_{i,y_2 y_2}=q_i(y_1,y_2;\kappa_i)
  I_{\PT_i} \,,
    \quad I_{\PT_i} \equiv \left\{\begin{array}{ll}
        1 \,, & (y_1,y_2) \in \PT_i \\
    0 \,, & (y_1,y_2) \notin \PT_i  \end{array}\right. \,
  \label{mfpt:fprob_1}\\
  {\mathcal F}_{i} \sim \frac{C_i}{2}\log\rho_0 + o(1)\,,  \quad
  \mbox{as} \quad \rho_0\equiv (y_1^2+y_2^2)^{1/2}\to \infty
                            \,. \label{mfpt:fprob_2}
\end{gather}
\esub Here $C_i=C_i(\kappa_i)$ while the charge density
$q_i(y_1,y_2;\kappa_i)$ is given in (\ref{mfpt:wc_charge}).

For an arbitrary patch shape, in Appendix \ref{app_h:inn2} we show
that $E_i(\kappa_i)$ is determined by
\begin{equation}  \label{eq:Ei_general0}
  E_i(\kappa_i) = - \frac{1}{2\pi^2} \int\limits_{\PT_i}\int\limits_{\PT_i}
  q_i(\y;\kappa_i)\, q_i(\y^{\prime};\kappa_i)
\, \log|\y-\y^{\prime}|\, d\y \,d\y^{\prime} \,.
\end{equation}
In addition, in the limit $\kappa_i \to 0$, we derive in Appendix
\ref{app_h:inn2} that to leading order
\begin{equation}  \label{eq:Ei_asympt0}
  E_i(\kappa_i) \sim e_i C_i^2(\kappa_i)\,,  \quad \mbox{with} \quad
  e_i \equiv - \frac{1}{2|\PT_i|^2}
  \int\limits_{\PT_i}\int\limits_{\PT_i} \log|\y-\y^{\prime}|\, d\y \,
  d\y^{\prime}\,,
\end{equation}
where $C_i(\kappa_i)\sim{\kappa_i |\PT_i|/(2\pi)}$ for $\kappa_i\ll
1$.

The next result, also proved in Appendix \ref{app_h:inn2},
characterizes $E_i$ when $\PT_i$ is a disk:

\begin{lemma}\label{lemma:Ej_kappa} 
When $\PT_i$ is the disk $y_1^2+y_2^{2}\leq a_i^2$, we have
\begin{equation}\label{mfpt:Ej_all}  
  E_i =  E_i(\kappa_i) = -\frac{\log{a_i}}{2} [C_i(\kappa_i)]^2 + a_i^2 \,
  {\mathcal E}_i(\kappa_i a_i) \,,
\end{equation}
where $C_i(\kappa_i)$ is given by (\ref{eq:Cmu_def0}), and
\begin{equation} 
{\mathcal E}_i(\mu) = 2 \int_{0}^1 \frac{1}{\rho_0}
\left(\int_{0}^{\rho_0}  a_i q_i(\eta a_i;\mu/a_i) \, \eta \,
  d\eta\right)^2 \, d\rho_0 
\end{equation}
corresponds to the unit disk, with $q_i(\rho_0;\kappa_i)$
as given in (\ref{mfpt:Ej_part}). The asymptotic behavior of $E_i(\kappa_i)$ is
\bsub \label{mfpt:Ej_asy}
  \begin{align}
  E_{i} &\sim E_{i}(\infty)\equiv -\frac{2 a_i^2}{\pi^2}\left(
    \log{a_i} + \log{4} - \frac{3}{2} \right)\,,   \quad \mbox{as}\quad
  \kappa_i\to \infty\,, \label{mfpt:Ej_asy_large} \\
  E_i & \sim \frac{\kappa_i^2 a_i^4}{8} \left( \frac{1}{4}-\log{a_i}\right)\,,
  \quad \mbox{as} \quad \kappa_i\to 0 \,.  \label{mfpt:Ej_asy_small}
\end{align}
\esub
\end{lemma}

Although there is no explicit formula for $E_i(\kappa_i)$ for
arbitrary $\kappa_i>0$ even when $\PT_i$ is a disk, in Appendix
\ref{sec:Ei} we show how it can be numerically computed to high
precision by expanding the charge density $q_i$ in terms of Steklov
eigenfunctions.  Moreover, when $\PT_i$ is a disk, we provide the
heuristic approximation (see Appendix \ref{sec:Ei})
\bsub \label{mfpt:E_heur}
\begin{equation}\label{mfpt:E_heur_1}
   E_{i}(\kappa_i)\approx E_{i}^{\rm app}(\kappa_i) =
  -\frac{a_i^2\log{a_i}}{2}
  \left[{\mathcal C}^{\rm app}(\kappa_i a_i)\right]^2 + a_i^2
  {\mathcal E}^{\rm app}(\kappa_i a_i)\,, \qquad \kappa_i > 0\,,
\end{equation}
where
${\mathcal C}^{\rm app}(\mu)$ is given by the sigmoidal approximation
(\ref{mfpt:sigmoidal_2}), and we define (see (\ref{eq:Eapp}) of
Appendix \ref{sec:Ei})
\begin{equation}  \label{mfpt:Eapp}
 {\mathcal E}^{\rm app}(\mu) = \left[{\mathcal C}^{\rm app}(\mu)\right]^2 
\biggl(\frac34 - \log{2} + \frac{1}{\frac{1}{\log{2} - 5/8} + 5.17 \,
  \mu^{0.81}} \biggr)\,.
\end{equation}
\esub In Fig.~\ref{fig:Ekappa} we show that, over the full range
$\kappa_i >0$, (\ref{mfpt:E_heur_1}) agrees remarkably well,
with only a maximal relative error of $0.7\%$, with corresponding
numerical results as computed using the methodology outlined in
Appendix \ref{sec:Ei}. This heuristic approximation also ensures that
the required limits from (\ref{mfpt:Ej_asy}) as $\kappa_i\to 0$
and $\kappa_i \to \infty$ hold.

\begin{figure}
\begin{center}
\includegraphics[width=88mm]{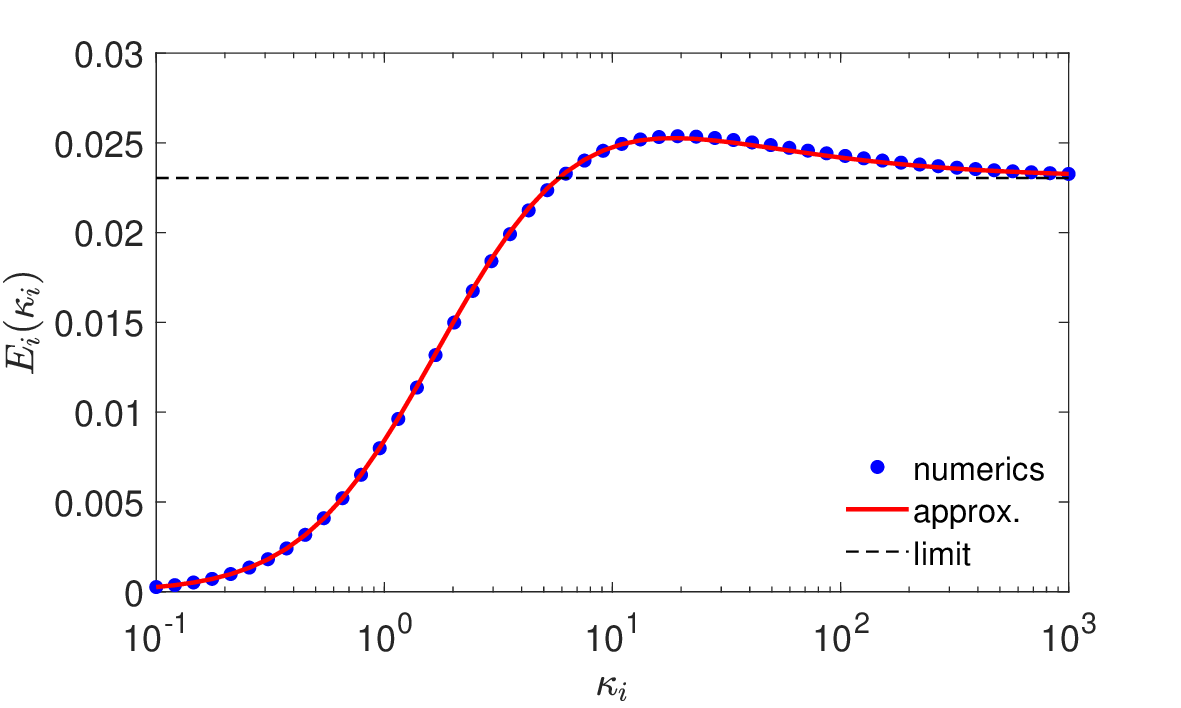} % {E_kappa2.eps}
\end{center}
\caption{
For a circular patch of radius $a_i=1$, the heuristic approximation
${\mathcal E}^{\rm app}(\kappa_i)$ (solid line) from (\ref{mfpt:Eapp})
is compared with $E_i(\kappa_i) = {\mathcal E}_i(\kappa_i)$ (filled
circles) given by (\ref{mfpt:Ej_all}) and computed via the numerical
approach described in Appendix \ref{sec:Ei}, with the series
truncation after 1000 terms.  The dashed horizontal line is the
asymptotic limiting value $(3 - 4\log 2)/\pi^2$ consistent with
(\ref{mfpt:Ej_asy_large}). }
\label{fig:Ekappa}
% [E, kappa] = A_Ward_Steklov_3d_disk_E_fig0(E);
\end{figure}

\subsection{The Surface Neumann Green's Function}\label{prel:gs}

The asymptotic solutions at each order in the outer region, defined at
${\mathcal O}(1)$ distances from the surface patches, is represented
in terms of the surface Neumann Green's function $G_s(\x;\x_i)$, which
is the unique solution to
\begin{equation}
\Delta_{\x} G_s = \frac{1}{|\Omega|} \,, \quad  \x\in \Omega \,; \qquad
\partial_n G_s =  \delta(\x-\x_i)\,, \quad \x \in \partial \Omega \,; \qquad
 \int_{\Omega} G_s \, d\x  = 0 \,, \label{mfpt:sph}
\end{equation}
with $\x_i\in \pa$.  For the unit sphere $\Omega$, the exact solution
to (\ref{mfpt:sph}), as derived in Appendix A of \cite{Cheviakov10},
is
\begin{equation}\label{mfpt:gs_exact}
G_s(\x;\x_i) = \frac{1}{2 \pi \left|\x-\x_i\right|} +
\frac{|\x|^2 + 1}{8\pi} + \frac{1}{4
\pi } \log\left(\frac{2}{1 - \x {\bf \cdot} \x_i +
\left|\x-\x_i\right|}\right) - \frac{7}{10 \pi} \,,
\end{equation}
with $|\x_i|=1$.  The following result characterizes the local
behavior of $G_{s}$ as $\x\to\x_i$ in terms of the local geodesic
coordinates $\y$ defined in (\ref{mfpt:innvar}).

\begin{lemma}\label{lemma:green} 
As $\x\to\x_i$ with $|\x_i|=1$, we have
\begin{equation}\label{mfpt:gs_loc1}
    G_s(\x;\x_i)\sim \frac{1}{2\pi|\x-\x_i|} -\frac{1}{4\pi}
    \log\bigl(|\x-\x_i|+1 -|\x|\bigr) + \frac{\log{2}}{4\pi} -\frac{9}{20\pi}     + o(1) \,.
\end{equation}
In terms of the local geodesic coordinates
$\y=\eps^{-1}{\mathcal Q}_i^{T}(\x-\x_i)$, where ${\mathcal Q}_i$ is the
orthogonal matrix from (\ref{app_g:change}) of Appendix \ref{app_g:geod}, we
have that
\begin{equation}\label{mfpt:gs_locm}
    G_s\sim \frac{1}{2\pi \eps |\y|} - \frac{1}{4\pi}\log\left(\frac{\eps}{2}
    \right) + \frac{y_3(y_1^2+y_2^2)}{4\pi |\y|^3} - \frac{1}{4\pi}
    \log(|\y| + y_3) + R_{s} + o(1)\,,
\end{equation}
where $R_{s}\equiv -9/(20\pi)$ is the regular part.   
\end{lemma}

\begin{proof}
We use the law of cosines with $|\x_i|=1$ to get $2\x {\bf \cdot}\x_i=|\x|^2+1
  -|\x-\x_i|^2$, so that (\ref{mfpt:gs_exact}) becomes
\begin{equation}\label{mfpt:gs_exact_1}
G_s(\x;\x_i) = \frac{1}{2 \pi \left|\x-\x_i\right|} +
\frac{|\x|^2 + 1}{8\pi} + \frac{1}{4
  \pi } \log\left(\frac{4}{\left(|\x-\x_i|+1\right)^2 - |\x|^2}\right)
- \frac{7}{10 \pi} \,. 
\end{equation}
Upon using $a^2-b^2=(a-b)(a+b)$, we let $\x\to\x_i$ with $|\x_i|=1$ to
obtain that (\ref{mfpt:gs_exact_1}) reduces to (\ref{mfpt:gs_loc1}).
Finally, as $\x\to\x_i$ we use $1-|\x|\sim\eps y_3$ in (\ref{mfpt:gs_loc1}),
together with (\ref{app_g:loc}) of Appendix \ref{app_g:geod}, to reduce
(\ref{mfpt:gs_loc1}) to (\ref{mfpt:gs_locm}).
\end{proof}

%%%%%%%%%%%%%%%%%%%%%%%%%%%%%%%%%%%%%%%%%%%%%%%%%%%%%%%%%%%%%%%%%%%%%
\section{The Mean First-Reaction Time}\label{mfpt_sec:expan}

In this section, we investigate the MFRT in the small-patch
limit and derive its three-term expansion, which is valid for
arbitrary reactivities.  We also discuss its asymptotic behavior for
fixed reactivities, as well as the homogenization of the spherical
boundary.  The asymptotic tools developed in this section will be
applied for solving the three other problems in the sections below.

\subsection{Asymptotic Analysis}

We use the method of matched asymptotic expansions to construct
solutions to (\ref{mfpt:ssp}) in the limit $\eps \to 0$. In the outer
region away from the Robin patches we expand the outer solution as
\begin{equation}
    u \sim \eps^{-1} U_0 + U_1 + \eps \log\left( \frac{\eps}{2} \right)
   U_2 + \eps U_3 + \cdots \,, \label{mfpt_b:outex}
\end{equation}
where $U_0$ is a constant to be determined and where $U_k$ for $k\geq 1$
satisfies
\begin{equation}
  \Delta_{\x} U_k = - \delta_{k1} \,, \quad \x \in \Omega \,; 
  \qquad \partial_n U_k = 0 \,, \quad \x\in \partial\Omega\backslash
\lbrace{\x_1,\ldots,\x_N\rbrace} \,.  \label{mfpt_b:Uk}
\end{equation}
Here $\delta_{k1}=1$ if $k=1$ and $\delta_{k1}=0$ for $k>1$. Our
asymptotic analysis below provides singularity behaviors for each
$U_{k}$ as $\x\to \x_i$, for $i=1,\ldots,N$.  The non-analytic term in
$\eps$ in (\ref{mfpt_b:outex}) arises from the subdominant logarithmic
term in the local behavior (\ref{mfpt:gs_locm}) of the surface Neumann
Green's function.

In the inner region near the $i$-th Robin patch we introduce the
local geodesic coordinates (\ref{mfpt:innvar}) and expand each
inner solution as
\begin{equation}
  u  \sim \eps^{-1} V_{0i} + \log\left( \frac{\eps}{2} \right) V_{1i} +
  V_{2i}  + \ldots \,. \label{mfpt_b:innex}
\end{equation}
Upon substituting (\ref{mfpt_b:innex}) into (\ref{mfpt:local}), we obtain
that $V_{ki}$ for $k=0,1,2$ satisfies
\bsub \label{mfpt_b:Vk}
\begin{align}
  \Delta_{\y} V_{ki} &= \delta_{k2} \left( 2y_{3} V_{0i,y_3 y_3} + 2 V_{0i,y_3}
                       \right) \,, \quad
   \y \in \R_{+}^{3} \,, \label{mfpt_b:Vk_1}\\
   -\partial_{y_3} V_{ki} + \kappa_i V_{ki} &=0 \,, \quad y_3=0 \,,\,
    (y_1,y_2)\in \PT_i\,,  \label{mfpt_b:Vk_2}\\
    \partial_{y_3} V_{ki} &=0 \,, \quad y_3=0 \,,\, (y_1,y_2)\notin \PT_i
    \,, \label{mfpt_b:Vk_3}
\end{align}
\esub  
where $\delta_{22}=1$ and $\delta_{k2}=0$ if $k=0,1$. 

Since the leading-order matching condition is $V_{0i}\sim U_0$ as
$|\y|\to\infty$, we have 
\begin{equation}\label{mfpt_b:v0sol}
    V_{0i} = U_0 \left( 1 - w_{i} \right) \,,
\end{equation}
where $w_{i}$ is the solution to (\ref{mfpt:wc}), defined on the
tangent plane to the sphere at $\x=\x_i$, which has the far-field
behavior (\ref{mfpt:wc_4}) in terms of $C_i$ and $\DT_i$. The matching
condition is that the local behavior of the outer expansion
(\ref{mfpt_b:outex}) as $\x\to\x_i$ must agree with the far-field
behavior of the inner expansion (\ref{mfpt_b:innex}), so that
\begin{equation} \label{mfpt_b:mat_1}
  \begin{split}
  \frac{U_0}{\eps} + U_1 + \eps \log\left( \frac{\eps}{2} \right)
  U_2 + &\eps U_3 + \ldots \\
   \sim \frac{U_0}{\eps} & \left( 1 - \frac{C_i}{|\y|} - \frac{\DT_i {\bf \cdot}
       \y}{|\y|^3} \right)  + 
  \log\left( \frac{\eps}{2} \right) V_{1i} + V_{2i} + \ldots \,.
  \end{split}
\end{equation}
Since $|\y|\sim\eps^{-1}|\x-\x_i|$ from (\ref{app_g:change}) of
Appendix \ref{app_g:geod}, it follows that the outer correction $U_1$
must satisfy (\ref{mfpt_b:Uk}) with the singular behavior $U_{1} \sim
- {U_0 C_{i}/|\x-\x_i|}$ as $\x\to \x_i$ for $i=1,\ldots,N$.  In this
way, $U_1$ satisfies
\bsub \label{mfpt_b:U1prob}
\begin{gather}
  \Delta_{\x} U_{1}= -1 \,, \quad \x\in \Omega \,; \qquad
  \partial_n U_1=0 \,, \quad \x\in \partial\Omega\backslash
  \lbrace{\x_1,\ldots,\x_N\rbrace} \,, \\
  U_1\sim -\frac{U_0 C_{i}}{|\x-\x_i|}\,,  \quad \mbox{as} \quad
  \x\to\x_i \in \partial\Omega \,, \quad i=1,\ldots,N \,.
\end{gather}
\esub

From the divergence theorem, the solvability condition for
(\ref{mfpt_b:U1prob}) is that $|\Omega|=2\pi U_0\sum_{i=1}^{N} C_{i}$,
which determines $U_0$ as
\begin{equation} 
  U_0 = \frac{|\Omega|}{2\pi\overline{C}}\,,  \quad \mbox{where} \quad
  \overline{C}\equiv   \sum_{j=1}^{N} C_{j} \,. \label{mfpt_b:U0sol}
\end{equation}
The solution to (\ref{mfpt_b:U1prob}) is represented in terms of the
surface Neumann Green's function of (\ref{mfpt:gs_exact}) as
\begin{equation}
  U_1 = \overline{U}_1 - 2\pi U_0 \sum_{j=1}^{N} C_j G_{s}(\x;\x_j) \,, \quad
  \mbox{where} \quad  \overline{U}_1\equiv |\Omega|^{-1}\int_{\Omega} U_1 \, d\x \,.
  \label{mfpt_b:U1sol}
\end{equation}
We recall that the coefficients $C_i=C_i(\kappa_i)$, defined by
(\ref{mfpt:wc_4}) and (\ref{mfpt:wc_charge}), have the asymptotic
behavior for both small and large $\kappa_i$ given in Lemma
\ref{lemma:Cj_kappa} for circular Robin patches.  As in
\cite{Cheviakov10}, we need to expand the unknown constant
$\overline{U}_1$ in (\ref{mfpt_b:U1sol}) as
\begin{equation}\label{mfpt_b:swit}
  \overline{U}_1 = \overline{U}_{10} \log\left(\frac{\eps}{2}\right) +
  \overline{U}_{11}\,,
\end{equation}
where $\overline{U}_{10}$ and $\overline{U}_{11}$ are constants
independent of $\eps$, which we determine below. We note that the
$\overline{U}_{10}\log\left({\eps/2}\right)$ term in
(\ref{mfpt_b:swit}) is a ``switchback term'' \cite{Lagerstrom88}
and effectively corresponds to inserting a constant term between
${U_0/\eps}$ and $U_1$ in the outer expansion (\ref{mfpt_b:outex}).

To proceed to higher order, we expand $U_1$ as $\x\to \x_i$ by using
the local behavior (\ref{mfpt:gs_locm}) of $G_s$ near the $i$-th
patch.  The matching condition (\ref{mfpt_b:mat_1}) becomes
\begin{multline}
  \frac{U_0}{\eps}\left(1 - \frac{C_i}{|\y|}\right) +
  \left(\frac{U_0 C_i}{2} + \overline{U}_{10}\right)
  \log\left( \frac{\eps}{2} \right) + \frac{U_0 C_i}{2}
 \left( \log(y_3+|\y|) - \frac{y_3 (y_1^2+y_2^2)}{|\y|^3} \right)\\
 + U_0 \beta_{i} + \overline{U}_{11} + \eps \log\left( \frac{\eps}{2} \right) U_2
 + \eps U_3 + \ldots \\
 \sim \frac{U_0}{\eps}
 \left( 1 - \frac{C_i}{|\y|} - \frac{\DT_i {\bf \cdot}
       \y}{|\y|^3}\right) +\log\left( \frac{\eps}{2} \right)
 V_{1i} + V_{2i} + \ldots \,.
 \label{mfpt_b:mat_2}
\end{multline}
Here the constant $\beta_i$ is defined by the $i$-th component of the
matrix-vector product
\begin{equation}\label{mfpt_b:Bi}
  \beta_i =  -2\pi \left({\mathcal G}_s \vc\right)_{i} \,,
\end{equation}
where $\vc\equiv(C_1,\ldots,C_N)^T$ and ${\mathcal G}_{s}$ is the
symmetric Green's matrix defined by
\begin{equation}\label{mfpt_b:green_mat}
    {\mathcal G}_s \equiv \left ( 
\begin{array}{cccc}
 R_s & G_{12} & \cdots & G_{1N} \\
 G_{21} & R_s & \cdots   &G_{2N} \\
 \vdots & \vdots  &\ddots  &\vdots\\ 
 G_{N1} &\cdots & G_{N,N-1} & R_s
\end{array}
\right ) \,, \quad R_s \equiv - \frac{9}{20\pi} \,, \quad G_{ij} \equiv
  G_{s}(\x_i;\x_j) \,,
\end{equation}
where $R_{s}$ is the regular part given in the local expansion
(\ref{mfpt:gs_locm}). The matrix entries $G_{ij}$ can be calculated from
(\ref{mfpt:gs_exact}).

Upon comparing the
${\mathcal O}\left(\log\left(\eps/2\right)\right)$ terms on both
sides of (\ref{mfpt_b:mat_2}) we conclude that we must have
$V_{1i}\sim \overline{U}_{10} + {U_0 C_i/2}$ as $|\y| \to \infty$,
where $V_{1i}$ satisfies the inner problem (\ref{mfpt_b:Vk}) with
$k=1$. This solution is determined in terms of $w_{i}$ of
(\ref{mfpt:wc}) by
\begin{equation}
  V_{1i} = \left( \frac{U_0 C_i}{2} + \overline{U}_{10} \right)
  \left(1 - w_{i}\right) \,. \label{mfpt_b:V1sol}
\end{equation}
The far-field behavior (\ref{mfpt:wc_4}) for $w_{i}$ yields for
$\rho=|\y|\to\infty$  that
\begin{equation}
    V_{1i} \sim \left( \frac{U_0 C_i}{2} + \overline{U}_{10} \right) \left(1 - 
      \frac{C_i}{|\y|} +\cdots  \right)\,,
    \quad \mbox{as} \quad |\y| \to \infty\,.
        \label{mfpt_b:V1ff}
\end{equation}

Next, we substitute (\ref{mfpt_b:V1ff}) into the matching condition
(\ref{mfpt_b:mat_2}). We conclude that the solution $U_2$ to
(\ref{mfpt_b:Uk}) has the singular behavior
$U_2 \sim - \left(\tfrac12 U_0 C_i +\overline{U}_{10} \right)
{C_i/|\x-\x_i|}$ as $\x\to \x_i$. Therefore, $U_2$ satisfies
\bsub \label{mfpt_b:U2prob}
\begin{gather}
  \Delta_{\x} U_{2} = 0 \,, \quad \x\in \Omega \,; \qquad
  \partial_n U_2=0 \,, \quad \x\in \partial\Omega\backslash
  \lbrace{\x_1,\ldots,\x_N\rbrace} \,, \\
  U_2\sim -\left( \frac{U_0 C_i}{2} +\overline{U}_{10} \right) \frac{C_i}
  {|\x-\x_i|}\,, \quad \mbox{as} \quad
  \x\to\x_i \in \partial\Omega \,, \quad i=1,\ldots,N \,.
\end{gather}
\esub By using the divergence theorem, (\ref{mfpt_b:U2prob}) is
solvable only when $\overline{U}_{10}$ satisfies
\begin{equation}
  \frac{\overline{U}_{10}}{U_0} = -\frac{\vc^T\vc}{2 \overline{C}}\,,  \quad
 \mbox{where}\quad  \vc^T\vc = \sum_{j=1}^{N} C_{j}^2 \,.
  \label{mfpt_b:u10}
\end{equation}
Then, the solution to (\ref{mfpt_b:U2prob}) is represented in terms of
an unknown constant $\overline{U}_{2}$ as
\begin{equation}
  U_2 = \overline{U}_2 -2\pi \sum_{j=1}^{N} C_j \left( \frac{U_0 C_j}{2} +
    \overline{U}_{10} \right)  G_{s}(\x;\x_j)  \,.
  \label{mfpt_b:U2}
\end{equation}

Next, we match the ${\mathcal O}(1)$ terms in (\ref{mfpt_b:mat_2}). We
obtain that $V_{2i}$ satisfies (\ref{mfpt_b:Vk}) with $k=2$ together
with the far-field behavior
\begin{equation}
  V_{2i} \sim \beta_i U_0 + \overline{U}_{11} + \frac{U_0 C_i}{2} \left(
    \log(y_3 + |\y|) - \frac{y_3 (y_1^2 + y_2^2)}{|\y|^3} \right)\,, \quad
   \mbox{as} \quad |\y| \to \infty \,. 
\label{mfpt_b:V2inf}
\end{equation}
Since $V_{0i}=U_0(1-w_{i})$ from (\ref{mfpt_b:v0sol}), we decompose
$V_{2i}$ as 
\begin{equation}\label{mfpt_b:V2decom_1}
  V_{2i} = U_0 \left( \Phi_{2i} +
  \left(\beta_{i} + \frac{\overline{U}_{11}}{U_0}\right)(1-w_{i})\right) \,,
\end{equation}
and obtain from (\ref{mfpt_b:Vk}) and (\ref{mfpt_b:V2inf}) that $\Phi_{2i}$
satisfies
\bsub \label{mfpt_b:Phi2}
\begin{align}
  \Delta_{\y} \Phi_{2i} &= - \left( 2y_{3} w_{i,y_3 y_3} + 2 w_{i,y_3}
                       \right) , \quad
   \y \in \R_{+}^{3} \,, \label{mfpt_b:Phi2_1}\\
   -\partial_{y_3} \Phi_{2i} + \kappa_i \Phi_{2i} &=0 \,, \quad y_3=0 \,,\,
    (y_1,y_2)\in \PT_i\,,  \label{mfpt_b:Phi2_2}\\
    \partial_{y_3} \Phi_{2i} &=0 \,, \quad y_3=0 \,,\, (y_1,y_2)\notin \PT_i
                               \,, \label{mfpt_b:Phi2_3}\\
  \Phi_{2i} &\sim  \frac{C_i}{2} \left(
    \log(y_3 + |\y|) - \frac{y_3 (y_1^2 + y_2^2)}{|\y|^3} \right)\,, \quad
   \mbox{as} \quad |\y| \to \infty \,. \label{mfpt_b:Phi2_4} 
\end{align}
\esub

In Appendix \ref{app_h:inn2} we analyze the solution to
(\ref{mfpt_b:Phi2}) and we determine the monopole term
$E_i=E_{i}(\kappa_i)$ in the refined far-field behavior, defined by the
limiting behavior
\begin{equation}\label{mfpt_b:Phi2_ff}
  \Phi_{2i} - \frac{C_i}{2} \left(
    \log(y_3 + |\y|) - \frac{y_3 (y_1^2 + y_2^2)}{|\y|^3} \right) \sim
  \frac{E_i}{|\y|}\,, \quad \mbox{as} \quad |\y| \to \infty \,.
\end{equation}
For an arbitrary patch shape $\PT_i$, $E_i(\kappa_i)$ is given by
(\ref{eq:Ei_general0}).  Some properties of $E_i(\kappa_i)$ were
summarized in \S \ref{prel:high}. In particular, when $\PT_i$ is a
disk, the limiting asymptotics of $E_i(\kappa_i)$ for $\kappa_i\ll 1$
and $\kappa_i\gg 1$ are given in (\ref{mfpt:Ej_asy}) of Lemma
\ref{lemma:Ej_kappa}.  In addition, an accurate heuristic global
approximation for $E_i$ was given in (\ref{mfpt:E_heur}).

Finally, we determine $\overline{U}_{11}$ from a solvability condition
for the problem for the outer correction $U_3$ in
(\ref{mfpt_b:outex}).  To do so, we substitute (\ref{mfpt_b:Phi2_ff})
into (\ref{mfpt_b:V2decom_1}) and use $w_{i}\sim {C_i/|\y|}$ as
$|\y|\to \infty$. We conclude that $V_{2i}$ satisfies the refined
far-field behavior
\begin{equation}\label{mfpt_b:V2ff_refine}
  \begin{split}
    V_{2i} &\sim \beta_i U_0 + \overline{U}_{11} + \frac{U_0 C_i}{2} \left(
      \log(y_3 + |\y|) - \frac{y_3 (y_1^2 + y_2^2)}{|\y|^3} \right) \\
    &\qquad +\frac{U_0 E_i}{|\y|} - \left(\beta_i U_0
        +\overline{U}_{11}\right) \frac{C_i}{|\y|} \,, \quad   \mbox{as}
    \quad |\y| \to \infty \,. 
  \end{split}
\end{equation}
The second line in (\ref{mfpt_b:V2ff_refine}) is the first of two terms
that needs to be accounted for by $U_3$ in the matching condition
(\ref{mfpt_b:mat_2}).  The second term is the dipole term in
(\ref{mfpt_b:mat_2}), which arises from (\ref{mfpt:wc_4}). This term
is written in terms of outer variables using (\ref{app_g:change}) of
Appendix \ref{app_g:geod}.

In this way, we conclude from (\ref{mfpt_b:Uk}), (\ref{mfpt_b:mat_2})
and (\ref{mfpt_b:V2ff_refine}) that $U_3$ must satisfy 
\bsub
\label{mfpt_b:U3prob}
\begin{align}
  \Delta_{\x} U_{3} &= 0 \,, \quad \x\in \Omega \,; \qquad
             \partial_n U_3=0 \,, \quad \x\in \partial\Omega\backslash
                      \lbrace{\x_1,\ldots,\x_N\rbrace} \,, \\
  U_3 &\sim \frac{\left[U_0 E_i -\left(\beta_i U_0 +\overline{U}_{11}\right)C_i
        \right]}{|\x-\x_i|} \nonumber \\
&\qquad -U_0 \frac{\DT_i {\bf \cdot} {\mathcal Q}_i^T (\x-\x_i)}{|\x-\x_i|^3}\,,
    \quad \mbox{as} \quad \x\to\x_i \in \partial\Omega \,, \quad
                      i=1,\ldots,N \,,
\end{align}
\esub 
where the orthogonal matrix ${\mathcal Q}_i$ is defined in
(\ref{app_g:change}) in terms of the basis vectors of the geodesic
coordinate system and $\DT_i$ is the dipole vector in
(\ref{mfpt:wc_4}).  By using the divergence theorem,
(\ref{mfpt_b:U3prob}) has a solution if and only if
$\overline{U}_{11}$ satisfies
\begin{equation}
  \frac{\overline{U}_{11}}{U_0} = \frac{1}{\overline{C}}
  \left(\sum_{j=1}^{N} E_j -
    \sum_{j=1}^{N} \beta_j C_j\right) \,, \label{mfpt_b:u11_1}
\end{equation}
where we find that the contribution from the dipole term vanishes
identically by symmetry since $\DT_i$ has the form
$\DT_i=(p_{1i},p_{2i},0)^T$.  Finally, by using (\ref{mfpt_b:Bi}) for
$\beta_i$, we get
\begin{equation}\label{mfpt_b:u11}
  \frac{\overline{U}_{11}}{U_0} = \frac{2\pi}{\overline{C}} \vc^T
  {\mathcal G}_s \vc + \frac{\overline{E}}{\overline{C}} \,, \quad
  \mbox{where} \quad  \overline{E}=\sum_{j=1}^{N} E_j \,.
\end{equation}

We summarize our main result for the dimensionless MFRT $u(\x)$ and
the volume-averaged MFRT $\overline{u}$ in the small-patch limit in the
following proposition. We also provide the corresponding dimensional
result for (\ref{mfpt:ssp0}), based on the scalings
(\ref{intro:scalings}) and (\ref{intro:mfpt_scale}).

\begin{prop}\label{mfpt_b:main_res} 
As $\eps \to 0$, the asymptotic solution to (\ref{mfpt:ssp}) in the
unit sphere $\Omega$ is given in the outer region $|\x-\x_i|\gg
{\mathcal O}(\eps)$ for $i=1,\ldots,N$ by
\begin{equation}\label{mfpt_b:main_res_1}
  \begin{split}
    u(\x) &\sim \frac{U_0}{\eps} \left[ 1 +
      \eps\log\left(\frac{\eps}{2}\right)
      \frac{\overline{U}_{10}}{U_0} + \eps\left(\frac{\overline{U}_{11}}{U_0} -2\pi \sum_{j=1}^{N}
        C_j G_{s}(\x;\x_j)\right)
    \right.\\
    & \qquad \left. + \eps^2\log\left(\frac{\eps}{2}\right)\left(\frac{\overline{U}_2}{U_{0}} -2\pi
        \sum_{j=1}^{N} C_j \left(\frac{C_j}{2}+\frac{\overline{U}_{10}}{U_0}\right)
        G_{s}(\x;\x_j) \right) +
      {\mathcal O}(\eps^2) \right] \,.
  \end{split}
\end{equation}
The volume-averaged MFRT $\overline{u}$, defined by (\ref{mfpt:ubar}),
satisfies
\begin{equation}\label{mfpt_b:main_res_2}
  \overline{u}\sim \frac{U_0}{\eps} \left[ 1 +
    \eps\log\left(\frac{\eps}{2}\right)
    \frac{\overline{U}_{10}}{U_0} + \eps \frac{\overline{U}_{11}}{U_0} +
    {\mathcal O}\left(\eps^2\log \eps\right)\right] \,.
\end{equation}
In (\ref{mfpt_b:main_res_1}) and (\ref{mfpt_b:main_res_2}), $U_0$,
$\overline{U}_{10}$ and $\overline{U}_{11}$ are determined in terms of
$\vc=(C_1,\ldots,C_N)^T$, $\overline{C}= \sum_{j=1}^{N} C_j$,
$\overline{E}= \sum_{j=1}^{N} E_j$, and the Green's matrix ${\mathcal
G}_s$ from (\ref{mfpt_b:green_mat}) and (\ref{mfpt:gs_exact}) by
\begin{equation}\label{mfpt_b:main_U}
 U_0=\frac{2}{3\overline{C}} \,,\qquad
  \frac{\overline{U}_{10}}{U_0}= -\frac{\vc^T\vc}{2\overline{C}}\,, \qquad
  \frac{\overline{U}_{11}}{U_0} = \frac{2\pi}{\overline{C}} \vc^T {\mathcal G}_s
  \vc + \frac{\overline{E}}{\overline{C}} \,.
\end{equation}
In (\ref{mfpt_b:main_res_1}), $\overline{U}_2$ remains unknown and can
only be found at higher order.
For a single patch ($N=1$), we have
\begin{equation}\label{mfpt_b:main_U_single}
  U_0=\frac{2}{3C_1} \,,\qquad
  \frac{\overline{U}_{10}}{U_0}= -\frac{C_1}{2}\,, \qquad
  \frac{\overline{U}_{11}}{U_0} = - \frac{9}{10}C_1 + \frac{E_1}{C_1} \,.
\end{equation}
In terms of the dimensional variables, we use (\ref{intro:scalings})
and (\ref{intro:mfpt_scale}) to conclude for a sphere of radius $R$ and
for a collection of Robin patches with maximum diameter $L$ that
\begin{equation}\label{mfpt_b:dimen}
  \overline{U} \sim \frac{R^2}{D} \overline{u}\,.
\end{equation}
Here in calculating $\overline{u}$ in (\ref{mfpt_b:main_res_2}) we set
$C_i=C_i\left({L\K_i/D}\right)$ and $E_i=E_i\left({L\K_i/D}\right)$ in
(\ref{mfpt_b:main_U}) and evaluate the Green's matrix ${\mathcal G}_s$
at $\x_i={\X_i/R}$.
\end{prop}

Although in (\ref{mfpt_b:main_res_1}) the coefficient $\overline{U}_2$
can only be determined at higher order, the spatial dependence of the
$\eps^2\log\left({\eps/2}\right)$ correction term is completely
specified up to this unknown constant. In (\ref{mfpt_b:main_res_1})
and (\ref{mfpt_b:main_res_2}), $C_i=C_i(\kappa_i)$ is determined by
the solution to (\ref{mfpt:wc}), while $E_i=E_i(\kappa_i)$ is given by
(\ref{eq:Ei_general0}) for an arbitrary shaped patch and by
(\ref{mfpt:Ej_all}) when the Robin patches are disks. Their asymptotic
behaviors when $\PT_i$ is a disk are given for small and large
reactivities in Lemmas \ref{lemma:Cj_kappa} and \ref{lemma:Ej_kappa}.

We emphasize several features of our main result for the MFRT:

(i) The coefficient $\overline{U}_{11}$ in
(\ref{mfpt_b:main_res_2}) depends on the spatial configuration
$\lbrace{\x_1,\ldots,\x_N\rbrace}$ of the centers of the Robin patches
on the surface of the unit sphere via the Green's matrix ${\mathcal
G}_s$, defined in (\ref{mfpt_b:green_mat}) and (\ref{mfpt:gs_exact}).
Therefore, the effect of the location of the patches is only
revealed at the third order in the asymptotic expansion.

(ii) To numerically calculate the asymptotic result
(\ref{mfpt_b:main_res_2}) for the volume-averaged MFRT, we need only
to numerically compute $C_i(\kappa_i)$ from a PDE solution of
(\ref{mfpt:wc}) and $E_i(\kappa_i)$ from the quadrature in
(\ref{eq:Ei_general0}).  However, when $\PT_i$ is a disk, the use of
the heuristic but accurate approximations (\ref{eq:Cmu_approx}) and
(\ref{mfpt:E_heur}) that closely predict respectively $C_i$ and $E_i$
for all $\kappa_i>0$, provide explicit approximations for the
coefficients in the asymptotic expansion (\ref{mfpt_b:main_U}) of the
volume-averaged MFRT (see \cite{Grebenkov26} for the applicability of
the approximation (\ref{eq:Cmu_approx}) to noncircular patches).

(iii) We observe from Lemmas \ref{lemma:Cj_kappa} and
\ref{lemma:Ej_kappa} that both ${\overline{U}_{10}/U_0}$ and
${\overline{U}_{11}/U_0}$ are ${\mathcal O}(\kappa_i)$ as $\kappa_i\to
0$ for $i=1,\ldots,N$. As a result, the expansion
(\ref{mfpt_b:main_res_2}) in $\eps$ remains uniformly valid in the
limit $\kappa_i\to 0$ for each Robin patch.

(iv) Proposition \ref{mfpt_b:main_res} extends the previous result
from \cite{Cheviakov10} that dealt with the special case of $N$
perfectly reactive ($\kappa_i=\infty$) disk-shape patches.

(v) As discussed in \S \ref{sec:results}, the leading-order term
is applicable to any bounded domain $\Omega$ with a smooth boundary;
in this more general setting, the dimensional MFRT reads
\begin{equation}  \label{eq:MFRT_leading}
\T(\x) \sim  \overline{U} \sim \frac{|\Omega|}{2\pi D C_{\rm tot}} \,,
\qquad C_{\rm tot} = L \bigl[C_1(L\K_i/D) + \cdots + C_N(L\K_i/D)\bigr] ,
\end{equation}
where $L = \max_i \{L_i\}$ is the size of the largest patch, whereas
each reactive capacitance $C_i$ depends implicitly on the patch size
$L_i$ via $a_i = L_i/L$.  This is an extension of the classical
relation that expresses the MFPT in terms of the capacitance $C_{\rm
tot}$ of a perfectly absorbing target.  Here, our structured target is
composed of multiple partially reactive patches, and the effective
capacitance $C_{\rm tot}$ is naturally expressed as the sum of the
reactive capacitances (up to a factor $L$, due to our definition of
$C_i$ to be dimensionless).  One sees that, to leading order, the
patches are fully represented by their reactive capacitances, which
are simply additive.  Moreover, one can substitute a partially
reactive patch of arbitrary shape by a perfectly reactive circular
patch of radius $L_i^{\rm eff} = 2L/(\pi C_i(L\K_i/D))$ such that
(\ref{eq:MFRT_leading}) remains unchanged.  However, this
``equivalence'' breaks down at higher-order terms in
(\ref{mfpt_b:main_res_1}) that account for the spatial arrangement
and the diffusional screening (or diffusive interactions) between
patches. 

(vi) When the reactivities $\K_i$ are fixed, one has $\kappa_i
\to 0$ in the limit $\eps\to 0$; using (\ref{eq:Ci_limit0}), one gets
$L C_i(\kappa_i) \approx \K_i L_i^2/(2\pi D)$ so that
(\ref{eq:MFRT_leading}) is reduced to $\overline{U} \sim
|\Omega|/(\K_1 L_1^2 + \cdots + \K_N L_N^2)$ in the leading order (see
\S \ref{sec:moderateK} for further details).  This relation
indicates that a Robin patch with fixed reactivity becomes
reaction-limited as $\eps\to 0$, i.e., the MFRT scales as $1/L_i^2$,
in sharp contrast to the $1/L_i$ scaling for Dirichlet patches.  This
is another manifestation of the singular nature of the limit $\K_i \to
\infty$ toward the Dirichlet condition.

For the special case of $N$ {\em identical} patches, we have
$\overline{C} = N C_1$, $\overline{E} = N E_1$, $U_0 = {2/(3N C_1)}$,
$\overline{U}_{10}/U_0 = - C_1/2$, $\overline{U}_{11}/U_0 = 2\pi C_1
({\bf e}^T \G {\bf e})/N + E_1/C_1$, with $\evec = (1,\ldots,1)^T$, so
that the asymptotic expansion (\ref{mfpt_b:main_res_2}) applied to the
dimensional volume-averaged MFRT in (\ref{mfpt_b:dimen}) yields
\begin{equation}  \label{eq:MFPT_Nidentical}
  \overline{U} \sim \frac{|\Omega|}{2\pi D N R} \left( \frac{1}{\eps C_1} +
    \frac{1}{2} \log\left({2/\eps}\right) 
    + \frac{2\pi}{N} ({\bf e}^T \G {\bf e}) + \frac{E_1}{C_1^2}
    + {\mathcal O}(\eps \log\eps)\right) \,,
\end{equation}
where $|\Omega|={4\pi R^3/3}$. For instance, for $N$ circular
perfectly reactive patches of radius $\eps R$, we have $C_1={2/\pi}$
and $E_1=(3-4\log 2)/\pi^2$ from (\ref{mfpt:cj_large_a}) and
(\ref{mfpt:Ej_asy_large}), respectively, so that
(\ref{eq:MFPT_Nidentical}) reduces to
\begin{equation}  \label{eq:MFPT_Nidentical_Kinf}
  \overline{U} \sim \frac{|\Omega|}{4D N R} \left(\frac{1}{\eps} +
    \frac{1}{\pi} \log\left({1/\eps}\right) 
    + \frac{4}{N} ({\bf e}^T \G {\bf e}) + \frac{3}{2\pi} - \frac{\log{2}}{\pi}
     + {\mathcal O}(\eps \log\eps)\right) \,.
\end{equation}
This expression corrects and extends, with its inclusion of the
${\mathcal O}(1)$ term, the seminal result by Singer {\it et al.}
\cite{Singer06a} that was obtained for a single circular patch
(note that the factor $1/\pi$ in front of the logarithmic term was
missing in \cite{Singer06a}).

\subsection{Numerical Comparison}\label{mfpt_sec:numerics}

We aim at validating the asymptotic results by comparison with a
numerical solution of the original PDE (\ref{mfpt:ssp}).  Different
numerical methods are available for this purpose, including a
hierarchical, fast multipole method \cite{Kaye20}, spectral expansions
\cite{Grebenkov19b}, finite-element methods \cite{Cheviakov13}, Monte
Carlo methods \cite{Plunkett24,Ye25}. However, most of these methods
are either specifically designed for perfectly reactive patches (i.e.,
Dirichlet boundary condition), or limited to not-too-small partially
reactive patches.  For our validation purposes, we resort to a basic
finite-element method (FEM) implemented in Matlab and focus on the
unit sphere with two identical circular patches of angle $\epsilon$,
centered at the north and south poles.  The symmetry of this setting
allows one to reduce the original 3-D problem to a planar one that can
be solved more readily numerically.   Indeed, in the spherical
coordinates $(r,\theta,\phi)$, the PDE (\ref{mfpt:ssp}) reads
\bsub \label{eq:planar_dec}
\begin{align}
  \frac{1}{r^2} \partial_r (r^2\, \partial_r u) +
  \frac{1}{r^2 \sin\theta} \partial_\theta (\sin \theta\, \partial_\theta u)
  & = - 1 \,, 
\quad (r,\theta) \in (0,1)\times (0,\pi) \,, \\
  \eps\partial_{r} u + \kappa_1 u & = 0\,, \quad r = 1, \,
                                    \theta \in(0,\epsilon) \,, \\ 
  \partial_{r} u & = 0\,, \quad r = 1, \,
                   \theta \in(\epsilon,\pi-\epsilon) \,, \\ 
  \eps\partial_{r} u + \kappa_2 u & = 0\,, \quad r = 1, \,
                                    \theta \in(\pi-\epsilon,\pi) \,, 
\end{align}
\esub 
which is independent of the angle $\phi$ (note that $\eps =
\sin(\epsilon)$, and thus $\eps \approx \epsilon$ for small patches).
Alternatively, one could use cylindrical coordinates $(r,z,\phi)$ to
deal with a simpler representation of the Laplace operator.

Figure~\ref{fig:MFPT} illustrates the volume-averaged MFRT
$\overline{u}$ from (\ref{mfpt_b:main_res_2}).  While the asymptotic
results were derived in the limit $\eps\to 0$, the volume-averaged
MFRT was computed numerically for a broad range of $\epsilon$, up to
${\pi/2}$, when one patch covers a hemisphere, whereas two patches
together cover the whole sphere (in this limiting case, one has
$\overline{u} = 1/15 + 1/(3\kappa)$, that served us for an independent
validation of the FEM solution).  Remarkably, Fig.~\ref{fig:MFPT}
reveals that the asymptotic formula (\ref{mfpt_b:main_res_2}) remains
accurate up to $\eps \simeq 0.8$, i.e., far beyond its expected range
of applicability.  In turn, Fig.~\ref{fig:MFPT_error} shows that the
relative error of the asymptotic formula (\ref{mfpt_b:main_res_2}),
which was predicted to be ${\mathcal O}(\eps^2 \log \eps)$ in a
general setting, is actually ${\mathcal O}(\eps^2)$ for this example.
Minor deviations at very small $\eps$ can be attributed to the
inaccuracy of the FEM solution, despite a rather small meshsize used.

\begin{figure}
  \centering
     \begin{subfigure}[b]{0.49\textwidth}  
      \includegraphics[width =\textwidth]{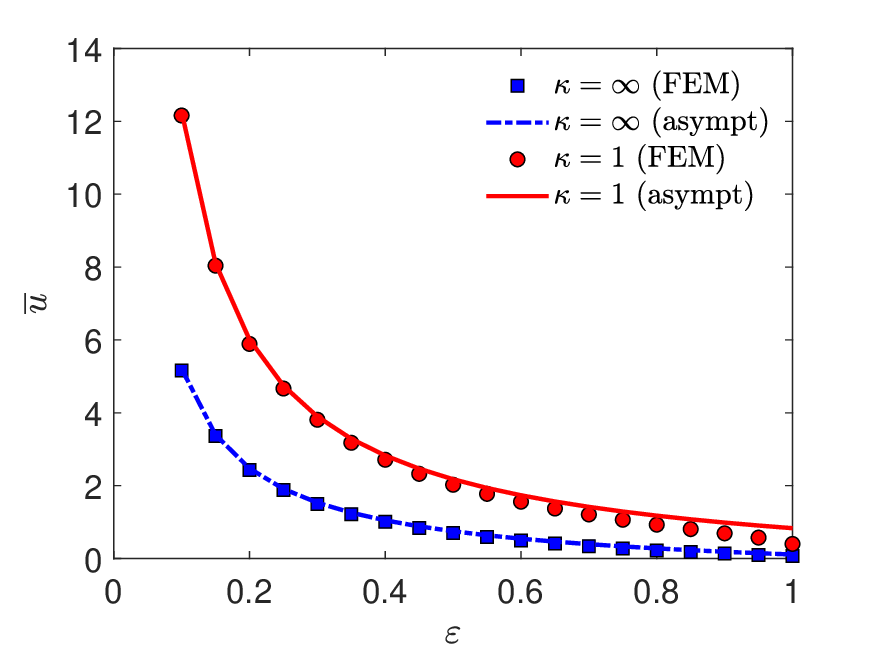} % {MFPT_new1.eps} 
        \caption{Volume-averaged MFRT versus $\eps$}
        \label{fig:MFPT}
    \end{subfigure}  
    \begin{subfigure}[b]{0.49\textwidth}
      \includegraphics[width=\textwidth]{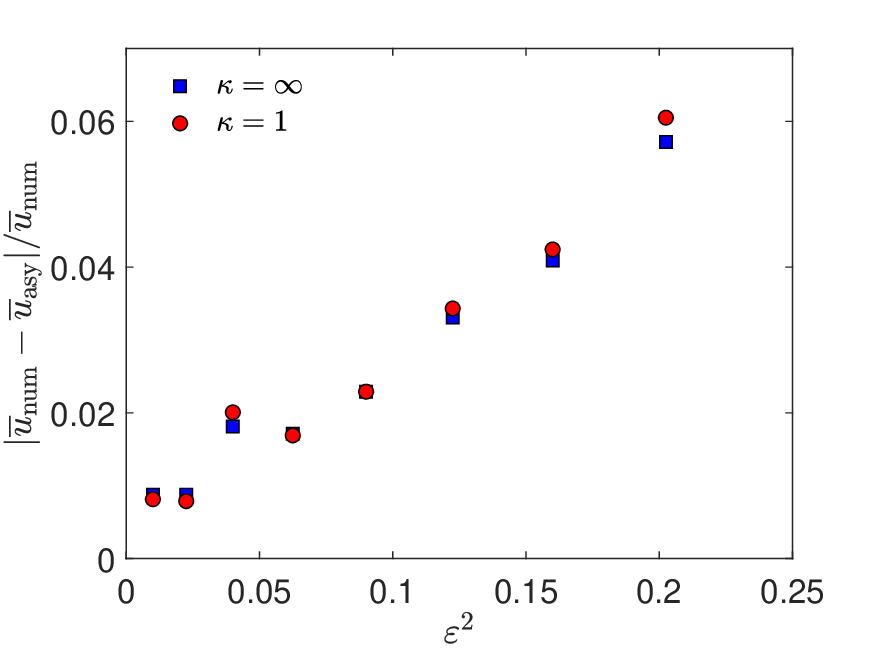}  % {MFPT_error.eps}
        \caption{Relative error versus $\eps^2$} 
        \label{fig:MFPT_error}
    \end{subfigure}
\caption{
{\bf (a)} Dimensionless volume-averaged MFRT $\overline{u}$ to two
identical circular patches of size $\eps$ (with $a_1 = a_2 = 1$),
located at the north and south poles of the unit sphere, with
$\kappa_1 = \kappa_2 = \infty$ (squares) or $\kappa_1 = \kappa_2 = 1$
(circles).  The curves present the asymptotic formula
(\ref{mfpt_b:main_res_2}), whereas symbols indicate the FEM solution
with the maximal meshsize $\hmax = 0.001$.
{\bf (b)} The relative error of the asymptotic formula
(\ref{mfpt_b:main_res_2}), as compared to the FEM solution.  }
% A_Ward_Steklov_3d_MFPT_check3b;
% A_Ward_Steklov_3d_MFPT_check3c;
%%%% A_Ward_Steklov_3d_MFPT_check3(Etau,EFEM);
\end{figure}

\subsection{Expansion for Moderate Reactivities}
\label{sec:moderateK}

The crucial advantage of our asymptotic analysis is that the
expansions (\ref{mfpt_b:main_res_1}) and (\ref{mfpt_b:main_res_2})
remain valid for any reactivities $\kappa_i$, even in the limit
$\kappa_i \to \infty$ of perfectly reactive patches.  In contrast,
former approaches usually fixed {\it finite} reactivities $\K_i$ and
then analyzed the limit $\eps \to 0$ such that $\kappa_i = \eps R
\K_i/D \to 0$ according to (\ref{intro:scalings}).  In this limit, one
can use the asymptotic behaviors of $C_i(\kappa_i)$ and
$E_i(\kappa_i)$ as $\kappa_i\to 0$ to derive more explicit expansions.
Throughout this subsection, we aim to express the volume-averaged MFRT
in terms of finite reactivities $\K_i$ and the small parameter $\eps$.

We first rewrite the Taylor expansion (\ref{eq:Cmu_Taylor}) for the
reactive capacitance in terms of dimensional parameters. Upon using
$a_i=L_i/(\eps R)$ from (\ref{intro:scalings}) together with
(\ref{eq:c1_def}) for $c_{1i}$, we obtain
\begin{equation}
  \begin{split}
C_i & = \frac{1}{2\pi \eps R D} \biggl(|\pa_i| \K_i - 2\pi
c_{2i} \frac{L_i^3 \K_i^2}{D} + 2\pi c_{3i} \frac{L_i^4 \K_i^3}{D^2} +
\ldots\biggr) \,.
\end{split}
\end{equation}
Here the dimensionless coefficients $c_{2i}$ and $c_{3i}$ are given in
(\ref{eq:c1_def}) for arbitrary patch shapes, while their exact
values for circular patches are given in (\ref{eq:cn_exact}).  We
emphasize that they depend only on the shape of the $i$-th patch, not
on its size. As a consequence, we get for $\overline{C}=\sum_{j=1}^{N} C_j$
that
\begin{equation}  \label{eq:averageC_eps}
\overline{C} = \frac{|\pa|}{2\pi \eps R D} 
\left(\overline{\K}^{(1)} - \overline{\K}^{(2)} + \overline{\K}^{(3)}
+ \0(\eps^5)\right) \,,
\end{equation}
where $|\pa|=4\pi R^2$ is the surface area of the sphere and where we
introduced the weighted reactivities defined by
\begin{subequations}
\begin{align}
  \overline{\K}^{(1)} & \equiv \frac{1}{|\pa|} \sum\limits_{i=1}^N |\pa_i| \K_i
                        = \sum\limits_{i=1}^N K_i \equiv \overline{K}\,,
\quad  \mbox{with} \quad K_i \equiv \K_i \frac{|\pa_i|}{|\pa|} \,, \\
  \overline{\K}^{(n)} & \equiv \frac{2\pi}{D^{n-1} |\pa|} \sum\limits_{i=1}^N c_{ni}
                        L_i^{n+1} \K_i^n \,, \quad \mbox{for} \quad n=2,3,
                        \ldots \,.
\end{align}
\end{subequations}
Since $L_i \sim \0(\eps)$ and $|\pa_i| \sim \0(\eps^2)$, we have
$\overline{\K}^{(n)} \sim \0(\eps^{n+1})$ for $n\geq 1$.

Substitution of the expansion (\ref{eq:averageC_eps}) into
(\ref{mfpt_b:main_U}), while using the binomial approximation, yields
\begin{equation}  \label{mfpt:U0_K}
U_0 = \frac{4\pi \eps R D}{3 |\pa| \overline{K}}
\left(1 + \frac{\overline{\K}^{(2)}}{\overline{K}} +
\frac{[\overline{\K}^{(2)}]^2 - \overline{\K}^{(3)} \overline{K}}{\overline{K}^2}
 + \0(\eps^5)\right) \,.
\end{equation}
We also derive from (\ref{mfpt_b:main_U}) that
\begin{equation}    \label{mfpt:U10_U0_K} 
  \frac{\overline{U}_{10}}{U_0}  = - \frac{1}{2\overline{C}}
  \sum\limits_{i=1}^N C_i^2 = - \frac{|\pa|}{4\pi \eps DR \overline{K}}
  \sum\limits_{i=1}^N K_i^2 + \0(\eps^2)  \,.
\end{equation}
To estimate ${\overline{U}_{11}/U_0}$ in (\ref{mfpt_b:main_U}), we first use
$E_i\sim C_i^{2}e_i$ from (\ref{eq:Ei_asympt0}) to obtain to leading order
that
\begin{equation}
  \frac{\overline{E}}{\overline{C}} = \frac{|\pa|}{2\pi \eps RD \overline{K}}
  \sum\limits_{i=1}^N e_i K_i^2 + \0(\eps^2)\,.
\end{equation}
In this way, we obtain in terms of ${\bf K}=(K_1,\ldots,K_N)^{T}$ that
\begin{equation} \label{mfpt:U11_U0_K}
\frac{\overline{U}_{11}}{U_0} = \frac{|\pa|}{2\pi \eps RD \overline{K}} 
\left(2\pi ({\bf K}^T \G {\bf K}) + \sum\limits_{i=1}^N e_i K_i^2 \right)
+ \0(\eps^2) \,.
\end{equation} 

Substituting the expansions (\ref{mfpt:U0_K}), (\ref{mfpt:U10_U0_K}),
and (\ref{mfpt:U11_U0_K}) into (\ref{mfpt_b:main_res_2}) and
(\ref{mfpt_b:dimen}), we finally obtain the following {\em four-term}
expansion for the dimensional volume-averaged MFRT:
\begin{equation}\label{mfpt:U_K}
 \begin{split}
\overline{U} & \sim \frac{|\Omega|}{|\pa|}
   \biggl(\underbrace{\frac{1}{\overline{K}}}_{\propto \eps^{-2}} +
 \underbrace{\frac{\overline{\K}^{(2)}}{\overline{K}^2}}_{\propto \eps^{-1}} 
               + \underbrace{ \log(2/\eps) \frac{|\pa| ({\bf K}^T {\bf K})}
               {4\pi DR \overline{K}^2}}_{\propto \log(\eps)}
               + \underbrace{\frac{[\overline{\K}^{(2)}]^2 - \overline{\K}^{(3)}
               \overline{K}}{\overline{K}^3}}_{\propto 1}  \\  
             & \qquad \qquad + \underbrace{\frac{|\pa|}{2\pi RD \overline{K}^2}
               \biggl( 2\pi ({\bf K}^T \G {\bf K})
        + \sum\limits_{i=1}^N e_i K_i^2 \biggr)}_{\propto 1} + \, o(1)\biggr)\,.
\end{split}
\end{equation}
This is the main result of this subsection.  Several comments are in
order: 

(i) The first (leading-order) term in (\ref{mfpt:U_K}) depends only on
the reactivities and surface areas of the patches. The second and the
third terms, providing the contributions $\0(\eps^{-1})$ and
$\0(\log \eps)$, also depend on the shape of the patches (via the
coefficients $c_{2i}$ and $c_{3i}$). Finally, the last term in
(\ref{mfpt:U_K}) of order $\0(1)$ incorporates the details on the
spatial arrangement of patches via the Green's matrix $\G$.

(ii) For $N$ identical patches of a common reactivity $\K$ and surface
area $|\pa_1|$, (\ref{mfpt:U_K}) reduces, with $\evec=(1,\ldots,1)^T$, to
\begin{equation} \label{mfpt:U_K_identical}
  \begin{split}
\overline{U} & \sim \frac{|\Omega|}{N}
\biggl(\frac{1}{\K |\pa_1|} + 2\pi c_{21}
\frac{L_1^3}{D |\pa_1|^2} + \frac{\log(2/\eps)}{4\pi DR}   
    + \frac{1}{2\pi RD} \biggl(e_1 + \frac{2\pi}{N} ({\bf e}^T \G {\bf e})
           \biggr)   \\
           &  \qquad \qquad + \frac{2\pi \K L_1^6}{D^2 |\pa_1|^3}
           \left(2\pi c_{21}^2 - c_{31} \frac{|\pa_1|}{L_1^2}\right)
   + o(1) \biggr)\,.
  \end{split}
\end{equation}
Interestingly, the second and the third terms do not depend on the
reactivity $\K$.  One might thus naively expect that these terms
remain valid even in the limit $\K \to \infty$.  However, this is not
true, as revealed by comparison of these terms with the first two
terms of the expansion (\ref{eq:MFPT_Nidentical}), which is valid for
any $\K$.  One sees that the logarithmic terms are indeed the same.
However, the coefficients in front of the contribution $\0(\eps^{-1})$
are in general different (e.g., for circular patches, the second term
in (\ref{mfpt:U_K_identical}) is $8/(3\pi^2 DR\eps)$, whereas the
first term in (\ref{eq:MFPT_Nidentical_Kinf}) is $1/(4DR\eps)$).  This
distinction originates from the singular character of the limit $\K
\to \infty$.  More specifically, the expansions (\ref{mfpt:U_K}) and
(\ref{mfpt:U_K_identical}) are based on the asymptotic behavior of
$C_i(\kappa_i)$ as $\kappa_i\to 0$, which is not applicable when
$\kappa_i = \infty$.

(iii) For $N$ identical circular patches of radius $\eps R = L_1$, we
have $c_{2}={4/(3\pi)}$ and $c_3\approx 0.3651$ as given in
(\ref{eq:cn_exact}). Moreover, since $E_1\sim {C_1^{2}/8}$ as
$\kappa_i\to 0$ from (\ref{mfpt:Ej_asy_small}), we identify from
(\ref{eq:Ei_asympt0}) that $e_1={1/8}$. In this way,
(\ref{mfpt:U_K_identical}) yields the explicit result
\begin{equation} \label{eq:MFRT_Ncircular}
  \begin{split}
\overline{U} & \sim \frac{|\Omega|}{N \pi}
\biggl(\frac{1}{\K R^2 \eps^2} + \frac{8}{3\pi R D \eps} 
               + \frac{1}{4DR} \log(2/\eps)  + \biggl(\frac{64}{9\pi^2} -
               2c_3\biggr) \frac{\K}{D^2}  \\ 
               & \qquad \qquad + \frac{1}{2RD} \biggl(\frac18 + \frac{2\pi}{N}
               ({\bf e}^T \G {\bf e}) \biggr) + o(1)\biggr) \,.
\end{split}
\end{equation}
First, we compare this expansion with an approximate expansion that
was derived in \cite{Grebenkov17a} for a single circular patch by
using a constant-flux approximation.  Despite the approximate
character of the expansion from \cite{Grebenkov17a}, its first three
terms turn out to be identical with those in our exact expansion
(\ref{eq:MFRT_Ncircular}).
Secondly, the expansion (\ref{eq:MFRT_Ncircular}) also agrees with the
expansion (5.22) from \cite{Grebenkov25}, which was obtained for a
single circular patch by a different method.

We conclude this subsection with two comments: 

(i) To leading order, the volume-averaged MFRT behaves as
$\overline{U} \sim |\Omega|/(|\pa| \overline{K})$, so that the
trapping effect of small patches is analogous to a parallel connection
of wires, whose resistances are inversely proportional to wire
cross-sectional areas $|\partial\Omega_i|$ and charge career's
mobilities (here $\K_i$) in electrostatics. Diffusion screening
between patches (i.e., their competition for diffusing particles)
appears only at higher orders.

(ii) The volume-averaged MFRT scales as $\eps^{-2}$, i.e., as the
surface area of the patches, which is drastically different from the
case of perfectly reactive patches, for which $\overline{U}$ scales as
$\eps^{-1}$. This crucial distinction between perfectly and partially
reactive patches was emphasized in \cite{Grebenkov17a} (see also
\cite{Grebenkov18a}).

\subsection{Effective Reactivity in the Homogenized Limit}\label{mfpt:homog}

We now derive a scaling law that characterizes the (dimensionless)
effective reactivity $\keff$ corresponding to a large number of
identical circular patches, with a common radius $\eps$ and reactivity
$\kappa$, that are uniformly distributed on the surface on the unit
sphere. A similar analysis has been done for the case of Dirichlet
patches in \cite{Cheviakov13}.

For this homogenized limit we define $\ueff=\ueff(|\x|)$ as the
radially symmetric solution to
\begin{equation}\label{homog:bvp}
  \Delta_{\x} \ueff = - 1 \,, \quad 0\leq |\x| < 1\,;
  \qquad \partial_{n} \ueff + \keff  \ueff=0 \,, \quad |\x|=1 \,.
\end{equation}
The solution to (\ref{homog:bvp}) and the homogenized volume-averaged
MFRT, defined as $\overline{u}_{\rm eff}\equiv |\Omega|^{-1} \int_{\Omega} \ueff
\, d\x$, are
\begin{equation}\label{homog:bvp_sol}
  \ueff(\x) =  \frac{1}{6}\left(1-|\x|^2\right) + \frac{1}{3 \keff}  \qquad
  \mbox{and} \qquad \overline{u}_{\rm eff} =
  \frac{1}{15} + \frac{1}{3\keff} \,.
\end{equation}

Since the reactive patches are all disks with radius $\eps$ and
reactivity $\kappa$, we set $C_i=C(\kappa)$ and $E_i=E(\kappa)$ in
(\ref{mfpt_b:main_res_2}) and (\ref{mfpt_b:main_U}) to obtain that
\begin{equation}\label{homog:uasy_ave}
  \overline{u}\sim \frac{2}{3N \eps C} \left[
    1 +\frac{\eps C}{2}\log\left({2/\eps}\right) + \eps C
    \left( \frac{P}{N} + \frac{E}{C^2} \right) \right] \,,
\end{equation}
where
$P=P(\x_1,\ldots,\x_N)\equiv 2\pi {\bf e}^T {\mathcal G}_s {\bf e}$
with ${\bf e}=(1,\ldots,1)^T$. Upon using the entries of the Green's
matrix ${\mathcal G}_s$ in (\ref{mfpt_b:green_mat}), as can be
calculated from (\ref{mfpt:gs_exact}), we obtain that
\begin{equation}\label{homog:peval}  
    P =-\frac{9N^2}{10} + N(N-1)\log{2} +
    2 {\mathcal H}(\x_1,\ldots,\x_N) \,,
\end{equation}
where the discrete energy ${\mathcal H}$ on the unit sphere is
defined by
\begin{equation}\label{homog:H}
    {\mathcal H}(\x_1,\ldots,\x_N)=\sum_{i=1}^{N}\sum_{j=i+1}^{N}
    \left( \frac{1}{|\x_i-\x_j|} - \frac{1}{2}\log|\x_i-\x_j| -\frac{1}{2}
      \log\left(2+|\x_i-\x_j|\right) \right) \,.
\end{equation}

For the homogenized limit, we require $\overline{u}=\overline{u}_{{\rm
eff}}$.  By using (\ref{homog:bvp_sol}), and in terms of ${\mathcal
H}$, we obtain after a little algebra that (\ref{homog:uasy_ave})
determines $\keff$ as
\begin{equation}\label{homog:uasy_nave}
  \frac{1}{\keff} \sim \frac{2}{N \eps C} \left[
    1 + \frac{\eps C}{2}\log\left({2/\eps}\right) + \eps C
    \left(\frac{E}{C^2} -  N +  (N-1)\log{2} + \frac{2{\mathcal H}}{N}
        \right) \right] \,.
\end{equation}

As derived formally in Appendix of \cite{Cheviakov13} (see also \S
4 of \cite{Cheviakov10}), for a large collection of uniformly
distributed patches with centers at $\x_i$ for $i=1,\ldots,N$, we have
for $N\to \infty$ that
\bsub \label{homog:hasy}
\begin{equation} \label{homog:hasy_1}
   {\mathcal H}\sim \frac{N^2}{2}\left(1-\log{2}\right) + b_1 N^{3/2}
   + b_2 N\log{N} + b_3 N + {\mathcal O}\left(N^{1/2}\right) \,,
 \end{equation}
 with coefficients
 \begin{equation}\label{homog:hasy_2}
   b_1=-\frac{1}{2}\,, \quad b_2=-\frac{1}{8} \,, \quad
   b_3=\frac{1}{2}\left(\log{2}-\frac{1}{4}\right) \,.
 \end{equation}
\esub

To derive a scaling law for $\keff$ we substitute (\ref{homog:hasy})
into (\ref{homog:uasy_nave}) and write the result in terms of the
surface fraction of patches, defined by $f\equiv {N\pi\eps^2/(4\pi)}$.
For the dilute fraction limit $f\ll 1$, we obtain after some algebra
that (\ref{homog:uasy_nave}) reduces to
\begin{equation}\label{homog:keff_final}
   \keff \sim \frac{2f C}{\eps} \left[
     1 + 4 b_1 C \sqrt{f} + \eps C \left(\frac{E}{C^2} - \frac{1}{4}
       -\frac{1}{4} \log f  \right) \right]^{-1} \,, 
\end{equation}
with $b_1 = -1/2$.  We remark that, as shown in \cite{Lindsay17} for a
related problem, the estimate $b_1\approx -0.5523$ should be slightly
better than using $b_1=-{1/2}$ as it accounts for the small defects
from the uniformly distributed patch assumption. Such defects will
always occur when tiling on a sphere. Related homogenization results,
but valid only for $\kappa=\infty$, have been derived in
\cite{Cheviakov13} and in \cite{Lindsay17} for the interior and
exterior problems, respectively.

This main result characterizes $\keff$ in terms of both $C$ and $E$,
which depend on the common local reactivity $\kappa$ of the patches.
In particular, for $\kappa\to \infty$, we set $a_i=1$ in
(\ref{mfpt:cj_large_a}) and (\ref{mfpt:Ej_asy_large}) to obtain $C
\sim {2/\pi}$ and ${E/C^2} \sim (3 - 4\log 2)/(2\pi)$, so that
(\ref{homog:keff_final}) becomes
\begin{equation}\label{homog:keff_big}  
\keff \sim \frac{4f}{\pi \eps} \left[
     1 + \frac{8 b_1}{\pi} \sqrt{f} + \frac{\eps}{\pi}
     \left(1 - \log{4}  -\frac{1}{2} \log{f} \right) \right]^{-1} \,, \quad
   \mbox{for}\quad
   \kappa\gg 1\,.
\end{equation}
Alternatively, for $\kappa\to 0$, (\ref{mfpt:cj_small_b}) and
(\ref{mfpt:Ej_asy_small}) with $a_i=1$ yield that $C\sim {\kappa/2}$ and $E\sim
{\kappa^2/32}$ so that (\ref{homog:keff_final}) becomes
\begin{equation}\label{homog:keff_small}
   \keff \sim \frac{f\kappa}{\eps} \left[
     1 + 2 b_1 \kappa \sqrt{f} - \frac{\eps \kappa}{16} \left(1 + 2\log{f}
     \right) \right]^{-1} \,,
   \quad \mbox{for} \quad \kappa\ll 1\,.
\end{equation}

In order to obtain an explicit approximation for $\keff$ valid for all
$\kappa>0$, one can use in (\ref{homog:keff_final}) the heuristic
approximations for $C$ and $E$ from \S \ref{mfpt_sec:prelim}.  Upon
setting $a_i=1$, we can use $C(\kappa)\approx \Cmu^{\rm app}(\kappa)$
and $E(\kappa) \approx {\mathcal E}^{\rm app}(\kappa)$, where the
explicit functions $\Cmu^{\rm app}(\kappa)$ and ${\mathcal E}^{\rm
app}(\kappa)$ are given by (\ref{mfpt:sigmoidal_2}) and
(\ref{mfpt:E_heur}), respectively.

Finally, we reformulate the homogenized result for the
dimensional problem (\ref{mfpt:ssp0}) in a sphere of radius $R$, 
covered by circular patches of a common radius $L\ll R$ and common
dimensional reactivity $\K_i=\K$ for $i=1,\ldots,N$. Upon recalling
(\ref{intro:scalings}) and $\eps={L/R}$, we use
(\ref{homog:keff_final}) to identify that
\begin{equation}\label{homog:dimen}
    \Keff = \frac{D}{L} \keff\sim 
  \frac{2 D R}{L^2} f C \left[
     1 + 4 b_1 C \sqrt{f} + \eps C \left(\frac{E}{C^2}-\frac{1}{4}
       - \frac14 \log f \right) \right]^{-1} \,, 
\end{equation}
where $C=C\left({L\K/D}\right)$ and $E=E\left({L\K/D}\right)$, and
$\Keff$ has units of ${\mbox{length}/\mbox{time}}$.  This result for
$\Keff$ pertains to the low patch area fraction
$f={N\left({L/R}\right)^2/4}\ll 1$.

\subsection{Laplacian Eigenvalue Problem with Reactive Patches}
\label{sec:mfrt:eig}

In this subsection, we briefly consider the following eigenvalue
problem for the Robin Laplacian in the unit sphere $\Omega$:
\bsub \label{mfpt_eig:ssp}
\begin{align}
  \Delta_{\x} \phi +\lambda \phi & = 0 \,, \quad \x \in \Omega \,; \qquad
   \int_{\Omega} \phi^2 \, d\x =1 \,,  \label{mfpt_eig:ssp_1}\\
  \eps\partial_{n} \phi + \kappa_i \phi & = 0\,, \quad \x \in
     \partial\Omega^{\eps}_i  \,, \quad i=1,\ldots,N \,, \\
  \partial_{n} \phi & = 0 \,, \quad \x \in \partial \Omega_r
  =\partial\Omega\backslash\partial\Omega_{a}\,. \label{mfpt_eig:ssp_2}
\end{align}
\esub 
In \S 3 of \cite{Cheviakov10}, a three-term expansion for the
principal (lowest) eigenvalue $\lambda_0$ of (\ref{mfpt_eig:ssp}) was
derived for a collection of well-separated perfectly reactive
($\kappa_i=\infty$) locally circular patches $\Omega_{i}^{\eps}$. We
now derive the corresponding result for partially reactive and
arbitrary-shaped patches. Instead of repeating an inner-outer
expansion analysis similar to that done for the MFRT and in \S 3 of
\cite{Cheviakov10}, we derive our result primarily from an
eigenfunction expansion solution of (\ref{mfpt:ssp}).

Labeling $\lambda_j=\lambda_j(\eps)$ and $\phi_j=\phi_j(\x;\eps)$ for
$j\geq 0$ to be the eigenpairs of (\ref{mfpt_eig:ssp}), for which
$\lambda_0\to 0$ as $\eps\to 0$ and $\lambda_j={\mathcal O}(1)$ as
$\eps \to 0$, we represent the solution to (\ref{mfpt:ssp}) as
$u=\sum_{j=0}^{\infty}\lambda_j^{-1}\phi_j\left(\int_{\Omega}\phi_{j}\,
d\x\right)$.  By calculating the
average $\overline{u}\equiv |\Omega|^{-1}\int_{\Omega}u\,d\x$, we
conclude that
\begin{equation}\label{sec:eig:ubar}
  \overline{u}=\frac{\left(\int_{\Omega} \phi_0\, d\x\right)^2}
  {|\Omega|\lambda_0} +  \sum_{j=1}^{\infty} \frac{\left(\int_{\Omega}\phi_j\,
      d\x\right)^2}{|\Omega|\lambda_j} \,,
\end{equation}
where a three-term expansion for $\overline{u}$ was given in
(\ref{mfpt_b:main_res_2}) of Proposition \ref{mfpt_b:main_res}.

As in \cite{Cheviakov10}, the principal eigenvalue $\lambda_0$ and
the corresponding eigenfunction $\phi_0$ in the outer region have the expansion
\begin{equation}\label{sec_eig:eigpair}
  \begin{split}
    \lambda_0 &=\eps\lambda_{00}+\eps^2\log\left(\frac{\eps}{2}\right)
    \lambda_{01} + \eps^2\lambda_{02} + {\mathcal O}(\eps^3\log\eps) \,,\\
   \phi_0 &= \phi_{00} + \eps \phi_{01} + \eps^2\log\left(\frac{\eps}{2}\right)
    \phi_{02} + \eps^2\phi_{03} + \ldots \,.
  \end{split}
\end{equation}

By substituting the expansion for $\phi_0$ into the normalization
condition in (\ref{mfpt_eig:ssp_1}), collecting powers of $\eps$, and
ignoring negligible ${\mathcal O}(\eps^3)$ contributions from the
inner regions near the patches, we obtain that
$\phi_{00}={1/|\Omega|^{1/2}}$, $\int_{\Omega}\phi_{01}\, d\x=0$,
$\int_{\Omega}\phi_{02}\, d\x=0$, and $\int_{\Omega}\phi_{03}\,
d\x\neq 0$. As a result, we estimate that
\begin{equation}\label{sec_eig:int_0}
  \left( \int_{\Omega} \phi_0 \, d\x\right)^2 \sim
  \left( |\Omega|^{1/2} + {\mathcal O}(\eps^2)\right)^2 \sim |\Omega| +
  {\mathcal O}(\eps^2) \,.
\end{equation}
Next, by the orthogonality property
$\int_{\Omega}\phi_j\phi_0\, d\x=0$ for $j\geq 1$, we use
(\ref{sec_eig:eigpair}) to estimate
\begin{equation}\label{sec_eig:int_j}
  0 = \int_{\Omega} \phi_j \phi_0\, d\x \sim \frac{1}{|\Omega|^{1/2}} \int_{\Omega}
  \phi_j \, d\x + \eps \int_{\Omega}\phi_{j}\phi_{01}\, d\x \,,
\end{equation}
which yields that
$\left(\int_{\Omega} \phi_j\, d\x\right)^2= {\mathcal O}(\eps^2)$ for
$j\geq 1$. By using this estimate and (\ref{sec_eig:int_0})
in (\ref{sec:eig:ubar}), we conclude that
$\overline{u} \sim \lambda_0^{-1}\left[ 1 + {\mathcal O}(\eps^2)
\right] + {\mathcal O}(\eps^2)$, which yields
\begin{equation}\label{sec_eig:uave_1}
  \lambda_0 \sim \frac{1+{\mathcal O}(\eps^2)}
  {\overline{u}-{\mathcal O}(\eps^2)}\,.
\end{equation}
By using the expansion (\ref{mfpt_b:main_res_2}) for $\overline{u}$
in (\ref{sec_eig:uave_1}) it follows that we can neglect the
${\mathcal O}(\eps^2)$ terms in (\ref{sec_eig:uave_1}), so that
\begin{equation}\label{sec_eig:lam0_1}
\lambda_0 \sim \frac{\eps}{U_0} \left( 1 + \eps \log\left(\frac{\eps}{2}\right)
    \frac{\overline{U}_{10}}{U_0} + \eps \frac{\overline{U}_{11}}{U_0}
     + {\mathcal O}(\eps^2 \log\eps) \right)^{-1} \,.
\end{equation}
Then, by using $(1+y)^{-1}\sim 1-y+{\mathcal O}(y^2)$ for $|y|\ll 1$,
we obtain the expansion
(\ref{sec_eig:eigpair}) for $\lambda_0$ in which
$\lambda_{01}={1/U_0}$, $\lambda_{02}=-{\overline{U}_{10}/U_0^2}$ and
$\lambda_{03}=-{\overline{U}_{11}/U_0^2}$.  Finally, by using
(\ref{mfpt_b:main_U}), we obtain our main result that
\begin{equation}\label{sec_eig:end}
  \lambda_0 \sim \frac{2\pi \overline{C}}{|\Omega|}\eps  +
  \eps^2\log\left(\frac{\eps}{2}\right)\frac{\pi \vc^{T}\vc}{|\Omega|}\\
   -\frac{2\pi \eps^2}{|\Omega|} \left(2\pi\vc^{T}{\mathcal G}_s\vc +
    \overline{E} \right) + {\mathcal O}(\eps^3\log^2\eps)\,,
\end{equation}
where $|\Omega|={4\pi/3}$, $\vc=(C_1,\ldots,C_N)^T$,
$\overline{C}=\sum_{i=1}^{N}C_i$, $\overline{E}=\sum_{i=1}^{N}E_i$,
and ${\mathcal G}_s$ is the Green's matrix, which depends on the patch
locations.  As for the MFRT, the leading-order term of this expansion
is valid for any bounded domain with a smooth boundary.

For the special case of perfectly reactive locally circular patches
for which $C_i={2a_i/\pi}$ and $E_i=E_i(\infty)$ is given in
(\ref{mfpt:Ej_asy_large}), we readily observe that (\ref{sec_eig:end})
agrees with the result in Proposition 3.1 of \cite{Cheviakov10}. For
circular patches, by using our heuristic approximations in
(\ref{mfpt:sigmoidal_2}) and (\ref{mfpt:E_heur}) for $C_i$ and $E_i$,
(\ref{sec_eig:end}) gives an explicit three-term approximation for the
principal eigenvalue $\lambda_0$ over the full range $\kappa_i>0$ of
reactivities.

%%%%%%%%%%%%%%%%%%%%%%%%%%%%%%%%%%%%%%%%%%%%%%%%%%%%%%%%%%%%%%%%%%%%%%
\section{The Splitting Probability}\label{split:intro}

We use the method of matched asymptotic expansions to approximate
solutions to (\ref{split:ssp}) as $\eps\to 0$. Since the MFRT and
splitting probability problems have a similar structure, our asymptotic
analysis for (\ref{split:ssp}) relies heavily on that done in
\S \ref{mfpt_sec:expan}.

\subsection{Asymptotic Analysis}\label{split:expan}

In the outer region, we expand
\begin{equation}
    u \sim U_0 + \eps U_1 + \eps^2 \log\left( \frac{\eps}{2} \right)
   U_2 + \eps^2 U_3 + \cdots \,, \label{split_b:outex}
\end{equation}
where $U_0$ is a constant to be determined and where $U_k$ for $k\geq 1$
satisfies
\begin{equation}
  \Delta_{\x} U_k = 0 \,, \quad \x \in \Omega \,; 
  \qquad \partial_n U_k = 0 \,, \quad \x\in \partial\Omega\backslash
  \lbrace{\x_1,\ldots,\x_N\rbrace} \,.  \label{split_b:Uk}
\end{equation}
The singularity behavior for $U_k$ as $\x\to\x_i$ will be found by
matching.

To construct the inner solution we introduce local geodesic
coordinates near each $\x_i\in\partial\Omega$ to obtain
(\ref{mfpt:local}) with
$-\partial_{y_3}V+\kappa_iV=\delta_{i1}\kappa_i$ on $y_3=0$ and
$(y_1,y_2)\in \PT_i$.  In the inner region near the $i$-th Robin patch
we expand the inner solution as
\begin{equation}
  u  \sim  V_{0i} + \eps \log\left( \frac{\eps}{2} \right) V_{1i} +
  \eps V_{i2}  + \ldots \,. \label{split_b:innex}
\end{equation}
Owing to the different boundary condition on the target
$\partial\Omega_1^{\eps}$, we find that $V_{0i}$ satisfies
\bsub \label{split_b:V0}
\begin{align}
  \Delta_{\y} V_{0i} &= 0 \,, \quad
   \y \in \R_{+}^{3} \,, \label{split_b:V0_1}\\
  -\partial_{y_3} V_{0i} + \kappa_i V_{0i} &= \delta_{i1}\kappa_i \,,
                                             \quad y_3=0 \,,\,
    (y_1,y_2)\in \PT_i\,,  \label{split_b:V0_2}\\
    \partial_{y_3} V_{0i} &=0 \,, \quad y_3=0 \,,\, (y_1,y_2)\notin \PT_i
    \,. \label{split_b:V0_3}
\end{align}
\esub
Moreover, for $k=1,2$, we have that $V_{ki}$ satisfies
\bsub \label{split_b:Vk}
\begin{align}
  \Delta_{\y} V_{ki} &= \delta_{k2} \left( 2y_{3} V_{0i,y_3 y_3} + 2 V_{0i,y_3}
                       \right) \,, \quad
   \y \in \R_{+}^{3} \,, \label{split_b:Vk_1}\\
   -\partial_{y_3} V_{ki} + \kappa_i V_{ki} &=0 \,, \quad y_3=0 \,,\,
    (y_1,y_2)\in \PT_i\,,  \label{split_b:Vk_2}\\
    \partial_{y_3} V_{ki} &=0 \,, \quad y_3=0 \,,\, (y_1,y_2)\notin \PT_i
    \,. \label{mfpt_b:split_3}
\end{align}
\esub  

In terms of $w_{i}=w_{i}(\y;\kappa_i)$, which satisfies
(\ref{mfpt:wc}), the solution to (\ref{split_b:V0}) is
\begin{equation}\label{split_b:V0_sol}
  V_{0i}= U_0 + \left(\delta_{i1} - U_0\right) w_{i} \,, \quad \mbox{for}
  \quad i=1,\ldots,N\,.
\end{equation}
By using the far-field behavior (\ref{mfpt:wc_4}) for $w_i$, we have
\begin{equation*}
  V_{0i}\sim U_0 + \left(\delta_{i1} - U_0\right)\left(
    \frac{C_i}{|\y|} + \frac{\DT_i {\bf \cdot} \y}{|\y|^3} + \cdots
  \right) \,, \quad  \mbox{as} \quad |\y|\to \infty\,.
\end{equation*}
The matching condition as $\x\to \x_i$ and $|\y|\to \infty$ is that
\begin{equation}  \label{split_b:mat_1}
  \begin{split}
  U_0 + \eps U_1 + & \eps^2 \log\left( \frac{\eps}{2} \right)
  U_2 + \eps^2 U_3 + \ldots \\
  & \sim U_0 + \left(\delta_{i1} - U_0\right)\left(
    \frac{C_i}{|\y|} + \frac{\DT_i {\bf \cdot} \y}{|\y|^3} \right) +
   \eps \log\left( \frac{\eps}{2} \right) V_{1i} + \eps V_{2i} + \ldots \,.
\end{split}
\end{equation}
Upon using $|\y|\sim\eps^{-1}|\x-\x_i|$, (\ref{split_b:mat_1}) gives
the singularity behavior for $U_1(\x)$ as $\x\to \x_i$.

At order ${\mathcal O}(\eps)$ in the outer expansion
(\ref{split_b:outex}) we obtain that $U_1(\x)$ satisfies
\bsub\label{split_b:U1prob}
\begin{gather}
  \Delta_{\x} U_{1}= 0 \,, \quad \x\in \Omega \,; \qquad
  \partial_n U_1=0 \,, \quad \x\in \partial\Omega\backslash
  \lbrace{\x_1,\ldots,\x_N\rbrace} \,, \\
  U_1\sim -\frac{\left(U_{0}-\delta_{i1}\right)C_i}{|\x-\x_i|}\,,
  \quad \mbox{as} \quad
  \x\to\x_i \in \partial\Omega \,, \quad i=1,\ldots,N \,.
\end{gather}
\esub
The solvability condition for (\ref{split_b:U1prob}) determines $U_0$ as
\begin{equation}\label{split_b_b:U0sol}
  U_0 = \frac{C_1}{\overline{C}}\,,  \quad \mbox{where} \quad
  \overline{C}\equiv \sum_{j=1}^{N} C_{j} \,. 
\end{equation}
In terms of the surface Neumann Green's function $G_s$, given in
(\ref{mfpt:gs_exact}), the solution to (\ref{split_b:U1prob}) is
\begin{equation}
  U_1 = \overline{U}_1 -2\pi \sum_{j=1}^{N} \left(U_{0}-\delta_{j1}\right)
  C_j G_{s}(\x;\x_j) \,, \quad
  \mbox{where} \quad  \overline{U}_1\equiv |\Omega|^{-1}\int_{\Omega} U_1 
  \, d\x \,.  \label{split_b:U1sol}
\end{equation}
As similar to that for the MFRT problem in \S \ref{mfpt_sec:expan}, we
expand $\overline{U}_1$ as
\begin{equation}\label{split_b:U1cons}
  \overline{U}_1 = \overline{U}_{10} \log\left(\frac{\eps}{2}\right) +
  \overline{U}_{11}\,,
\end{equation}
where $\overline{U}_{10}$ and $\overline{U}_{11}$ are constants
independent of $\eps$, which are to be determined.

Next, we expand $U_1$ in (\ref{split_b:U1sol}) as $\x\to\x_i$ to obtain,
after some algebra, that
\begin{equation}\label{split_b:U1_expan}
  \begin{split}
    U_{1} &\sim -\left(U_{0}-\delta_{i1}\right) \frac{C_i}{\eps |\y|} +
    \left(\overline{U}_{10} + \frac{\left(U_0-\delta_{i1}\right)C_i}{2} \right)
    \log\left(\frac{\eps}{2}\right)  + \gamma_i + \overline{U}_{11}\\
    & \qquad + \frac{\left(U_0-\delta_{i1}\right)C_i}{2}
    \left( \log(y_3+|\y|) - \frac{y_3(y_1^2+y_2^2)}{|\y|^3}\right)  \,,
  \end{split}
\end{equation}
where $\gamma_i$ is the $i$-th component of the vector ${\bm \gamma}$
defined by
\begin{equation}\label{split_b:gamma}  
  {\bm \gamma} \equiv 2\pi C_1 \left(R_s,G_{12},\ldots,G_{1N}\right)^{T} - 2\pi U_0 {\mathcal G}_s \vc \,. 
\end{equation}
Here $R_s=-{9/(20 \pi)}$, $G_{1j}=G_{s}(\x_1;\x_j)$, and ${\mathcal G}_s$
is the Green's matrix of (\ref{mfpt_b:green_mat}).

Upon substituting (\ref{split_b:U1_expan}) into the matching condition
(\ref{split_b:mat_1}), we conclude that the dominant
${\mathcal O}\left(\eps\log\left({\eps/2}\right)\right)$ terms
determine the far-field behavior for the inner correction $V_{1i}$ in
(\ref{split_b:innex}). In particular, we find that $V_{1i}$ satisfies
(\ref{split_b:Vk}) with $k=1$ subject to
$ V_{1i} \sim \overline{U}_{10} + {\left( U_0 -
    \delta_{i1}\right)C_i/2}$ as $|\y|\to \infty$. In terms of $w_{i}$,
satisfying (\ref{mfpt:wc}), we obtain that $V_{1i}$ is given by
\begin{equation}\label{split_b:V1sol}
  V_{1i}= \left(\overline{U}_{10} + \frac{\left(U_0 -\delta_{i1}\right)C_i}{2}
  \right)\left(1-w_{i}\right)\,.
\end{equation}

We substitute (\ref{split_b:V1sol}) into the matching condition
(\ref{split_b:mat_1}), where we use $w_{i}\sim{ C_i/|\y|}$ as
$|\y|\to \infty$ from (\ref{mfpt:wc_4}). This provides the required
singularity behavior of the outer correction $U_2$ in
(\ref{split_b:outex}). In this way, we find that $U_2$ satisfies
\bsub \label{split_b:U2prob}
\begin{gather}
  \Delta_{\x} U_{2} = 0 \,, \quad \x\in \Omega \,; \qquad \partial_n
  U_2=0 \,, \quad \x\in \partial\Omega\backslash
  \lbrace{\x_1,\ldots,\x_N\rbrace} \,, \\
  U_2\sim -\left( \overline{U}_{10} + \frac{\left(
      U_0-\delta_{i1}\right)C_i}{2} \right) \frac{C_i} {|\x-\x_i|}\,, \quad
  \mbox{as} \quad \x\to\x_i \in \partial\Omega \,, \quad i=1,\ldots,N
  \,.
\end{gather}
\esub 
From the divergence theorem we find that (\ref{split_b:U2prob}) is
solvable only when $\overline{U}_{10}$ satisfies
$2\overline{U}_{10}\sum_{j=1}^{N} C_j=-\sum_{j=1}^{N}
C_j^2\left(U_0-\delta_{j1} \right)$. Upon using
(\ref{split_b_b:U0sol}) for $U_0$, we obtain in terms of
$\vc=(C_1,\ldots,C_N)^T$ and $\overline{C}= \sum_{j=1}^{N}C_j$
that
\begin{equation}\label{split_b:U10}
  \frac{\overline{U}_{10}}{U_0} = -\frac{1}{2\overline{C}} \left(
    \vc^T\vc - \overline{C} C_1 \right)\,.
\end{equation}
Then, the solution to (\ref{split_b:U2prob}) is written in terms of an
unknown constant $\overline{U}_{2}$ as
\begin{equation}\label{split_b:U2}
  U_2 = \overline{U}_2 -2\pi \sum_{j=1}^{N} \left(\overline{U}_{10} +
    \frac{\left( U_0 -\delta_{j1}\right)C_j}{2} \right) C_j G_{s}(\x;\x_j)
    \,.
\end{equation}

Next, we continue our expansion to higher order by substituting
(\ref{split_b:U1_expan}) into (\ref{split_b:mat_1}). We conclude
that $V_{2i}$ satisfies (\ref{split_b:Vk}) with $k=2$ subject to the
far-field behavior
\begin{equation}
  V_{2i} \sim \frac{\left(U_0-\delta_{i1}\right) C_i}{2} \left(
    \log(y_3 + |\y|) - \frac{y_3 (y_1^2 + y_2^2)}{|\y|^3} \right) +
    \gamma_i + \overline{U}_{11} \,, \quad \mbox{as} \quad |\y| \to \infty \,. 
\label{split_b:V2inf}
\end{equation}
Since $V_{0i}= U_0 + \left(\delta_{i1} - U_0\right) w_{i}$ from
(\ref{split_b:V0_sol}), we decompose $V_{2i}$ as 
\begin{equation}\label{split_b:V2decom_1}
  V_{2i} = \left(U_0-\delta_{i1}\right) \left[ \Phi_{2i} +
  \frac{\left(\gamma_i+\overline{U}_{11}\right)}{U_{0}-\delta_{i1}}
   (1-w_{i})\right] \,.
\end{equation}
We conclude from (\ref{split_b:Vk}) with $k=2$ and
(\ref{split_b:V2inf}) that $\Phi_{2i}$ satisfies (\ref{mfpt_b:Phi2})
of \S \ref{mfpt_sec:expan}.  As a result, the refined far-field
behavior of $\Phi_{2i}$ is given by (\ref{mfpt_b:Phi2_ff}) in terms of
$E_i=E_i(\kappa_i)$, which is determined by the far-field behavior of
(\ref{mfpt:inn2_probh}) (see also Appendix \ref{app_h:inn2}).

To determine $\overline{U}_{11}$ we will impose a solvability
condition on the problem for the outer correction $U_3$ in
(\ref{split_b:outex}). To obtain the singularity behavior for $U_3$
as $\x\to \x_i$, we substitute (\ref{mfpt_b:Phi2_ff}) and $w_{i}\sim
{C_i/|\y|}$ as $|\y|\to \infty$ into (\ref{split_b:V2decom_1}) to
conclude that $V_{2i}$ satisfies the refined far-field behavior
\begin{equation}\label{split_b:V2ff_refine}
  \begin{split}
    V_{2i} &\sim \gamma_i + \overline{U}_{11} +
    \frac{\left(U_0-\delta_{i1}\right) C_i}{2} \left(
      \log(y_3 + |\y|) - \frac{y_3 (y_1^2 + y_2^2)}{|\y|^3} \right) \\
    &\qquad + \frac{E_i\left(U_0-\delta_{i1}\right)}{|\y|} -
    \frac{\left(\gamma_i+\overline{U}_{11}\right)C_{i}}{|\y|} \,,
    \quad \mbox{as} \quad |\y| \to \infty \,.
  \end{split}
\end{equation}
The second line in (\ref{split_b:V2ff_refine}) is one of the two terms
that provides the singularity behavior for $U_3$ in the matching
condition (\ref{split_b:mat_1}). The other term is the dipole term
given in (\ref{split_b:mat_1}), which is written in terms of outer
variables using (\ref{app_g:change}) of Appendix \ref{app_g:geod}.

In this way, upon substituting (\ref{split_b:V2ff_refine}) into the
matching condition (\ref{split_b:mat_1}), we conclude that $U_3$ in the
outer expansion (\ref{split_b:outex}) must satisfy
\bsub \label{split_b:U3prob}
\begin{align}
  \Delta_{\x} U_{3} &= 0 \,, \quad \x\in \Omega \,; \qquad
                      \partial_n U_3=0 \,, \quad \x\in \partial\Omega\backslash
                      \lbrace{\x_1,\ldots,\x_N\rbrace} \,, \\
  U_3 &\sim \frac{E_i \left(U_0-\delta_{i1}\right)}{|\x-\x_i|} 
        - \frac{\left(\gamma_i + \overline{U}_{11}\right)C_i}{|\x-\x_i|}
        \nonumber\\
        & \quad + \left(\delta_{i1}-U_0\right)
    \frac{\DT_i {\bf \cdot} {\mathcal Q}_i^T (\x-\x_i)}{|\x-\x_i|^3}\,,
         \quad \mbox{as} \quad \x\to\x_i \in \partial\Omega \,, \quad
                i=1,\ldots,N \,.
\end{align}
\esub 
From the divergence theorem, (\ref{split_b:U3prob}) has a solution if
and only if $\sum_{j=1}^{N} E_j\left(U_0-\delta_{j1}\right) =
\sum_{j=1}^{N} \left(\gamma_j+\overline{U}_{11}\right)C_j$, where we
observe that the contribution from the dipole term again vanishes
identically.  By using (\ref{split_b:gamma}) for $\gamma_i$, we solve
for $\overline{U}_{11}$ as
\begin{equation*}
  \overline{U}_{11} \overline{C} = U_0\sum_{j=1}^{N} E_j - E_1 +
  2\pi U_0 \vc^{T} {\mathcal G}_s \vc - 2\pi C_1 \left({\mathcal G}_s \vc
  \right)_{1}\,,
\end{equation*}
which yields
\begin{equation}\label{split_b:u11_1}
  \overline{U}_{11} = \frac{U_0\sum_{j=1}^{N} E_j}{\overline{C}} -
   \frac{E_1}{\overline{C}}
  + \frac{2\pi C_1}{\overline{C}} \left(
    \frac{\vc^T {\mathcal G}_s \vc}{\overline{C}} - \left(
      {\mathcal G}_s \vc \right)_{1} \right) \,.
\end{equation}

We summarize our main result for the splitting probability $u(\x)$ and
the volume-averaged splitting probability $\overline{u} =
|\Omega|^{-1}\int_{\Omega} u(\x) \, d\x$ in the small-patch
limit as follows:

\begin{prop}\label{splitt_b:main_res} 
As $\eps \to 0$, the asymptotic solution to (\ref{split:ssp}) is given
in the outer region $|\x-\x_i|\gg {\mathcal O}(\eps)$ for
$i=1,\ldots,N$ by
\begin{equation}\label{split_b:main_res_1}
  \begin{split}
  u(\x) &\sim U_{0}\left[ 1 + \eps\log\left(\frac{\eps}{2}\right)
    \frac{\overline{U}_{10}}{U_0} + \eps\left(\frac{\overline{U}_{11}}{U_0} + 
      2\pi \sum_{j=1}^{N} \frac{\left(\delta_{j1}-U_0\right)}{U_0}
      C_jG_{s}(\x;\x_j) \right)
    \right.\\
    & \left. + \eps^2\log\left(\frac{\eps}{2}\right)\left(\frac{\overline{U}_2}{U_{0}} + 
	2\pi \sum_{j=1}^{N} C_j \left(\frac{C_j\left(\delta_{j1}-U_0\right)}{2U_0}
          -\frac{\overline{U}_{10}}{U_0}\right)
      G_{s}(\x;\x_j) \right) + {\mathcal O}(\eps^2)
  \right].
  \end{split}
\end{equation}
The volume-averaged splitting probability $\overline{u}$ satisfies
\begin{equation}\label{split_b:main_res_2}
  \overline{u}\sim U_0 \left[ 1 +
    \eps\log\left(\frac{\eps}{2}\right)
    \frac{\overline{U}_{10}}{U_0} + \eps \frac{\overline{U}_{11}}{U_0} +
    {\mathcal O}\left(\eps^2\log \eps\right)\right] \,.
\end{equation}
In (\ref{split_b:main_res_1}) and (\ref{split_b:main_res_2}), $U_0$,
$\overline{U}_{10}$ and $\overline{U}_{11}$ are determined in terms of
$\vc=(C_1,\ldots,C_N)^T$, $\overline{C}=\sum_{j=1}^{N} C_j$,
$\overline{E} = \sum_{j=1}^{N} E_j$, and the Green's matrix
${\mathcal G}_s$ from (\ref{mfpt_b:green_mat}) by
\begin{equation}\label{split_b:main_U}
  \begin{split}
  U_0 &=\frac{C_1}{\overline{C}} \,,\qquad
  \frac{\overline{U}_{10}}{U_0}= -\frac{1}{2\overline{C}} \left(
    \vc^T\vc - C_1 \overline{C}\right) \,, \\
  \frac{\overline{U}_{11}}{U_0} &= \left( \frac{\overline{E}}{\overline{C}}-
    \frac{E_1}{C_1}\right) + 2\pi \left(\frac{\vc^T {\mathcal G}_s \vc}
    {\overline{C}} - \left({\mathcal G}_s \vc \right)_{1} \right) \,.
  \end{split}
\end{equation}
In (\ref{split_b:main_res_1}), $\overline{U}_2$ remains unknown and
can only be found at higher order, $C_i=C_i(\kappa_i)$ is obtained
from (\ref{mfpt:wc}), while $E_i=E_i(\kappa_i)$ is given by
(\ref{eq:Ei_general0}) for arbitrary-shaped patches and by
(\ref{mfpt:Ej_all}) when the patches are disks. Their asymptotic
behaviors for small and large reactivities are given in Lemmas
\ref{lemma:Cj_kappa} and \ref{lemma:Ej_kappa}.  When the patches are
disks, heuristic approximations for $C_i$ and $E_i$ valid for all
$\kappa_i>0$ are given in (\ref{mfpt:sigmoidal_2}) and
(\ref{mfpt:E_heur}).

To obtain the corresponding result in terms of dimensional variables,
we use (\ref{intro:scalings}) to conclude that we need only replace
$C_i=C_i\left({L\K_i/D}\right)$ and $E_i=E_i\left({L\K_i/D}\right)$ in
(\ref{split_b:main_res_2}) and (\ref{split_b:main_U}), while
evaluating the Green's matrix ${\mathcal G}_s$ at $\x_i={\X_i/R}$.
\end{prop}

We remark that in the limit $\kappa_i\to 0$ for $i=1,\ldots,N$, it
follows from Lemma \ref{lemma:Cj_kappa} and Lemma \ref{lemma:Ej_kappa}
that $C_i \sim {\kappa_i a_i^2/2}$ and $E_i={\mathcal
  O}(\kappa_i^2)$. In this limit, we obtain from
(\ref{split_b:main_U}) that $U_0={\mathcal O}(1)$, while both
${\overline{U}_{10}/U_0}$ and ${\overline{U}_{11}/U_0}$ are
${\mathcal O}(\kappa_i)$ as $\kappa_i\to 0$. As a result, we conclude
that the asymptotic expansions in (\ref{split_b:main_res_1}) and
(\ref{split_b:main_res_2}) remain well-ordered in $\eps$ in the limit
$\kappa_i\to 0$ for $i=1,\ldots,N$.  Moreover, since 
\begin{equation}
  U_0 \approx \frac{\kappa_1 a_1^2}{\kappa_1 a_1^2 + \cdots + \kappa_N a_N^2} \,,
  \quad \mbox{for} \quad \kappa_i\ll 1 \,, \quad i=1,\ldots,N \,,
\end{equation}
we observe that the leading-order splitting probability is determined
by the relative reactive surface $\kappa_1 a_1^2$ of the first patch
as compared to other patches.  When all reactivities $\K_i$ are finite
and fixed, one has $\kappa_i = \eps \K_i R/D \to 0$ as $\eps\to 0$, so
that
\begin{equation}
  \overline{U} \sim U_0 \approx \frac{|\pa_1| \K_1}{|\pa_1| \K_1  +
    \cdots + |\pa_N| \K_N}  \,,
\end{equation}
to leading order in $\eps$.  This shows that to leading-order the
trapping capacity of the $i$-th patch is simply the product of its
reactivity and surface area.  We emphasize that this approximate
relation is not valid if at least one reactivity is infinite (see a
numerical example in the next subsection).

To complete this subsection, we briefly mention another physical
interpretation of the considered setting.  Upon multiplying the
splitting probability by a prescribed concentration $c_{\rm ext}$,
$c(\X) = c_{\rm ext} U(\X)$, we transform the original BVP
(\ref{split:ssp_0}) into
\bsub 
\begin{align}
  \Delta c &= 0 \,, \quad \X \in \Omega \,, \\
  D\partial_{n} c + \K_i c &= \delta_{i1}\K_i c_{\rm ext} \,, \quad
      \X \in \partial \Omega_i \,, \quad i=1,\ldots,N \,, \\
  \partial_{n} c &= 0 \,, \quad \X \in \partial \Omega_r \,.
\end{align}
\esub
Here $c(\X)$ is a steady-state concentration of particles in the
presence of a {\it source} on the Robin patch $\partial \Omega_1$ and
partially reactive sinks on the other patches $\partial \Omega_i$ with
$i = 2,\ldots,N$.  Rewriting the boundary condition on
$\partial\Omega_1$ as $-D\partial_n c = \K_1(c - c_{\rm ext})$, one
sees that the diffusive flux density from the bulk in the left-hand
side is equal to an exchange term on the right.  For instance, this
term can represent an exchange of particles across a semi-permeable
membrane $\partial \Omega_1$ between the domain $\Omega$ and an
exterior reservoir with a prescribed concentration $c_{\rm ext}$.  A
similar problem in heat transfer is known as Newton cooling.

\subsection{Numerical Example}\label{split_sec:numerics}

As an example, we consider two circular patches located at the north
and south poles of the unit sphere.  Figure~\ref{fig:split}
illustrates how the volume-averaged splitting probability
$\overline{u}$ depends on the reactivity $\kappa_1 = \kappa$ of the
first patch, where the second patch is assumed to be perfectly
reactive ($\kappa_2 = \infty$).  We consider two scenarios: two
patches of the same radius (i.e., $a_1 = a_2 = 1$), and two patches of
different radii ($0.5 = a_2 < a_1 = 1$).  In both cases, the
asymptotic result (\ref{split_b:main_res_2}) is in very close
agreement with the FEM solution of the BVP (\ref{split:ssp}) over a
broad interval of reactivities, ranging from $10^{-2}$ to $10^2$.
When two patches are of the same radius, the splitting probability
approaches $1/2$ in the limit $\kappa\to\infty$.  Curiously, even for
a weak reactivity (e.g., $\kappa = 0.1$), the first patch has a
non-negligible chance of capturing the particle.

\begin{figure}
\begin{center}
\includegraphics[width=88mm]{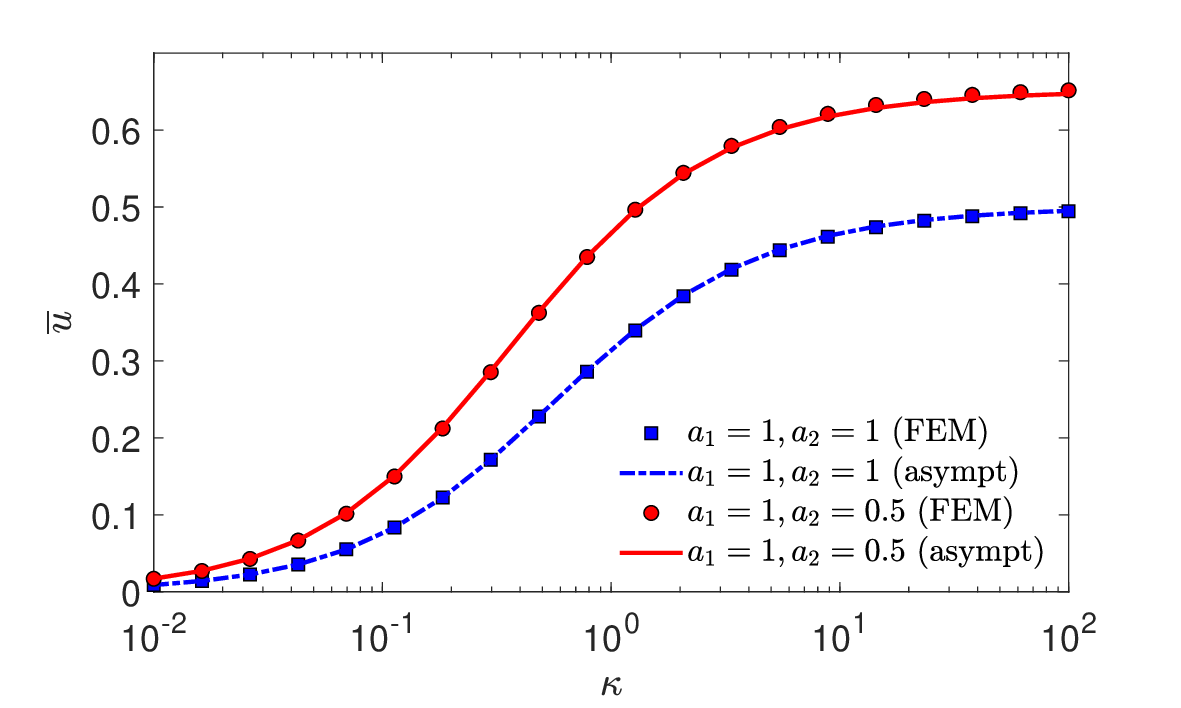} % {split.eps}
\end{center}
\caption{
The volume-averaged splitting probability $\overline{u}$ to the first
circular patch of reactivity $\kappa_1 = \kappa$ in the presence of
the second circular patch of infinite reactivity ($\kappa_2 =
\infty$).  Two patches are located at the north and south poles of the
unit sphere.  Two configurations are considered: (i) patches of
equal radii ($a_1 = a_2 = 1$, $\eps = 0.2$) and (ii) patches of
different radii ($a_1 = 1$, $a_2 = 0.5$, $\eps = 0.4$).  The curves
present the asymptotic formula (\ref{split_b:main_res_2}), while the
symbols indicate a FEM solution with the maximal meshsize $\hmax =
0.0025$.}
\label{fig:split}
\end{figure}

\section{The Steklov-Dirichlet-Neumann (SDN)
 Eigenvalue Problem}\label{stekDN:intro}

Next, we analyze the SDN eigenvalue problem (\ref{sdn:eig}) in the
unit sphere $\Omega$, where we refer to $\partial \Omega_1^{\eps}$ as
the Steklov patch and to $\partial \Omega_i^{\eps}$ for $i=2,\ldots,N$
as Dirichlet patches.  As previously, each boundary patch
$\partial\Omega_i^{\eps}$ having diameter ${\mathcal O}(\eps)\ll 1$ is
assumed to be simply-connected with a smooth boundary, but with an
otherwise arbitrary shape and satisfies
$\partial\Omega_i^{\eps}\to\x_i\in \partial\Omega$.  For convenience,
in our analysis below we will normalize the Steklov eigenfunctions of
(\ref{sdn:eig}) by
\begin{equation}\label{sdn:eig_norm}
  \int_{\partial\Omega_{1}^{\eps}} u^2 \, d{\bf s} =1\,,
\end{equation}
where the surface area element on the unit sphere in geodesic
coordinates is $d{\bf s}=\cos(\xi_1)d\xi_1 d\xi_2$ (see Appendix
\ref{app_g:geod}).  

As in \S \ref{mfpt_sec:expan} and \S \ref{split:intro}, the dependence
of the inner solution near the Steklov patch $\PT_1$ on the spectral
parameter $\sigma$ at each order in $\eps$ will rely on properties of
the parameterized solution $w_{1}\equiv w_{1}(\y;-\sigma)$ to the
canonical problem (\ref{mfpt:wc}).  Here the dependences of the
monopole term $C_1$ and the dipole term $\DT_1$ on $\sigma$ are
obtained by setting $\kappa_1=-\sigma$ in (\ref{mfpt:wc}).  However,
the ``negative reactivity'' $\kappa_1 < 0$ presents the crucial
difference with respect to the previously considered problems in \S
\ref{mfpt_sec:expan} and \S \ref{split:intro}.  In fact, from the
Steklov spectral representation (\ref{eq:Cmu_def0}), it follows that
$C_1(-\sigma)$ has simple poles at the eigenvalues $\mu_{k1}$, with
$k\geq 0$, of the local Steklov eigenvalue problem (\ref{eq:Psi_def}),
which is defined near the patch $\PT_1$, with nontrivial spectral
weights $d_{k1} \ne 0$ (see Appendix \ref{sec:Cmu} for details; we
also recall Fig.~\ref{fig:Cmu} where the poles of $C_1$ were shown for
a circular patch of unit radius). We label this resonant set as
\begin{equation}\label{sdnp:poles}
  {\mathcal P}_1 \equiv \bigcup\limits_{k=0}^\infty \left\{ \mu_{k1} ~\vert~
    \textrm{if} ~ d_{k1} \ne 0 \right\} 
\end{equation}
(i.e., this set includes only the indices $k$ for which $d_{k1} \ne
0$).  As an eigenvalue $\sigma = \sigma(\eps)$ of the SDN problem
(\ref{sdn:eig}) depends on $\eps$, a small-$\eps$ expansion of
$w_1(\y;-\sigma(\eps))$ would naturally lead to $w_1(\y;-\sigma_0)$,
where $\sigma_{0} = \lim\limits_{\eps\to 0}\sigma(\eps)$.  However,
the solution $w_{1}(y;-\sigma_0)$ exists only when $\sigma_0\notin
{\mathcal P}_1$.  This preliminary consideration clarifies the need to
distinguish between two cases in our analysis below:

(I) $\sigma_0 \notin {\mathcal P}_1$, in which case an eigenpair
$\{\sigma, u\}$ is called {\em non-resonant};

(II) $\sigma_0 \in {\mathcal P}_1$, in which case an eigenpair
$\{\sigma, u\}$ is called {\em near-resonant}.

In the next subsection, we will determine the asymptotic behavior of
the non-resonant SDN eigenvalues.  In turn, in \S \ref{sec:sdn_degen},
we will show that near-resonant eigenvalues do not exist for the SDN
problem (\ref{sdn:eig}) with a single Steklov patch.  In other words,
the near-resonant case is not possible for this SDN problem.  In
contrast, this case will re-appear in our analysis in \S \ref{stekN}
for the SN problem (\ref{sn:eig}) with multiple Steklov patches.

Before undertaking our asymptotic analysis, we outline an important
auxiliary result related to the Steklov patch.
When $\sigma\notin {\mathcal P}_1$, we observe upon substituting
$\kappa_1 = -\sigma$ into (\ref{mfpt:wc}) and differentiating it with
respect to $\sigma$ that
\begin{equation}  \label{eq:wc_def}
w_{c1}=w_{c1}(\y;-\sigma)\equiv \partial_{\sigma} w_{1}(\y;-\sigma)
\end{equation}
satisfies
\bsub \label{sdnp:wc}
\begin{align}
    \Delta_{\y} w_{c1} &=0 \,, \quad \y \in \R_{+}^{3} \,, \label{sdnp:wc_1}\\
    \partial_{y_3} w_{c1} + \sigma w_{c1} &=1-w_1 \,, \quad y_3=0 \,,\,
    (y_1,y_2)\in \PT_1\,,  \label{sdnp:wc_2}\\
    \partial_{y_3} w_{c1} &=0 \,, \quad y_3=0 \,,\, (y_1,y_2)\notin \PT_1
    \,, \label{sdnp:wc_3}\\
  w_{c1}&\sim - \frac{C_{1}^{\prime}(-\sigma)}{|\y|} -
                {  \frac{\DT_1^{\prime}(-\sigma) {\bf \cdot} \y}{|\y|^3}}
         + \cdots\,,  \quad \mbox{as}\quad
    |\y|\to \infty \,, \label{sdnp:wc_4}
\end{align}
where we have defined
\begin{equation}\label{sdnp:wc_deriv}
  C_{1}^{\prime}(-\sigma) \equiv \left. \left(\frac{dC_1(\kappa)}{d\kappa}
  \right)\right|_{\kappa
    = -\sigma}
  \qquad \mbox{and} \qquad \DT_1^{\prime}(-\sigma) \equiv \left. \left(\frac{d
    \DT_1(\kappa)}{d\kappa}\right)\right|_{\kappa = -\sigma}\,.
\end{equation}
\esub
The identification of this problem satisfied by $w_{c1}$ is key for
solving the different inner problems at each order of the inner
expansion near the Steklov patch $\PT_1$.  A general spectral
representation (\ref{eq:wc_spectral}) for $w_{c1}$ is established in
(\ref{eq:wc_spectral}) of Appendix \ref{sec:Cmu}.

\subsection{Non-Resonant Case}\label{sdn_sec:asymptotics}

In the outer region, we expand the solution to (\ref{sdn:eig}) away
from all the patches as
\begin{equation}
    u \sim U_0 + \eps U_1 + \eps^2 \log\left( \frac{\eps}{2} \right)
   U_2 + \eps^2 U_3 + \cdots \,. \label{sdn:outex}
\end{equation}
We will initially seek a solution where $U_0\neq 0$ is a
constant. This non-zero leading-order outer solution will be used
below to satisfy the normalization condition (\ref{sdn:eig_norm}).  In
Remark \ref{sdn:remark2} below we will briefly discuss whether one can
find SDN eigenpairs for the case where $U_0=0$.

Upon substituting (\ref{sdn:outex}) into (\ref{sdn:eig}) 
we obtain that $U_j$ for $j=1,2,3$ satisfies the
outer problems
\begin{equation} \label{sdn:Uk}
  \Delta_{\x} U_j = 0 \,, \quad \x \in \Omega \,; 
  \qquad \partial_n U_j = 0 \,, \quad \x\in \partial\Omega\backslash
  \lbrace{\x_1,\ldots,\x_N\rbrace} \,.
\end{equation}
Our analysis below provides singularity behaviors for each $U_{j}$
as $\x\to \x_i$.

The novel feature of our analysis of (\ref{sdn:eig}) is that
each Steklov eigenvalue $\sigma=\sigma(\eps)$ must be expanded as
\begin{equation}\label{sdn:stek_eig}
  \sigma=\sigma(\eps)=\sigma_{0}+
  \eps\log\left(\frac{\eps}{2} \right) \sigma_{1} + \eps
  \sigma_{2} + \ldots\,.
\end{equation}
The coefficients $\sigma_{j}$ for $j=0,1,2$, which are
independent of $\eps$, will be determined below by ensuring that the
outer problems are solvable at each order. We emphasize that by using
(\ref{sdn:stek_eig}) in (\ref{sdn:eig_2}) we will obtain distinct
boundary conditions on the Steklov patch at each order of the inner
expansion.

In the inner region near the $i$-th patch we introduce the local
geodesic coordinates (\ref{mfpt:innvar}) and expand each inner
solution as
\begin{equation}
  u  \sim V_{0i} + \eps \log\left( \frac{\eps}{2} \right) V_{1i} +
   \eps V_{i2}  + \ldots \,. \label{sdn:innex}
\end{equation}
For each Dirichlet patch with $i=2,\ldots, N$, we readily
obtain that $V_{ji}$ for $j=0,1,2$ satisfies
\bsub \label{sdn_d:Vi}
\begin{align}
  \Delta_{\y} V_{ji} &= \delta_{j2} \left( 2y_{3} V_{0i,y_3 y_3} + 2 V_{0i,y_3}
                       \right) \,, \quad
   \y \in \R_{+}^{3} \,, \label{sdn_d:Vi_1}\\
   V_{ji} &=0 \,, \quad y_3=0 \,,\,
    (y_1,y_2)\in \PT_i\,,  \label{sdn_d:Vi_2}\\
    \partial_{y_3} V_{ji} &=0 \,, \quad y_3=0 \,,\, (y_1,y_2)\notin \PT_i
    \,, \label{sdn_d:Vi_3}
\end{align}
\esub where $\delta_{22}=1$, $\delta_{j2}=0$ if $j=0,1$, and
$\PT_i \asymp \eps^{-1}\partial\Omega_i^{\eps}$.  In contrast, on the
rescaled Steklov patch $\PT_1 \asymp
\eps^{-1}\partial\Omega_1^{\eps}$, we obtain that $V_{j1}$ for
$j=0,1,2$ satisfies
\bsub \label{sdn_s:V1}
\begin{align}
  \Delta_{\y} V_{j1} &= \delta_{j2} \left( 2y_{3} V_{01,y_3 y_3} + 2 V_{01,y_3}
                       \right) \,, \quad
   \y \in \R_{+}^{3} \,, \label{sdn_s:V1_1}\\
  \partial_{y_3} V_{j1} + \sigma_{0} V_{j1} &=-(1-\delta_{j0})
        \sigma_{j} V_{01} \,, \quad y_3=0 \,,\,
    (y_1,y_2)\in \PT_1\,,  \label{sdn_s:V1_2}\\
    \partial_{y_3} V_{j1} &=0 \,, \quad y_3=0 \,,\, (y_1,y_2)\notin \PT_1
    \,. \label{sdn_s:V1_3}
\end{align}
\esub

Since the leading-order matching condition between inner and outer
solutions is that $V_{0i}\sim U_0$ as $|\y|\to\infty$, the leading-order
inner solution, obtained from  (\ref{sdn_d:Vi}) and (\ref{sdn_s:V1}), is
\begin{equation}\label{sdn:v0i}
  V_{0i}=U_0\left(1-w_{0i}\right) \,.
\end{equation}
In (\ref{sdn:v0i}), $w_{0i}(\y)$ is defined in terms of the solution
$w_i(\y;\kappa_i)$ to (\ref{mfpt:wc}) by
\begin{equation}\label{sdn:V0isol}
  w_{0i}(\y) \equiv \left\{\begin{array}{ll} w_{1}(\y;-\sigma_0)
                             \,, & i=1, \\
   w_i(\y;\infty) \,, & i=2,\ldots,N.  \end{array}\right. 
\end{equation}
Our assumption $\sigma_0\notin {\mathcal P}_1$ implies that
$C_1(-\sigma_0)$ and $w_{1}(\y;-\sigma_0)$ are well-defined.

Upon using (\ref{mfpt:wc_4}) to obtain the far-field behavior for
$w_{0i}$, we match the inner and outer expansions to conclude that the
outer correction $U_1$ in (\ref{sdn:outex}) satisfies
\bsub\label{sdn:U1prob}
\begin{align}
  \Delta_{\x} U_{1} &= 0 \,, \quad \x\in \Omega \,; \qquad
  \partial_n U_1=0 \,, \quad \x\in \partial\Omega\backslash
  \lbrace{\x_1,\ldots,\x_N\rbrace} \,, \label{sdn:U1prob_1}\\
  U_1 & \sim -\frac{C_1(-\sigma_{0})U_0}{|\x-\x_1|}\,,
  \quad \mbox{as} \quad
  \x\to\x_1\in \partial\Omega \,, \\
   U_1 & \sim -\frac{C_i(\infty)U_0}{|\x-\x_i|}\,, \quad \mbox{as} \quad
  \x\to\x_i \in \partial\Omega\,, \quad i=2,\ldots,N \,.
\end{align}
\esub 
The solvability condition for (\ref{sdn:U1prob}), together with the
assumption that $U_0\neq 0$, provides the following nonlinear
algebraic equation for the leading-order term $\sigma_{0}$ in the
expansion (\ref{sdn:stek_eig}) of a Steklov eigenvalue:
\begin{equation}\label{sdn:sigma_0}
  C_1\left(-\sigma_{0}\right) = {\mathcal N} \equiv
  -\sum_{i=2}^{N} C_i(\infty) \,.
\end{equation}
Since $C_i(\infty)>0$, we conclude that ${\mathcal N} <0$.  The
spectral expansion (\ref{eq:Cmu_def0}) ensures that, for any
${\mathcal N} < 0$, (\ref{sdn:sigma_0}) has infinitely many solutions
that we denote as $\sigma_0^{(k)}$, with $k = 0,1,\ldots$.  Moreover,
since $C_1(-\sigma_0)$ decreases monotonically as $\sigma_0$ increases
between its poles in ${\mathcal P}_1$ (see (\ref{eq:dCmu}) and
Fig.~\ref{fig:Cmu}), each solution $\sigma_0^{(k)}$ is simple and lies
between two consecutive poles.  In our analysis below, we will
determine the next-order corrections $\sigma_1^{(k)}$ and
$\sigma_2^{(k)}$ to the dominant contribution $\sigma_0^{(k)}$.  As
our asymptotic analysis is applicable to any $k$, we omit the
superscript $^{(k)}$ for brevity.

With $\sigma_{0}$ determined by (\ref{sdn:sigma_0}), the solution to
(\ref{sdn:U1prob}) is represented in terms of an unknown constant
$\overline{U}_1$ as
\begin{equation}\label{sdn:U1sol}
  U_1(\x) = \overline{U}_1 -2\pi U_0 \sum_{j=1}^{N} C_j G_{s}(\x;\x_j)\,,
 \quad
 \mbox{where} \quad  C_j \equiv \left\{\begin{array}{ll}
        C_1(-\sigma_{0})  \,, & j=1 \\
  C_j(\infty) \,, & j=2,\ldots,N\,.  \end{array}\right.
\end{equation}
Here $G_s$ is the surface Neumann Green's function for the sphere given
in (\ref{mfpt:gs_exact}).

To proceed to higher order, we expand $U_1(\x)$ in (\ref{sdn:U1sol})
as $\x\to\x_i$ to obtain in terms of local geodesic coordinates that
\begin{equation}\label{sdn:U1_expan}
  \begin{split}
    U_{1} &\sim -\frac{C_i U_0}{|\y|} + \frac{U_0 C_i}{2}
    \log\left(\frac{\eps}{2}\right)
    + \frac{U_0 C_i}{2} \left( \log(y_3+|\y|) -
      \frac{y_3(y_1^2+y_2^2)}{|\y|^3}\right) \\
    & \qquad + U_0 \beta_i +\overline{U}_1\,,
  \end{split}
\end{equation}
where $\beta_i$ is the $i$-th component of the vector ${\bm \beta}$
defined by
\begin{equation}\label{sdn:beta}
  {\bm \beta} \equiv -2\pi {\mathcal G}_s \vc \,, \qquad
  \mbox{where} \qquad \vc=(C_1,\ldots,C_N)^T\,.
\end{equation}
Here $C_i$ for $i=1,\ldots,N$ is defined in (\ref{sdn:U1sol}) and
${\mathcal G}_s$ is the Green's matrix in (\ref{mfpt_b:green_mat}).

Upon matching the inner and outer expansions for the
$\eps\log\left({\eps/2} \right)$ terms, we conclude that the inner
correction $V_{1i}$ in (\ref{sdn:innex}) must satisfy
$V_{1i}\sim {U_0 C_i/2}$ as $|\y|\to\infty$. As a result, we seek a
solution to (\ref{sdn_d:Vi}) and (\ref{sdn_s:V1}) for $k=1$ in the
form
\begin{equation}\label{sdn:V1form}
  V_{1i} = \frac{U_0C_i}{2} \left(1- w_{1i}\right) \,, \quad \mbox{for}
  \quad i=1,\ldots,N \,.
\end{equation}
By using the problem (\ref{sdnp:wc}) satisfied by
$w_{c1}(\y;-\sigma_{0})$ to account for the inhomogeneous term in
(\ref{sdn_s:V1_2}) on the Steklov patch, we use superposition to get
that
\bsub \label{sdn:V1sol}
\begin{align}
  V_{1i} &=\frac{U_0 C_i(\infty)}{2} \left( 1 - w_{i}(\y;\infty)\right)  \,,
           \quad i=2,\ldots, N \,,\\   \label{sdn:V1sol_1}
  V_{11} &=  -U_{0} \sigma_{1} w_{c1}(\y;-\sigma_0) +
           \frac{U_0 C_1(-\sigma_0)}{2}
           \left(1 - w_{1}(\y;-\sigma_{0})\right) \,.
\end{align}
\esub
Then, by using the far-field behaviors in (\ref{mfpt:wc_4}) and
(\ref{sdnp:wc_4}), we find for $|\y|\to\infty$ that
\bsub\label{sdn:V1sol_ff}
\begin{align}
  V_{1i} & \sim \frac{U_0 C_i(\infty)}{2}\left( 1 - \frac{C_{i}(\infty)}{|\y|}
           \right) +\cdots \,, \quad i=2,\ldots,N \,,\\     
  V_{11} &\sim U_0 \sigma_{1} \frac{ C_{1}^{\prime}(-\sigma_0)}
           {|\y|} + \frac{U_0 C_1(-\sigma_0)}{2}
           \left(1 - \frac{C_1(-\sigma_{0})}{|\y|} \right) +\cdots \,,
\end{align}
\esub
where the neglected higher-order terms are dipole contributions.

Upon matching the monopole terms in the far-field behavior
(\ref{sdn:V1sol_ff}) to the outer correction $U_2$ in
(\ref{sdn:outex}), we obtain from (\ref{sdn:Uk}) that $U_2$ satisfies
\bsub\label{sdn:U2prob}
\begin{gather}
  \Delta_{\x} U_{2} = 0 \,, \quad \x\in \Omega \,; \qquad
  \partial_n U_2=0 \,, \quad \x\in \partial\Omega\backslash
  \lbrace{\x_1,\ldots,\x_N\rbrace} \,, \label{sdn:U2prob_1}\\
  U_2  \sim -\frac{U_0}{2}\frac{C_i^2}{|\x-\x_i|} + U_0 \sigma_{1}
        \delta_{i1} \frac{ C_{1}^{\prime}(-\sigma_0)}{|\x-\x_i|}\,,
          \quad \mbox{as} \quad
  \x\to\x_i\in \partial\Omega \,, \quad i=1,\ldots, N \,, \label{sdn:U2prob_2}
\end{gather}
\esub
where $C_i$ for $i=1,\ldots,N$ is defined in (\ref{sdn:U1sol}) and
$\delta_{i1}=1$ if $i=1$ and $\delta_{i1}=0$ for $i=2,\ldots,N$. The
solvability condition for (\ref{sdn:U2prob}) yields
$U_0\sum_{i=1}^{N} C_i^2 = 2 U_0 \sigma_{1} C_1^{\prime}(-\sigma_0)$. For
$U_0\neq 0$, this expression determines the coefficient
$\sigma_{1}$ in (\ref{sdn:stek_eig}) as
\begin{equation}\label{sdn:sigma_1}
  \sigma_{1} = \frac{1}{2 C_1^{\prime}(-\sigma_0)} \left(
    \left[ C_1(-\sigma_0)\right]^2 + \sum_{i=2}^{N}
     \left[ C_i(\infty) \right]^2 \right)\,,
\end{equation}
where $\sigma_0$ is a root of (\ref{sdn:sigma_0}).
Upon using (\ref{sdn:sigma_0}) directly in (\ref{sdn:sigma_1}),
we can write $\sigma_{1}$ equivalently as
\begin{equation}\label{sdn:sigma_1n}
  \sigma_{1} = \frac{1}{2 C_1^{\prime}(-\sigma_0)} \left(
    \left[ \sum_{i=2}^{N} C_i(\infty) \right]^2 + \sum_{i=2}^{N}
     \left[ C_i(\infty) \right]^2 \right) \,.
\end{equation}
From (\ref{eq:dCmu}) we recall that $C_1(\kappa)$ increases
monotonically in $\kappa$ between its poles so that
$C_1^{\prime}(\kappa)$, evaluated at $\kappa = -\sigma_0$, is strictly
positive.  As a result, the denominator in (\ref{sdn:sigma_1n}) never
vanishes, and $\sigma_1$ is well-defined and strictly positive.  Then,
the solution to (\ref{sdn:U2prob}) is represented in terms of an
unknown constant $\overline{U}_2$ as
\begin{equation}\label{sdn:U2sol}
  U_2(\x) = \overline{U}_2 -2\pi U_0 \sum_{j=1}^{N}
  \left( \frac{C_j^2}{2} -\sigma_{1}
    C_1^{\prime}(-\sigma_0) \delta_{j1}\right) G_{s}(\x;\x_j) \,.
\end{equation}

Finally, we calculate the eigenvalue correction $\sigma_{2}$ in
(\ref{sdn:stek_eig}). To do so, we observe from the matching condition
between inner and outer solutions that the ${\mathcal O}(1)$ terms in
(\ref{sdn:U1_expan}) provide the following far-field behavior for the
inner correction $V_{2i}$ in (\ref{sdn:innex}):
\begin{equation}\label{sdn:v2_ff}
   V_{2i} \sim \frac{U_0 C_i}{2} \left( \log(y_3+|\y|) -
     \frac{y_3(y_1^2+y_2^2)}{|\y|^3}\right) + U_0 \beta_i +\overline{U}_1\,,
   \quad   \mbox{as} \quad |\y|\to \infty\,.
 \end{equation}

For the Dirichlet patches, and as similar to the
analysis for the MFRT and splitting probability problems, the solution
to (\ref{sdn_d:Vi}) with $k=2$ subject to (\ref{sdn:v2_ff}) is
\begin{equation}\label{sdn:v2sol_i}
  V_{2i} = U_0 \Phi_{2i} + \left(U_0 \beta_i+\overline{U}_1\right)
  \left(1- w_{i}(\y;\infty)\right)\,, \quad \mbox{for} \quad
  i=2,\ldots, N \,,
\end{equation}
where $\Phi_{2i}$ satisfies (\ref{mfpt_b:Phi2}) with $\kappa_i=\infty$
so that $\Phi_{2i}=0$ on $\PT_i$.  As a result, by using
(\ref{mfpt_b:Phi2_ff}) for the far-field behavior of $\Phi_{2i}$, we
conclude for $i=2,\ldots,N$ that $V_{2i}$ has the refined far-field
behavior
\begin{equation}\label{sdn_d:V2iff_refine}
  \begin{split}
    V_{2i} &\sim U_0 \beta_i +\overline{U}_1 + \frac{U_0 C_i(\infty)}{2} \left(
      \log(y_3 + |\y|) - \frac{y_3 (y_1^2 + y_2^2)}{|\y|^3} \right) \\
    &\qquad + \left[E_i(\infty)-\left(\beta_i+ \frac{\overline{U}_1}{U_0}
        \right) C_i(\infty)\right] \frac{U_0}{|\y|}
     \,, \quad   \mbox{as}    \quad |\y| \to \infty \,. 
  \end{split}
\end{equation}
For an arbitrarily-shaped patch, $E_i(\infty)$ is obtained by setting
$\kappa_i=\infty$ in (\ref{eq:Ei_general0}), while for a locally
circular patch it is given in (\ref{mfpt:Ej_all}) of Lemma
\ref{lemma:Ej_kappa}.  The second line in (\ref{sdn_d:V2iff_refine})
is one of the two terms that needs to be accounted for by the outer
correction $U_3$ in the matching condition. The other term is the
dipole contribution from (\ref{mfpt:wc_4}) and (\ref{sdn:v0i}).

In contrast, for the Steklov patch, we use superposition to determine
that the solution to (\ref{sdn_s:V1}) with $k=2$ subject to
(\ref{sdn:v2_ff}) is
\begin{equation}\label{sdn:v2sol_1}
  V_{21} = U_0 \Phi_{21} + \left(U_0 \beta_1+\overline{U}_1\right)
  \left(1-w_{1}(\y;-\sigma_0)\right)
  - U_0 \sigma_2 w_{c1}(\y;-\sigma_{0}) \,,
\end{equation}
where $\Phi_{21}$ satisfies (\ref{mfpt_b:Phi2}) (with $i=1$) in which we set
$\kappa_1=-\sigma_{0}$ and $C_1=C_1(-\sigma_0)$. As a
result, the refined far-field behavior of $V_{21}$ is
\begin{equation}\label{sdn_d:V21ff_refine}
  \begin{split}
    V_{21} &\sim U_0 \beta_1 + \overline{U}_1 + \frac{U_0 C_1(-\sigma_{0})}
    {2}  \left(
      \log(y_3 + |\y|) - \frac{y_3 (y_1^2 + y_2^2)}{|\y|^3} \right) \\
    &\qquad + \left[E_1(-\sigma_{0})-
    \left(\beta_1 + \frac{\overline{U}_1}{U_0}\right) C_1(-\sigma_0)
    + \sigma_{2} C_1^{\prime}(-\sigma_0)\right] \frac{U_0}{|\y|}\,,
    \quad \mbox{as} \quad |\y| \to \infty \,.
  \end{split}
\end{equation}
Here $E_1(-\sigma_{0})$ is calculated by setting $\kappa_i=-
\sigma_0$ in (\ref{eq:Ei_general0}). 

The monopole terms in the far-field behaviors
(\ref{sdn_d:V21ff_refine}) and (\ref{sdn_d:V2iff_refine}) together
with the dipole terms from the leading order inner solutions $V_{0i}$
in (\ref{sdn:v0i}) provide the required singularity behavior for the
outer correction $U_3$ in (\ref{sdn:outex}). In this way, we obtain
from (\ref{sdn:Uk}) that $U_3$ satisfies
\begin{equation}\label{sdn:U3prob}
  \begin{split}
  \Delta_{\x} U_{3} &= 0 \,, \quad \x\in \Omega \,; \qquad
  \partial_n U_3=0 \,, \quad \x\in \partial\Omega\backslash
  \lbrace{\x_1,\ldots,\x_N\rbrace} \,, \\
  U_3  & \sim \left[ E_i - \left(\beta_i +
      \frac{\overline{U}_1}{U_0}\right)C_i + \sigma_2 C_{1}^{\prime}(-\sigma_0)
  \delta_{i1}\right] \frac{U_0}{|\x-\x_i|} \\
  & \qquad -U_0 \frac{\DT_i {\bf \cdot} {\mathcal Q}_i^T (\x-\x_i)}{|\x-\x_i|^3}
\,, \quad \mbox{as} \quad
  \x\to\x_i\in \partial\Omega \,, \quad i=1,\ldots, N \,,
\end{split}
\end{equation}
where the orthogonal matrix ${\mathcal Q}_i$ is defined in
(\ref{app_g:change}) in terms of the basis vectors of the geodesic
coordinate system. In (\ref{sdn:U3prob}) we have defined $E_i$ and
$\DT_i$ for $i=1,\ldots,N$ by
\begin{equation}\label{sdn:Ei}
 E_i \equiv \left\{\begin{array}{ll}
        E_1(-\sigma_{0})  \,, & i=1, \\
  E_i(\infty) \,, & i=2,\ldots,N \,,  \end{array}\right.  \qquad
 \DT_i \equiv \left\{\begin{array}{ll}
        \DT_1(-\sigma_{0})  \,, & i=1, \\
   \DT_i(\infty) \,, & i=2,\ldots,N .  \end{array}\right. 
\end{equation}
Here both $\DT_i(\infty)$ and $\DT_i(-\sigma_0)$ are defined
by (\ref{mfpt:wc_4}).

In deriving the solvability condition for (\ref{sdn:U3prob}), we
observe that the dipole terms again do not contribute and we obtain
\begin{equation}\label{sdn:sigma2_1}
  U_0 \sum_{i=1}^{N} E_i - U_0 \sum_{i=1}^{N} \beta_i C_i -\overline{U}_1
  \sum_{i=1}^{N} C_i + U_0
  \sigma_{2} C_1^{\prime}(-\sigma_{0}) =0 \,,
\end{equation}
where $\sum_{i=1}^{N} C_i=0$ from the leading-order result
(\ref{sdn:sigma_0}) for $\sigma_0$. As a result, (\ref{sdn:sigma2_1})
shows that $\sigma_2$ is independent of the constant
$\overline{U}_1$. This term is determined by a higher-order evaluation
of the normalization condition for the Steklov eigenfunction.

Then, by using (\ref{sdn:beta}) for $\beta_i$, we solve (\ref{sdn:sigma2_1})
to determine $\sigma_{2}$  as
\begin{equation}\label{sdn:sigma_2}
  \sigma_{2}=-\frac{1}{C_1^{\prime}(-\sigma_{0})} \left[
    2\pi \vc^T {\mathcal G}_s \vc + \sum_{i=1}^{N} E_i\right]\,, \quad
  \mbox{where} \quad  \vc \equiv (C_1,\ldots,C_N)^T\,,
\end{equation}
where ${\mathcal G}_s$ is the Green's matrix in
(\ref{mfpt_b:green_mat}). Through this Green's matrix, it follows that
$\sigma_{2}$ depends on the overall spatial configuration of the
Steklov and Dirichlet patches on the boundary.  We summarize our
result as follows:

\begin{prop}\label{sdn:main_res} 
As $\eps \to 0$, there are eigenvalues $\sigma=\sigma(\eps)$ of the
Steklov-Dirichlet-Neumann problem (\ref{sdn:eig}) that have the
three-term asymptotics
\bsub
\begin{align}\label{sdn:stek_eigex}
  \sigma &=\sigma_{0}+
  \eps\log\left(\frac{\eps}{2} \right) \sigma_{1} + \eps
  \sigma_{2} + {\mathcal O}(\eps^2\log \eps)\,, 
\end{align}
where $\sigma_{0}$, $\sigma_{1}$ and $\sigma_{2}$ are respectively
given by (\ref{sdn:sigma_0}), (\ref{sdn:sigma_1n}) and
(\ref{sdn:sigma_2}).  The corresponding eigenfunctions, restricted to
$\PT_1$, have the three-term asymptotics
\begin{equation}
  u\vert_{\PT_1} = V_{01} + \eps\log\left(\frac{\eps}{2} \right)
                               V_{11} + \eps
  V_{21} + {\mathcal O}(\eps^2\log \eps )\,,
\end{equation}
\esub
where $V_{01}$, $V_{11}$ and $V_{21}$ are respectively given by
(\ref{sdn:v0i}), (\ref{sdn:V1sol_1}) and (\ref{sdn:v2sol_1}).  Here
$u\vert_{\PT_1}$ is given up to constants $U_0\neq 0$ and
$\overline{U}_1$, while its spatial behavior is determined by the
functions $w_1(\y;-\sigma_0)$ and $w_{c1}(\y;-\sigma_0)$, which admit
the spectral expansions (\ref{eq:wi_spectral}) and
(\ref{eq:wc_spectral}), respectively.  For a circular patch, these
expansions allow for their efficient numerical computation (see
Appendix \ref{sec:Cmu}).
\end{prop}

We emphasize that with our assumption $U_0\neq 0$, the Steklov
eigenfunction is not solely concentrated on the Steklov patch. Instead
it has a long-range extension into the outer region as a result of the
presence of the other small Dirichlet patches. To determine $U_0$ we
substitute $u\sim V_{01}=U_{0}\left(1-w_1(\y;-\sigma_0)\right)$ into
the normalization condition (\ref{sdn:eig_norm}). Upon using $d{\bf
s}=\cos(\xi_1)d\xi_1 d\xi_2=\eps^2 \cos(\eps y_1) dy_1 dy_2
\sim \eps^2 dy_1 dy_2$, we calculate that
\begin{equation}\label{sdn:norm_u0} 
  U_0 \sim \eps^{-1} \left(\int_{\PT_1} \left[1-w_1(\y;-\sigma_0)
      \right]^2 d\y\right)^{-1/2} \,.
\end{equation}
In turn, the constant $\overline{U}_1$ that appears in $V_{21}$,
remains unknown and can only be found at higher order.

\begin{remark} \label{sdn:remark1} 
Proposition \ref{sdn:main_res} offers a straightforward numerical
procedure to construct three-term asymptotic expansions of the SDN
eigenpairs in the small-patch limit when $U_0\neq 0$.  Moreover, since
(\ref{sdn:sigma_0}) has infinitely many solutions, one can 
construct SDN eigenvalues on any desired, but large enough,
interval $(\sigma_{\rm min}, \sigma_{\rm max})$.  To characterize the
associated eigenfunctions on the Steklov patch $\Gamma_1$, we
use $V_{01}=U_{0}\left[1-w_1(y;-\sigma_0)\right]$ to leading
order, together with the divergence theorem applied to the problem
(\ref{mfpt:wc}) for $w_{1}$, to obtain for $U_0\neq 0$ that
\begin{equation}
C_{1}(-\sigma_0)=-\frac{\sigma_0}{2\pi} \int_{\Gamma_1} \left(1-w_1
\right)\, dy_1 dy_2 = -\frac{\sigma_0}{2\pi U_0} \int_{\Gamma_1}
V_{01}\, dy_1 dy_2 \,.
\end{equation}
Since we must have $C_1(-\sigma_0)\neq 0$ from (\ref{sdn:sigma_0}), we
conclude that to leading order all non-resonant SDN
eigenfunctions in Proposition \ref{sdn:main_res} are such that
$\int_{\Gamma_1}V_{01}\, dy_1 dy_2\neq 0$.  For a circular patch
$\Gamma_1$, this implies that these SDN eigenfunctions must be axially
symmetric on the patch.
\end{remark}

\begin{remark} \label{sdn:remark2} 
If we remove our requirement $U_0\neq 0$ and set $U_0 = 0$, so that
the bulk solution is now asymptotically small, we can construct a SDN
eigenpair where the Dirichlet patches have only a very weak influence
on the SDN eigenvalue.  In this situation, the SDN eigenfunction is
concentrated on the Steklov patch, and to leading-order is unaffected
by the presence of the Dirichlet patches, as if they were absent.  The
asymptotic behavior for a single circular Steklov patch was studied in
\cite{Grebenkov25}.  For an arbitrary patch, to construct such a
solution, we let $\mu_{k1}^{N}>0$ and $\Psi_{k1}^{N}$ for $k\geq 1$ be
the eigenpairs of the local Steklov problem (\ref{eq:Psi_def_N}) near
$\Gamma_1$ that satisfies, up to a normalization condition, the local
boundary value problem
\begin{subequations}  \label{SDN:Psi_def_N}
\begin{align}  \label{SDN:VkN_eq} 
\Delta_{\y} \Psi_{k1}^N & = 0\,, \quad \y \in \R_+^3 \,,\\
  \partial_{y_3}\Psi_{k1}^N + \mu_{k1}^N \Psi_{k1}^N &=0 \,, \label{SDN:VkN_stek}
\quad y_3=0 \,,\, (y_1,y_2)\in \PT_1\,,\\  \label{SDN:VkN_Neumann}
  \partial_{y_3} \Psi_{k1}^N & = 0 \,,
\quad y_3=0 \,,\, (y_1,y_2)\notin \PT_1\,,\\  \label{SDN:VkN_inf}
   \Psi_{k1}^N(\y) & \sim {\mathcal O}(|\y|^{-2}) \,,
                                    \quad \textrm{as}\quad |\y|\to \infty\,,
\end{align}
\end{subequations}
where there is no monopole behavior in the far-field
(\ref{SDN:VkN_inf}).  Then, a leading-order SDN eigenpair for
(\ref{sdn:eig}) is obtained by taking $\sigma_0=\mu_{k1}^{N}$ for some
index $k\geq 1$, and by choosing $V_{01}=\Psi_{k1}^{N}$ as the
leading-order inner solution near the Steklov patch $\Gamma_1$, and
$V_{0i}=0$ as the leading-order inner solution near the Dirichlet
patches $\Gamma_i$ for $i=2,\ldots,N$.  Owing to the fast decay
(\ref{SDN:VkN_inf}), the outer (bulk) solution away from the patches
is ${\mathcal O}(\eps^2)$ smaller than that of the SDN eigenfunction
evaluated on the Steklov patch.  For this leading-order construction,
we find by applying the divergence theorem to (\ref{SDN:Psi_def_N})
that $\int_{\Gamma_1} V_{01}\, dy_1 dy_2=0$.  As a result, when
$\Gamma_1$ is a circular patch, all of the non-axially symmetric
eigenfunctions on the patch will satisfy this condition.  It is an
open problem to asymptotically construct a higher-order approximation
for these SDN eigenpairs.
\end{remark}

\subsection{Example of Identical Circular Dirichlet Patches}
\label{sdn:example}

Our main result can be simplified considerably for the special case of
$N-1$ identical locally circular Dirichlet patches $\pa_i^\eps$, each
of radius $\eps a$ (i.e., $a_i = a$ for $i = 2,\ldots,N$).  For this
situation, we have $C_i(\infty)= 2a/\pi$ for $i=2,\ldots,N$, so that
from (\ref{sdn:sigma_0}) $\sigma_{0}$ is a root of the nonlinear
algebraic equation
\begin{equation}\label{sdn:stek_eig0} 
  C_1(-\sigma_0) =-\frac{2a(N-1)}{\pi} \,.
\end{equation}
In addition, we reduce the expression (\ref{sdn:sigma_1n}) for
$\sigma_{1}$ to
\begin{equation}\label{sdn:stek_eig1}  
  \sigma_{1} = \frac{2a^2}{\pi^2 C_1^{\prime}(-\sigma_0)} N(N-1)  \,.
\end{equation}
Moreover, we can use (\ref{sdn:stek_eig0}) together with
(\ref{mfpt:Ej_asy_large}) for $E_i(\infty)$, so that to determine
$\sigma_2$ in (\ref{sdn:sigma_2}) we need only to set
\begin{equation}\label{sdn:sigma_2id}  
  \vc =\frac{2a}{\pi} \left( 1-N,1,\ldots,1\right)^T \,; \quad
   E_i=-\frac{2a^2}{\pi^2} \left(\log{a} + \log{4}-\frac{3}{2}
  \right) \,, \quad i=2,\ldots, N\,.
\end{equation}
In this way, the numerical evaluation of the three-term expansion for
$\sigma$ in (\ref{sdn:stek_eigex}) only involves solving the
root-finding problem (\ref{sdn:stek_eig0}) for $\sigma_0$ and then
calculating $E_1(-\sigma_0)$ and $C_1^{\prime}(-\sigma_0)$
numerically.  To calculate $E_1(-\sigma_0)$ for a circular Steklov
patch, we use the decomposition (\ref{appE:new}) of Appendix
\ref{sec:Ei} together with the numerical results shown in
Fig.~\ref{fig:Ei}.

To qualitatively illustrate our theory, we consider a Steklov patch of
arbitrary shape and examine the limit $a\to 0$ with a fixed $N$.  In
this limit, the right-hand side of (\ref{sdn:stek_eig0}) vanishes, and
$\sigma_0$ is determined by solving $C_{1}(-\sigma_0) = 0$.  These
solutions correspond to the eigenvalues $\mu_{k1}^N$ of another {\em
local} exterior Steklov problem defined by (\ref{eq:Psi_def_N}) of
Appendix \ref{sec:Cmu0} near the Steklov patch $\PT_1$, with
a Neumann-like boundary condition (\ref{eq:VkN_inf}) at infinity (see
\cite{Grebenkov25} for details; in particular, Table 1 from
\cite{Grebenkov25} reports $\mu_{k1}^N$ for a circular Steklov patch).
Indeed, as the Dirichlet patches vanish, one recovers the conventional
Steklov problem with a single Steklov patch (see Remark
\ref{sdn:remark2}).  In particular, the principal eigenvalue
$\sigma^{(0)}$ of the SDN problem should approach $0$ for a Steklov
patch of arbitrary shape. Substituting the asymptotic relation
(\ref{eq:Ci_limit0}) into the left-hand side of (\ref{sdn:stek_eig0}),
we obtain to leading-order in $a$ and $\eps$ that
\begin{equation} 
\sigma^{(0)} \approx  \sigma^{(0)}_0 \approx \frac{4a(N-1)}{|\PT_1|} 
\,, \quad \mbox{as} \quad a\to 0 \,.
\end{equation}

\subsection{Numerical Comparison}\label{stekDN:num}

A numerical solution of the SDN spectral problem was obtained via a
finite-element method (FEM) described in
\cite{Chaigneau24,Grebenkov25b}.  For an accurate computation, one
needs to ensure that a tetrahedral mesh of the domain is sufficiently
refined near small patches, which requires a numerical diagonalization
of very large matrices.  To avoid these technical issues, we restrict
our analysis to two circular patches ($N = 2$), located at the north
and south poles.  The axial symmetry of this setting reduces the
original three-dimensional setting to a planar one, and is similar to
the formulation of (\ref{eq:planar_dec}).  Since our analysis leading
to Proposition \ref{sdn:main_res} does not access the asymptotic
behavior of non-axially-symmetric eigenfunctions (see Remarks
\ref{sdn:remark1} and \ref{sdn:remark2}), we restrict the numerical
comparison exclusively to the eigenvalues that correspond to axially
symmetric eigenfunctions.

Table \ref{tab:SDN_eigenvalues} illustrates the accuracy of the
three-term asymptotic formula (\ref{sdn:stek_eigex}) by comparing its
predictions (last column) to FEM solutions with different maximal
meshsizes $h_{\rm max}$, with smaller meshsizes yielding more accurate
solutions.  Even though the obtained numerical values of
$\sigma^{(k)}$ did not fully converge to the true eigenvalues, further
refinement of the mesh yielded matrices that are too large to be
treated on a laptop.  Nevertheless, the numerical values reported in
the 6th column are very close to the predictions from our three-term
asymptotic formula given in the last column.  This example serves as a
numerical validation of the asymptotic formula (\ref{sdn:stek_eigex}).
Further analysis of the SDN problem and its applications will be
reported elsewhere.

\begin{table}
\begin{center}
\begin{tabular}{|c|c|c|c|c|c|c|}  \hline
$\eps$ & $h_{\rm max}$  &  0.01  & 0.005 & 0.0025 & 0.002 & asympt. \\  \hline
       & $\sigma^{(0)}$ & 0.536 & 0.548 & 0.555 & 0.556 & 0.556      \\ 
$0.1$  & $\sigma^{(1)}$ & 4.130 & 4.097 & 4.109 & 4.114 & 4.146      \\
       & $\sigma^{(2)}$ & 7.836 & 7.406 & 7.321 & 7.316 & 7.338      \\  \hline
       & $\sigma^{(0)}$ & 0.528 & 0.535 & 0.538 & 0.539 & 0.529  \\
$0.2$  & $\sigma^{(1)}$ & 4.035 & 4.051 & 4.067 & 4.071 & 4.088  \\
       & $\sigma^{(2)}$ & 7.318 & 7.250 & 7.253 & 7.256 & 7.282  \\  \hline
\end{tabular}
\end{center}
\caption{
The first three SDN eigenvalues (that correspond to axially symmetric
eigenfunctions) for the unit sphere with two circular patches of
radius $\eps$ (with $a_1 = a_2 = 1$), located at the north and south
poles.  Columns 3-6 present the numerical results by FEM with
different maximal meshsizes $h_{\rm max}$.  The last column indicates
the three-term asymptotic relation (\ref{sdn:stek_eigex}), which is
seen to compare very favorably with the numerical result on the most
refined mesh.}
\label{tab:SDN_eigenvalues}
% [sigma,sigma0,sigma1,sigma2,  muk,dk] = A_Ward_Steklov_3d_SDN_sphere();
% [mu] = A_Ward_Steklov_3d_FEM_SDN(kshow);
\end{table}

Figure \ref{fig:eigen_SDN1} illustrates the behavior of the first
three axially-symmetric eigenfunctions for circular patches of radii
$\eps_1 = \eps_2 = 0.2$.  For a clearer visualization, we plot these
eigenfunctions along the boundary of the sphere that corresponds to
the plane cross-section at $\phi = 0$ in spherical coordinates
$(r,\theta,\phi)$ (i.e., $r = 1$).  The Steklov patch is located on
the south pole ($\pi-\epsilon_2 < \theta < \pi$), whereas the
Dirichlet patch is located on the north pole ($0 < \theta <
\epsilon_1$), where $\eps_i = \sin(\epsilon_i)$.  On the Steklov
patch $\PT_1$, the first eigenfunction $u^{(0)}$ is positive, whereas
the other eigenfunctions necessarily change the sign to ensure their
orthogonality on $\PT_1$.  In turn, far from both patches, each
eigenfunction is nearly constant, according to the far-field expansion
(\ref{sdn:outex}), with $U_0 \ne 0$, see Remark
\ref{sdn:remark1}.  Note that the constant $U_0$, given by
(\ref{sdn:norm_u0}), depends on the eigenvalue and thus is different
for each eigenfunction.  For comparison, we plot in
Fig. \ref{fig:eigen_SDN2} three other eigenfunctions that are not
axially symmetric.  The axial symmetry of the eigenvalue problem
(\ref{sdn:eig}) for this example allows one to construct
such eigenfunctions by imposing their angular dependence as
$e^{im\phi}$, with an integer $m$ (see the related discussion in
\cite{Grebenkov25}).  For this illustration, we chose $m = 1$.  In
sharp contrast to Fig. \ref{fig:eigen_SDN1}, these eigenfunctions are
essentially concentrated near the Steklov patch, being ${\mathcal
O}(\eps)$ away from the patch.  In fact, one has $U_0 = 0$ for any
non-axially symmetric eigenfunction on the circular patch, see Remark
\ref{sdn:remark2}.

\begin{figure}
  \centering
     \begin{subfigure}[b]{0.49\textwidth}  
      \includegraphics[width =\textwidth]{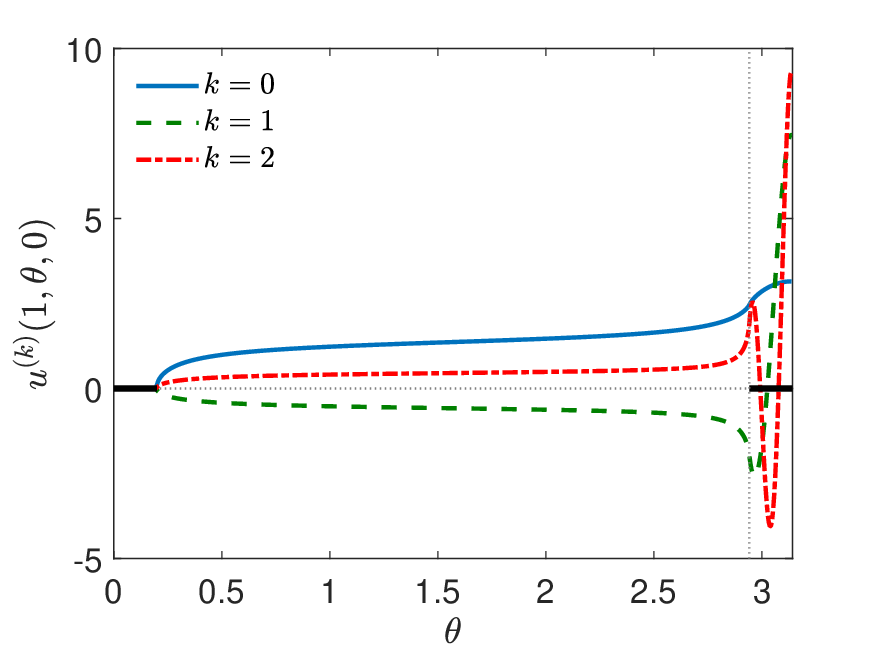} % SDN_eigenfunction_m0.eps}  
        \caption{Axially symmetric case}
        \label{fig:eigen_SDN1}
    \end{subfigure} 
    \begin{subfigure}[b]{0.49\textwidth}
      \includegraphics[width=\textwidth]{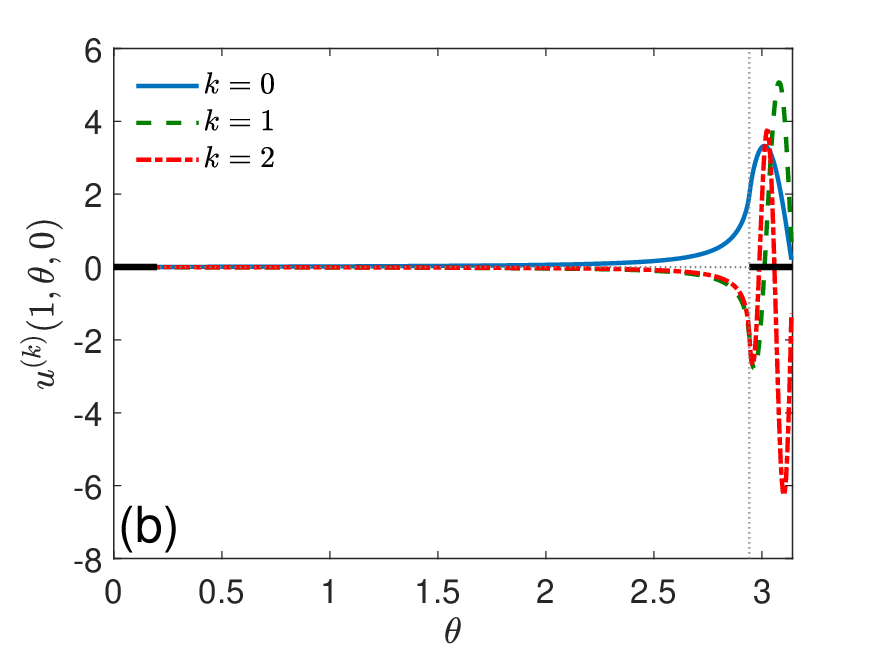} % SDN_eigenfunction_m1.eps} 
        \caption{Non-axially symmetric case} 
        \label{fig:eigen_SDN2}
    \end{subfigure}
\caption{
Eigenfunctions of the SDN problem for the unit sphere with two
circular patches of radii $\eps_1 = \eps_2 = 0.2$ (with $a_1 = a_2 =
1$), located at the north and south poles.  They are plotted in the
spherical coordinates $(r,\theta,\phi)$, by fixing $r = 1$ and $\phi =
0$.  Black thick intervals indicate the locations of the Dirichlet
patch (left, $0 < \theta < \epsilon_2$) and the Steklov patch (right,
$\pi - \epsilon_1 < \theta < \pi$), with $\eps_i = \sin(\epsilon_i)$.
{\bf (a)} The first three axially symmetric eigenfunctions that
correspond to the eigenvalues $\sigma^{(k)}$ reported in Table
\ref{tab:SDN_eigenvalues}. {\bf (b)} The first three non-axially
symmetric eigenfunctions that exhibit the dependence $\propto
e^{i\phi}$.  These FEM solutions \cite{Chaigneau24,Grebenkov25b} were
obtained with the maximal meshsize $\hmax = 0.005$. }
% [mu, v, V, pet] = A_Ward_Steklov_3d_SDN_sphere_eigen(mu, v, V, pet);
\end{figure}

\subsection{Near-Resonant Case}\label{sec:sdn_degen}

Next, we inspect whether there may exist a near-resonant SDN eigepair
for the spectral problem (\ref{sdn:eig}) for which the limiting value
$\sigma_{0}=\lim\limits_{\eps\to 0} \sigma(\eps)$ belongs to
${\mathcal P}_{1}$ (see (\ref{sdnp:poles})).  More specifically, let
us assume that there exists an index $k^{\prime} \geq 0$ such that
$\sigma_0 = \mu_{k^{\prime}1}$ for some {\em simple} eigenvalue
$\mu_{k^{\prime}1}$ of the local Steklov eigenvalue problem
(\ref{eq:Psi_def}) of Appendix \ref{sec:Cmu} for which
$d_{k^{\prime}1}\neq 0$ (see also Remark \ref{rem:simple} below).
We normalize the corresponding local Steklov eigenfunction
$\tilde{\Psi}_{k^{\prime}1}$ to be the unique solution of
\begin{subequations}  \label{snd:Psi_def}
\begin{align}  \label{snd:Vk_eq}
  \Delta_{\y} \tilde{\Psi}_{k^{\prime}1} & = 0 \,, \quad \y \in \R_+^3 \,,\\
  \label{snd:Vk_stek}
  \partial_{y_3} \tilde{\Psi}_{k^{\prime}1} + \sigma_0 \tilde{\Psi}_{k^{\prime}1} &=0\,,
                                 \quad y_3=0 \,,\, (y_1,y_2)\in \PT_1\,,
  \\  \label{snd:Vk_Neumann}
  \partial_{y_3} \tilde{\Psi}_{k^{\prime}1} & = 0 \,,\quad y_3=0 \,,\,
                                     (y_1,y_2)\notin \PT_1\,,
  \\  \label{snd:Vk_inf}
  \tilde{\Psi}_{k^{\prime}1}(\y) & \sim \frac{1}{|\y|} + {\mathcal O}\left(|\y|^{-2}\right)
                  \quad \textrm{as}  \quad |\y|\to \infty\,,
\end{align}
\end{subequations}
where $\sigma_0=\mu_{k^{\prime}1}$ and
$\PT_1 \asymp \eps^{-1}\partial\Omega_{1}^{\eps}$.  Here the tilde symbol
highlights that this normalization is different from the conventional
one used in Appendix \ref{sec:Cmu}.

With $\sigma_0=\mu_{k^{\prime}1}$ we again expand the outer solution,
the eigenvalue $\sigma(\eps)$, and the inner solutions as in
(\ref{sdn:outex}), (\ref{sdn:stek_eig}) and (\ref{sdn:innex}),
respectively. However, in place of (\ref{sdn:v0i}) and
(\ref{sdn:V0isol}), the leading-order inner solutions near the Steklov
patch and the Dirichlet patches are now
\begin{equation}
  V_{01}=A_1 \tilde{\Psi}_{k^{\prime}1}(\y) \,; \qquad
  V_{0i}=U_0\left(1-w_{i}(\y;\infty)\right) \,, \quad i=2,\ldots,N \,,
\end{equation}
where $A_1$ is a constant to be determined. Since $V_{01}\to 0$ as
$|\y|\to\infty$, we can only match the far-field behaviors of the
inner solutions to a leading-order constant outer solution $U_0$ when
$U_0=0$. As a result, since $U_0=0$, we must have $V_{0i}\equiv 0$ for
$i=2,\ldots,N$. In this way, in place of (\ref{sdn:U1prob}) the outer
correction $U_1$ must now satisfy
\bsub\label{sdn:degen:U1prob}
\begin{align}
  \Delta_{\x} U_{1} &= 0 \,, \quad \x\in \Omega \,; \qquad
  \partial_n U_1=0 \,, \quad \x\in \partial\Omega\backslash
  \lbrace{\x_1\rbrace} \,, \label{sdn:degen:U1prob_1}\\
  U_1 & \sim \frac{A_1}{|\x-\x_1|}\,, \quad \mbox{as} \quad
        \x\to\x_1\in \partial\Omega \,.
\end{align}
\esub
The solvability condition for (\ref{sdn:degen:U1prob}) yields that
$A_1=0$.  As a result, $V_{01}\equiv 0$, and we must have
$U_1=\overline{U}_1$ in $\Omega$, where $\overline{U}_1$ is an unknown
constant.

Proceeding to higher order it is readily established that
$\overline{U}_1=0$. We conclude that one cannot construct a
{\em nontrivial} solution to (\ref{sdn:eig}) with limiting behavior
$\sigma_0=\mu_{k^{\prime}1}$.  As a consequence, there is no
near-resonant eigenpair of the SDN problem (\ref{sn:eig}) with a
single Steklov patch.

\begin{remark}  \label{rem:simple}  
From the beginning of \S \ref{sec:sdn_degen}, we restricted our
analysis to the near-resonant cases, for which $\sigma_0 =
\mu_{k^{\prime}1} \in {\mathcal P}_1$ such that $\mu_{k^{\prime}1}$ is simple.
As a consequence, our statement that there is no near-resonant
eigenpair of the SDN problem (\ref{sn:eig}) is established only under
the assumption that all eigenvalues in the resonant set ${\mathcal
P}_1$ are simple.  For a circular patch, numerical evidence suggests
that this assumption does hold, but its rigorous validation presents
an interesting open problem.  For an arbitrary patch, however, the
assumption on the simplicity of eigenvalues from ${\mathcal P}_1$ may
not hold.  However, we expect that this assumption can be relaxed,
i.e., one can use any nontrivial element of the eigenspace associated
to $\mu_{k^{\prime}1}$ in the analysis above.  A proof of this
statement is beyond the scope of this paper.
\end{remark}

%%%%%%%%%%%%%%%%%%%%%%%%%%%%%%%%%%%%%%%%%%%%%%%%%%%%%%%%%%%%%%%%%%%%%%%
\section{The Steklov-Neumann (SN) Eigenvalue Problem}
\label{stekN}

Finally, we address the SN eigenvalue problem (\ref{sn:eig}) in the unit
sphere $\Omega$, with $N\geq 1$ Steklov patches $\pa_i^{\eps}$.
We impose the normalization condition
\begin{equation}\label{sn:eig_4}
  \int_{\partial\Omega_{a}} u^2 \, d{\bf s} =\sum_{i=1}^{N}
  \int_{\partial\Omega_{i}^{\eps}} u^2 d{\bf s} = 1\,.
\end{equation}
This spectral problem has a discrete spectrum, with a countable set of
nonnegative eigenvalues $\{\sigma^{(m)}(\eps)\}$ that are enumerated
by an integer index $m = 0,1,2,\ldots$ that accumulates to infinity
\cite{Levitin}.  Since the principal eigenvalue,
which corresponds to a constant eigenfunction, is
$\sigma^{(0)}(\eps) = 0$, independently of $\eps$, we exclude it from
our analysis below.  Owing to the orthogonality of the Steklov
eigenfunctions, or more simply by applying the divergence theorem to
(\ref{sn:eig}), we must have for any other Steklov eigenvalue
$\sigma^{(m)}(\eps) > 0$ with $m = 1,2,\ldots$ that the
corresponding eigenfunction $u^{(m)}$ satisfies
\begin{equation}\label{sn:eig_5}
  \int_{\partial\Omega_a} u^{(m)} \, d{\bf s}=\frac{\eps}{\sigma^{(m)}(\eps)}
 \sum_{i=1}^{N} \int_{\partial\Omega_i^{\eps}} \partial_{n} u^{(m)} \, d{\bf s}=0 \,.
\end{equation}
For convenience, in our analysis below we omit the superscript
$^{(m)}$ to highlight that our asymptotic analysis is not specific to
a particular value of $m$.

In analogy with the SDN problem studied in \S \ref{stekDN:intro}, we
need to consider both non-resonant and near-resonant cases.  While the
near-resonant case was not possible for the SDN problem with a single
Steklov patch, it will present one of the challenging features of the
SN analysis here.

We introduce the resonant set
\begin{equation}\label{sn:poles0} 
  {\mathcal P} \equiv  \bigcup_{i=1}^{N} {\mathcal P}_i\,, \quad
  \mbox{where} \quad
  {\mathcal P}_i \equiv \bigcup_{k=0}^{\infty} \left\{ \mu_{ki} ~\vert~ d_{ki}
    \ne 0 \right\}\,,
\end{equation}
where $\mu_{ki}$ with $k = 0,1,\ldots$ are the eigenvalues of the {\em
local} Steklov problem (\ref{eq:Psi_def}) near the $i$-th Steklov
patch $\PT_i$ for $i=1,\ldots,N$, with nontrivial spectral weights
$d_{ki} \ne 0$ (see of Appendix \ref{sec:Cmu} for details).  We
distinguish two situations according to the limiting value
$\sigma_{0}=\lim\limits_{\eps\to 0}\sigma(\eps)$ of a SN eigenvalue
$\sigma = \sigma(\eps)$:

(I) If $\sigma_0 \notin {\mathcal P}$, the eigenpair $\{\sigma, u\}$
is called non-resonant. In this case, both $C_i(-\sigma_0)$ and
$w_{i}(\y;-\sigma_0)$ are well-defined for each $i=1,\ldots,N$. The
analysis for this non-resonant case, which is similar to that done in
\S \ref{sdn_sec:asymptotics} for the SDN problem, will be performed in
\S \ref{sec:sn:non}.

(II) If $\sigma_0 \in {\mathcal P}$, the eigenpair $\{\sigma, u\}$ is
called near-resonant. In this case, some $w_{i}(\y;-\sigma_0)$ may be
undefined and thus must be replaced by suitable eigenfunctions of the
local Steklov problem (\ref{eq:Psi_def}) associated to $\sigma_0$.
Even though the asymptotic analysis of the near-resonant case is
possible to undertake in more generality, we will restrict our
attention below to one relevant setting.  More specifically, we assume
that $N\geq 2$ and that there are exactly $M$ {\em identical} patches,
with $2\leq M\leq N$, which we relabel by
$\partial\Omega_{1}^{\eps}=\ldots=\partial\Omega_{M}^{\eps}\equiv
\partial\Omega_{c}^{\eps}$.  For these identical patches, with a common
patch shape $\PT_c\asymp \eps^{-1}\partial\Omega_{c}^{\eps}$, there is
a common spectrum $\{\mu_{kc}\}_{k\geq 0}$ of the local Steklov
problem (\ref{eq:Psi_def}) of Appendix \ref{sec:Cmu}.  In \S
\ref{sec:sn_degen} we will show that the global SN problem
(\ref{sn:eig}) admits $M-1$ {\em near-resonant} eigenvalues (counting
multiplicity) such that $\sigma_0 \in {\mathcal P}$.  The
corresponding  eigenfunctions will be shown to concentrate on
the $M$ near-resonant patches.  We will also derive a three-term
asymptotic behavior for these global SN eigenvalues.

\begin{remark}
We remark that for this setting, it is essential that $M>1$.  In fact,
if the near-resonant condition occurs on only one patch, i.e. if
$\sigma_0=\mu_{ki}$ for some simple eigenvalue $\mu_{ki}$ of
(\ref{eq:Psi_def}) of a unique patch $i\in \lbrace{1,\ldots,N\rbrace}$
with $d_{ki}\neq 0$, it is readily shown, as similar to that done in
\S \ref{sec:sdn_degen} for the SDN problem, that the SN problem
(\ref{sn:eig}) only admits the trivial solution (see also Remark
\ref{rem:simple}).  Hence, for $M=1$, there is no SN eigenpair for
(\ref{sn:eig}) with such limiting asymptotics $\sigma_0$. We emphasize
that when there are two or more identical patches, our analysis for
the non-resonant case will only capture a subset of the SN eigenvalues
for the global problem (\ref{sn:eig}). The remaining global SN
eigenvalues will be in near-resonance with eigenvalues of the local
Steklov problem (\ref{eq:Psi_def}) on the identical patches. The
corresponding eigenfunctions will concentrate on these identical
patches. We remark that more intricate near-resonant cases are
possible for specific non-generic situations such as when there are
two patch indices $i_1$ and $i_2$ and two eigenvalue indices $k_1$ and
$k_2$ for which $\mu_{k_1 i_1} = \mu_{k_2 i_2}$.  Although the
construction of a SN eigenpair for which $\sigma_0 = \mu_{k_1 i_1} =
\mu_{k_2 i_2}$ can be done in a similar way as for the identical patch
case undertaken in \S \ref{sec:sn_degen}, we will not consider this
special case below.
\end{remark}

\subsection{Non-Resonant Case}\label{sec:sn:non}

Since the analysis for the SN problem (\ref{sn:eig}) in the
non-resonant case is very similar to that done in \S
\ref{sdn_sec:asymptotics} for the SDN problem (\ref{sdn:eig}), we will
only briefly outline the main steps to determine $\sigma(\eps)$.  In
analogy to the solution of the SDN problem, we seek to construct
eigenpairs of (\ref{sn:eig}) for which the leading-order outer
solution $U_0$ is non-vanishing (i.e.  $U_0\neq 0$).

In the outer region, we expand $u$ as in (\ref{sdn:outex}) to obtain
(\ref{sdn:Uk}) at each order of the expansion. In addition, each
nontrivial Steklov eigenvalue is expanded as in
(\ref{sdn:stek_eig}). In the inner region near each Steklov patch, we
expand the inner solution in terms of geodesic coordinates as in
(\ref{sdn:innex}), to derive that $V_{ji}$ for $j=0,1,2$, and for each
$i=1,\ldots,N$, satisfies
\bsub \label{sn_s:Vi}
\begin{align}
  \Delta_{\y} V_{ji} &= \delta_{j2} \left( 2y_{3} V_{0i,y_3 y_3} + 2 V_{0i,y_3}
                       \right) \,, \quad
   \y \in \R_{+}^{3} \,, \label{sn_s:Vi_1}\\
  \partial_{y_3} V_{ji} + \sigma_{0} V_{ji} &=-(1-\delta_{j0})
        \sigma_{j} V_{0i} \,, \quad y_3=0 \,,\,
    (y_1,y_2)\in \PT_i\,,  \label{sn_s:Vi_2}\\
    \partial_{y_3} V_{ji} &=0 \,, \quad y_3=0 \,,\, (y_1,y_2)\notin \PT_i
    \,. \label{sn_s:Vi_3}
\end{align}
\esub 

For each Steklov patch, we set  $w_{i}=w_{i}(\y;-\sigma)$ as the
solution of (\ref{mfpt:wc}), where the dependence of $w_i$, $C_i$ and
$\DT_i$ on $\sigma$ is obtained by setting $\kappa_i=-\sigma$ in
(\ref{mfpt:wc}).  As similar to the analysis of the SDN problem in
 \S \ref{stekDN:intro}, we define
\begin{equation}  \label{sneq:wc_def}
w_{ci}=w_{ci}(\y;-\sigma)\equiv \partial_{\sigma} w_{i}(\y;-\sigma),
\end{equation}
which satisfies, for each $i=1,\ldots,N$, the following inner problem:
\bsub \label{snp:wc}
\begin{align}
    \Delta_{\y} w_{ci} &=0 \,, \quad \y \in \R_{+}^{3} \,, \label{snp:wc_1}\\
    \partial_{y_3} w_{ci} + \sigma w_{ci} &=1-w_i \,, \quad y_3=0 \,,\,
    (y_1,y_2)\in \PT_i\,,  \label{snp:wc_2}\\
    \partial_{y_3} w_{ci} &=0 \,, \quad y_3=0 \,,\, (y_1,y_2)\notin \PT_i
    \,, \label{snp:wc_3}\\
  w_{ci}&\sim - \frac{C_{i}^{\prime}(-\sigma)}{|\y|} -
                {  \frac{\DT_i^{\prime}(-\sigma) {\bf \cdot} \y}{|\y|^3}}
         + \cdots\,,  \quad \mbox{as}\quad
    |\y|\to \infty \,. \label{snp:wc_4}
\end{align}
\esub 

Since $\sigma_0 \notin {\mathcal P}$, it follows that $C_i(-\sigma_0)$
and $w_{i}(\y;-\sigma_0)$ are well-defined.  Then, in terms of the
constant leading-order outer solution $U_0$, which will be found below
by the normalization condition (\ref{sn:eig_4}), the leading order
inner solution for each $i=1,\ldots,N$, as obtained by setting $j=0$
in (\ref{sn_s:Vi}), is
\begin{equation}\label{sn:v0i}
  V_{0i}=U_0 \left(1 - w_{i}(\y;-\sigma_{0})\right) \,.
\end{equation}

Upon matching the far-field of $V_{0i}$ to the outer solution, we find
that $U_1$ satisfies
\bsub\label{sn:U1prob}
\begin{align}
  \Delta_{\x} U_{1} &= 0 \,, \quad \x\in \Omega \,; \qquad
  \partial_n U_1=0 \,, \quad \x\in \partial\Omega\backslash
  \lbrace{\x_1,\ldots,\x_N\rbrace} \,, \label{snn:U1prob_1}\\
  U_1 & \sim -\frac{C_i(-\sigma_{0})U_0}{|\x-\x_i|}\,,
  \quad \mbox{as} \quad
  \x\to\x_i\in \partial\Omega \,, \quad i=1,\ldots,N \,.
\end{align}
\esub The solvability condition for (\ref{sn:U1prob}) is that
\begin{equation}\label{sn:solv1}
  U_0 \sum_{i=1}^{N} C_i(-\sigma_0)=0 \,.
\end{equation}
For $U_0\neq 0$, we conclude that the leading-order Steklov
eigenvalue $\sigma_0$ is a root of the following nonlinear algebraic
equation:
\begin{equation}\label{sn:sigma_0}
  {\mathcal N}(\sigma_0)=0 \,, \quad \mbox{where} \quad
  {\mathcal N}(\sigma_{0}) \equiv \sum_{i=1}^{N} C_i(-\sigma_0)\,.
\end{equation}
The spectral expansion (\ref{eq:Cmu_def0}) ensures that
${\mathcal N}(\sigma_0)$ increases monotonically between its
consecutive poles so that (\ref{sn:sigma_0}) has infinitely many
solutions that we denote as $\sigma_0^{(k)}$.  These solutions lie
between consecutive poles but finding their explicit locations is in
general more difficult than for the SDN problem with a single Steklov
patch.  As earlier, we omit the superscript $^{(k)}$ for brevity.

With $\sigma_0$ determined in this way, the solution to
(\ref{sn:U1prob}) for $U_{1}$ is written in terms of the
surface Neumann Green's function $G_s$ and an unknown constant
$\overline{U}_1$ as
\begin{equation}\label{sn:U1sol}
  U_1(\x) = \overline{U}_1 -2\pi U_0 \sum_{j=1}^{N} C_j(-\sigma_0) G_{s}(\x;\x_j)
  \,.
\end{equation}

To proceed to higher order, we expand $U_1$ as $\x\to \x_i$ to derive
(\ref{sdn:U1_expan}), where we now label $C_i=C_i(-\sigma_0)$ for
$i=1,\ldots,N$.  Upon matching to the inner solution we conclude that
$V_{1i}\sim {U_0C_i/2}$ as $|\y|\to \infty$ for $i=1,\ldots,N$. Upon
solving the problem (\ref{sn_s:Vi}) for $V_{1i}$ with this limiting
behavior, we obtain that
\begin{equation}\label{sn:V1isol}
    V_{1i} = -U_{0} \sigma_{1} w_{ci}(\y;-\sigma_0) +
           \frac{U_0 C_i(-\sigma_0)}{2}
           \left(1 - w_{i}(\y;-\sigma_{0})\right) \,, \quad
           i=1,\ldots,N \,.
\end{equation}
The far-field behavior for $V_{1i}$ as $|\y|\to \infty$ is
\begin{equation}\label{sn:Visol_ff}
  V_{1i} \sim U_0 \sigma_{1}\frac{ C_{i}^{\prime}(-\sigma_0)}
           {|\y|} + \frac{U_0 C_i(-\sigma_0)}{2}
           \left(1 - \frac{C_i(-\sigma_{0})}{|\y|} \right) +\cdots \,,
\end{equation}
where the neglected higher-order far-field terms are dipole contributions.

The monopole terms in (\ref{sn:Visol_ff}) provide the singularity
behavior for the outer correction $U_2$ in (\ref{sdn:outex}). In this
way, we find that $U_2$ satisfies (\ref{sdn:U2prob_1}) subject to
\begin{equation}\label{sn:U2prob_2}
  U_2  \sim -\frac{U_0}{2}\frac{\left[C_i(-\sigma_0)\right]^2}
  {|\x-\x_i|} + U_0 \sigma_{1}
        \frac{ C_{i}^{\prime}(-\sigma_0)}{|\x-\x_i|}\,,
          \quad \mbox{as} \quad
  \x\to\x_i\in \partial\Omega \,, \quad i=1,\ldots, N \,.
\end{equation}
The solvability condition for this problem for $U_2$ is that
\begin{equation}\label{sn:sigma_10}
  -\frac{U_0}{2} \sum_{i=1}^{N} \left[ C_i(-\sigma_0)\right]^2 +
  U_0 \sigma_1  \sum_{i=1}^{N} C_i^{\prime}(-\sigma_0)=0 \,.
\end{equation}
Under the condition that $U_0\ne 0$, (\ref{sn:sigma_10}) determines
$\sigma_1$ as
\begin{equation}\label{sn:sigma_1}
  \sigma_{1} = \frac{1}{2} \frac{\sum_{i=1}^{N} \left[
      C_i(-\sigma_0)\right]^2}{\sum_{i=1}^{N} C_i^{\prime}(-\sigma_0)
    } \,,
\end{equation}
where $\sigma_0$ is a root of (\ref{sn:sigma_0}). From
(\ref{eq:dCmu}), it follows that $C_i$ increases monotonically between
its poles so that $C_i^{\prime}(-\sigma_0)\neq 0$.  As a result, the
denominator in (\ref{sdn:sigma_1}) never vanishes, and $\sigma_1$ is
well-defined and strictly positive.  With $\sigma_1$ determined in
this way, the solution to (\ref{sdn:U2prob_1}) with
(\ref{sn:U2prob_2}) is given in terms of an unknown constant
$\overline{U}_2$ by
\begin{equation}\label{sn:U2sol}
  U_2(\x) = \overline{U}_2
  -2\pi U_0 \sum_{j=1}^{N} \left( \frac{\left[C_j(-\sigma_0)
      \right]^2}{2} -\sigma_{1} C_j^{\prime}(-\sigma_0)
  \right) G_{s}(\x;\x_j) \,.
\end{equation}

Finally, we determine $\sigma_2$. We readily obtain (\ref{sdn:v2_ff})
for the far-field behavior for the inner correction $V_{2i}$, which
satisfies (\ref{sn_s:Vi}) with $k=2$. In analogy with (\ref{sdn:v2sol_1}),
we determine $V_{2i}$ for $i=1,\ldots,N$ as
\begin{equation}\label{sn:v2sol_i}
V_{2i} = U_0 \Phi_{2i} + \left(U_0 \beta_i +\overline{U}_1\right)
\left(1-w_{i}(\y;-\sigma_0)\right)
  - U_0 \sigma_2 w_{ci}(\y;-\sigma_{0}) \,,
\end{equation}
where $\Phi_{2i}$ satisfies (\ref{mfpt_b:Phi2}) in which we set
$\kappa_i=-\sigma_{0}$ and $C_i=C_i(-\sigma_0)$.  As a result, the
refined far-field behavior of $V_{2i}$ for each $i=1,\ldots,N$ is
\begin{equation}\label{sn:V2iff_refine}
  \begin{split}
    V_{2i} &\sim U_0 \beta_i +\overline{U}_1 +
    \frac{U_0 C_i(-\sigma_{0})}{2}
    \left(
      \log(y_3 + |\y|) - \frac{y_3 (y_1^2 + y_2^2)}{|\y|^3} \right) \\
    &\qquad + \left[ E_i(-\sigma_{0})-
      \left(\beta_i+\frac{\overline{U}_1}{U_0}\right)
      C_i(-\sigma_0)+ \sigma_{2}
        C_i^{\prime}(-\sigma_0)\right] \frac{U_0}{|\y|}\,,
    \quad \mbox{as} \quad |\y| \to \infty \,.
  \end{split}
\end{equation}
Here $E_i(-\sigma_{0})$ is obtained by setting $\kappa_i=- \sigma_0$
in (\ref{eq:Ei_general0}).

As similar to the analysis of the SDN problem in \S \ref{sdn_sec:asymptotics},
the monopole terms in the far-field behavior (\ref{sn:V2iff_refine}),
together with dipole term in the far-field of $V_{0i}$, provide the
singularity behavior for the outer correction $U_3$. In this way, we get
that $U_3$ satisfies
\begin{equation}\label{sn:U3prob}
  \begin{split}
  \Delta_{\x} U_{3} &= 0 \,, \quad \x\in \Omega \,; \qquad
  \partial_n U_3=0 \,, \quad \x\in \partial\Omega\backslash
  \lbrace{\x_1,\ldots,\x_N\rbrace} \,, \\
  U_3  & \sim \left[ E_i(-\sigma_0) - \left(\beta_i +
      \frac{\overline{U}_1}{U_0}\right)C_i(-\sigma_0) + \sigma_2
    C_{i}^{\prime}(-\sigma_0)\right] \frac{U_0}{|\x-\x_i|} \\
  & \qquad -U_0 \frac{\DT_i(-\sigma_0) {\bf \cdot} {\mathcal Q}_i^T
    (\x-\x_i)}{|\x-\x_i|^3}\,,
  \quad \mbox{as} \quad
  \x\to\x_i\in \partial\Omega \,, \quad i=1,\ldots, N \,,
\end{split}
\end{equation}
where the orthogonal matrix ${\mathcal Q}_i$ is defined in
(\ref{app_g:change}) in terms of the basis vectors of the geodesic
coordinate system. The solvability condition for (\ref{sn:U3prob}) is
that 
\begin{equation}\label{sn:sigma_20}
  U_0 \left( \sum_{i=1}^{N} E_{i}(-\sigma_0) - \sum_{i=1}^{N}
    \beta_i C_{i}(-\sigma_0) +\sigma_2 \sum_{i=1}^{N} C_{i}^{\prime}(-\sigma_0)
    \right) - \overline{U}_1 \sum_{i=1}^{N} C_{i}(-\sigma_0)=0\,,
\end{equation}
which determines $\sigma_2$.  Since $\sigma_0$ satisfies
$\sum_{i=1}^{N} C_i(-\sigma_0)=0$ when $U_0\neq 0$, as seen from
(\ref{sn:sigma_0}) and (\ref{sn:solv1}), we observe that $\sigma_2$ is
independent of the unknown normalization constant $\overline{U}_1$. As
a result, upon recalling that
$\beta_i=-2\pi \left({\mathcal G}_s\vc\right)_i$ from
(\ref{sdn:beta}), we conclude from (\ref{sn:sigma_20}) that, when
$U_0\neq 0$, $\sigma_2$ is given by
\begin{equation}\label{sn:sigma_2}
  \sigma_{2} = -\frac{1}{\sum_{i=1}^{N} C_i^{\prime}(-\sigma_0)}
  \left( 2\pi \vc^{T} {\mathcal G}_s \vc + \sum_{i=1}^{N}
    E_i(-\sigma_0) \right) \,,
\end{equation}
where ${\mathcal G}_s$ is the Green's matrix and
$\vc=\left(C_1(-\sigma_0),\ldots,C_N(-\sigma_0)\right)^T$.
We summarize our result as follows:

\begin{prop}\label{sn:main_res} 
As $\eps \to 0$, consider the eigenvalues $\sigma=\sigma(\eps)$ of the
Steklov-Neumann problem (\ref{sn:eig}), for which
$\sigma(\eps)\to\sigma_0 \notin {\mathcal P}$, where the resonant set
${\mathcal P}$ is defined by (\ref{sn:poles0}). These Steklov
eigenvalues and the associated eigenfunctions, restricted to $\PT_i$,
have the three-term asymptotics
\bsub
\begin{align}\label{sn:stek_eigex}
  \sigma &=\sigma_{0}+
  \eps\log\left(\frac{\eps}{2} \right) \sigma_{1} + \eps
  \sigma_{2} + {\mathcal O}(\eps^2\log \eps )\,, \\
  u\vert_{\PT_i} &= V_{0i} + \eps\log\left(\frac{\eps}{2} \right)
                               V_{1i} + \eps
  V_{2i} + {\mathcal O}(\eps^2\log \eps)\,, \quad i=1,\ldots,N \,,
\end{align}
\esub where $\sigma_{0}$, $\sigma_{1}$ and $\sigma_{2}$ are
respectively determined by (\ref{sn:sigma_0}), (\ref{sn:sigma_1}) and
(\ref{sn:sigma_2}).  Moreover, $V_{0i}$, $V_{1i}$ and $V_{2i}$ for
each $i=1,\ldots,N$ are respectively given by (\ref{sn:v0i}),
(\ref{sn:V1isol}), and (\ref{sn:v2sol_i}).  Here $u|_{\PT_i}$ is given
up to constants $U_0\neq 0$ and $\overline{U}_1$, whereas its spatial
behavior is determined by the functions $w_i(\y;-\sigma_0)$ and
$w_{ci}(\y;-\sigma_0)$, which admit the spectral expansions
(\ref{eq:wi_spectral}) and (\ref{eq:wc_spectral}), respectively.  For
a circular patch, these expansions can be readily calculated
numerically, as shown in Appendix
\ref{sec:Cmu}.
\end{prop}

By using the normalization condition (\ref{sn:eig_4}), together with
the leading-order inner solutions (\ref{sn:v0i}), we can determine
$U_0$ as
\begin{equation}\label{sn:norm_u0} 
  U_0 \sim \eps^{-1} \left[\sum_{i=1}^{N} \int_{\PT_i}
      \left[1-w_i(\y;-\sigma_0)\right]^2 d\y\right]^{-1/2} \,.
\end{equation}
In turn, the constant $\overline{U}_1$ that appears in $V_{2i}$,
remains unknown and can only be found at higher order.  Since $U_0\neq
0$, the class of SN eigenpairs given in Proposition
\ref{sn:main_res} results from a global interaction of the Steklov
patches through the outer (bulk) solution.

\begin{remark} \label{sn:remark2} 
In analogy with the SDN eigenvalue problem analysis (see Remarks
\ref{sdn:remark1} and \ref{sdn:remark2}), the analysis leading to
Proposition \ref{sn:main_res} does not give access to {\em all} SN
eigenpairs.  In particular, as for the SDN problem, there are
eigenpairs for which the leading-order bulk solution $U_0$ vanishes.
For circular Steklov patches, this situation will always occur for
eigenfunctions that are not axially symmetric on the patches.  To
leading order, these SN eigenvalues have limiting behavior
$\sigma_0=\mu_{ki}^{N}$, for some index $k\geq 1$ and patch index
$i\in \lbrace{1,\ldots,N\rbrace}$, where $\mu_{ki}^{N}>0$ is an
eigenvalue of the local Steklov problem (\ref{SDN:Psi_def_N}) in which
$\Gamma_1$ is replaced by $\Gamma_i$. The corresponding eigenfunction
concentrates on the $i$-th patch and is only weakly influenced by the
other patches.  In addition, other SN eigenvalues corresponding to the
near-resonant case will be recovered in \S \ref{sec:sn_degen}.
\end{remark}

\subsection{Numerical Comparison}\label{sn:example:nondegen}

In \S \ref{mfpt_sec:numerics}, \S \ref{split_sec:numerics} and \S
\ref{stekDN:num}, we used a finite-element method for validating the
asymptotic formulas.  However, obtaining accurate numerical results by
this method for small patches requires using very fine meshes, which
typically results in prohibitively long computations.  To achieve a
more accurate computation of the SN eigenvalues for a single patch or
for two antipodal patches on a sphere, in Appendix \ref{appf:numer} we
outline an alternative spectral method, based on \cite{Grebenkov19b},
which exploits properties of axially symmetric harmonic functions.
Using this more refined numerical approach we now give two examples to
illustrate our main result (\ref{sn:stek_eigex}).

In addition, we inspect the possible advantage of using of an
alternative ``geodesic convention'' for the radius of a curved patch
in numerical computations.  We recall that $\eps$ was introduced in
(\ref{intro:scalings}) as the maximal half-diameter of the orthogonal
projection of the rescaled patch $\pa_i^{\eps}$ when mapped onto the
tangent plane to the sphere at $\x_i$. This is the ``standard
convention'' that we used in the numerical examples of previous
sections.  Denoting by $\epsilon$ the polar angle of the spherical
cap, we identify that $\epsilon = \sin^{-1}(\eps)$, and so we can
interpret $\epsilon$ as the ``geodesic radius'' of the cap.
For small patches, one has $\epsilon \sim \eps (1 + {\mathcal
O}(\eps^2))$, so that replacing $\eps$ by $\epsilon$ has no effect on
our asymptotic formulas, up to the relative error of order ${\mathcal
O}(\eps^2)$.  However, if $\eps$ is not small, this relative error can
deteriorate the accuracy of the asymptotic formulas.  In fact, since
the flattened patch $\PT_i$ aims at reproducing the effect of the
curved patch $\pa_i^{\eps}$, its rescaling by the ``geodesic radius''
$\epsilon$ instead of $\eps$ would preserve the surface area of the
patch and thus may provide more accurate results (see the remark at
the end of Appendix \ref{app_g:geod}).  In this ``geodesic
convention'', we retain $\eps$ in our asymptotic formulas but will
perform numerical computations for spherical caps with the geodesic
radius set to be $\eps$, not $\sin^{-1}(\eps)$, as earlier.  In Table
\ref{tab:SN} below, we compare the asymptotic formulas with the
numerical results obtained by both standard and geodesic conventions,
and the latter turns out to yield a better agreement.  For this
reason, we adopt the ``geodesic convention'' for the numerical
examples of this section.  We stress that the convention concerns
exclusively the numerical part (the results obtained via the spectral
method of Appendix \ref{appf:numer}) and does not alter any asymptotic
formula.

\vspace*{0.2cm}
\noindent {\bf Example I:} For a single circular patch $\pa_1^\eps$ of radius
$\eps$ (i.e., $N = 1$ and $a_1 = 1$), the condition (\ref{sn:sigma_0})
for $\sigma_0$ reads as $C_1(-\sigma_0)=0$.  The spectral expansion
(\ref{eq:Cmu_def0}) of $C_1(\kappa_1)$ implies that it has infinitely
many nontrivial zeros, which we denote as $-\mu_{k1}^N$ with
$k=1,2,\ldots$ (see Fig.~\ref{fig:Cmu}).  As shown in Lemma
\ref{lem:Cmu} of Appendix \ref{sec:Cmu0}, these roots $\mu_{k1}^{N}$
are in fact eigenvalues of a local Steklov eigenvalue problem
(\ref{eq:Psi_def_N}) defined near the patch $\PT_1$, for which the
corresponding eigenfunctions satisfy a far-field Neumann-like
condition (\ref{eq:VkN_inf}).  These zeros determine the leading-order
behavior of the associated SN eigenvalues for the global SN problem
(\ref{sn:eig}) as $\sigma_0^{(k)} =
\mu_{k1}^N$.  This leading-order behavior was studied in
\cite{Grebenkov25}, and the numerical values of $\mu_{k1}^{N}$ were
reported in Table I of \cite{Grebenkov25}.  In turn, the asymptotic
formula (\ref{sn:stek_eigex}) gives the next-order corrections as
$\sigma_1=0$ from (\ref{sn:sigma_1}) and
$\sigma_2=-{E_1(-\sigma_0)/C_1^{\prime}(-\sigma_0)}$ from
(\ref{sn:sigma_2}).  In this way, from (\ref{sn:stek_eigex}), the
first four SN eigenvalues (that correspond to axially symmetric
eigenfunctions) are predicted to have the two-term asymptotic behavior
\begin{subequations}  \label{sn:example_1all}
\begin{align}\label{sn:example_1}
  \sigma_{\rm asy}^{(1)}
  & \sim 4.121 -0.573\, \eps+ {\mathcal O}(\eps^2\log\eps)\,, \quad
    \sigma_{\rm asy}^{(2)} \sim 7.342 -0.552\, \eps+
    {\mathcal O}(\eps^2\log\eps)\,, \\
  \sigma_{\rm asy}^{(3)}
  & \sim 10.517 - 0.542\, \eps+ {\mathcal O}(\eps^2\log\eps)\,, \quad
    \sigma_{\rm asy}^{(4)} \sim 13.677 - 0.535\, \eps+
    {\mathcal O}(\eps^2\log\eps)\,.
\end{align}
\end{subequations}

In Fig.~\ref{fig:SN_patch1} we plot the difference between the
numerically computed values $\sigma^{(k)}_{\rm num}$ of the first two
SN eigenvalues and their asymptotic approximations in
(\ref{sn:example_1}), as a function of $\eps^2$.  The observed linear
dependence on $\eps^2$ indicates that: (i) the first two terms of
(\ref{sn:example_1}) are correct, and (ii) the next-order term is
${\mathcal O}(\eps^2)$ as the coefficient of the error term ${\mathcal
O}(\eps^2\log\eps)$ seems to vanish.  We remark that the contour
plots of the related eigenfunctions were shown in Fig. 6 of
\cite{Grebenkov25}.

\begin{figure}
\begin{center}
\includegraphics[width=80mm]{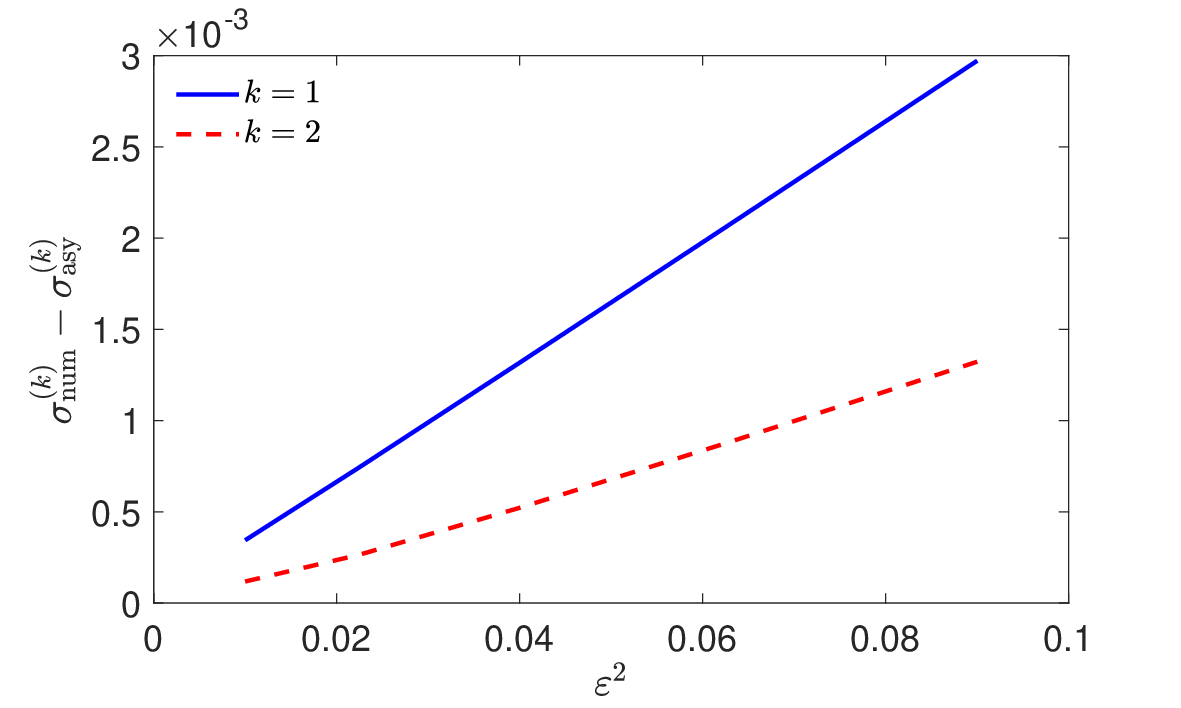} % SN_patch1_error_new.eps}
\end{center}
\caption{ 
Asymptotic behavior of the first two SN eigenvalues (that correspond
to axially-symmetric eigenfunctions) for a single circular patch of
radius $\eps$ on the unit sphere.  Each curve shows the difference
between the numerical value $\sigma^{(k)}_{\rm num}$, computed using
the spectral method from Appendix \ref{appf:numer} (with the
truncation order $\nmax = 4000$ and the ``geodesic convention''), and
its asymptotic value $\sigma^{(k)}_{\rm asy}$ given in
(\ref{sn:example_1}).  This difference is shown as a function of
$\eps^2$ to highlight the order of the error estimate in
(\ref{sn:example_1}).}
\label{fig:SN_patch1}
% [Mu2] = A_Ward_Steklov_3d_sphereSN1_fig;
\end{figure}

\vspace*{0.2cm}

\noindent {\bf Example II:} Next, we consider the special case of $N$
circular Steklov patches of distinct radii $a_i>0$, for
$i=1,\ldots,N$, so that $a_i\neq a_j$ for all $i\neq j$.  Then, using
the scaling law (\ref{eq:Cmu_scaling}), we obtain from
(\ref{sn:sigma_0}) that $\sigma_{0}^{(k)}$ are the roots of ${\mathcal
N}(\sigma_0)=0$, where
\begin{equation}\label{sn:sigma_iden_disk}
  {\mathcal N}(\sigma_0) = \sum_{i=1}^{N} a_i {\mathcal C}(-\sigma_{0} a_i)\,.
\end{equation}
Here ${\mathcal C}(\mu)$ is readily computed via the spectral
expansion (\ref{eq:Cmu_def0}) for any rescaled patch $\PT_i/a_i$.
Owing to the monotonicity of ${\mathcal N}(\sigma_0)$ between
consecutive poles, which readily follows from the monotonicity of
${\mathcal C}(\mu)$ established in (\ref{eq:dCmu}), we conclude that
between any two consecutive poles of ${\mathcal N}(\sigma_0)$ the
function ${\mathcal N}(\sigma_0)=0$ must have a unique root.  For each
such root, the asymptotic result (\ref{sn:stek_eigex}) can then be
used to determine a three-term asymptotic expansion for this
particular SN eigenvalue.

To illustrate this result, we numerically compute the eigenvalues of
the SN problem (\ref{sn:eig}) for two circular patches of radii
$a_1\eps$ and $a_2\eps$ (with $a_2 = 1$), located at the north and
south poles of the unit sphere (note that the trivial principal
eigenvalue $\sigma^{(0)} = 0$ will be excluded from our discussion; we
also focus on the eigenvalues that correspond to axially symmetric
eigenfunctions).  Figure~\ref{fig:sigma_SN1} shows an excellent
agreement between the asymptotic result in (\ref{sn:stek_eigex}) and
the numerically computed eigenvalues as the radius $a_1$ of the
smaller patch is varied on $(0,1)$.  We observe that as $a_1\to 0$, we
recover the eigenvalues, written in the form $\mu_{k2}^{N}$, for the
SN problem with a single Steklov patch $\PT_2$.
The observed behavior of the eigenvalues allows us to push the
analogy to a single Steklov patch even further.  When $\eps$ is small,
one might expect that the two well-separated patches do not almost
``feel'' each other.  This (over-)simplified picture suggests that, to
leading order, the spectrum of the SN problem with two patches
would be the union of the spectra of the two SN problems with a single
patch, either $\PT_1$, or $\PT_2$.  In Fig.~\ref{fig:sigma_SN2},
four thin horizontal lines present the asymptotic values $\sigma_{\rm
asy}^{(k)}$ from (\ref{sn:example_1all}) of the first four eigenvalues
for a single Steklov patch $\PT_2$ (as if $\PT_1$ was absent).
In turn, the thick solid and dashed lines present the asymptotic values
$\sigma_{\rm asy}^{(k)}/a_1$ from (\ref{sn:example_1all}) of the first
two eigenvalues for a single Steklov patch $\PT_1$ (as if
$\PT_2$ was absent).  For comparison, symbols show the numerically
computed eigenvalues of the SN problem with two patches; these symbols
are identical with those shown in Fig.~\ref{fig:sigma_SN1} but just
colored differently.  One observes an excellent agreement between
symbols and curves that {\em partly} validates the intuitive idea of
the patches not feeling each other.  However, there are points (shown
by triangles) that are not captured by either of the asymptotic
relations for single patches.  To outline their dependence on $a_1$,
we added the dashed curve $0.95/a_1$, in which the prefactor $0.95$
was obtained from fitting, i.e., it does not correspond to any
limiting eigenvalue $\mu_{ki}^N$.  The presence of such points
highlights that the interaction between two patches is still relevant
but it mainly affects the smallest (nontrivial) eigenvalue.

\begin{figure}
  \centering
     \begin{subfigure}[b]{0.49\textwidth}  
      \includegraphics[width =\textwidth]{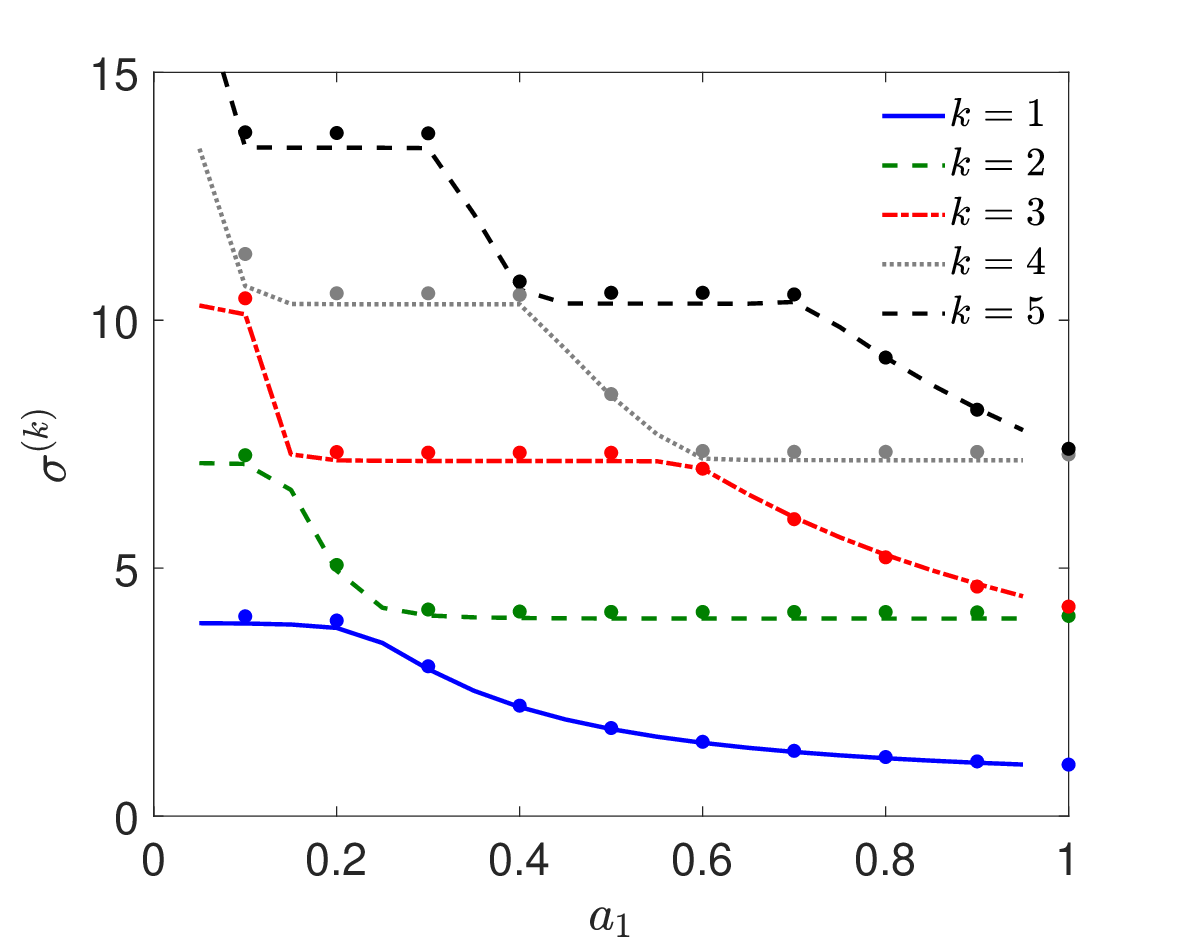} % sigma_SN_new4.eps} 
        \caption{Comparison with asymptotics (\ref{sn:stek_eigex})}
        \label{fig:sigma_SN1}
    \end{subfigure}  
    \begin{subfigure}[b]{0.49\textwidth}
      \includegraphics[width=\textwidth]{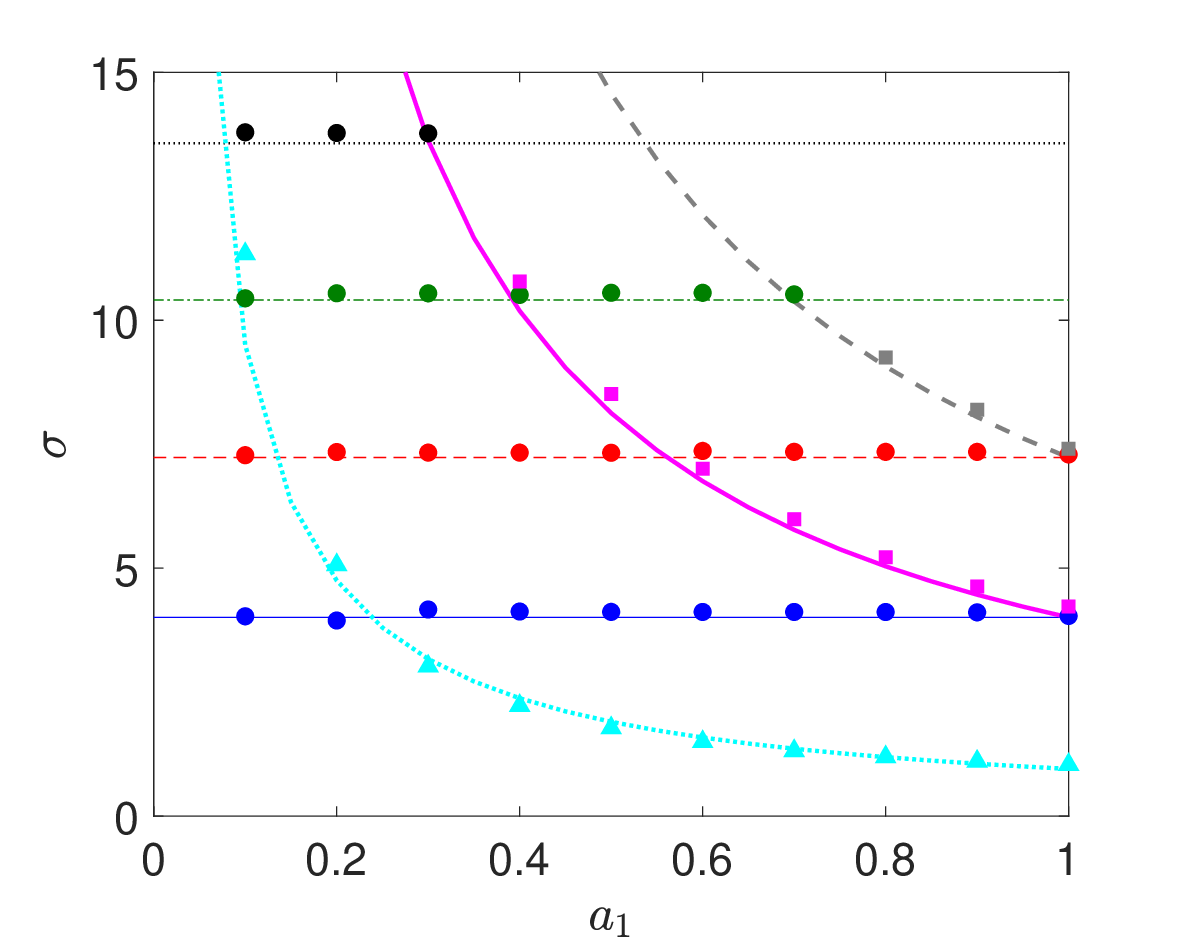} % {sigma_SN_view4.eps}
        \caption{Comparison with asymptotics (\ref{sn:example_1all})} 
        \label{fig:sigma_SN2}
    \end{subfigure}
\caption{
{\bf (a)} The first five SN eigenvalues $\sigma^{(k)}$ (that
correspond to axially symmetric eigenfunctions) for two circular
patches of radii $\eps a_1$ and $\eps a_2$, with $0<a_1<1$ and $a_2 =
1$, located at the north and south poles of the unit sphere, with
$\eps = 0.2$.  Symbols illustrate the numerical values computed by the
spectral method from Appendix \ref{appf:numer} (with the truncation
order $\nmax=1000$ and the ``geodesic convention''), whereas thick
lines indicate the asymptotic formula (\ref{sn:stek_eigex}).
{\bf (b)} Symbols show the same eigenvalues computed numerically but
colored differently, along with the shown curves: the thin horizontal
lines present $\sigma_{\rm asy}^{(k)}/a_2$ with $k= 1,2,3,4$ from
(\ref{sn:example_1all}); thick solid and dashed curves present
$\sigma_{\rm asy}^{(k)}/a_1$ with $k = 1,2$ from
(\ref{sn:example_1all}); thick dotted line presents $0.95/a_1$. }
% [sigma,sigma_0, a1] = A_Ward_Steklov_3d_SN_sphere2c(sigma,sigma_0);
\end{figure}

Let us now consider the limit $a_1\to 1$, which corresponds to the
setting of two identical patches.  Figure~\ref{fig:sigma_SN1} shows
that the asymptotic theory of \S \ref{sec:sn:non} does not account for
the closely spaced SN eigenvalues that are computed numerically, and
instead accurately captures only one of these two eigenvalues.
To qualitatively explain this discrepancy for two identical patches,
we observe that the poles of ${\mathcal N}(\sigma_0)$, as
characterized by the resonant set ${\mathcal P}$ in (\ref{sn:poles0})
with $N=2$, are no longer all distinct. In addition, the leading-order
SN eigenvalue for two identical patches with $a_1=a_2=1$ reduces from
(\ref{sn:sigma_iden_disk}) to simply finding the roots of ${\mathcal
C}(-\sigma_0)=0$.  As a result, our asymptotic theory when applied to
two identical patches would predict that, for each $k = 1,2,\ldots$,
there is a {\em unique} leading-order approximation
$\sigma_{0}^{(k)}=\mu_{k1}^{N} (=\mu_{k2}^N)$ for the SN eigenvalue,
satisfying $C_1(-\mu_{k1}^{N})=0$.  We emphasize that these are simply
the leading-order SN eigenvalues for a single patch. In particular,
the asymptotic result (\ref{sn:stek_eigex}) of \S \ref{sec:sn:non}
would erroneously predict that the two smallest SN eigenvalues have
the {\em same} two-term asymptotics $\sigma_{\rm asy}^{(k)} \sim
\mu_{k1}^{N} + \eps E_1(-\mu_{k1}^{N})/C_1^{\prime}(-\mu_{k1}^{N})$.
This is precisely the asymptotic result given in (\ref{sn:example_1})
for a single Steklov patch.
These results are shown by the limiting values of the green and black
dashed curves at the right endpoint in Fig.~\ref{fig:sigma_SN1}.
However, our asymptotic theory fails to account for the additional
nearby SN eigenvalue at the right end of the red curve in
Fig.~\ref{fig:sigma_SN1}, as well as a further closely spaced SN
eigenvalue near the right end of the black dashed curve.  Moreover,
the asymptotic analysis of \S \ref{sec:sn:non} does not predict the
lowest SN eigenvalue at the right end of the lower blue curve in
Fig.~\ref{fig:sigma_SN1}.  Indeed, the second line of Table
\ref{tab:SN} reports the first five nontrivial SN eigenvalues for two
identical patches with $\eps=0.2$ that were calculated by the spectral
method from Appendix \ref{appf:numer} (with the ``geodesic
convention''), which we consider as benchmarks.  These results are the
SN eigenvalues at the right ends of the curves in
Fig.~\ref{fig:sigma_SN1}.  For comparison, the second line of Table
\ref{tab:SN} yields the first two SN eigenvalues for two identical
patches as predicted by the asymptotic theory of \S
\ref{sec:sn:non} as obtained by setting $\eps=0.2$ in
(\ref{sn:example_1}).  As a result, we conclude that when applied to
the case of identical patches, the asymptotic theory of \S
\ref{sec:sn:non} only accounts for a subset of the true SN
eigenvalues.  We now remedy this deficiency by refining our asymptotic
theory to treat the setting of multiple identical patches.

\begin{table}
\centering
\begin{tabular}{|c|c|c|c|c|c|}  \hline
$k$                         & 1      &   2    &   3    &   4    &   5    \\  \hline     
Standard convention         & 1.023  & 3.980  & 4.166  & 7.183  & 7.295  \\
Geodesic convention         & 1.031  & 4.008  & 4.195  & 7.233  & 7.349  \\
Non-resonant asymptotics    &        & 4.006  &        & 7.232  &        \\
Near-resonant asymptotics   & 1.008  &        & 4.190  &        & 7.342  \\  \hline
\end{tabular}
\caption{
The first five SN eigenvalues (that correspond to axially symmetric
eigenfunctions) for two identical circular patches of radius $\eps =
0.2$ located on the north and south poles of the unit sphere.  The
first line presents the numerical results obtained by the spectral
method presented in Appendix \ref{appf:numer}, with the truncation
order $\nmax = 3000$ that ensures that all shown digits of these
values are exact.  The second line presents the numerical results
obtained by the same method, under the ``geodesic convention'' (all
the shown digits are as well exact).  The third line gives our
three-term expansion (\ref{sn:stek_eigex}) for the non-resonant case
(it is identical to the single-patch asymptotics
(\ref{sn:example_1all})).  The fourth line gives the three-term
expansion (\ref{sn:degen:ex1}), that will be derived in \S
\ref{sec:sn_degen} for the near-resonant case (see below). }
\label{tab:SN}
% [sigma, sigma0,sigma1,sigma2] = A_Ward_Steklov_3d_SN_sphere2_asympt_eps(0.2);
\end{table}

\subsection{Near-Resonant Case}\label{sec:sn_degen}

We now give a specific nontrivial illustration of the near-resonant
case that will always occur when there are $M$ identical patches, with
$2\leq M\leq N$.  With a suitable relabeling of the patch indices, we
label the common patch shape as $\partial\Omega_{c}^{\eps}=
\partial\Omega_{i}^{\eps}$ (with $a_i = a_c$) for $i=1,\ldots,M$. On
these identical patches, there is a common spectrum, labeled by
$\{\mu_{kc}\}_{k\geq 0}$, for the local Steklov problem
(\ref{eq:Psi_def}) of Appendix \ref{sec:Cmu}.

We assume that
$\sigma_0=\lim\limits_{\eps\to 0}\sigma(\eps)=\mu_{k^{\prime} c}$ for
some simple eigenvalue $\mu_{k^{\prime} c}$ of (\ref{eq:Psi_def}) for
which $d_{k^{\prime}c}\neq 0$. The corresponding local Steklov
eigenfunction from (\ref{eq:Psi_def}), labeled by
$\tilde{\Psi}_{k^{\prime}c}$, is taken to be the unique solution to
\begin{subequations}  \label{sn:Psi_def}
\begin{align}  \label{sn:Vk_eq}
  \Delta_{\y} \tilde{\Psi}_{k^{\prime}c} & = 0 \,, \quad \y \in \R_+^3 \,,\\
  \label{sn:Vk_stek}
  \partial_{y_3} \tilde{\Psi}_{k^{\prime}c} + \sigma_0
  \tilde{\Psi}_{k^{\prime}c} &=0 \,,
                                  \quad y_3=0 \,,\, (y_1,y_2)\in \PT_c\,,
  \\  \label{sn:Vk_Neumann}
  \partial_{y_3} \tilde{\Psi}_{k^{\prime}c} & = 0 \,, \quad y_3=0 \,,\,
                                      (y_1,y_2)\notin \PT_c\,,
  \\  \label{sn:Vk_inf}
  \tilde{\Psi}_{k^{\prime}c}(\y) & \sim \frac{1}{|\y|} +
                                   {\mathcal O}\left(|\y|^{-2}\right)
                           \quad \textrm{as}  \quad |\y|\to \infty\,,
\end{align}
\end{subequations}
where $\PT_c \asymp \eps^{-1}\partial\Omega_c^\eps$.  The tilde
highlights that we changed here the normalization of
$\tilde{\Psi}_{k^{\prime}c}$ by imposing (\ref{sn:Vk_inf}).  Comparing
this decay with the asymptotic behavior (\ref{eq:Psi_asympt}) of an
equivalent eigenfunction $\Psi_{k^{\prime}c}$ with the
conventional $L^2(\PT_c)$ normalization, we deduce that
\begin{equation}  \label{eq:tildePsi_Psi}  
  \tilde{\Psi}_{k^{\prime}c} = \frac{2\pi}{\mu_{k^\prime c}\, d_{k^{\prime}c}}
  \Psi_{k^{\prime}c} \,,
\end{equation}
with $d_{k^\prime c}$ being defined in (\ref{eq:dj}).
We conclude that
\begin{equation}  \label{eq:tildePsi_norm}
  \int\limits_{\PT_c} [\tilde{\Psi}_{k^{\prime}c}(\y)]^2 \, d\y =
  \left(\frac{2\pi}{\mu_{k^\prime c} \, d_{k^{\prime}c}}\right)^2 \,.
\end{equation}
Moreover, in our analysis below, we assume that the remaining $N-M$
patches are not in near-resonance in the sense that $\sigma_0\neq
\mu_{ki}$ for all $k\geq 0$ and all $i=M+1,\ldots N$.

As similar to the analysis in \S \ref{sec:sn:non}, in the outer region
we expand $u$ as in (\ref{sdn:outex}) to obtain (\ref{sdn:Uk}) at each
order. We then expand the Steklov eigenvalue as in
(\ref{sdn:stek_eig}), where $\sigma_0=\mu_{k^{\prime}c}$.  In the
inner regions near each Steklov patch, we expand the inner solution as
in (\ref{sdn:innex}) to obtain the inner problems (\ref{sn_s:Vi}) at
each order.

In contrast to the analysis for the non-resonant case, we obtain in
place of (\ref{sn:v0i}) that the leading-order inner solutions are now
\begin{equation}\label{sn:degen:v0i}
  \begin{split}
    V_{0i} &=A_{i} \tilde{\Psi}_{k^{\prime}c}(\y) \,, \quad i=1,\ldots,M \,,\\
    V_{0i} &=U_0 \left(1 - w_{i}(\y;-\sigma_{0})\right) \,, \quad
    i=M+1,\ldots,N \,,
  \end{split}
\end{equation}
where $A_1,\ldots,A_M$ are constants to be determined.  As
$|\y|\to \infty$, we have $V_{0i}\to 0$ for $i=1,\ldots,M$ while
$V_{0i}\to U_0$ for $i=M+1,\ldots,N$. This implies that we can only
match to the leading-order constant outer solution $U_0$ when
$U_0=0$. As a result, the leading-order inner solutions near the
non-resonant patches vanish, i.e.~$V_{0i}=0$ for
$i=M+1,\ldots,N$.

Next, by matching the far-field behavior of $V_{0i}$ for
$i=1,\ldots,M$ to the outer correction $U_1$ by using the far-field
(\ref{sn:Vk_inf}) for $\tilde{\Psi}_{k^{\prime}c}$, we obtain that
$U_1$ satisfies
\bsub\label{sn:degen:U1prob}
\begin{align}
  \Delta_{\x} U_{1} &= 0 \,, \quad \x\in \Omega \,; \qquad
  \partial_n U_1=0 \,, \quad \x\in \partial\Omega\backslash
  \lbrace{\x_1,\ldots,\x_M\rbrace} \,, \label{sn:degen:U1prob_1}\\
  U_1 & \sim \frac{A_i}{|\x-\x_1|}\,, \quad \mbox{as} \quad
        \x\to\x_i\in \partial\Omega \,, \quad i=1,\ldots,M \,.
\end{align}
\esub
The solvability condition for (\ref{sn:degen:U1prob}) yields that
\begin{equation}\label{sn:degen:suma}
  \sum_{i=1}^{M} A_i = 0 \,.
\end{equation}
In terms of an unknown constant $\overline{U}_1$, the
solution to (\ref{sn:degen:U1prob}) is
\begin{equation}\label{sn:degen:u1solve}
  U_1 = \overline{U}_1 +2\pi \sum_{j=1}^{M} A_j G_{s}(\x;\x_j) \,, \quad
\end{equation}
where $G_s(\x;\x_j)$ is the surface Neumann Green's function of
(\ref{mfpt:gs_exact}). By using the local behavior of $G_{s}$ given in
(\ref{mfpt:gs_locm}) in terms of geodesic coordinates, we obtain as
$\x\to \x_i$ that for the resonant patches
\bsub
\begin{equation}\label{sn:degen:U1loc_r}
\begin{split}
       U_{1} &\sim \frac{A_i}{\eps |\y|} - \frac{A_i}{2}
    \log\left(\frac{\eps}{2}\right)
    - \frac{A_i}{2} \left( \log(y_3+|\y|) -
      \frac{y_3(y_1^2+y_2^2)}{|\y|^3}\right)  \\
    & \qquad + \beta_{ci} +\overline{U}_1\,, \quad
    \mbox{for} \quad i=1,\ldots,M \,.
\end{split}
\end{equation}
In (\ref{sn:degen:U1loc_r}), we have defined $\beta_{ci}$ as the
$i$-th component of the vector ${\bm \beta}_c$ defined by
\begin{equation}\label{sn:degen:beta}
  {\bm \beta}_c \equiv 2\pi {\mathcal G}_{sc} \vac \,, \qquad
  \mbox{where} \qquad \vac=(A_1,\ldots,A_M)^T\,.
\end{equation}
Here ${\mathcal G}_{sc}$ is the $M\times M$ Green's matrix representing
long-range interactions over the resonant patches, defined by
\begin{equation}\label{sn:degen:green_mat}
    {\mathcal G}_{sc}  \equiv \left ( 
\begin{array}{cccc}
 R_s & G_{12} & \cdots & G_{1M} \\
 G_{21} & R_s & \cdots   &G_{2M} \\
 \vdots & \vdots  &\ddots  &\vdots\\ 
 G_{M1} &\cdots & G_{M,M-1} & R_s
\end{array}
\right ) \,, \quad R_s \equiv - \frac{9}{20\pi} \,, \quad G_{ij} \equiv
  G_{s}(\x_i;\x_j) \,.
\end{equation}
\esub
In contrast, for the non-resonant patches, we have as $\x\to\x_i$ that
\begin{equation}
  U_1  \sim \overline{U}_1 + 2\pi \sum_{j=1}^{M} A_j G_{s}(\x_i;\x_j) \,,
  \quad \mbox{for} \quad i=M+1,\ldots, N \,. \label{sn:degen:U1loc_nr}
\end{equation}

We observe upon comparing (\ref{sn:degen:U1loc_r}) and
(\ref{sn:degen:U1loc_nr}) that the ${\mathcal O}(\log\eps)$ term only
occurs for the resonant patches. As a result, in the inner expansion
(\ref{sdn:innex}) we conclude that $V_{1i}=0$ for the non-resonant
patches $i=M+1,\ldots,N$. Alternatively, for the resonant patches
$i=1,\ldots,M$, we obtain from the matching condition between the
inner and outer solutions that $V_{1i}$ satisfies
\bsub \label{sn:degen:V1i}
\begin{align}
  \Delta_{\y} V_{1i} &= 0\,, \quad \y \in \R_{+}^{3} \,, \label{sn:degen:Vi_1}\\
  \partial_{y_3} V_{1i} + \sigma_0 V_{1i} &= -\sigma_1 V_{0i} \,, \quad y_3=0 \,,\,
    (y_1,y_2)\in \PT_c\,,  \label{sn:degen:Vi_2}\\
    \partial_{y_3} V_{1i} &=0 \,, \quad y_3=0 \,,\, (y_1,y_2)\notin \PT_c
                            \,, \label{sn:degen:Vi_3}\\
  V_{1i}& \sim - \frac{A_i}{2} + {\mathcal O}(|\y|^{-1}) \,, \quad
  \mbox{as} \quad |\y|\to \infty \,. \label{sn:degen:Vi_4}
\end{align}
\esub

To derive the solvability condition for (\ref{sn:degen:V1i}), which will
determine $\sigma_1$, we need the following lemma:

\begin{lemma}\label{sn:degen:lemma} 
Consider the inhomogeneous problem for $V(\y)$ given by
\bsub \label{sn:degen:Vlem}
\begin{align}
  \Delta_{\y} V &= 0\,, \quad \y \in \R_{+}^{3} \,, \label{sn:degen:Vlem_1}\\
  \partial_{y_3} V + \sigma_0 V &= {\mathcal R}(y_1,y_2) \,, \quad y_3=0 \,,\,
    (y_1,y_2)\in \PT_c\,,  \label{sn:degen:Vlem_2}\\
    \partial_{y_3} V &=0 \,, \quad y_3=0 \,,\, (y_1,y_2)\notin \PT_c
                            \,, \label{sn:degen:Vlem_3}\\
  V & \sim V_{\infty} + {\mathcal O}(|\y|^{-1}) \,, \quad
  \mbox{as} \quad |\y|\to \infty \,, \label{sn:degen:Vlem_4}
\end{align}
\esub
where $V_{\infty}$ is a constant.  A necessary and sufficient
condition for (\ref{sn:degen:Vlem}) to have a solution is that
\begin{equation}\label{sn:degen:solve}
  \int_{\PT_c} \tilde{\Psi}_{k^{\prime} c} {\mathcal R} \, dy_1 dy_2 =
  2\pi V_{\infty} \,,
\end{equation}
where $\tilde{\Psi}_{k^{\prime}c}$ is the unique solution to
(\ref{sn:Psi_def}) with $\sigma_0=\mu_{k^{\prime}c}$.  When
(\ref{sn:degen:solve}) holds, the solution $V$ is unique up to adding
an arbitrary multiple of $\tilde{\Psi}_{k^{\prime}c}$.
\end{lemma}

\begin{proof} To prove the necessity of (\ref{sn:degen:solve}) we apply
  Green's second identity to $V$ and $\tilde{\Psi}_{k^{\prime}c}$ over
  a large hemisphere of radius $R$ in the upper  half-space to
  obtain
\begin{equation}\label{sn:degen:proof_1}
  \begin{split}
    \int\limits_{\R_{+}^{3}} \left(V \Delta_{\y} \tilde{\Psi}_{k^{\prime}c} -
      \tilde{\Psi}_{k^{\prime}c}
      \Delta_{\y} V \right) d\y &=  \int\limits_{\PT_c} \left[
      \tilde{\Psi}_{k^{\prime} c}  \left(\partial_{y_3} V + \sigma_0 V\right) -
      V  \left(\partial_{y_3} \tilde{\Psi}_{k^{\prime} c} + \sigma_0
          \tilde{\Psi}_{k^{\prime}c}\right) \right]  d\y  \\
  & \qquad + 2\pi \lim_{R\to \infty} R^2 \left(
    V \frac{\partial \tilde{\Psi}_{k^{\prime}c}}{\partial |\y|} -
    \tilde{\Psi}_{k^{\prime}c} \frac{\partial V}{\partial |\y|} \right)\Big{\vert}_{
    |\y|=R}\,.
  \end{split}
\end{equation}
Then, upon imposing the conditions (\ref{sn:degen:Vlem_2}) and
(\ref{sn:Vk_stek}) on the patches, together with using the far-field
behaviors (\ref{sn:degen:Vlem_4}) and (\ref{sn:Vk_inf}), we readily
obtain that (\ref{sn:degen:proof_1}) reduces to (\ref{sn:degen:solve}).
This proves the necessity of (\ref{sn:degen:solve}).

We now prove the sufficiency of (\ref{sn:degen:solve}). Since
$\{\Psi_{kc}|_{\PT_c}\}$ form a complete orthonormal basis of
$L^2(\PT_c)$, any harmonic function $V$ in $\R_{+}^3$ that decays
at infinity and satisfies the mixed Robin-Neumann conditions
(\ref{sn:degen:Vlem_2}) and (\ref{sn:degen:Vlem_3}) can be represented
in terms of the eigenfunctions $\Psi_{kc}$.  In other words, a
general solution $V$ to (\ref{sn:degen:Vlem}) can be written as
\begin{equation}
V(\y) = V_\infty + \sum\limits_{k=0}^\infty \nu_k \, \Psi_{kc}(\y) \,,
\end{equation}
with suitable coefficients $\nu_k$.  By construction, it
satisfies (\ref{sn:degen:Vlem_1}), (\ref{sn:degen:Vlem_3}), and
(\ref{sn:degen:Vlem_4}).  Substituting this representation into
(\ref{sn:degen:Vlem_2}), multiplying it by $\Psi_{jc}$, integrating
over $\y\in\PT_c$ and using the orthonormality of $\{\Psi_{kc}\}$, we
obtain that
\begin{equation}\label{sn:lemma_coeff}
  \nu_j (\sigma_0 - \mu_{jc}) = \int\limits_{\PT_c}
  ({\mathcal R}(\y) - \sigma_0 V_\infty) \Psi_{jc}(\y) \, d\y \,. 
\end{equation}
Since $\sigma_0 = \mu_{k^{\prime}c}$, we observe that we must have
$\int_{\PT_c}({\mathcal R}(\y) - \sigma_0 V_\infty)
\Psi_{k^{\prime}c}(\y) \, d\y=0$, which is equivalent to
(\ref{sn:degen:solve}), as is readily seen by using the divergence
theorem.  When this condition holds, $\nu_{k^{\prime}}$ remains
undetermined (a free parameter). The other coefficients $\nu_j$ for
any $j\ne k^{\prime}$ are uniquely given by (\ref{sn:lemma_coeff}).

Finally, when (\ref{sn:degen:solve}) holds, the general solution $V$
to (\ref{sn:degen:Vlem}) can be written as
$V=V_p + B \tilde{\Psi}_{k^{\prime}c}$, where $B$ is an arbitrary
constant and where $V_p$ is the particular solution of
(\ref{sn:degen:Vlem}) satisfying
$V_p\sim V_{\infty} + {\mathcal O}(|\y|^{-2})$ as $|\y|\to
\infty$. Since $\tilde{\Psi}_{k^{\prime}c}\sim {1/|\y|}$ as
$|\y|\to \infty$, it follows that $V\sim V_{\infty}+{B/|\y|}$ as
$|y|\to\infty$, where $B$ is arbitrary.
\end{proof}

To determine $\sigma_1$ from (\ref{sn:degen:V1i}) we simply apply the
solvability condition (\ref{sn:degen:solve}) of Lemma
\ref{sn:degen:lemma} where we set $V_{\infty}={-A_i/2}$, and
${\mathcal R}= -\sigma_1 V_{0i}$ with
$V_{0i}=A_i \tilde{\Psi}_{k^{\prime} c}$. In this way, by using
(\ref{eq:tildePsi_norm}), we readily determine in terms of
$\mu_{k^{\prime}c}$ and the weight $d_{k^{\prime}c}=
\int_{\Gamma_c}\Psi_{k^{\prime}c} \, d\y$ that
\begin{equation}\label{sn:degen:sigma_1}  
  \sigma_1 = \frac{\pi}{\int_{\PT_c} (\tilde{\Psi}_{k^{\prime} c})^2
  \, d\y } = \frac{\mu_{k^{\prime}c}^2 d_{k^{\prime}c}^2}{4\pi}\,.
\end{equation}

Without loss of generality, as shown in Lemma \ref{sn:degen:lemma} we
are free to impose that $V_{1i}\sim -{A_i/2} + {\mathcal
O}(|\y|^{-2})$ as $|\y|\to\infty$, which ensures that $V_{1i}$ is
unique.  As a result, from the matching condition we obtain that the
outer correction $U_2$ in (\ref{sdn:outex}) satisfies (\ref{sdn:Uk}),
with no singularities at any $\x_i$ for $i=1,\ldots,N$.  We conclude
that $U_2=\overline{U}_2$, where the constant $\overline{U}_2$ can
only be obtained at higher order.

Next, we proceed to determine the SN eigenvalue correction $\sigma_2$.
For the non-resonant patches $i=M+1,\ldots,N$, we set $V_{0i}=0$ in
(\ref{sn_s:Vi}) and use the local behavior (\ref{sn:degen:U1loc_nr})
to derive that the inner correction $V_{2i}$ satisfies
\bsub \label{sn:ndegen:V2i}
\begin{align}
  \Delta_{\y} V_{2i} &= 0\,, \quad \y \in \R_{+}^{3} \,, \label{sn:ndegen:V2i_1}\\
  \partial_{y_3} V_{2i} + \sigma_0 V_{2i} &= 0 \,, \quad y_3=0 \,,\,
    (y_1,y_2)\in \PT_c\,,  \label{sn:ndegen:V2i_2}\\
    \partial_{y_3} V_{2i} &=0 \,, \quad y_3=0 \,,\, (y_1,y_2)\notin \PT_c
                            \,, \label{sn:ndegen:V2i_3}\\
  V_{2i}& \sim \overline{U}_1 + 2\pi \sum_{j=1}^{M} A_j G_{s}(\x_i;\x_j) \,, \quad
  \mbox{as} \quad |\y|\to \infty \,. \label{sn:ndegen:V2i_4}
\end{align}
\esub
The solution to (\ref{sn:ndegen:V2i}) for $i=M+1,\ldots,N$ is
\begin{equation}\label{sn:ndegen:V2i_sol}
  V_{2i} = \left(\gamma_i + \overline{U}_1\right) \left( 1 - w_{i}(\y;-\sigma_0)
  \right)\,, \quad \mbox{where} \quad
  \gamma_i\equiv 2\pi \sum_{j=1}^{M} A_j G_{s}(\x_i;\x_j) \,,
\end{equation}
which has the far-field behavior 
\begin{equation}\label{sn:ndegen:V2i_ff}
  V_{2i}\sim \left(\gamma_i + \overline{U}_1\right) \left( 1 -
    \frac{C_i(-\sigma_0)}{|\y|} \right) \,, \quad \mbox{as} \quad
  |\y|\to \infty \,, \quad i=M+1,\ldots,N \,.
\end{equation}

In contrast, for the resonant patches $i=1,\ldots,M$, we obtain from
(\ref{sn_s:Vi}), together with the ${\mathcal O}(1)$ terms in the
local behavior (\ref{sn:degen:U1loc_r}), that the inner correction
$V_{2i}$ satisfies
\bsub \label{sn:degen:V2i}
\begin{align}
  \Delta_{\y} V_{2i} &= 2y_3 V_{0i,y_3y_3} + 2 V_{0i,y_3}\,, \quad
                       \y \in \R_{+}^{3} \,, \label{sn:degen:V2i_1}\\
  \partial_{y_3} V_{2i} + \sigma_0 V_{2i} & = -\sigma_2 V_{0i}\,, \quad y_3=0 \,,\,
       (y_1,y_2)\in \PT_c\,,  \label{sn:degen:V2i_2}\\
  \partial_{y_3} V_{2i} &=0 \,, \quad y_3=0 \,,\, (y_1,y_2)\notin \PT_c
                          \,, \label{sn:degen:V2i_3}\\
  V_{2i} & \sim \beta_{ci} + \overline{U}_1  -\frac{A_i}{2}
          \left[\log\left(y_3 + |\y|\right) -
          \frac{y_3(y_1^2+y_2^2)}{|\y|^3} \right] \,, 
         \,\, \mbox{as} \quad |\y|\to \infty \,. \label{sn:degen:V2i_4}
\end{align}
\esub To derive the solvability condition for (\ref{sn:degen:V2i}), we
first need to decompose $V_{2i}$ so as to account for the
inhomogeneous term in the PDE (\ref{sn:degen:V2i_1}) as well as the
term in the square bracket in (\ref{sn:degen:V2i_4}) in the far-field
behavior.  More specifically, and as very similar to the analysis in Lemma
\ref{lemma:Phi2} of Appendix \ref{app_h:inn2}, we decompose $V_{2i}$
as
\begin{equation}\label{sn:degen:V2decomp}
  V_{2i}=V_{2ip} + V_{2iH} \,,
\end{equation}
where $V_{2ip}$ is given explicitly in terms of $V_{0i}=A_i \tilde{\Psi}_{k^{\prime}c}$ by
\begin{equation}\label{sn:degen:V2p}
  V_{2ip}=\frac{y_3^2}{2}V_{0i,y_3} + \frac{y_3}{2}V_{0i} - \frac{1}{2}
  \int_{0}^{y_3} V_{0i}(y_1,y_2,\eta)\, d\eta + A_{i}
  {\mathcal F}_{c}(y_1,y_2)\,,
\end{equation}
where ${\mathcal F}_{c}$ is the unique solution to
\bsub \label{sn:degen:fcal}
\begin{gather}
  \Delta_{S}{\mathcal F}_{c} = \left(\frac{1}{2} \partial_{y_3}
  \tilde{\Psi}_{k^{\prime}c}\vert_{y_3=0} \right) I_{\PT_c} \,; \quad
    I_{\PT_c} \equiv \left\{\begin{array}{ll}
        1 \,, & (y_1,y_2) \in \PT_c \\
       0 \,, & (y_1,y_2) \notin \PT_c\,,  \end{array}\right.
                       \label{sn:degen:fcal_1}\\
{\mathcal F}_{c} \sim
\left( \frac{1}{4\pi} \int_{\PT_c} \partial_{y_3} \tilde{\Psi}_{k^{\prime}c}
  \vert_{y_3=0} \, d\y\right)
 \log\rho_0 +  o(1)\,,  \quad \mbox{as} \quad
                            \rho_0\equiv (y_1^2+y_2^2)^{1/2}\to \infty
                            \,, \label{sn:degen:fcal_2}
\end{gather}                 
\esub 
with $\Delta_{S}{\mathcal F}_{c} \equiv {\mathcal F}_{c,y_1
y_1}+ {\mathcal F}_{c,y_2 y_2}$.  By applying the divergence theorem
to (\ref{sn:Psi_def}) we calculate $\int_{\PT_c}\partial_{y_3}
\tilde{\Psi}_{k^{\prime}c}\vert_{y_3=0}
\, d\y=-2\pi$. In addition, upon using the relation (\ref{sn:Vk_stek})
on $\PT_c$ the solution to (\ref{sn:degen:fcal}) is written in terms of
the free-space Green's function in the plane as
\begin{equation}  \label{eq:Fc_solution}
  {\mathcal F}_{c}(\y) = -
  \frac{\mu_{k^\prime c}}{4\pi} \int\limits_{\PT_c}
  \tilde{\Psi}_{k^{\prime}c}(\y^{\prime}) \log|\y-\y^{\prime}| \, d\y^{\prime} \,,
\end{equation}
which satisfies ${\mathcal F}_c\sim -\tfrac{1}{2}\log\rho_0+ o(1)$ as
$\rho_0\to\infty$. 

Then, by repeating a similar calculation as in the proof of Lemma
\ref{lemma:Phi2}, we conclude that $V_{2iH}$ in (\ref{sn:degen:V2decomp})
for $i=1,\ldots,M$ satisfies
\bsub \label{sn:degen:V2iH}
\begin{align}
  \Delta_{\y} V_{2iH} &= 0\,, \quad \y \in \R_{+}^{3} \,, \label{sn:degen:V2iH_1}\\
  \partial_{y_3} V_{2iH} + \sigma_0 V_{2iH} &= -\sigma_2 V_{0i} -\sigma_0 A_i
                                    {\mathcal F}_c \,, \quad y_3=0 \,,\,
    (y_1,y_2)\in \PT_c\,,  \label{sn:degen:V2iH_2}\\
    \partial_{y_3} V_{2iH} &=0 \,, \quad y_3=0 \,,\, (y_1,y_2)\notin \PT_c
                            \,, \label{sn:degen:V2iH_3}\\
  V_{2iH}& \sim  \beta_{ci} + \overline{U}_1  \,, \quad
  \mbox{as} \quad |\y|\to \infty \,. \label{sn:degen:V2iH_4}
\end{align}
\esub

To determine $\sigma_2$ from (\ref{sn:degen:V2iH}) we simply apply the
solvability condition (\ref{sn:degen:solve}) of Lemma
\ref{sn:degen:lemma} in which we set $V_{\infty}=\beta_{ci}+\overline{U}_1$ and
${\mathcal R}= -\sigma_2 V_{0i}-\sigma_0A_i {\mathcal F}_c$ with
$V_{0i}=A_i \tilde{\Psi}_{k^{\prime} c}$. This yields for $i=1,\ldots,M$ that
\begin{equation}\label{sn:degen:sigma2_old}
  -2\pi \left(\beta_{ci} + \overline{U}_1\right) = A_i \left[
    \sigma_2 \int_{\PT_c} \left(\tilde{\Psi}_{k^{\prime}c}\right)^2 \, d\y +
    \int_{\PT_c} \sigma_0 \tilde{\Psi}_{k^{\prime}c} \, {\mathcal F}_c
    \, d\y \right]\,,
\end{equation}
where we identify $\sigma_0 \tilde{\Psi}_{k^{\prime}c}=-
\partial_{y_3}\tilde{\Psi}_{k^{\prime}c}$ on $\PT_c$ and
$\int_{\PT_c}(\tilde{\Psi}_{k^{\prime}c})^2 \, d\y =
{\pi/\sigma_1}$ from (\ref{sn:degen:sigma_1}).  

Finally, upon recalling (\ref{sn:degen:beta}) for $\beta_{ci}$ and the
constraint (\ref{sn:degen:suma}), we obtain a matrix eigenvalue
problem for $\vac\equiv (A_1,\ldots,A_M)^{T}$ and an eigenvalue
parameter $\alpha$ given by
\bsub \label{sn:degen:mat_all}
\begin{equation}
  {\mathcal G}_{sc} \vac + \frac{\overline{U}_1}{2\pi} \evec_{M}
  =\alpha \vac \,, \qquad \evec_{M}^{T}\vac = 0 \,,
  \label{sn:degen:mat_1}
\end{equation}
where $\evec_M\equiv (1,\ldots,1)^{T}\in \R^M$. Here $\sigma_2$ is
related to $\alpha$ by
\begin{equation}
  \sigma_2 = -\frac{\sigma_1}{\pi} \left[ 4\pi^2 \alpha +
    {\mathcal J} \right] \,,  \label{sn:degen:sigma_2}
\end{equation}
with ${\mathcal J}$ given by
\begin{equation}\label{sn:degen:sigma_3}
  {\mathcal J} \equiv \int_{\PT_c} \sigma_0
  \, \tilde{\Psi}_{k^{\prime}c} \, {\mathcal F}_c \, d\y 
   = - \frac{\mu_{k^{\prime}c}^2}{4\pi} \int\limits_{\PT_c}
   \int\limits_{\PT_c} \tilde{\Psi}_{k^{\prime}c}(\y)\,
    \tilde{\Psi}_{k^{\prime}c}(\y^{\prime})\, \log |\y-\y^{\prime}|\, d\y \,
   d\y^{\prime}\,,
\end{equation}
where in the last equality, we used $\sigma_0=\mu_{k^{\prime}c}$ and
(\ref{eq:Fc_solution}).  By using (\ref{eq:tildePsi_Psi}),
this relation can also be written in terms of the conventionally
normalized eigenfunction $\Psi_{k^{\prime}c}$ as
\begin{equation}\label{sn:degen:sigma_3bis}
  {\mathcal J}  = - \frac{\pi}{d_{k^{\prime}c}^2} \int\limits_{\PT_c}
   \int\limits_{\PT_c} \Psi_{k^{\prime}c}(\y)\,
    \Psi_{k^{\prime}c}(\y^{\prime})\, \log |\y-\y^{\prime}|\, d\y \,
   d\y^{\prime}\,,
\end{equation}
where $d_{k^{\prime}c}=\int_{\Gamma_c}\Psi_{k^{\prime}c}\, d\y$.  This
relation shows that ${\mathcal J}$, and thus the associated correction
$\sigma_2$ to the SN eigenvalue, are independent of the normalization
of $\Psi_{k^{\prime}c}$.

Since $V_{0i}=A_i \tilde{\Psi}_{k^{\prime}c}$ for $i=1,\ldots,M$ and
$V_{0i}\equiv 0$ for $i=M+1,\ldots,N$, the PDE normalization condition
(\ref{sn:eig_4}) provides the following normalization condition for
the matrix eigenvalue problem (\ref{sn:degen:mat_all}):
\begin{equation}
  \vac^T \vac = \sum_{j=1}^{M} A_j^2 \sim \frac{1}{\eps \int_{\PT_c}
    [\tilde{\Psi}_{k^{\prime}c}]^2 \, d\y} \,. \label{sn:degen:mat_norm}
\end{equation}
\esub

By taking the inner product of (\ref{sn:degen:mat_1}) with $\evec_M$, we
can isolate $\overline{U}_1$ as
\begin{equation}\label{sn:degen:U1_mat}
  \overline{U}_{1} = - \frac{2\pi}{M} \evec_M^{T} {\mathcal G}_{sc}
   \vac \,,
\end{equation}
where $\vac$ and $\alpha$ is an eigenpair of the $M\times M$ matrix
eigenvalue problem
\begin{equation}\label{sn:degen:avec} 
  \left({\bf I} - \frac{\evec_{M} \evec_{M}^{T}}{M}  \right) {\mathcal G}_{sc}
  \vac = \alpha \vac \,,
  \quad \mbox{with} \quad \evec_{M}^T \vac=0  \,,
\end{equation}
with normalization (\ref{sn:degen:mat_norm}). Here ${\bf I}$ is the
$M\times M$ identity matrix.

We complete our analysis by deriving the problem for the outer
correction $U_3$ in (\ref{sdn:outex}), which satisfies (\ref{sdn:Uk})
with singularity conditions at the patch locations. For the resonant
patches, as discussed in the proof of Lemma \ref{sn:degen:lemma}, we
can impose that $V_{2iH}\sim \beta_{ci}+ \overline{U}_1 + {B_i/|\y|}$
as $|\y|\to \infty$, where $B_i$ for $i=1,\ldots,M$ are unknown
constants. For the non-resonant patches, we have that
(\ref{sn:ndegen:V2i_ff}) provides the singularity behavior for $U_3$. In this
way, we obtain that $U_3$ satisfies
\begin{equation}\label{sn:degen:U3prob}
  \begin{split}
  \Delta_{\x} U_{3} &= 0 \,, \quad \x\in \Omega \,; \qquad
  \partial_n U_3=0 \,, \quad \x\in \partial\Omega\backslash
  \lbrace{\x_1,\ldots,\x_N\rbrace} \,, \\
  U_3  & \sim -\frac{\left(\gamma_i +\overline{U}_1\right)C_i(-\sigma_0)}
  {|\x-\x_i|}\,, \quad \mbox{as} \quad \x\to\x_i \,, \quad i=M+1,\ldots, N ,\\
  U_3  & \sim \frac{B_i}{|\x-\x_i|}\,, \quad \mbox{as} \quad \x\to\x_i \,,
  \quad i=1,\ldots,M \,.
\end{split}
\end{equation}
The solvability condition for (\ref{sn:degen:U3prob}) provides one
equation for $(B_1,\ldots,B_M)^{T}$:
\begin{equation}\label{sn:degen:U3prob_solv}
  \sum_{i=1}^{M} B_i = \sum_{i=M+1}^{N} \left(\gamma_i+\overline{U}_1\right)
  C_i(-\sigma_0)\,,
\end{equation}
where $\gamma_i$ is defined in (\ref{sn:ndegen:V2i_sol}). A higher-order
analysis can in principle be undertaken to determine a matrix system
for $(B_1,\ldots,B_M)^T$.

We summarize our result in the following proposition.

\begin{prop}\label{sn:degen:main_res}
Suppose that there are exactly $M$ identical patches, with $2\leq
M\leq N$, with a common patch shape $\partial\Omega_{c}^{\eps}=
\partial\Omega_{i}^{\eps}$ for $i=1,\ldots,M$. Let 
$\{\mu_{kc}\}_{k\geq 0}$ be the spectrum of the local Steklov problem
(\ref{eq:Psi_def}) of Appendix \ref{sec:Cmu} on the local common patch
$\PT_c \asymp \eps^{-1}\partial\Omega_{c}^{\eps}$.  Then, for any
$k\geq 0$ such that $d_{kc}\neq 0$, the Steklov-Neumann (SN) problem
(\ref{sn:eig}) has $M-1$ eigenvalues $\sigma(\eps)$ (counting
multiplicity) for which $\sigma_0=\lim\limits_{\eps\to
0}\sigma(\eps)=\mu_{kc}$, where $\mu_{kc}$ is assumed to be simple.  A
three-term expansion for these eigenvalues is
\begin{equation}\label{sn:degen:main_sigma}
  \sigma =\sigma_{0}+ \eps\log\left(\frac{\eps}{2} \right) \sigma_{1} + \eps
  \sigma_{2} + {\mathcal O}(\eps^2\log \eps)\,,
\end{equation}
where $\sigma_1$ is given in (\ref{sn:degen:sigma_1}) and $\sigma_2$
is related via (\ref{sn:degen:sigma_2}) to $M-1$ eigenpairs
$\alpha$ and $\vac=(A_1,\ldots,A_M)^T$ of the matrix eigenvalue
problem (\ref{sn:degen:avec}), with normalization condition
(\ref{sn:degen:mat_norm}). Moreover, the local behavior of the
eigenfunctions on the resonant and non-resonant patches are
\begin{equation}\label{sn:degen:main_v}
\begin{split}
  u\vert_{\PT_i} &= A_i \tilde{\Psi}_{k^{\prime}c} + {\mathcal O}(\eps) \,,
  \quad i=1,\ldots,M \,, \\
  u\vert_{\PT_i} &\sim \eps\left(\gamma_i + \overline{U}_1\right)
  \left( 1 - w_{i}(\y;-\sigma_0) \right)\,, \quad i=M+1,\ldots,N \,,
\end{split}
\end{equation}
where $\tilde{\Psi}_{k^{\prime}c}$ satisfies (\ref{sn:Psi_def}). In
(\ref{sn:degen:main_v}), $\overline{U}_1$ and $\gamma_i$ are defined
in (\ref{sn:degen:U1_mat}) and (\ref{sn:ndegen:V2i_sol}),
respectively.  In the outer region, the SN eigenfunction is given by
\begin{equation}\label{sn:degen:main_out}
  u = \eps \left(\overline{U}_1 +2\pi \sum_{j=1}^{M} A_j G_{s}(\x;\x_j)
  \right) + {\mathcal O}(\eps^2) \,.
\end{equation}
\end{prop}

We emphasize that in our analysis of this near-resonant case, the
eigenvalue corrections $\sigma_1$ and $\sigma_2$ are obtained upon
applying solvability conditions to the inner problems defined near the
resonant patches. For the non-resonant case studied in \S
\ref{sec:sn:non}, these correction terms were found from solvability
conditions on the outer solution. Moreover, in contrast to our main
result in Proposition \ref{sn:main_res} for the non-resonant case, we
observe that the SN eigenfunctions for this near-resonant case are
concentrated primarily on the resonant patches and that the outer
solution in (\ref{sn:degen:main_out}) is now ${\mathcal O}(\eps)$
smaller than that for the non-resonant case.

The results in Proposition \ref{sn:degen:main_res} for $\sigma_1$ and
$\sigma_2$ can be simplified for the special case where the common
patch shape $\PT_c$ is a disk of radius one ($a_c = 1$), for which
$\tilde{\Psi}_{k^{\prime}c}$ and
$\partial_{y_3}\tilde{\Psi}_{k^{\prime}c}$, depend only on
$\rho_0=(y_1^2+y_2^2)^{1/2}$ when $y_3=0$ and $(y_1,y_2)\in\PT_c$.  In
this special case, where ${\mathcal F}_c={\mathcal F}_{c}(\rho_0)$, we
write ${\mathcal J}$ in (\ref{sn:degen:sigma_3}) as
\begin{equation}\label{sn:degen:j1}
  {\mathcal J} = 2\pi \sigma_{0} \int_{0}^{1}
  \rho_0 \tilde{\Psi}_{k^{\prime}c}\, {\mathcal F}_{c}(\rho_0) \, d\rho_0 \,.
\end{equation}
Since the problem (\ref{sn:degen:fcal}) for ${\mathcal F}_c$ is
radially symmetric, we find that its first integral is $\rho_0
{\mathcal F}_{c\rho_0}=
-\tfrac{\sigma_0}{2}\int_{0}^{\rho_0}\eta\tilde{\Psi}_{k^{\prime}c}(\eta)\,
d\eta$ for $0<\rho_0<1$. Upon integrating (\ref{sn:degen:j1}) by
parts, we use this first integral together with ${\mathcal
F}_{c}(1)=0$ and $\sigma_0=\mu_{k^{\prime}c}$, to conclude that
\begin{subequations}\label{sn:degen:all}
\begin{align}\label{sn:degen:j}
  {\mathcal J} & = \pi \mu_{k^\prime c}^2 \int_{0}^{1}
  \frac{1}{\rho_0} \left( \int_{0}^{\rho_0} \eta \tilde{\Psi}_{k^\prime c}(\eta)
    \, d\eta \right)^2  \, d\rho_0\,, \\   \label{sn:degen:j2}
& = \frac{4\pi^3}{d_{k^\prime c}^2} \int_{0}^{1}
  \frac{1}{\rho_0} \left( \int_{0}^{\rho_0} \eta \Psi_{k^\prime c}(\eta)
    \, d\eta \right)^2  \, d\rho_0 \,,
\end{align}
\end{subequations}
where we used (\ref{eq:tildePsi_Psi}) in the second equality.

\subsection{Numerical Comparison}\label{sec:sn_examples_degen}

We now consider two examples of our theory for the near-resonant case.

\vspace*{0.2cm}
\noindent {\bf Example I:} We first apply our theory to Example II of
\S \ref{sn:example:nondegen} with two identical anti-podal circular
patches of radius $\eps$ centered at the north and south poles, for
which $M=N=2$ and $a_1 = a_2 = 1$.  For such anti-podal patches
centered at $\x_1=(0,0,1)$ and $\x_2=(0,0,-1)$, the $2\times 2$
Green's matrix ${\mathcal G}_{sc}$ in (\ref{sn:degen:green_mat}) is
circulant symmetric and so the only eigenpair of (\ref{sn:degen:avec})
is $\vac=A_c(1,-1)^T$ and $\alpha=R_s-G_s(\x_1;\x_2)$, where $A_c$ is
a normalization constant. Upon using $R_s=-{9/(20\pi)}$ and
(\ref{mfpt:gs_exact}) for $G_s(\x_1;\x_2)$, we calculate that
$\alpha={\left(\log{2}-1\right)/(4\pi)}$.  By inserting the
superscript $^{(k)}$ to distinguish asymptotics for different
eigenvalues, we conclude from (\ref{sn:degen:main_sigma}) and
(\ref{sn:degen:sigma_2}) that there are SN eigenvalues of
(\ref{sn:eig}) for each $k\geq 0$ given by
\begin{equation}\label{sn:degen:ex1}  
  \sigma(\eps) \sim \mu_{kc} +
  \eps\log\left(\frac{\eps}{2} \right) \sigma_{1}^{(k)} -
  \frac{\eps \sigma_1^{(k)}}{\pi} \left[ \pi (\log{2}-1) +
    {\mathcal J}^{(k)}\right]
  + \ldots \,,
\end{equation}
where $\sigma_1^{(k)}$ and ${\mathcal J}^{(k)}$ are given by
(\ref{sn:degen:sigma_1}) and (\ref{sn:degen:all}), respectively.  The
numerical values for some local Steklov eigenvalues $\mu_{kc}$ and
their weights $d_{kc}$ are given in Table \ref{table:muk_disk} of
Appendix \ref{app:circle}. These values can be used to calculate
$\sigma_{1}^{(k)}$ from (\ref{sn:degen:sigma_1}).

To illustrate the accuracy of the three-term expansion
(\ref{sn:degen:ex1}), we employ the accurate numerical method from
Appendix \ref{appf:numer} to compute the numerical values $\sigma_{\rm
num}^{(j)}$, enumerated by the index $j = 1,2,3,\ldots$, that are
considered as benchmarks (we recall that these SN eigenvalues
correspond to axially symmetric eigenfunctions).
Figure~\ref{fig:sigma_SN_two} presents the difference between
$\sigma_{\rm num}^{(j)}$ with $j = 1,3,5,$ and the three-term
expansion (\ref{sn:degen:ex1}) for near-resonant SN eigenvalues with
indices $k = 1,2,3$.  We also present the difference between
$\sigma_{\rm num}^{(j)}$ with $j= 2,4$, and the asymptotic values
given by (\ref{sn:example_1all}) for non-resonant eigenvalues for a
single patch with indices $k= 1, 2$.  In both cases, the difference is
shown as a function of $\eps^2$ to outline the correct form of the
asymptotic relations.  We can therefore conclude that these relations
are very accurate.  We also observe that the asymptotic relation for
the first (nontrivial) eigenvalue $\sigma^{(1)}$ is the least
accurate.  Note also that the asymptotic relations for non-resonant
eigenvalues are more accurate than those for near-resonant ones.

\begin{figure}
  \centering
  \includegraphics[width = 85mm]{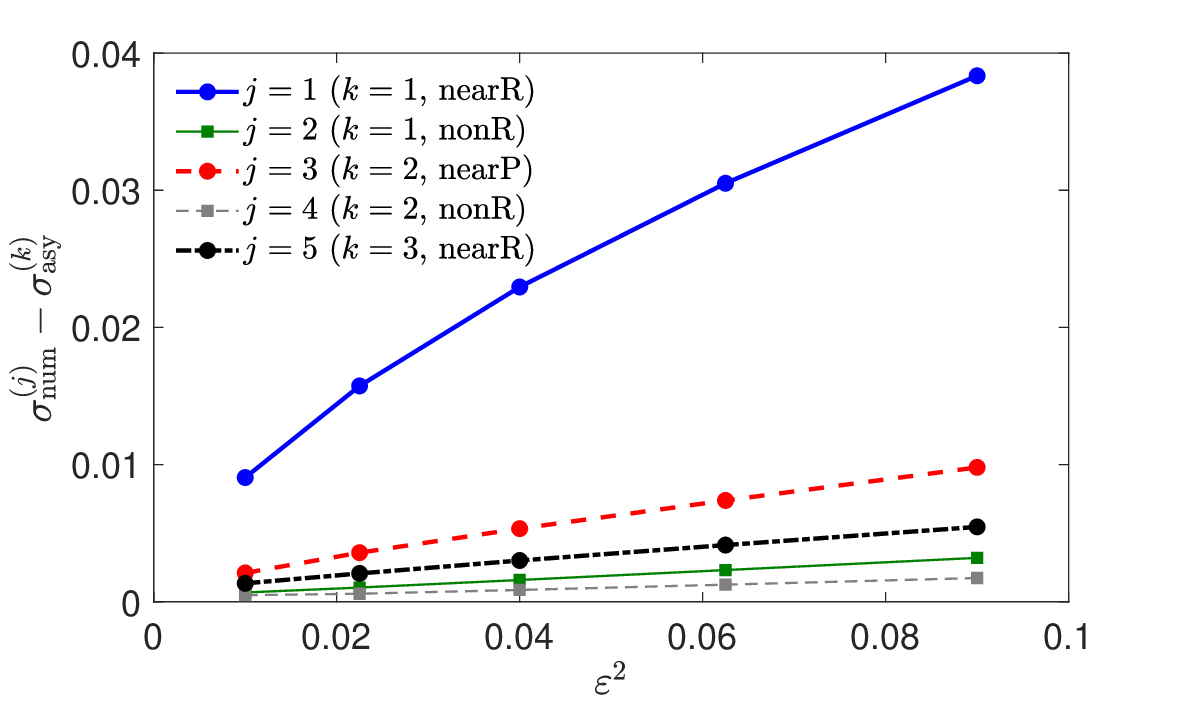} % {sigma_SN_two2new.eps} 
\caption{
Validation of the three-term expansions for the SN eigenvalues (that
correspond to axially symmetric eigenfunctions) in the case of two
identical circular patches of radius $\eps$ located at the north and
south poles of the unit sphere.  The difference between the numerical
values $\sigma_{\rm num}^{(j)}$ obtained by the spectral method from
Appendix \ref{appf:numer} (with truncation order $\nmax=2000$ and
``geodesic convention''), and the asymptotic values $\sigma_{\rm
asy}^{(k)}$ given by (\ref{sn:degen:ex1}) for near-resonant
eigenvalues with $k= 1,2,3$ and by (\ref{sn:example_1all}) for
non-resonant eigenvalues with $k = 1,2$.}
\label{fig:sigma_SN_two}
% [Mu, eps1] = A_Ward_Steklov_3d_sphereSN2_fig5(Mu);
\end{figure}

Figure \ref{fig:SD_eigen} illustrates the behavior of the first
three nontrivial axially-symmetric eigenfunctions $u^{(k)}$ of the SN
problem with two circular patches of radii $\eps_1 = \eps_2 = 0.2$,
located at the north and south poles.  We recall that the principal
eigenvalue $\sigma^{(0)} = 0$ corresponds to a trivial (constant)
eigenfunction that we exclude from our analysis.  In turn, the
eigenfunctions plotted in Figure \ref{fig:SD_eigen} correspond to the
first three eigenvalues $\sigma^{(k)}$ reported in Table \ref{tab:SN}.
As discussed earlier, the eigenpair with $k = 2$ is non-resonant,
whereas the eigenpairs with $k = 1$ and $k = 3$ are near-resonant.  In
particular, in the non-resonant case, the eigenfunction away from both
patches approaches a constant $U_0 \ne 0$, which is fixed by the
normalization.  In contrast, the near-resonant eigenfunctions vanish
in the far-field.  Since the eigenvalues $\sigma^{(2)}$ and
$\sigma^{(3)}$ are relatively close to each other (see Table
\ref{tab:SN}), the associated eigenfunctions almost coincide on the
``left'' patch and are almost antisymmetric on the ``right'' patch, to
ensure their mutual orthogonality.

In addition, the SN eigenfunctions in our symmetric antipodal patch
configuration are necessarily either symmetric, or antisymmetric with
respect to the horizontal plane $x_3 = 0$.  We observe that the
non-resonant eigenfunctions are symmetric, given that the two patches
``communicate'' with each other through the bulk ($U_0 \ne 0$).  In
turn, the near-resonant eigenfunctions are antisymmetric; they vanish
at the center and take small values in the bulk; in this case, the
``communication'' between two patches is of resonance type.  We
emphasize, however, that the above symmetry argument is specific to
our example of antipodal patches; in contrast, our asymptotic analysis
and the established distinction between non-resonant and near-resonant
cases remain valid for any mutual arrangement of the identical patches
on the sphere. 

\begin{figure}
\begin{center}
\includegraphics[width=80mm]{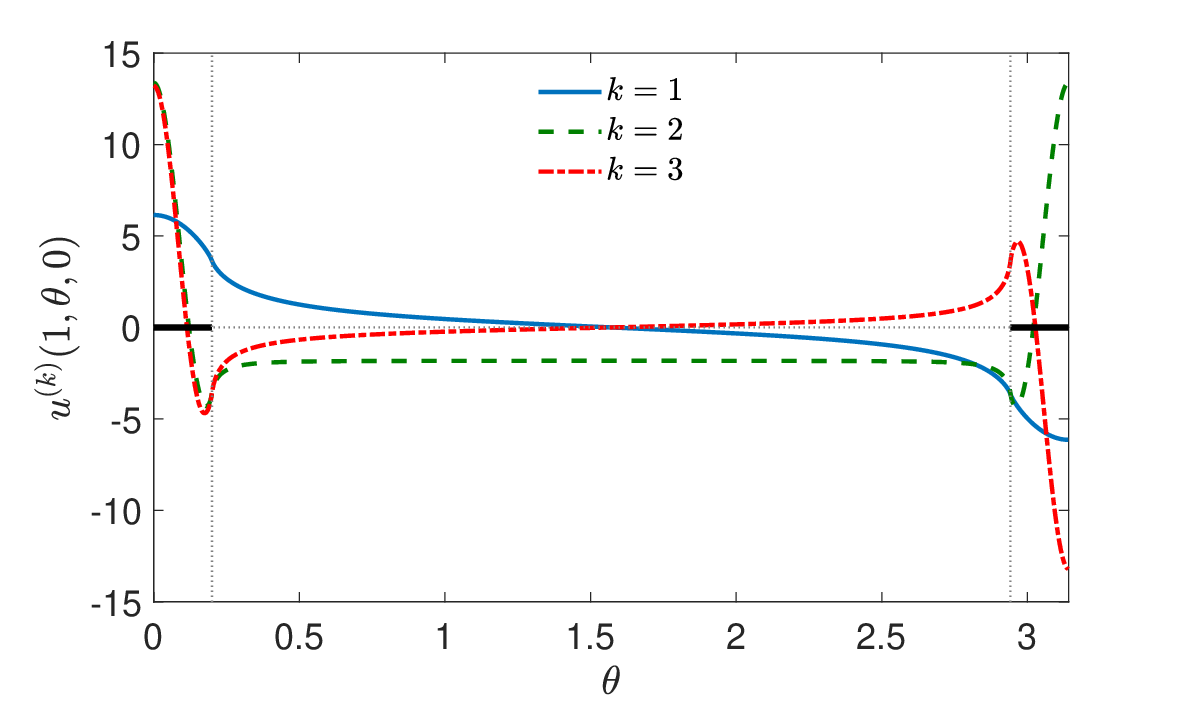} % SD2_eigenfunction.eps} 
\end{center}
\caption{
The first three nontrivial axially-symmetric eigenfunctions $u^{(k)}$
of the SN problem with two circular patches of radii $\eps_1 = \eps_2
= 0.2$ (with $a_1 = a_2 = 1$), located at the north and south poles.
They are plotted in the spherical coordinates $(r,\theta,\phi)$, by
fixing $r = 1$ and $\phi = 0$.  Black thick intervals indicate the
locations of the Steklov patches ($0 < \theta < \epsilon_1$ and
$\pi-\epsilon_2 < \theta < \pi$), with $\eps_i = \sin(\epsilon_i)$.
These eigenfunctions correspond to the eigenvalues $\sigma^{(k)}$
reported in Table \ref{tab:SN}.  We used the spectral method from
Appendix \ref{appf:numer}, with the truncation order $\nmax = 1000$
and ``geodesic convention''.}
\label{fig:SD_eigen}
% [vth,theta, mu,V,K] = A_Ward_Steklov_3d_sphereSN2_u_fig(mu,V,K);
\end{figure}

\vspace*{0.2cm}
\noindent {\bf Example II:} Next, we suppose that we have $N$
identical circular patches of a common radius $\eps$ (i.e., all
$a_i = 1$), with centers located at the vertices of one of the five
largest platonic solids that can be inscribed within the unit
sphere. For this spatial configuration of patches, where $M=N$ and
$N\in \lbrace{4,6,8,12,20\rbrace}$, the Green's matrix ${\mathcal
G}_{sc}$ is symmetric with a constant row sum, so that ${\mathcal
G}_{sc}\evec_M=\alpha_M\evec_M$. Note that this matrix is in
general not circulant.  It follows that, up to a normalization
condition, the $N-1$ mutually orthogonal solutions to
(\ref{sn:degen:avec}) are
\begin{equation}\label{sx:ex2:ac}
  \vac_j = \vq_j \,, \quad \mbox{where} \quad {\mathcal G}_{sc}\vq_j=\alpha_j
  \vq_j\,, \quad \vq_j^{T}\evec_{M}=0 \quad \mbox{for} \quad j=2,\ldots,N\,.
\end{equation}
In addition, we obtain from (\ref{sn:degen:U1_mat}), that we must have
$\overline{U}_1=0$ in (\ref{sn:degen:main_out}).  We conclude from
(\ref{sn:degen:main_sigma}) and (\ref{sn:degen:sigma_2}) that, for
each $k\geq 0$, there are $(N-1)$ SN eigenvalues of (\ref{sn:eig}),
enumerated by $j = 1,\ldots,N-1$, given by
\begin{equation}\label{sn:degen:ex2}  
  \sigma(\eps)\sim \mu_{kc} +
  \eps\log\left(\frac{\eps}{2} \right) \sigma_{1}^{(k)} -
  \frac{\eps \sigma_1^{(k)}}{\pi} \left( 4\pi^2 \alpha_j +
    {\mathcal J}^{(k)}\right)  + \ldots \,, \quad j=2,\ldots,N \,.
\end{equation}
We remark that $\sigma_1^{(k)}$ and ${\mathcal J}^{(k)}$, as given by
(\ref{sn:degen:sigma_1}) and (\ref{sn:degen:j}) respectively, are
independent of $N$ and of the locations of the vertices. In this way,
determining a three-term expansion for the $(N-1)$ SN eigenvalues that
are in near-resonance with an eigenvalue $\mu_{kc}$ of the local
Steklov problem reduces to finding all of the eigenpairs of the
Green's matrix ${\mathcal G}_{sc}$ that are orthogonal to $\evec_M$.

In addition, from our analysis in \S \ref{sec:sn:non}, we observe from
our main result in Proposition \ref{sn:main_res} that all non-resonant
eigenvalues are close to roots of $C(-\sigma_0)=0$.  These roots are
the eigenvalues $\mu_{k}^{N}$ of the local Steklov eigenvalue problem
(\ref{eq:Psi_def_N}) of Appendix \ref{sec:Cmu0} satisfying the
far-field Neumann condition (\ref{eq:VkN_inf}). In this case, the
expansion (\ref{sn:stek_eigex}) is independent of the vertex locations
of the platonic solids and reduces to
\begin{equation}
  \sigma(\eps)
  \sim \mu_{k}^{N} + \eps \frac{E(-\mu_{k}^{N})}{C^{\prime}(-\mu_{k}^{N})}
  + \cdots \,,
\end{equation}
which is the same result derived in Example I of \S
\ref{sn:example:nondegen} for a single Steklov patch.

\section{Discussion}\label{sec:discussion}

For a collection of small partially reactive patches of arbitrary
shape located on the boundary of a sphere in three dimensions, we have
explored four different PDE problems with mixed boundary conditions:
the Poisson equation for the mean first-reaction time (MFRT), the
Laplace equation for the splitting probability, the
Steklov-Dirichlet-Neumann (SDN) eigenvalue problem, and the
Steklov-Neumann (SN) eigenvalue problem.  To analyze these problems,
we have developed and implemented a unified mathematical approach,
based on the method of matched asymptotic expansions combined with
spectral theory.  For the MFRT, our analysis extends that of
\cite{Cheviakov10} where only perfectly reactive and locally circular
patches were considered.  In each case, our three-term asymptotic
results in the small-patch limit have been favorably compared with
full numerical results.

The main focus of this paper was on the derivation of accurate
asymptotic expansions for partially reactive patches.  As in the case
of Dirichlet patches (perfect sinks), these expansions can find
numerous applications in diverse disciplines such as chemistry,
biophysics, computational biology, neurosciences, and ecology
\cite{Bressloff13,Holcman13,Holcman14,Holcman,Benichou14,Lindenberg,Grebenkovbook}.
For instance, different asymptotic formulas for the MFRT presented in
\S \ref{mfpt_sec:expan} allow one to examine  the relative roles of
reactivities, sizes, shapes and mutual arrangements of the patches
onto the efficiency of the boundary for capturing diffusive particles.
We provided a detailed analysis of the reactive capacitance
$C_i(\kappa_i)$ that incorporates reactivities, sizes and shapes of
the patches into the leading-order term.  However, higher-order terms
are mandatory to uncover the effect of the spatial organization of
reactive patches onto transport and reaction rates.  For instance,
does a clustered organization of receptors present a trapping
advantage as compared to a uniform distribution?  (see
\cite{Taflia11,Freche11} and references therein).  Such optimization
problems can now be addressed with the help of our asymptotic
relations.
In particular, \S \ref{mfpt:homog} presents the homogenization
procedure, in which a collection of partially reactive patches can be
replaced by an effective Robin condition characterized by an
equivalent reactivity $\K_{\rm eff}$, which extends the earlier
studies for perfectly reactive patches \cite{Cheviakov10,Cheviakov13}.
In other words, such homogenized descriptions substitute a structured
heterogeneous reactive boundary by a much simpler homogeneous one,
allowing one to upscale the problem by eliminating small patches and
thus considerably reducing computational cost.

In the case of fixed reactivities $\K_i$, the small-patch limit
$\eps\to 0$ implies $\kappa_i \to 0$ so that the trapping efficiency
of the patches is progressively reduced.  In \S \ref{sec:moderateK},
we described the four-term expansion for the MFRT, which includes all
the relevant terms $\OO(\eps^{-2})$, $\OO(\eps^{-1})$, $\OO(\log
\eps)$, and $\OO(1)$.  These asymptotic results rely on the Taylor
expansion of the reactive capacitance and incorporate the patch shapes
more explicitly, through the coefficients $c_{2i}$ and $c_{3i}$.
These coefficients can be found directly from (\ref{eq:c1_def}),
without the need for solving the local Steklov eigenvalue problem.

An important byproduct of the MFRT analysis is the three-term
asymptotic expansion for the principal (lowest) eigenvalue $\lambda_0$
of the Laplace operator in the presence of partially reactive patches
(see \S \ref{sec:mfrt:eig}).  This eigenvalue determines the long-time
behavior of various time-dependent characteristics of the considered
diffusion-reaction system such as the propagator, the survival
probability, and the probability density function of the
first-reaction time.

In addition to the MFRT, in \S \ref{split:intro} we analyzed the
splitting probabilities that are often employed to describe the
competition between patches that capture the diffusive particles in
parallel.  In particular, one can fix the positions of all patches
except one, and then inspect the relative role of that particular
patch on the overall capture efficiency.  Moreover, if different
patches represent distinct reaction pathways, the splitting
probabilities characterize their relative contributions.  For
instance, one patch may represent a target site of interest (e.g., a
receptor or an enzyme), whereas other patches can be holes or
channels, through which the particles may escape from the domain.  In
this setting, the splitting probability describes the successful
reaction on the patch of interest, prior to the escape.  We also
provided alternative physical interpretations in terms of a
steady-state concentration with a source term on one patch, as well as
Newton cooling in heat transfer.

Finally, the conventional Robin boundary condition describes the most
basic type of surface reactions.  As briefly discussed in
\S \ref{all:intro}, choosing a suitable (non-exponential)
distribution for the threshold $\hat{\ell}$ in the definition
(\ref{eq:tau_def}) of the first-reaction time $\tau$, one can
introduce more sophisticated surface reaction mechanisms such as
progressive passivation or activation of the patch by its encounters
with the diffusive particle.  In turn, a PDE description of these
mechanisms employs Steklov-Neumann or Steklov-Dirichlet-Neumann
spectral problems \cite{Grebenkov20,Grebenkov23a,Grebenkov23}.  In
\S \ref{stekDN:intro} and \S \ref{stekN}, we present the asymptotic
expansions of Steklov eigenvalues and eigenfunctions, which determine
other diffusion-reaction characteristics (e.g.,
Ref. \cite{Grebenkov25} explains how these spectral expansions give
access to an asymptotic formula for the MFRT in the presence of a
single small patch with a general surface reaction).  Further
extensions and practical implications of these asymptotic results in
chemical physics and computational biology will be presented
elsewhere.

We now discuss several open problems that are directly related to our
study.

(i) It would be worthwhile to develop an accurate numerical scheme,
such as in \cite{Lindsay18a} and \cite{Kaye20} for the exterior
problem, to validate the homogenization result
(\ref{homog:keff_small}) for the effective reactivity $\keff$ for a
large collection of small, but equi-distributed, patches on the
boundary of a sphere. 

(ii) Our analysis has been focused only on problems that are {\em
interior} to the sphere.  In the companion article
\cite{Grebenkov-Ward25}, we derived the effective reactivity rate for
the {\em exterior} problem where there is a large number of
equi-distributed partially reactive patches of arbitrary shape, but
with small area, on the boundary of a sphere.  This analysis extends
that of \cite{Lindsay17} where an effective Robin boundary condition
was derived for the case of perfectly reacting but locally circular
boundary patches. In turn, it would be worthwhile to analyze the SDN
and SN problems {\em exterior} to a sphere that has a collection of
small partially reactive surface patches.

(iii) Our analysis of the SDN and SN problems has left open a few
technical issues.  In particular, our analysis has provided only the
leading-order term for the Steklov eigenvalue for the special case
where the corresponding eigenfunction is not axially-symmetric on a
circular patch.  In addition, for the SN problem, in \S
\ref{sec:sn_degen} we have only analyzed in detail the {\em
near-resonant} case that will occur when a subset of the patches are
identical.  We did not provide a similar analysis to investigate {\em
near-resonance} behavior in other, less generic, situations, involving
non-identical patches, or when the local Steklov eigenvalue problem
near a patch has non-degenerate eigenvalues.

(iv) Given the central role of the reactive capacitance
$C_i(\kappa_i)$ and the monopole coefficient $E_i(\kappa_i)$ in our
asymptotic results, it is worthwhile to more fully characterize these
key quantities for an arbitrarily-shaped patch. Some monotonicity
principles for $C_i(\kappa_i)$ with regards to the patch shape and the
reactivity $\kappa_i$ have been derived in \cite{Grebenkov26}.
However, there are no such results to date for $E_i(\kappa_i)$. The
determination of these two quantities for a given patch shape relies
on calculating the charge density $q_i$ (recall (\ref{mfpt:wc_charge})
and (\ref{eq:Ei_general0})).  For an elliptically shaped patch $\PT_i$
with semi-axes $a_1$ and $a_2$, by an extension of the analysis in
\cite{Strieder09}, we can calculate explicitly $C=C_i(\infty)$, $q =
q_i(y_1,y_2;\infty)$, and $E=E_i(\infty)$ for a perfectly reactive
patch as
\begin{equation}\label{disc:ellipse} 
C=\frac{a_{>}}{K(m)} \,, \,\,\, q
= \frac{C}{2a_1a_2} \frac{1}{\sqrt{1 - y_1^2/a_1^2 - y_2^2/a_2^2}}
\,, \,\,\, \frac{E}{C^2} =
-\frac{1}{2}\log\left(\frac{a_1+a_2}{2}\right) -\log{2} + \frac{3}{4}
\,, \end{equation} where $a_{<}=\min(a_1,a_2)$, $a_{>}=\max(a_1,a_2)$,
and $K(m)$ is the complete elliptic integral of the first kind with
modulus $m\equiv \left(1- {a_{<}^2/a_{>}^2}\right)^{1/2}$. However, we
are not aware of any other exact solutions for different patch
shapes.
  
(v) Our asymptotic analysis relies on the assumption of a smooth
boundary so that the patches become flat in the limit $\eps \to 0$. In
particular, the inner expansion near each patch involves the solution
$w_i(\y;\kappa_i)$ of the BVP (\ref{mfpt:wc}) in the upper half-space.
In turn, this solution admits a spectral expansion over the local
Steklov eigenfunctions satisfying (\ref{eq:Psi_def}) that yields the
spectral expansion (\ref{eq:Cmu_def0}) for the reactive capacitance.
However, some biomedical and neuroscience applications require solving
the narrow escape to structured targets possessing spine-like,
``hairy'', or ``bumpy'' shapes located on a smooth surface
\cite{Reingruber09,Greb-Skvor22,Greb-Skvor23}, such as surfaces
coated by polymers or nanopillars.  Therefore, in practice, one may
need to deal with a smooth boundary covered by small cones, pyramids,
capped cylinders or half-ellipsoids, which do not reduce to flat
patches even in the limit $\eps\to 0$.  The proposed approach can be
formally applied to such bumpy patches. In fact, the shape of the
patch only affects the local Steklov eigenvalues and eigenfunctions
that in turn determine the reactive capacitance $C_i(\kappa_i)$ and
the correction term $E_i(\kappa_i)$.  For instance, the local exterior
Steklov problem for (half)spheroidal patches was studied in detail in
\cite{Grebenkov24}.  Moreover, in the case of a half-spherical patch,
the analysis can actually be even simpler than for flat patches
because the local problems become rotation invariant.  It is therefore
interesting to investigate how ``bumpiness'' of the patches and their
shapes may affect the narrow escape problem (see also
\cite{Grebenkov26}).

(vi) As we stated earlier, the leading-order terms of the
derived asymptotic expansions are independent of the actual shape of
the bounded domain and thus readily applicable to an arbitrary bounded
3-D domain $\Omega$ with a smooth boundary $\partial\Omega$ that
contains a collection of $N$ partially reactive surface
patches. However, extending the analysis to higher-order to capture
the role of the curvature of the boundary and the inter-patch
interactions is a largely open research frontier. In
\cite{Cheviakov15}, a two-term asymptotic expansion for the MFPT in an
arbitrary 3-D domain for perfectly reactive patches revealed the
effect on the MFPT of the local mean curvature of the smooth
boundary. However this analysis did not capture the higher-order
inter-patch interactions, as this would have required a detailed study
of the singularity behavior as $\x\to\x_i$ of the surface Neumann
Green's function $G_s(\x;\x_i)$ satisfying (\ref{mfpt:sph}).  By using
techniques from microlocal analysis, it was established rigorously in
Proposition 1 of \cite{Tzou21} that this singularity behavior has the
form
\begin{align} \label{extent:sing}
  G_{s}(\x;\x_i) & = \frac{1}{2\pi|\x-\x_i|} -\frac{{\mathcal H}_i}{4\pi}
  \log\left(|\x-\x_i| +d(\x) \right) \\  \nonumber
& \qquad + R_{s}(\x_i) + e(\x;\x_i) + o(1)
  \,, \quad \mbox{as} \quad \x\to \x_i \,.
\end{align}
Here ${\mathcal H}_i$ is the mean curvature of $\partial\Omega$ at
$\x_i$, $d(\x)\equiv \mbox{dist}(\x,\partial\Omega)>0$ is the distance
to the boundary from $\x$, $e(\x;\x_i)$ is an anisotropic term that
depends on the specific path of approach to
$\x=\x_i\in\partial\Omega$, while the {\em regular part} $R_{s}(\x_i)$
is a {\em global quantity} that depends on $\Omega$ and $\x_i$.  From
the resolution of this singularity structure of $G_s$, a three-term
expansion for the MFPT for a single local circular or elliptically
shaped perfectly reactive boundary patch was rigorously derived for a
general Riemannian manifold \cite{Tzou21}.

An extension of this result to a collection of small partially
reactive patches of arbitrary shape on the smooth boundary of a
general 3-D domain presents an open problem.  More specifically, one
could aim to extend our asymptotic framework to provide a three-term
expansion for the MFRT, the splitting probability, and the Steklov
eigenvalues for both the SDN and SN problems.  Local geodesic
coordinates, or a tangential-normal coordinate system, would lead to
very similar local problems near each patch as derived herein.  The
main challenge for obtaining readily evaluated asymptotic formulas
would be to develop an efficient numerical scheme to compute $G_s$ and
$R_s$ by exploiting the intricate singularity behavior
(\ref{extent:sing}).  Finally, we remark that some specific 3-D domain
shapes where exact solutions for $G_s$ and $R_s$ may be available are
a (capped) cylinder or a spheroid, for which the Laplace operator
admits a separation of variables in curvilinear coordinates.  These
basic geometries arise naturally in cellular biology and neuroscience,
where partially reactive boundaries play an important role.

\section*{Acknowledgements} 
D.S.G. acknowledges the Simons Foundation for supporting his
sabbatical sojourn in 2024 at the CRM (CNRS -- Universit\'e de
Montr\'eal, Canada), and the Alexander von Humboldt Foundation for
support within a Bessel Prize award.  M.J.W. was supported by the
NSERC Discovery grant program. We are grateful to Prof. Daniel Gomez
of the University of New Mexico for providing the
Fig.~\ref{fig:geodesic} of the geodesic coordinates and to Prof. Sean
Lawley of the University of Utah for discussions related to the
sigmoidal approximation in (\ref{mfpt:sigmoidal_2}).  We are grateful
to the anonymous referees for their comments on the initial
manuscript.

\appendix
\renewcommand{\theequation}{\Alph{section}.\arabic{equation}}

\section{Geodesic Normal Coordinates to the Unit Sphere
$\Omega$}\label{app_g:geod}

We define geodesic normal coordinates
$\bxi=(\xi_1,\xi_2,\xi_3)^T\in \left({-\pi/2},{\pi/2}\right) \times
\left(-\pi,\pi\right)\times [0,1]$ in $\Omega\cup \partial\Omega$ so
that $\bxi=0$ corresponds to $\x_i\in\partial\Omega$, with $\xi_3>0$
corresponding to the interior of $\Omega$.  Geodesics on
$\partial\Omega$ are obtained by setting $\xi_3=0$ and fixing either
$\xi_1=0$ or $\xi_2=0$.  In terms of the spherical angles $\theta_i\in
(0,\pi)$ and $\varphi_i\in [0,2\pi)$ (see Fig.~\ref{fig:geodesic}),
and for $|\x_i|=1$, we define the orthonormal vectors
\begin{equation}\label{app_g:ortho}
  \x_i \equiv \begin{bmatrix}
           \cos\varphi_i \, \sin\theta_i \\
           \sin\varphi_i \, \sin\theta_i \\
           \cos\theta_i
         \end{bmatrix}\,, \,\,\,
         \vv_{2i}=\partial_{\theta} \x_i \equiv
\begin{bmatrix}
           \cos\varphi_i \, \cos\theta_i \\
           \sin\varphi_i \, \cos\theta_i \\
           -\sin\theta_i
         \end{bmatrix}\,, \,\,\,
         \vv_{3i}=\x_i {\bf \times} \partial_{\theta} \x_i \equiv
\begin{bmatrix}
           -\sin\varphi_i \\
           \cos\varphi_i \\
           0
         \end{bmatrix}\,.
\end{equation}
The vectors $\vv_{2i}$ and $\vv_{3i}$ span the tangent plane to the
sphere at $\x=\x_i$.  We now define the geodesic normal coordinates
$\bxi=(\xi_1,\xi_2,\xi_3)^T$ by the global transformation
\begin{equation}\label{app_g:global}
  \x(\bxi) = \left(1-\xi_3\right)\left( \cos\xi_1 \,
    \cos\xi_2 \, \x_i + \cos\xi_1 \,
    \sin\xi_2 \, \vv_{2i} + \sin\xi_1 \vv_{3i}\right)\,.
\end{equation}
We observe that $\xi_3$ measures the distance of $\x$ to
$\partial\Omega$.  The curves obtained by setting $\xi_3=0$, and
fixing either $\xi_2=0$ or $\xi_1=0$ are, respectively,
$\x(\xi_1,0,0)=\cos\xi_1 \, \x_i + \sin\xi_1 \, \vv_{3i}$ or
$\x(0,\xi_2,0)=\cos\xi_2 \, \x_i + \sin\xi_2 \, \vv_{2i}$, which
correspond to intersections of $\partial\Omega$ with planes spanned by
$\lbrace{\x_i,\vv_{3i}\rbrace}$ or $\lbrace{\x_i,\vv_{2i}\rbrace}$,
respectively.

To transform the Laplacian from Cartesian to geodesic coordinates, we
use (\ref{app_g:global}) to calculate the scale factors
$h_{\xi_j}\equiv \vert {\partial\x/\partial \xi_j}\vert$ for $j=1,2,3$
as
\begin{equation}\label{app_g:scale}
  h_{\xi_1} = (1-\xi_3) \,, \qquad h_{\xi_2}=(1-\xi_3)\cos\xi_1 \,, \qquad
  h_{\xi_3}=1 \,.
\end{equation}
For the transformation of a generic function
${\mathcal V}(\bxi)\equiv u\left(\x({\bxi})\right)$, we calculate that
\begin{equation*}
  \begin{split}
    \Delta_{\x} u &= \frac{1}{h_{\xi_1}h_{\xi_2}h_{\xi_3}} \left[
      \frac{\partial}{\partial \xi_1}\left( \frac{h_{\xi_2}h_{\xi_3}}
        {h_{\xi_1}} {\mathcal V}_{\xi_1} \right) +
      \frac{\partial}{\partial \xi_2}\left( \frac{h_{\xi_1}h_{\xi_3}}
        {h_{\xi_2}} {\mathcal V}_{\xi_2} \right) +
            \frac{\partial}{\partial \xi_3}\left( \frac{h_{\xi_1}h_{\xi_2}}
              {h_{\xi_3}} {\mathcal V}_{\xi_3} \right) \right] \,,\\
          &= \frac{(1-\xi_3)^{-2}}{\cos\xi_1} \left[
            \frac{\partial}{\partial \xi_3} \left( (1-\xi_3)^2
              \cos\xi_1 {\mathcal V}_{\xi_3}\right) +
         \frac{\partial}{\partial \xi_1} \left( \cos\xi_1 {\mathcal V}_{\xi_1}
            \right) +
            \frac{\partial}{\partial \xi_2} \left( \frac{1}{\cos\xi_1 }
              {\mathcal V}_{\xi_2}\right) \right] \,,
        \end{split}
\end{equation*}
which yields
\begin{equation}\label{app_g:lap}
\Delta_{\x} u = {\mathcal V}_{\xi_3\xi_3} - \frac{2}{1-\xi_3} {\mathcal V}_{\xi_3} +
\frac{1}{(1-\xi_3)^2\cos^{2}\xi_1} {\mathcal V}_{\xi_2\xi_2} +
\frac{1}{(1-\xi_3)^2 \cos\xi_1} \frac{\partial}{\partial_{\xi_1}}
\left(\cos\xi_1 {\mathcal V}_{\xi_1} \right)\,.
\end{equation}

Next, by introducing the inner, or local variables,
$\y=(y_1,y_2,y_3)^T$, defined by
\begin{equation}\label{app_g:innvar}
  \xi_1=\eps y_1 \,, \qquad \xi_2=\eps y_2 \,, \qquad \xi_3 =\eps y_3\,,
\end{equation}
we use the Taylor approximations $(1-\xi_3)^{-1}\sim 1+\eps y_3$,
$(1-\xi_3)^{-2}\sim 1+2\eps y_3$, $\cos^{2}\xi_1=1+{\mathcal O}(\eps^2)$ and
$\sin\xi_1\sim \eps y_1$. We readily obtain that (\ref{app_g:lap}) reduces
to (\ref{mfpt:local}).

To determine a two-term approximation for the Euclidean distance
$|\x-\x_i|$ near the patch, we use (\ref{app_g:innvar}) in
(\ref{app_g:global}). From a Taylor series approximation we obtain that
\bsub \label{app_g:vab_all}
\begin{equation}
  \x-\x_i = \eps \vb_0 - \eps^2 \vb_1 + {\mathcal O}(\eps^3)\,, \qquad
  |\x-\x_i|^2=\eps^2 \vb_0^T\vb_0 - 2 \eps^3 \vb_0^T \vb_1  +
  {\mathcal O}(\eps^4)\,,
\end{equation}
where $\vb_0$ and $\vb_1$ are defined by
\begin{equation}\label{app_g:vab}
  \vb_0 = - y_3 \x_i + y_2 \vv_{2i} + y_1 \vv_{3i} \,, \qquad
  \vb_1= \frac{1}{2} \left(y_1^2+y_2^2\right) \x_i + y_3 y_2 \vv_{2i} + y_3
  y_1 \vv_{3i} \,.
\end{equation}
\esub
By calculating
$\vb_0^T\vb_0 =y_1^2+y_2^2+y_3^2 \equiv \rho^2$ and
$\vb_0^T\vb_1={y_3\left(y_1^2+y_2^2\right)/2}$, we get
\begin{equation}\label{app_g:loc}
  |\x-\x_i|\sim \eps \rho - \frac{\eps^2 y_3}{2\rho} \left(y_1^2+y_2^2
    \right) + {\mathcal O}(\eps^3) \,, \quad
  \frac{1}{|\x-\x_i|} \sim \frac{1}{\eps \rho} \left( 1 +
    \frac{\eps y_3}{2\rho^2} (y_1^2+y_2^2) + {\mathcal O}(\eps^2)\right)\,.
\end{equation}

In matrix form, and to leading order in $\eps$, we can write
  (\ref{app_g:vab_all}) in terms of $\y=(y_1,y_2,y_3)^T$ and
  an orthogonal matrix ${\mathcal Q}_i$ as
\begin{equation}\label{app_g:change}
  \y \sim \eps^{-1} {\mathcal Q}_{i}^T (\x-\x_i) \,, \quad \mbox{where} \quad
  {\mathcal Q}_{i} \equiv \begin{bmatrix}
    \vert & \vert & \vert \\
    \vv_{3i}   & \vv_{2i} & -\x_i \\
    \vert & \vert & \vert
\end{bmatrix} \quad \rightarrow \quad |\y|\sim \eps^{-1}|\x-\x_i|\,.
\end{equation}

Since $|\x-\x_i|=\eps \rho + {\mathcal O}(\eps^3)$ when $y_3=0$
from (\ref{app_g:loc}), it follows that if the Robin patch
$\partial\Omega_i^{\eps}$ is a spherical cap with the radius $a$ of
the base (i.e., whose projection onto the tangent plane at $\x_i$ is a
disk of radius $a$), then in terms of local geodesic coordinates we
have that $\left(y_1^2+y_2^2\right)^{1/2}\leq a + {\mathcal
O}(\eps^2)$.  One sees that the geodesic radius of the boundary patch
is close to $a$, up to ${\mathcal O}(\eps^2)$.  This small error
between the orthogonally projected and on-surface patches does not
influence our three-term asymptotic results obtained in our analysis
of the four problems.

In addition, since the scale factor is $h_{\xi_3}=1$, the Robin
boundary condition on a circular patch is well-approximated in the
local geodesic coordinates by
\begin{equation*}
  - \partial_{y_3} U  + \kappa U=0 \,, \qquad
  \mbox{for} \quad y_3=0 \,, \,\, (y_1^2+y_2^2)^{1/2}\leq a \,.
\end{equation*}
Finally, by using the scale factors (\ref{app_g:scale}), the surface area
element on the unit sphere, as needed in
(\ref{sdn:eig_norm}), is $d{\bf s}= \left(h_{\xi_1}h_{\xi_2}\right)
\vert_{\xi_3=0}d\xi_1d\xi_2= \cos(\xi_1)d\xi_1 d\xi_2$.

\section{Asymptotic Behavior of $C_{i}(\kappa_i)$ as $\kappa_i\ll 1$ for the Disk}
\label{app_a:Ckappa}

We derive (\ref{mfpt:cj_small_b}) of Lemma \ref{lemma:Cj_kappa} for a
disk-shaped patch of radius $a_i$. To do so, we first introduce
$\W(\y)$ by
\begin{equation}\label{app_a:W}
  w_{i}(\y) = 1- C_i\W(\y)\,,
\end{equation}
so that, upon dropping the subscript $i$, we obtain from (\ref{mfpt:wc}) that
$\W$ satisfies
\bsub \label{app_a:Wc}
\begin{align}
  \Delta_{\y} \W &=0 \,, \quad \y \in \R_{+}^{3} \,, \label{app_a:wc_1}\\
  -\partial_{y_3} \W + \kappa \W &=0 \,, \quad y_3=0 \,,\,
                                   (y_1,y_2)\in \PT\,,  \label{app_a:wc_2}\\
  \partial_{y_3} \W &=0 \,, \quad y_3=0 \,,\, (y_1,y_2)\notin \PT
                      \,, \label{app_a:wc_3}\\
  \W &\sim  B(\kappa) -\frac{1}{|\y|} + \ldots\,,  \quad \mbox{as}\quad
       |\y|\to \infty \,, \label{app_a:wc_4}
\end{align}
\esub where the neglected term in (\ref{app_a:wc_4}) is a dipole and
where $B(\kappa)$ is related to $C(\kappa)$ by
\begin{equation}\label{app_a:BtoC}
  C(\kappa) = \frac{1}{B(\kappa)}\,.
\end{equation}
Here $\PT$ is a disk of radius $a$.  By applying the divergence
theorem over the hemisphere
$\Omega_R=\lbrace{\y = (y_1,y_2,y_3) \,|\,|\y|\leq R\,, \, y_3\geq 0
  \rbrace}$, with boundary $\partial\Omega_R$, and with $\partial_n$
denoting the outward normal derivative to $\partial\Omega_R$, we get
from (\ref{app_a:wc_4}) that
\begin{equation}\label{app_a:w_inf}
  \lim_{R\to\infty} \int_{\partial\Omega_R} \partial_n \W\, d{\bf s} =2\pi \,.
\end{equation}

For $\kappa\ll 1$, we expand the solution to (\ref{app_a:Wc}) as
\begin{equation}\label{app_a:W_exp}
  \W = \frac{b_0}{\kappa} + \left(\W_{1}+b_1\right) + \kappa
  \left(\W_{2}+b_2\right) + \kappa^2 \left(\W_{2}+b_3\right) + \ldots\,.
\end{equation}
We substitute (\ref{app_a:W_exp}) into (\ref{app_a:Wc}) and collect
powers of $\kappa$. At leading-order, we choose $b_0$ so that
$\W_{1}$ satisfies
\bsub \label{app_a:W1}
\begin{align}
    \Delta_{\y} \W_1 &=0 \,, \quad \y \in \R_{+}^{3} \,, \label{app_a:1w_c1}\\
    \partial_{y_3} \W_1 &= b_0 \,, \quad y_3=0 \,,\,
    (y_1,y_2)\in \PT\,,  \label{app_a:1wc_2}\\
    \partial_{y_3} \W_1 &=0 \,, \quad y_3=0 \,,\, (y_1,y_2)\notin \PT
    \,, \label{app_a:1wc_3}\\
    \lim_{R\to\infty} \int_{\partial\Omega_R} \partial_n \W_1\, d{\bf s} &=2\pi \,,
    \qquad \W_1 \sim -\frac{1}{|\y|} + o(1)\,,   \quad \mbox{as}\quad
    |\y|\to \infty \,. \label{app_a:1wc_4}
\end{align}
\esub The $o(1)$ condition in (\ref{app_a:1wc_4}) determines $\W_1$
uniquely.  By applying the divergence theorem to (\ref{app_a:W1}) over
$\Omega_R$, we let $R\to \infty$ to obtain $2\pi - b_0 \pi a^2=0$, so
that
\begin{equation}\label{app_a:b0}
  b_0 = \frac{2}{a^2} \,.
\end{equation}

At higher order in $\kappa$, for each $m=2,3,\ldots$, we will
choose $b_{m-1}$ so that $\W_{m}$ is the unique solution to
\bsub \label{app_a:Wm}
\begin{align}
    \Delta_{\y} \W_m &=0 \,, \quad \y \in \R_{+}^{3} \,, \label{app_a:mw_c1}\\
    \partial_{y_3} \W_m &= \W_{m-1}+b_{m-1} \,, \quad y_3=0 \,,\,
    (y_1,y_2)\in \PT\,,  \label{app_a:mwc_2}\\
    \partial_{y_3} \W_m &=0 \,, \quad y_3=0 \,,\, (y_1,y_2)\notin \PT
    \,, \label{app_a:mwc_3}\\
    \lim_{R\to\infty} \int_{\partial\Omega_R} \partial_n \W_m\, d{\bf s} &=0 \,,
    \qquad \W_m \sim o(1)\,,  \quad \mbox{as}\quad
    |\y|\to \infty \,. \label{app_a:mwc_4}
\end{align}
By applying the divergence theorem to (\ref{app_a:Wm}) we obtain that
$\int_{\PT} \left(\W_{m-1}+b_{m-1}\right)\, d{\bf s}=0$, which
determines $b_{m-1}$ as
\begin{equation}\label{app_a:mwc_inf}
  b_{m-1}=-\frac{1}{\pi a^2} \int_{\PT} \W_{m-1}(y_1,y_2,0)\,  dy_1\,
  dy_2 \,, \qquad m=2,3,\ldots \,.
\end{equation}
\esub

With $b_0$ determined in (\ref{app_a:b0}), we use the method of images
to calculate $\W_1$ as
\begin{equation}\label{app_a:w1sol}
  \W_1(\y) = -\frac{1}{\pi a^2} \int_{\PT} \frac{d\xi_1 d\xi_2}{\left[
      (\xi_1-y_1)^2+(\xi_2-y_2)^2+y_3^2\right]^{1/2}}\,.
\end{equation}
In particular, for $y_3=0$, and with $\rho_0=(y_1^2+y_2^2)^{1/2}\leq a$, the
double integral in (\ref{app_a:w1sol}) can be evaluated as
\begin{equation}\label{app_a:w1sol_0}
  \W_1(y_1,y_2,0)= -\frac{4}{\pi a} E\left({\rho_0/a}\right) \,,
  \quad 0\leq \rho_0\leq a \,,
\end{equation}
where $E(z)\equiv\int_{0}^{\pi/2}\sqrt{1-z^2\sin^{2}\theta}\, d\theta$ is the
complete elliptic integral of the second kind. By using
(\ref{app_a:mwc_inf}) with $m=2$, and exploiting radial symmetry, we conclude
that 
\begin{equation}\label{app_a:b1}
  b_1 = \frac{8}{\pi a^3} \int_{0}^{a} \rho_0 E\left({\rho_0/a}\right) \,
  d\rho_0 = \frac{8}{\pi a} \int_{0}^{1} z E(z) \, dz = \frac{16}{3\pi a}\,,
\end{equation}
where we have used $\int_{0}^{1} z E(z)\, dz={2/3}$ from (5.112) of
\cite{Gradstein}.

By using the method of images we calculate that
\begin{equation}\label{app_a:w2p}
  \W_2(\y)= \frac{1}{2\pi} \int_{\PT} \frac{4 (\pi a)^{-1}
    E\left({|\xi|/a}\right) - b_1}{\left[
        (\xi_1-y_1)^2+(\xi_2-y_2)^2+y_3^2\right]^{1/2}} \, d\xi_1 d\xi_2 \,,
\end{equation}
where $|\xi|=(\xi_1^2+\xi_2^2)^{1/2}$.  By evaluating (\ref{app_a:w2p}) on
$y_3=0$, we obtain on the patch $0\leq\rho_0\leq a$ that
\begin{equation}\label{app_a:w2p_zero}
  \W_{2}(y_1,y_2,0)=-\frac{2a}{\pi} b_1 E\left(\frac{\rho_0}{a}\right)
  + \frac{2}{\pi^2 a} \int_{\PT} \frac{E\left({|\xi|/a}\right)}
  {\left[
      (\xi_1-y_1)^2+(\xi_2-y_2)^2\right]^{1/2}} \, d\xi_1 d\xi_2 \,.
\end{equation}
Upon substituting (\ref{app_a:w2p_zero}) into (\ref{app_a:mwc_inf}) with
$m=3$, we obtain that
\begin{equation}\label{app_a:b2_1}
  b_2 =\frac{4b_1}{\pi a} \int_{0}^{a} \rho_0 E\left(\frac{\rho_0}{a}
  \right) \, d\rho_0 - \frac{4}{\pi^2 a^3} \int_{0}^{a}
  \rho_0 {\mathcal J}(\rho_0)\, d\rho_0 \,,
\end{equation}
where ${\mathcal J}(\rho_0)$ is defined by
\begin{equation}\label{app_a:jrho}
{\mathcal J}(\rho_0) \equiv \int_{\PT} \frac{E\left({|\xi|/a}\right)}
  {\left[
      (\xi_1-y_1)^2+(\xi_2-y_2)^2\right]^{1/2}} \, d\xi_1 d\xi_2\,,
  \qquad \rho_0=\sqrt{y_1^2+y_2^2} \,.
\end{equation}
By using (\ref{app_a:b1}) for $b_1$ and $\int_{0}^{1} z E(z)\, dz={2/3}$ we
can calculate the first term in (\ref{app_a:b2_1}). Then, by writing the
second integral in (\ref{app_a:jrho}) in polar coordinates we obtain that
\begin{equation}\label{app_a:b2_2}
b_2 =\frac{128}{9\pi^2} -\frac{8}{\pi^2}\int_{0}^{1} \eta{\mathcal H}(\eta)\,
d\eta\,, \quad \mbox{where} \quad {\mathcal H}(\eta) \equiv
\int_{0}^{\pi}\int_{0}^{1} \frac{ E(r) r\, dr \, d\theta}{\left[r^2+\eta^2 -
    2 r \eta\cos\theta\right]^{1/2}}\,.
\end{equation}

Next, we use $\int_{0}^{\pi} \left[1+\beta^2-2\beta \cos\theta\right]^{-1/2}
\, d\theta= 2(1+\beta)^{-1}K\left({2\sqrt{\beta}/(1+\beta)}\right)$ for
$0\leq \beta<1$, where $K(z)$ is the complete elliptic integral of the first
kind of modulus $z$ (see (3.6.17) of \cite{Gradstein}). We conclude from
(\ref{app_a:b2_2}) that
\begin{equation}\label{app_a:heval}
  {\mathcal H}(\eta) = \int_{0}^{1} \frac{2 r E(r)}{r+\eta}
  K\left( \frac{2\sqrt{\eta r}}{r+\eta}\right) \, dr\,.
\end{equation}
We label $A_0\equiv \int_{0}^{1}\eta {\mathcal H}(\eta)\, d\eta$, in which
we use (\ref{app_a:heval}) for ${\mathcal H}(\eta)$. Upon switching the
order of integration and decomposing the resulting integral into two
parts we obtain
\begin{equation}\label{app_a:A0}
  A_0=\int_{0}^{1} r E(r) \left[ \int_{0}^{r} \frac{2\eta}{r+\eta}
    K\left(\frac{2\sqrt{r\eta}}{r+\eta}\right) \, d\eta +
\int_{r}^{1} \frac{2\eta}{r+\eta}
K\left(\frac{2\sqrt{r\eta}}{r+\eta}\right) \, d\eta\right] \, dr\,.
\end{equation}
To ensure that the modulus of the elliptic functions are on $[0,1]$, we
introduce the new variables $s={\eta/r}$ and $s={r/\eta}$ in the first and
second integrals of (\ref{app_a:A0}), respectively. This yields that
\begin{equation}\label{app_a:A0_2}
  A_0=\int_{0}^{1} 2r^2 E(r) \left[ \int_{0}^{1} \frac{s}{1+s}
    K\left(\frac{2\sqrt{s}}{1+s}\right) \, ds +
   \int_{r}^{1} \frac{1}{s^2(1+s)}
K\left(\frac{2\sqrt{s}}{1+s}\right) \, ds \right] \, dr\,.
\end{equation}
Since $0\leq s\leq 1$, we use the Landen transformation
$K\left({2\sqrt{s}/(1+s)}\right)=(1+s)K(s)$ in (\ref{app_a:A0_2}) to
obtain that
\begin{equation}\label{app_a:A0_3}
  A_0 = \int_{0}^{1} 2 r^2 E(r) \left[ \int_{0}^{1} s K(s) \, ds +
    \int_{r}^{1} s^{-2} K(s) \, ds \right] \, dr\,.
\end{equation}
By using the indefinite integrals
$\int s^{-2} K(s) \, ds = -s^{-1}E(s)$ and
$\int s K(s)\, ds = E(s)-(1-s^2)K(s)$ from (6.12.05) and (6.10.01) of
\cite{Gradstein} together with $E(0)=K(0)={\pi/2}$, we obtain from
(\ref{app_a:A0_3}) that
\begin{equation}\label{app_a:A0_fin}
  A_0 = \int_{0}^{1} 2r^2 E(r) \left[ E(1) + \left(-E(1) + r^{-1}E(r)
    \right)\right] \, dr = \int_{0}^{1} 2r \left[E(r)\right]^2 \, dr\,.
\end{equation}
In this way, since $A_0=\int_{0}^{1}\eta {\mathcal H}(\eta)\, d\eta$ we
obtain from (\ref{app_a:b2_2}) and (\ref{app_a:A0_fin}) that
\begin{equation}\label{app_a:b2_final}
  b_2= \frac{128}{9\pi^2} - \frac{8}{\pi^2}
  \int_{0}^{1} 2r \left[E(r)\right]^2 \, dr\,.
\end{equation}

Finally, upon substituting $B(\kappa)\sim b_0\kappa^{-1}+ b_1 +
b_2\kappa$ into (\ref{app_a:BtoC}), we revert the expansion to
conclude that
\begin{equation}\label{app_a:ckappa}
  C(\kappa) \sim \frac{\kappa}{b_0} - \frac{b_1\kappa^2}{b_0^2}+
  \frac{\kappa^3}{b_0^3}\left(b_1^2 - b_0 b_2\right) + \ldots \,.
\end{equation}
Upon using (\ref{app_a:b0}), (\ref{app_a:b1}) and
(\ref{app_a:b2_final}) in (\ref{app_a:ckappa}), we obtain
(\ref{mfpt:cj_small_b}) of Lemma \ref{lemma:Cj_kappa}. Moreover, we
can also readily identify the first three terms in the Taylor expansion
(\ref{eq:Cmu_Taylor}) as given in (\ref{eq:cn_exact}).

\section{Inner Problem Beyond Tangent Plane Approximation}\label{app_h:inn2}

Omitting the subscript $i$ for the $i$-th patch, we now analyze the
following inner problem that arises at a higher order beyond the
tangent-plane approximation:
\bsub\label{app_h:inn2_prob}
\begin{align}
  \Delta_{\y} \Phi_2 &= - 2\left( y_3 w_{y_3 y_3} + w_{ y_3}\right)
                       \,, \quad \y \in \R_{+}^{3} \,, \label{app_h:inn2_1}\\
  -\partial_{y_3} \Phi_2 + \kappa \Phi_2 &= 0 \,, \quad y_3=0 \,,\,
                   (y_1,y_2)\in \PT\,,  \label{app_h:inn2_2}\\
  \partial_{y_3} \Phi_2 &=0 \,, \quad y_3=0 \,,\, (y_1,y_2)\notin \PT
                          \,, \label{app_h:inn2_3}\\
  \Phi_2 \sim \frac{C}{2}\log(y_3+\rho) - & \frac{C y_3}{2\rho^3}
  (y_1^2+y_2^2) + \frac{E}{\rho} + \ldots\,,  \quad
          \mbox{as} \quad \rho\to \infty \,, \label{app_h:inn2_4}
\end{align}
\esub where $\rho=(y_1^2+y_2^2+y_3^2)^{1/2}$ and $\PT$ is the Robin
patch, which is not necessarily circular. Here $w_{}(\y)$ is the
solution to the leading-order problem (\ref{mfpt:wc}), where
$C=C(\kappa)$ is the coefficient of the monopole in the far-field
behavior (\ref{mfpt:wc_4}).  Our goal is to determine the coefficient
$E$ of the monopole term in the far-field (\ref{app_h:inn2_4}). The
result is as follows:

\begin{lemma}\label{lemma:Phi2}
  The solution to (\ref{app_h:inn2_prob}) can be decomposed as
\begin{equation}\label{app_h:decomp}
  \Phi_2 = \Phi_{2p} + \Phi_{2h} \,,
\end{equation}
where 
\begin{equation}\label{app_h:phi_2p}
  \Phi_{2p}=-\frac{y_3^2}{2}w_{y_3} - \frac{y_3}{2}w + \frac{1}{2}
  \int_{0}^{y_3} w(y_1,y_2,\eta)\, d\eta + {\mathcal F}(y_1,y_2;\kappa)\,,
\end{equation}
and where ${\mathcal F}(y_1,y_2;\kappa)$, with $\Delta_{S}{\mathcal F} \equiv
{\mathcal F}_{y_1 y_1}+ {\mathcal F}_{y_2 y_2}$, is the unique solution to
\bsub \label{app_h:fcal}
\begin{gather}
  \Delta_{S}{\mathcal F} = q(y_1,y_2;\kappa) I_{\PT} \,; \,\,\,
  q(y_1,y_2;\kappa)\equiv - \left(\frac{1}{2} w_{y_3}\vert_{y_3=0}\right)\,,
  \,\,\, I_{\PT} \equiv \left\{\begin{array}{ll}
        1 \,, & (y_1,y_2) \in \PT \\
  0 \,, & (y_1,y_2) \notin \PT\,,  \end{array}\right. \label{app_h:fcal_1}\\
                            {\mathcal F} \sim \frac{C}{2}\log\rho_0 +
                            o(1)\,,  \quad \mbox{as} \quad
                            \rho_0\equiv (y_1^2+y_2^2)^{1/2}\to \infty
                            \,. \label{app_h:fcal_2}
\end{gather}                 
\esub The logarithmic growth as $\rho_0\to\infty$ in
(\ref{app_h:fcal_2}) follows by applying the divergence theorem to
(\ref{app_h:fcal_1}) and recalling (\ref{mfpt:wc_charge}) for $C$. The
$o(1)$ condition in the far-field (\ref{app_h:fcal_2}) specifies
${\mathcal F}$ uniquely. In addition, the far-field behavior is
\begin{equation}
  \Phi_{2p}\sim -\frac{C y_3(y_1^2+y_2^2)}{2\rho^3} +
  \frac{C}{2}\log(y_3+\rho)+ o\left({1/\rho}\right)\,, \quad \mbox{as} \quad
\rho\to\infty\,. 
\end{equation}
The remaining term $\Phi_{2h}$ in (\ref{app_h:decomp}) satisfies
\bsub\label{app_h:inn2_probh}
\begin{align}
\Delta_{\y} \Phi_{2h} &= 0 \,, \quad \y \in \R_{+}^{3} \,,\label{app_h:inn2_h1}\\
  -\partial_{y_3} \Phi_{2h} + \kappa \Phi_{2h} &= -\kappa {\mathcal F}
                                                 \,, \quad y_3=0 \,,\,
    (y_1,y_2)\in \PT\,,  \label{app_h:inn2_h2}\\
    \partial_{y_3} \Phi_{2h} &=0 \,, \quad y_3=0 \,,\, (y_1,y_2)\notin \PT
    \,, \label{app_h:inn2_h3}\\
    \Phi_{2h} & \sim \frac{E}{\rho}\,,  \quad
    \mbox{as} \quad \rho\to \infty \,, \label{app_h:inn2_h4}
  \end{align}
  \esub  
  where the monopole coefficient $E=E(\kappa)$ is given explicitly by
\bsub \label{app_h:eval}
\begin{equation}\label{app_h:eval_a}
  E= -\frac{1}{\pi} \int_{\PT} q(y_1,y_2;\kappa) \,
  {\mathcal F}(y_1,y_2;\kappa)\, dy_1 dy_2  \,,
\end{equation}
where
\begin{equation}\label{app_h:eval_b}
  {\mathcal F}(y_1,y_2;\kappa)=\frac{1}{4\pi}
  \int_{\PT} q(y_1^{\p},y_{2}^{\p},\kappa)
  \log\left( \left(y_1-y_1^{\p}\right)^2+\left(y_2-y_{2}^{\p}\right)^2\right)\,
  dy_1^{\p} dy_2^{\p}\,.
\end{equation}
\esub In this way, the far-field behavior (\ref{app_h:inn2_4}) holds.
Finally, in the limit $\kappa\to 0$, the leading-order asymptotics for
$E$ is given by (\ref{eq:Ei_asympt0}).
\end{lemma}

\begin{proof}
To establish this result, we first show that $\Phi_{2p}$ in
(\ref{app_h:phi_2p}) accounts for the the inhomogeneous terms in
(\ref{app_h:inn2_1}). We readily calculate that
\begin{equation}\label{app_h:phi2p_1}
  \Phi_{2p, y_3}=-\frac{y_3^2}{2}w_{y_3 y_3}-\frac{3}{2}y_3 w_{y_3} \,, \quad
  \Phi_{2p, y_3 y_3} = - \frac{y_3^2}{2} w_{y_3 y_3 y_3} - \frac{5}{2}
  y_3 w_{y_3 y_3} - \frac{3}{2} w_{y_3} \,.
\end{equation}
Moreover, we calculate
$\Delta_{S}\Phi_{2p}\equiv \Phi_{2p, y_1y_1}+\Phi_{2p, y_2 y_2}$, and by using
$w_{y_3y_3}=-\Delta_{S} w$ from (\ref{mfpt:wc_1}),  we derive
\begin{equation}\label{app_h:phi2p_s}
  \begin{split}
  \Delta_{S}\Phi_{2p} &= -\frac{y_3^2}{2}\partial_{y_3}\left[
    \Delta_{S} w\right] - \frac{y_3}{2}\Delta_{S}
  w + \frac{1}{2}\int_{0}^{y_3} \Delta_{S} w\, d\eta +
  \Delta_{S} {\mathcal F}  \\
      &= \frac{y_3^2}{2} w_{y_3 y_3 y_3} +
  \frac{y_3}{2} w_{y_3 y_3} - \frac{1}{2} w_{y_3} + \frac{1}{2} w_{y_3}
  \vert_{y_3=0} + \Delta_{S} {\mathcal F}\,.
  \end{split}
\end{equation}
Upon adding (\ref{app_h:phi2p_1}) and (\ref{app_h:phi2p_s}) we conclude that
\begin{equation}\label{app_h:phi2_add}
  \Delta_{\y} \Phi_{2p} = \Phi_{2p, y_3 y_3} + \Delta_{S}
  \Phi_{2p} =-2\left(y_3 w_{y_3 y_3} + w_{y_3}\right) +
  \left(\Delta_{S} {\mathcal F} + \frac{1}{2} w_{y_3}  \vert_{y_3=0}\right)\,,
\end{equation}
where the terms on the right-hand side in (\ref{app_h:phi2_add}) are
the inhomogeneous terms for this PDE for $\Phi_{2p}$. It follows that
$\Phi_{2p}$ satisfies (\ref{app_h:inn2_1}) when ${\mathcal F}$
satisfies (\ref{app_h:fcal_1}). Consequently, $\Phi_{2h}$ satisfies
the homogeneous problem (\ref{app_h:inn2_h1}).

To establish the boundary conditions in  (\ref{app_h:inn2_probh})
we observe from (\ref{app_h:phi2p_1}) and
(\ref{app_h:phi_2p}) that $\Phi_{2p, y_3}=0$ and $\Phi_{2p}={\mathcal F}$
on $y_3=0$. As a result, upon substituting  (\ref{app_h:decomp}) into
(\ref{app_h:inn2_2})--(\ref{app_h:inn2_3}) we obtain
(\ref{app_h:inn2_h2})--(\ref{app_h:inn2_h3}).

Next, we determine the asymptotic far-field behavior of $\Phi_{2p}$
as defined in (\ref{app_h:phi_2p}). We use $w\sim C\rho^{-1}$ as
$\rho\to \infty$ with $\rho=(y_3^2+\rho_0^2)^{1/2}$ where
$\rho_0\equiv (y_1^2+y_2^2)^{1/2}$. We readily calculate for $\rho\to\infty$
that
\bsub \label{app_h:ffes}
\begin{gather}
  -\frac{1}{2} y_3^2 w_{y_3} - \frac{1}{2} y_3 w \sim 
 -\frac{C}{2} \frac{y_3 \rho_0^2}{\rho^3} \,,
   \label{app_h:ffes_1} \\
  \frac{1}{2} \int_{0}^{y_3} w(y_1,y_2,\eta)\, d\eta \sim 
  \frac{C}{2} \int_{0}^{y_3} \frac{1}{\left(\eta^2+\rho_0^2\right)^{1/2}}
  \, d\eta \sim \frac{C}{2} \left[
  \log\left( y_3 + \rho \right) - \log\rho_0 \right] 
  \,. \label{app_h:ffes_2}
\end{gather}
\esub
In this way, it follows that $\Phi_{2p}$ in (\ref{app_h:phi_2p}) has the
{divergent} far-field behavior
\begin{equation}
  \Phi_{2p} \sim  \frac{C}{2} \log\left( y_3 + \rho \right)
 - \frac{C y_3}{2\rho^3} \left(y_1^2+y_2^2\right) 
  - \frac{C}{2} \log\rho_0 + {\mathcal F} \,. \label{app_h:phi2p_ff}
\end{equation}
Therefore, since
${\mathcal F}\sim \left({C/2}\right)\log\rho_0 + o(1)$ as
$\rho_0\to \infty$ as specified in (\ref{app_h:fcal_2}) it follows
from (\ref{app_h:phi2p_ff}) that
$\Phi_{2p}\sim -{C y_3(y_1^2+y_2^2)/(2\rho^3)} +
\left({C/2}\right)\log(y_3+\rho)+ o\left({1/\rho}\right)$ as
$\rho\to\infty$.  Finally, since (\ref{app_h:inn2_probh}) for
$\Phi_{2h}$ is a Neumann-Robin BVP with a spatially
inhomogeneous Robin condition on the patch, we have
$\Phi_{2h}={\mathcal O}(\rho^{-1})$ as $\rho\to \infty$. The
expression (\ref{app_h:eval}) for the monopole coefficient results
from using Green's second identity to the problems (\ref{mfpt:wc}) and
(\ref{app_h:inn2_probh}) for $w$ and $\Phi_{2h}$, respectively, over a
large hemisphere and by calculating ${\mathcal F}$ using the 2-D
free-space Green's function. In this way, the far-field behavior
(\ref{app_h:inn2_4}) for $\Phi_2$ holds.

Finally, we derive the limiting asymptotics (\ref{eq:Ei_asympt0}) for
$E$.  For $\kappa\to 0$, we obtain from (\ref{mfpt:wc}) that
$w={\mathcal O}(\kappa)$ so that
$q=-\left({1/2}\right)\partial_{y_3}w\sim {\kappa/2}$ on $\PT$. We
substitute (\ref{app_h:eval_b}) into (\ref{app_h:eval_a}) and use
$q\sim {\kappa/2}$. Upon eliminating $\kappa$ by using $C\sim {\kappa
|\PT|/(2\pi)}$ for $\kappa\ll 1$, as obtained from
(\ref{eq:Cmu_Taylor}) and (\ref{eq:c1_def}), we conclude for
$\kappa\to 0$ that
\begin{equation}\label{app_h:smallE}
  E \sim -\frac{\kappa^2}{8\pi^2}
  \int\limits_{\PT_i}\int\limits_{\PT_i} \log|\y-\y^{\prime}|\, d\y \,
  d\y^{\prime} = -\frac{C^2}{2|\PT|^2}
  \int\limits_{\PT_i}\int\limits_{\PT_i} \log|\y-\y^{\prime}|\, d\y \,
  d\y^{\prime}\,,
\end{equation}
which establishes (\ref{eq:Ei_asympt0}).  This completes the proof of
Lemma \ref{lemma:Phi2}.
\end{proof}

We remark that Lemma \ref{lemma:Phi2} applies to a patch $\PT$ of
arbitrary shape. However, when $\PT$ is a disk of radius $a$, the
expression for $E$ in (\ref{app_h:eval}) can be reduced to quadrature
and we can determine its limiting asymptotics, as was summarized in Lemma
\ref{lemma:Ej_kappa}.

When $\PT$ is a disk, $q$ and ${\mathcal F}$ in (\ref{app_h:fcal})
depend only on $\rho_0=(y_1^2+y_2^2)^{1/2}$, so that from
  (\ref{app_h:fcal}) we readily obtain that
  ${\mathcal F}={\mathcal F}(\rho_0;\kappa)$ satisfies
  $\left(\rho_0 {\mathcal F}_{\rho_0}\right)_{\rho_0}= \rho_0 q(\rho_0;\kappa)$. We
    integrate this ODE, impose that ${\mathcal F}(\rho_0;\kappa)$ has
    no singularity at $\rho_0=0$, and we substitute the resulting
    expression into (\ref{app_h:eval_a}). Exploiting radial
    symmetry, we determine $E$ as
\begin{equation}\label{app_h:e_rad}
  E = -2 \int_{0}^{a} \rho_0 q(\rho_0;\kappa) {\mathcal F}(\rho_0;\kappa)\, 
d\rho_0 \,; \qquad
  {\mathcal F}_{\rho_0} = \frac{1}{\rho_0} \int_{0}^{\rho_0} \eta
  q(\eta;\kappa) \, d\eta \,, \quad 0\leq \rho_0\leq a\,,
\end{equation}
with ${\mathcal F}=\left({C/2}\right)\log{a}$ at $\rho_0=a$. Next, we
integrate (\ref{app_h:e_rad}) by parts to obtain
\begin{equation}\label{app_h:e_by_parts}
  E=-2 \left[ {\mathcal F}(\rho_0;\kappa) \left(\int_{0}^{\rho_0}
        \eta q(\eta;\kappa)\, d\eta\right) \Big{\vert}_{0}^{a} -
        \int_{0}^{a} {\mathcal F}_{\rho_0} 
          \left(\int_{0}^{\rho_0}
        \eta q(\eta;\kappa)\, d\eta\right) \, d\rho_0 \right] \,.
\end{equation}
Then, upon using ${\mathcal F}=\left({C/2}\right)\log{a}$ at $\rho_0=a$,
$C=2\int_{0}^{a} \eta q(\eta;\kappa)\, d\eta$ and ${\mathcal F}_{\rho_0}$
from (\ref{app_h:e_rad}) we find that (\ref{app_h:e_by_parts}) reduces
to (\ref{mfpt:Ej_all}) in Lemma \ref{lemma:Ej_kappa}.

For $\kappa=\infty$, we calculate (\ref{mfpt:Ej_all}) analytically.
By using (\ref{mfpt:wc_q}) for $q=q(\rho_0;\infty)$,
$C=C(\infty)={2a/\pi}$ and
$\int_{0}^{\rho_0}{\eta/\sqrt{a^2-\eta^2}}\, d\eta=
a-\sqrt{a^2-\rho_0^2}$, we obtain from (\ref{mfpt:Ej_all}) that
\begin{align*}
  E(\infty) &= - \frac{\left[C(\infty)\right]^2}{2}\log{a} +
              \frac{2}{\pi^2}\int_{0}^{a}
              \frac{1}{\rho_0} \left[ a -
              \sqrt{a^2-\rho_0^2}\right]^2 \, d\rho_0\,,\\
            & = -\frac{2 a^2}{\pi^2}\log{a} +
              \frac{2a^2}{\pi^2}\int_{0}^{1}
              \frac{1}{x}\left(2-x^2 - 2\sqrt{1-x^2}\right)\, dx \,,\\
            &=-\frac{2a^2}{\pi^2}\log{a} +
              \frac{2a^2}{\pi^2} \left(\frac{3}{2}-\log{4}
              \right) \,.
\end{align*}
In this way, we recover the expression for $E(\infty)$ given in
(\ref{mfpt:Ej_asy_large}) of Lemma \ref{lemma:Ej_kappa}.

Finally, we calculate $E$ when $\kappa\ll 1$ and $\PT$ is a
disk. Instead of using (\ref{app_h:smallE}), we can proceed more
directly. When $\kappa\ll 1$, we find from (\ref{mfpt:wc}) that
$-\partial_{y_3}w\sim \kappa$ on $y_3=0$, $(y_1,y_2)\in \PT$, so that
from (\ref{app_h:fcal}), $q(\rho_0;\kappa)\sim {\kappa/2}$ on
$0\leq \rho_0\leq a$. By evaluating the integrals in
(\ref{mfpt:Ej_all}), and using $C\sim {\kappa a^2/2}$ for
$\kappa\ll 1$, we obtain (\ref{mfpt:Ej_asy_small}) of Lemma
\ref{lemma:Ej_kappa} for $E$ when $\kappa\ll 1$.

%%%%%%%%%%%%%%%%%%%%%%%%%%%%%%%%%%%%%%%%%%%%%%%%%%%%%%%%%%%%%%%%%%%%%%%
\section{Computation and Analysis of the Reactive Capacitance}
\label{sec:Cmu}

In this Appendix, we derive an exact representation for the reactivity
capacitance $C_i(\kappa_i)$ in the general case of an arbitrary patch.
We then analyze its properties for the case of a circular patch.  Some
additional properties of $C_i(\kappa_i)$ for noncircular patches are
discussed in \cite{Grebenkov26}.

\subsection{Arbitrary patch}

\subsubsection*{Spectral expansions}

The solution $w_i(\y;\kappa_i)$ of the mixed BVP (\ref{mfpt:wc}) is
the key element for our asymptotic analysis.  As this BVP is
formulated for the $i$-th patch $\PT_i$, independently of the other
patches, its solution can be searched separately for different patches
so that the index $i$ will be kept fixed throughout this appendix.  As
an explicit solution is not available even for circular patches, we
will seek a suitable spectral representation of $w_i(\y;\kappa_i)$.
For this purpose, we employ the {\em local} Steklov eigenvalue problem
defined in an upper half-space by
\begin{subequations}  \label{eq:Psi_def}
\begin{align}  \label{eq:Vk_eq}
\Delta_{\y} \Psi_{ki} & = 0 \,, \quad \y \in \R_+^3 \,,\\ \label{eq:Vk_stek}
  \partial_n \Psi_{ki} & = \mu_{ki} \Psi_{ki}\,, \quad y_3=0 \,,\,
                         (y_1,y_2)\in \PT_i\,,
  \\  \label{eq:Vk_Neumann}
  \partial_n \Psi_{ki} & = 0 \,, \quad y_3=0 \,,\, (y_1,y_2)\notin \PT_i\,,
  \\  \label{eq:Vk_inf}
  \Psi_{ki}(\y) & = {\mathcal O}\left({1/|\y|}\right) \quad \textrm{as}
               \quad |\y|\to \infty\,.
\end{align}
\end{subequations}
This spectral problem can be reduced to the exterior Steklov problem
in the whole space $\R^3$.  The latter has a discrete spectrum
\cite{Bundrock25}, whereas its eigenfunctions are necessarily either
symmetric, or antisymmetric with respect to the horizontal plane.  The
symmetric ones satisfy the Neumann boundary condition
(\ref{eq:Vk_Neumann}).  In the following, we focus only on these
symmetric eigenmodes and enumerate them by the index $k =
0,1,2,\ldots$ such that the associated eigenvalues form an increasing
sequence:
\begin{equation}
0 < \mu_{0i} < \mu_{1i} \leq \cdots \nearrow +\infty \,,
\end{equation}
with the principal eigenvalue $\mu_{0i}$ being simple and strictly
positive.  Importantly, the restrictions $\Psi_{ki}(\y)|_{\PT_i}$
onto $\PT_i$ form a complete orthonormal basis in $L^2(\PT_i)$:
\begin{equation}\label{eq:orhthog_D}
\int_{\PT_i} \Psi_{ki} \Psi_{k^{\prime}i} \, d\y= \delta_{k,k^{\prime}}\,.
\end{equation}
As discussed in \cite{Grebenkov25}, the restriction of the Neumann
Green's function onto $\PT_i$ is in fact the the kernel of an
integral operator that determines the eigenpairs $\mu_{ki}$ and
$\Psi_{ki}(\y)$.  In our setting, the restriction of the Neumann
Green's function in the half-space onto a patch on the horizontal
plane is $1/(2\pi|\y-\y^{\prime}|)$ and yields the following identity:
\begin{equation} \label{eq:NGreen_Psi}
  \frac{1}{2\pi |\y-\y^{\prime}|} =
  \sum\limits_{k=0}^\infty \frac{\Psi_{ki}(\y) \Psi_{ki}(\y^{\prime})}{\mu_{ki}}\,,
  \quad \mbox{for} \quad \y,\y^{\prime}\in \PT_i\,.
\end{equation}
This identity can also be recast as an eigenvalue problem
\begin{equation}  \label{eq:Psi_eigen}
  \int\limits_{\PT_i} \frac{1}{2\pi |\y - \y^{\prime}|} \Psi_{ki}(\y^{\prime}) \,
  d\y^{\prime} = \frac{1}{\mu_{ki}} \Psi_{ki}(\y)\,,  \quad \mbox{for}
  \quad \y\in \PT \,,
\end{equation}
whose eigenpairs are enumerated by $k = 0,1,\ldots$.

The completeness of the basis of $\{\Psi_{ki}\}$ in $L^2(\PT_i)$
allows us to decompose the restriction of $w_i(\y;\kappa_i)$ to
$\PT_i$ and thus to seek the solution of (\ref{mfpt:wc}) in the
form
\begin{equation}
w_i(\y; \kappa_i) = \sum\limits_{k=0}^\infty \eta_{ki} \Psi_{ki}(\y) \,.
\end{equation}
The unknown coefficients $\eta_{ki}$ can be found from imposing the
boundary condition (\ref{mfpt:wc_2}), which yields
\begin{equation}\label{eq:dj_zero}
  \sum\limits_{k^{\prime}=0}^\infty \eta_{k^{\prime}i}
  (\mu_{k^{\prime}i} + \kappa_i) \Psi_{k^{\prime}i}(\y) = \kappa_i \,,
\quad \mbox{for} \quad \y\in \PT_i\,.
\end{equation}
Multiplying (\ref{eq:dj_zero}) by $\Psi_{ki}(\y)$ and integrating the
resulting expression over $\PT_i$, we get
\begin{equation}  \label{eq:dj}
  \eta_{ki} = \frac{\kappa_i \, d_{ki}}{\mu_{ki} + \kappa_i}  \,, \quad \textrm{with}
  \quad d_{ki} = \int_{\PT_i} \Psi_{ki}\, d\y \,, 
\end{equation}
where we used the orthonormality condition (\ref{eq:orhthog_D}).  In this
way, we have deduced the following spectral representation:
\begin{equation}   \label{eq:wi_spectral}
  w_i(\y;\kappa_i) = \kappa_i \sum\limits_{k=0}^\infty
  \frac{d_{ki}}{\mu_{ki} + \kappa_i} \Psi_{ki}(\y) , \quad \y \in \R^3_+ \,.
\end{equation}
As a consequence, the charge density becomes
\begin{equation} \label{eq:qi}
q_i(\y; \kappa_i) = \frac12 \partial_n w_i(\y;\kappa_i) |_{\PT_i}
= \frac{\kappa_i}{2} \sum\limits_{k=0}^\infty
  \frac{d_{ki} \mu_{ki}}{\mu_{ki} + \mu} \Psi_{ki}(\y) \,, \quad \mbox{for}
  \quad \y\in \PT_i \,.
\end{equation}
Moreover, setting $\kappa_i = - \sigma$ and evaluating the derivative
with respect to $\sigma$ gives
\begin{equation}   \label{eq:wc_spectral}
w_{ci}(\y;-\sigma) = \partial_\sigma w_i(\y;-\sigma) =
- \sum\limits_{k=0}^\infty
\frac{\mu_{ki} d_{ki}}{(\mu_{ki} - \sigma)^2} \Psi_{ki}(\y) \,.
\end{equation}
These spectral expansions imply that if $\sigma = \mu_{ki}$ for some
integer $k\geq 0$, then there is no solution $w_{i}(\y;-\sigma)$.

\subsubsection*{Reactive capacitance}

According to the definition (\ref{mfpt:wc_charge}), the reactive
capacitance can be found as
\begin{equation} \label{eq:Cmu}
  C_i(\kappa_i) = \frac{1}{2\pi} \int\limits_{\PT_i} \partial_n w_i \, d\y
  = \frac{\kappa_i}{2\pi} \sum\limits_{k=0}^\infty
  \frac{\mu_{ki} d_{ki}^2}{\mu_{ki} + \kappa_i} \,.
\end{equation}
Equation (\ref{eq:Cmu}) is one of the main results of this appendix.
In particular, its derivative, given by (\ref{eq:dCmu}), is strictly
positive.  As a result, the capacitance $C_i(\infty)$ is an upper
bound for $C_i(\kappa_i)$ for $\kappa_i > 0$. Another upper bound,
which is useful for small $\kappa_i$, reads
\begin{equation}  \label{eq:Cmu_upper}
  C_i(\kappa_i) \leq \frac{\kappa_i}{2\pi} \sum\limits_{k=0}^\infty d_{ki}^2
  = \frac{\kappa_i \, |\PT_i|}{2\pi} \,,
\end{equation}
where the second equality follows from (\ref{eq:sum_dk}), which is
shown below.  We conclude that the eigenvalues $\mu_{ki}$ and the
spectral weights $d_{ki}$, for which $d_{ki}\neq 0$, fully determine
the reactive capacitance and its properties.  It is worth noting that
a dilation of the patch $\PT_i$ by $a_i > 0$ implies
\begin{equation}   \label{eq:rescaling}
  \mbox{If}~~ \PT^{\prime}_i = a_i\PT_i \,, \quad \mbox{then} \quad
  \mu^{\prime}_{ki} = \mu_{ki}/a_i, \quad d^{\prime}_{ki} = a_i \, d_{ki} \,,
\end{equation}
where we used the $L^2(\PT_i)$ normalization of eigenfunctions.

A practical implementation of the spectral expansion (\ref{eq:Cmu})
requires solving the eigenvalue problem (\ref{eq:Psi_eigen}) for a
given patch $\PT_i$.  An efficient discretization of this problem that
handles the explicit but weakly singular kernel $1/|\y-\y^{\prime}|$,
was presented in \cite{Grebenkov26}.  This numerical method allows one
to compute a number of eigenpairs $\{\mu_{ki}, \Psi_{ki}\}$ and the
related weights $d_{ki}$ to approximate the reactive capacitance over
the whole range of reactivities $0 < \kappa_i < \infty$.  A priori,
the convergence of the series in (\ref{eq:Cmu}) is not fast due to a
non-analytic behavior of $C_i(\kappa_i)$ at large $\kappa_i$
\cite{Grebenkov26} (see also \S \ref{sec:large_mu} below).  However,
our empirical observation suggests that the spectral weights $d_{ki}$
may decay rapidly with $k$ for some shapes (see Table
\ref{table:muk_disk} for the circular patch and
Refs. \cite{Grebenkov24,Grebenkov25a} for some other shapes), allowing
one to approximate $C_i(\kappa_i)$ accurately with a moderate number
of terms.  A systematic analysis of the convergence speed for
different patches presents an interesting open problem.

\subsubsection*{Small-reactivity behavior}

In the limit $\kappa_i\to 0$, the geometric series expansion of each
fraction in (\ref{eq:Cmu}) yields the Taylor expansion
\begin{equation} \label{eq:Cmu_Taylor0} 
C_i(\kappa_i) = - a_i
  \sum\limits_{n=1}^{\infty} c_{ni} \,(-\kappa_i a_i)^n  \,, \quad
  \mbox{with} \quad c_{ni} = \frac{1}{2\pi a_i^{n+1}} \sum\limits_{k=0}^\infty
  \frac{d_{ki}^2}{\mu_{ki}^{n-1}} \,.
\end{equation}
The coefficient $c_{1i}$ can be found by expanding the unity on the
complete basis of eigenfunctions $\{\Psi_{ki}\}$,
\begin{equation}  \label{eq:unity_Psi}
  \sum\limits_{k=0}^\infty d_{ki} \Psi_{ki}(\y) = 1 \,, \quad \mbox{for} \quad
  \y\in\PT_i\,,
\end{equation}
while its integral over $\y\in\PT_i$ yields
\begin{equation}  \label{eq:sum_dk}
\sum\limits_{k=0}^\infty d_{ki}^2 = |\PT_i| \,, 
\end{equation}
where $|\PT_i|$ denotes the area of $\PT_i$.  This shows that $F_{ki}
= d_{ki}^2/|\PT_i|$ can be interpreted as the relative weight of the
$k$-th eigenpair.  Applying the divergence theorem to
(\ref{eq:Psi_def}), we observe that the coefficient $d_{ki}$
determines the far-field behavior of $\Psi_{ki}(\y)$ in the form
\begin{equation}  \label{eq:Psi_asympt}
  \Psi_{ki}(\y) \sim \frac{\mu_{ki} d_{ki}}{2\pi |\y|} + \ldots \,, \quad
  \mbox{as}
  \quad |\y|\to\infty\,.
\end{equation}

According to (\ref{eq:sum_dk}), we get 
\begin{equation}
c_{1i} = \frac{|\PT_i|}{2\pi a_i^2} \,.  
\end{equation}
It is convenient to derive a closed-form representation for the
coefficients $c_{2i}$ and $c_{3i}$ to facilitate their numerical
computation without solving the Steklov eigenvalue problem.
To this end, we integrate (\ref{eq:Psi_eigen}) over
$\y\in\PT_i$ to get
\begin{equation}
  \int\limits_{\PT_i} \Psi_{ki}(\y)\, \omega_i(\y) \, d\y =
  \frac{d_{ki}}{\mu_{ki}} \,,  
\qquad \mbox{where} \qquad \omega_i(\y) \equiv \int\limits_{\PT_i}
\frac{d\y^{\prime}}{2\pi|\y-\y^{\prime}|} \,.
\end{equation}
Multiplying this relation by $d_{ki}/(2\pi)$ and summing over $k$, we find
\begin{align}  \nonumber
c_{2i} & = \frac{1}{2\pi a_i^3} \sum\limits_{k=0}^\infty \frac{d_{ki}^2}{\mu_{ki}} 
= \frac{1}{2\pi a_i^3} \sum\limits_{k=0}^\infty \int\limits_{\PT_i} 
      \Psi_{ki}(\y)\, \omega_i(\y) \, d\y \int\limits_{\PT_i} \Psi_{ki}(\y^{\prime})
  \, d\y^{\prime}  \\  \label{eq:c2_integral}
    & = \frac{1}{2\pi a_i^3} \int\limits_{\PT_i} \omega_i(\y)\,
    \left(\int\limits_{\PT_i} \underbrace{\sum\limits_{k=0}^\infty\Psi_{ki}(\y)\,
      \Psi_{ki}(\y^{\prime})}_{=\delta(\y-\y^{\prime})} \, d\y^{\prime}  \right) \,
       d\y\ = \frac{1}{2\pi a_i^3} \int\limits_{\PT_i} \omega_i(\y) \, d\y \,,
\end{align}
where we used the completeness of the basis of eigenfunctions
$\Psi_{ki}$ in $L^2(\PT)$.  In the same vein, we have
\begin{align}  \nonumber
c_{3i} & = \frac{1}{2\pi a_i^4} \sum\limits_{k=0}^\infty \frac{d_{ki}^2}{\mu_{ki}^2} 
= \frac{1}{2\pi a_i^4} \sum\limits_{k=0}^\infty \int\limits_{\PT_i}  \Psi_{ki}(\y)\, 
      \omega_i(\y)\, d\y \, \int\limits_{\PT_i} \Psi_{ki}(\y^{\prime}) \,
      \omega_i(\y^{\prime})\, d\y^{\prime} \\  \label{eq:c3_integral}
    & = \frac{1}{2\pi a_i^4} \int\limits_{\PT_i}\omega_i(\y) \, 
    \left(  \int\limits_{\PT_i}\omega_i(\y^{\prime})
      \underbrace{\sum\limits_{k=0}^\infty
      \Psi_{ki}(\y)\, \Psi_{ki}(\y^{\prime})}_{=\delta(\y-\y^{\prime})}
      \,  d\y^{\prime}\right) \,  d\y 
= \frac{1}{2\pi a_i^4} \int\limits_{\PT_i}  \omega_i^2(\y) \, d\y \,.
\end{align}
These two representations allow one to compute the coefficients
$c_{2i}$ and $c_{3i}$ numerically for any patch shape without solving
the exterior Steklov problem.  In other words, we managed to represent
these coefficients in purely geometric terms.  Moreover, the
exact analytical formulas for $c_{2i}$ were derived in
\cite{Grebenkov26} for rectangular, elliptical and circular-annular
patches.

A numerical computation of the function $\omega_i(\y)$ requires
integration of the singular kernel $1/|\y-\y^{\prime}|$.  To avoid
technical issues, it is convenient to recall that $\Delta_2
|\y-\y^{\prime}| = {1/|\y-\y^{\prime}|}$, where $\Delta_2$ is the
two-dimensional Laplace operator.  As a consequence, one has
\begin{equation}  \label{eq:omega_int}
  \omega_i(\y) = \frac{1}{2\pi} \int\limits_{\PT_i} \Delta_2 |\y-\y^{\prime}|
 \,   d\y^{\prime} \, 
 = -\frac{1}{2\pi} \int\limits_{\partial\PT_i} \frac{(\n_{\y^{\prime}}
   \cdot (\y-\y^{\prime}))}{|\y-\y^{\prime}|} \, d\y^{\prime} \,, 
\end{equation}
where $\n_{\y^{\prime}}$ is the unit normal vector to the boundary of
the patch, oriented outward from $\PT_i$.  In this way, the
original integral over a planar region $\PT_i$ is reduced to an
integral over its one-dimensional boundary that avoids singularities.
Using $\nabla_{\y} |\y-\y^{\prime}| = (\y -
\y^{\prime})/|\y-\y^{\prime}|$ and the divergence theorem, we can
rewrite the integral in (\ref{eq:c2_integral}) as
\begin{equation}  \label{eq:c2_integral2}
  c_{2i} = - \frac{1}{4\pi^2 a_i^3} \int\limits_{\partial\PT_i} d\y
  \int\limits_{\partial\PT_i}  |\y-\y^{\prime}| (\n_{\y} \cdot \n_{\y^{\prime}}) \,
  d\y^{\prime} \,.
\end{equation}

\subsection{Circular Patch}\label{app:circle}

For a circular patch $\PT_i$, the exterior Steklov problem was
studied in \cite{Grebenkov24}.  In particular, an efficient numerical
procedure for constructing the eigenvalues and eigenfunctions was
developed by using oblate spheroidal coordinates.  Moreover, the axial
symmetry of this setting implies that only axially symmetric
eigenfunctions do contribute to $w_i(\y;\kappa_i)$ in
(\ref{eq:wi_spectral}) and related quantities (in fact, the
coefficients $d_{ki}$ in (\ref{eq:dj}) vanish for non-axially
symmetric eigenfunctions).  As a consequence, we can focus exclusively
on axially symmetric eigenmodes that we still enumerate by the index
$k = 0,1,2,\ldots$.

\begin{table}
\begin{center}
\begin{tabular}{|c|c|c|c|c|c|c|c|c|} \hline
$k$            & 0      & 1      & 2      & 3      & 4      & 5      & 6      & 7      \\ \hline
$\mu_{ki}$     & 1.1578 & 4.3168 & 7.4602 & 10.602 & 13.744 & 16.886 & 20.028 & 23.169 \\ 
$d_{ki}    $   & 1.7524 & 0.2298 & 0.1000 & 0.0587 & 0.0397 & 0.0291 & 0.0225 & 0.0180 \\ 
$d_{ki}^2/\pi$ & 0.9775 & 0.0168 & 0.0032 & 0.0011 & 0.0005 & 0.0003 & 0.0002 & 0.0001 \\ \hline
$\mu^N_{ki}$   & 0      & 4.1213 & 7.3421 & 10.517 & 13.677 & 16.831 & 19.981 & 23.128 \\ 
$F_{ki}^N$     & 1      & 0.1195 & 0.0782 & 0.0587 & 0.0471 & 0.0394 & 0.0339 & 0.0297 \\ \hline
\end{tabular}
\end{center}
% [Cmu, muk,dk, B,C ] = A_Ward_Steklov_Cmu_3d(1);  muk([1:2:20])'   
% dk([1:2:20])'
% dk([1:2:20])'.^2/pi
% [mu, Minv, G, c,  W,v2norm,Vinf] = A_DN_disk_Neumann_Minv(50);  J = [2:2:length(mu)];
% mu(J(1:7))'
% pi*(Vinf(J(1:7)).^2)'
\caption{ 
The first eight Steklov eigenvalues $\mu_{ki}$ for the unit disk
$\PT_i$ ($a_i = 1$) in the upper half-space that correspond to axially
symmetric eigenfunctions, for which the weights $d_{ki}$ are nonzero.
These values were obtained via a numerical diagonalization of a
truncated matrix representing the Dirichlet-to-Neumann operator, with
the truncation order $100$ (see details in \cite{Grebenkov24}).  Note
that the reduction of the truncation order to $50$ does not affect the
shown digits, assuring the high accuracy of this computation.  The
last two rows present the first eight eigenvalues $\mu_{ki}^N$,
corresponding to axially symmetric eigenfunctions, of the Steklov
problem (\ref{eq:Psi_def_N}), and the associated weights $F_{ki}^N =
\pi [\Psi_{ki}^N(\infty)]^2$; see \cite{Henrici70,Grebenkov25} for
their numerical computation.}
\label{table:muk_disk}
\end{table}

Using the numerical procedure from \cite{Grebenkov24}, we compute
$\mu_{ki}$ and $d_{ki}$ by diagonalizing an appropriate truncated
matrix.  The first eight eigenvalues $\mu_{ki}$ and coefficients
$d_{ki}$ for the unit disk are shown in Table \ref{table:muk_disk}.
One sees that the principal eigenmode provides the dominant
contribution of $98\%$, the next one gives $1.7\%$, whereas all the
remaining eigenmodes are almost negligible.  In other words, the
infinite sum in (\ref{eq:Cmu_Taylor}) can be truncated to only a few
terms to get very accurate results for $c_{ni}$.  For instance,
keeping only the first two terms in the spectral expansion
(\ref{eq:Cmu_Taylor}) determining $c_{ni}$ and substituting the data
from Table \ref{table:muk_disk}, we find
\begin{equation}  \label{eq:cn_approx}
  c_{ni} \approx \frac{0.4888}{(1.1578)^{n-1}} + \frac{0.0084}{(4.3168)^{n-1}}\,,
  \quad \mbox{for} \quad n\geq 2\,.
\end{equation}
This explicit approximation gives $c_{2i} \approx 0.4241$ and $c_{3i}
\approx 0.3651$, which perfectly agree with the exact values from
(\ref{eq:cn_exact}).  While the exact computation of $c_{ni}$ rapidly
becomes very cumbersome (see Appendix \ref{app_a:Ckappa}), the
approximation (\ref{eq:cn_approx}) is fully explicit.

We also note that an alternative exact computation of $c_{2i}$ and
$c_{3i}$ to that done in Appendix \ref{app_a:Ckappa} can be achieved
using the integral representations (\ref{eq:c2_integral}) and
(\ref{eq:c3_integral}).  Setting $\y = (r,\phi)$ and $\y^{\prime} =
(1,\phi^{\prime})$ in polar coordinates, we first evaluate the
integral in (\ref{eq:omega_int}) as
\begin{equation}
  \omega_i(\y) = -\frac{1}{2\pi} \int\limits_0^{2\pi}
  \frac{r \cos(\phi-\phi^{\prime}) - 1}{\sqrt{1 + r^2 -
      2r \cos(\phi-\phi^{\prime})}} = \frac{2}{\pi} E(|\y|) \,,
\end{equation}
where $E(z)$ is the complete elliptic integral of the second kind.
As a consequence, we get immediately that
\begin{equation}
  c_{2i} = \frac{2}{\pi} \int\limits_{0}^{1}  r \, E(r) \, dr =
  \frac{4}{3\pi} \,,  \qquad
  c_{3i} = \frac{4}{\pi^2} \int\limits_{0}^{1} r \, \left[E(r)\right]^2\,
  dr \approx 0.3651 \,.
\end{equation}
Note that $c_{2i}$ could also be found directly from
(\ref{eq:c2_integral2}):
\begin{equation}
  c_{2i} = - \frac{1}{4\pi^2} \int\limits_0^{2\pi} \left(\int\limits_0^{2\pi}
    \sqrt{2-2\cos(\phi-\phi^{\prime})} \cos(\phi-\phi^{\prime}) \, d\phi^{\prime}
    \right) \, d\phi = \frac{4}{3\pi} \,.
\end{equation}

\subsection{The Large-Reactivity Limit}
\label{sec:large_mu}

While the fast decay of $d_{ki}^2$ allows one to keep only few terms
in the analysis of the small-reactivity limit $\kappa_i \to 0$, all
eigenmodes become relevant in the opposite limit $\kappa_i\to\infty$
of high reactivity.  In the case of the unit disk (i.e., $a_i = 1$),
the solution of (\ref{mfpt:wc}) can be related to the potential $\Phi$
introduced and studied in \cite{Guerin23}.  In fact, the divergence
theorem applied to (\ref{mfpt:wc}) yields $\kappa_i \int_{\PT_i}
(1-w_i(\y;\kappa_i)) \, d\y = 2\pi C_i(\kappa_i)$, so that
$(1-w_i)/C_i = 2\pi \Phi$, where $\Phi$ is the unique solution of
equations (12-14) from \cite{Guerin23}.  As a consequence, the
asymptotic relations (2,3,11) derived in \cite{Guerin23} (with $D =
1$) imply, in our notations, that
\begin{equation} \label{eq:Cmu_asympt}
  C_i(\kappa_i) \approx \frac{2}{\pi} -
  \frac{2\left( \log\kappa_i + \log{2} + \gamma_e +
1\right)}{\pi^2 \kappa_i}   \quad \mbox{as} \quad \kappa_i\to \infty \,,
\end{equation}
where $\gamma_e \approx 0.5772\ldots$ is the Euler constant.
Figure~\ref{fig:Cmu_asympt} illustrates an excellent agreement between
the numerically computed values of $C_i(\kappa_i)$ and the asymptotic
relation (\ref{eq:Cmu_asympt}).

\begin{figure}
\begin{center}
\includegraphics[width=88mm]{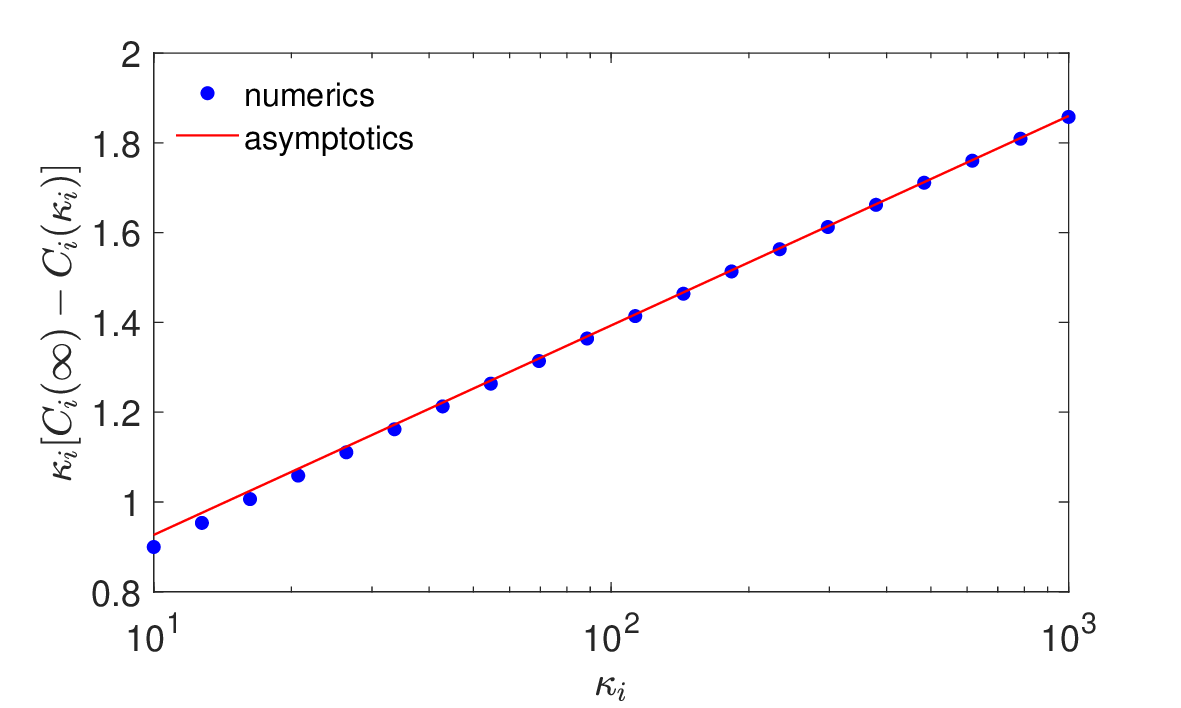} % {Cmu_asympt1.eps}
\end{center}
\caption{ 
Asymptotic behavior of $C_i(\kappa_i)$ at large $\kappa_i$ for the
unit disk $\PT_i$ ($a_i = 1$).  The numerical results (symbols)
computed from (\ref{eq:Cmu}), which was truncated after 100 terms, are
well-predicted by the asymptotic formula (\ref{eq:Cmu_asympt}) when
$\kappa_i$ is large.}
\label{fig:Cmu_asympt}
\end{figure}

\subsection{Alternative Representation and Zeros of the Function $C_i(\kappa_i)$}
\label{sec:Cmu0}

As discussed in \S \ref{stekN}, the leading-order term in the
asymptotic expansion of the SN problem is determined by
(\ref{sn:sigma_0}), which requires finding zeros of the sum $\sum_{i}
C_i(-\sigma_0)$.  We now discuss the relation between the zeros of
$C_i(-\sigma_0)$ and the spectrum of an additional local Steklov
eigenvalue problem (see also \cite{Grebenkov25,Grebenkov26}).

For this purpose, we consider an auxiliary exterior Steklov problem:
\begin{subequations}  \label{eq:Psi_def_N}
\begin{align}  \label{eq:VkN_eq}
\Delta_{\y} \Psi_{ki}^N & = 0\,, \quad \y \in \R_+^3 \,,\\
  \partial_n \Psi_{ki}^N & = \mu_{ki}^N \Psi_{ki}^N \,, \label{eq:VkN_stek}
\quad y_3=0 \,,\, (y_1,y_2)\in \PT_i\,,\\  \label{eq:VkN_Neumann}
  \partial_n \Psi_{ki}^N & = 0 \,,
\quad y_3=0 \,,\, (y_1,y_2)\notin \PT_i\,,\\  \label{eq:VkN_inf}
  |\y|^2 \, |\nabla \Psi_{ki}^N(\y)| & \to 0 \,,
                                    \quad \textrm{as}\quad |\y|\to \infty\,.
\end{align}
\end{subequations}
This spectral problem, which was formulated in \cite{Henrici70},
admits infinitely many nontrivial solutions $\{\mu_{ki}^N, \Psi_{ki}^N\}$,
enumerated by $k = 0,1,\ldots$, such that the eigenvalues form an
ordered sequence,
\begin{equation}
0 = \mu_{0i}^N < \mu_{1i}^N \leq \mu_{2i}^N \leq \ldots \nearrow +\infty,
\end{equation}
whereas the restrictions $\Psi_{ki}^N|_{\PT_i}$ onto $\PT_i$ form a
complete orthonormal basis of $L^2(\PT_i)$.  The spectral problem
(\ref{eq:Psi_def_N}) differs from the former problem
(\ref{eq:Psi_def}) by the imposed behavior of eigenfunctions at
infinity.  In fact, while (\ref{eq:Vk_inf}) can be interpreted as a
Dirichlet-type condition at infinity, (\ref{eq:VkN_inf}) implements a
vanishing flow condition, which is a sort of Neumann condition at
infinity.  The distinctions between these two cases were investigated
in a much more general setting in \cite{Arendt15,Bundrock25}.  It is
trivial to check that a constant function $\Psi_{0i}^N =
{1/\sqrt{|\PT_i|}}$ is a solution of (\ref{eq:Psi_def_N}), with the
trivial eigenvalue $\mu_{0i}^N = 0$.  As the other eigenfunctions
$\Psi_{ki}^N$ must be orthogonal to $\Psi_{0i}^N$ in $L^2(\PT_i)$, one
has $\int\nolimits_{\PT_i} \Psi_{ki}^N\, d\y= 0$ for any $k > 0$.
Note that the condition (\ref{eq:VkN_inf}) does not require vanishing
of $\Psi_{ki}^N$ at infinity, i.e., $\Psi_{ki}^N$ may have a constant
non-zero limit as $|\y|\to \infty$.  For a circular patch, the first
eight eigenvalues $\mu_{ki}^N$ corresponding to axially symmetric
eigenfunctions, are given in the next-to-last row of Table
\ref{table:muk_disk} (see \cite{Grebenkov25} for the details on their
numerical computation).

The two exterior Steklov problems (\ref{eq:Psi_def}) and
(\ref{eq:Psi_def_N}) provide complementary spectral tools for our
asymptotic analysis.  It is therefore instructive to discuss several
relations between the eigenpairs of these problems.  For some indices
$k$ and $k^{\prime}$, let us multiply (\ref{eq:Vk_eq}) by
$\Psi_{k^{\prime}i}^N$, multiply (\ref{eq:VkN_eq}) with $k^{\prime}$
by $\Psi_{ki}$, subtract these equations, integrate over a large
hemisphere, apply the Green's formula, and use the boundary conditions
and the asymptotic behavior at infinity to get
\begin{equation}  \label{eq:Psi_DN}
  \left(\mu_{ki} - \mu_{k^{\prime}i}^{N}\right)
  \int\limits_{\PT_i} \Psi_{ki}(\y) \Psi_{k^{\prime}i}^N(\y) \, d\y  =
  \mu_{ki} d_{ki} \Psi_{k^{\prime}i}^N(\infty) \,.
\end{equation}
This identity allows us to provide some classification of the spectra of
the two problems.  In (\ref{sn:poles0}), we introduced the resonant
set ${\mathcal P}_i$ as the union of all eigenvalues $\mu_{ki}$ for
which $d_{ki} \ne 0$:
\begin{equation}
  {\mathcal P}_i \equiv \bigcup_{k=0}^{\infty} \left\{ \mu_{ki} ~\vert~
    d_{ki} \ne 0 \right\}\,.
\end{equation}
The union of the remaining eigenvalues is then denoted by
\begin{equation}
 {\mathcal P}_i^0 \equiv \bigcup_{k=0}^{\infty}\left\{ \mu_{ki} ~\vert~ d_{ki} = 0
  \right\}\,.
\end{equation}
We emphasize that the intersection of these two sets is not
necessarily empty, i.e., there may exist indices $k \ne k^{\prime}$
such that $\mu_{ki} = \mu_{k^{\prime}i}$, $d_{ki} \ne 0$ and
$d_{k^{\prime}i} = 0$.  In analogy, one can define two sets of
eigenvalues for the spectral problem (\ref{eq:Psi_def_N}) as
\begin{subequations}
\begin{align}
  {\mathcal P}_i^N & \equiv \bigcup_{k=0}^{\infty} \left\{ \mu_{ki}^N ~\vert~
                     \Psi^N_{ki}(\infty) \ne 0 \right\}\,, \\
  {\mathcal P}_i^{0,N} & \equiv \bigcup_{k=0}^{\infty} \left\{ \mu_{ki}^N ~\vert~
                         \Psi^N_{ki}(\infty) = 0 \right\}\, .
\end{align}
\end{subequations}

We now prove the following statement.

\begin{lemma} One has
\begin{equation}  \label{eq:Pset}
{\mathcal P}_i^0 = {\mathcal P}_i^{0,N} \,, \qquad 
{\mathcal P}_i \cap {\mathcal P}_i^N = \emptyset \,.
\end{equation}
\end{lemma} 
\begin{proof}
Let us first prove that ${\mathcal P}_i^0 \subset {\mathcal
P}_i^{0,N}$ in the first relation.  If $\mu_{ki} \in {\mathcal
P}_i^0$, the right-hand side of (\ref{eq:Psi_DN}) is zero for all
$k^{\prime} = 0,1,\ldots$.  Since the eigenfunctions
$\{\Psi_{k^{\prime}i}^N\}$ form a complete basis of $L^2(\PT_i)$, an
eigenfunction $\Psi_{ki}$ cannot be orthogonal to all eigenfunctions
$\Psi_{k^{\prime}i}^N$, implying that there exists an index
$k^{\prime}$ such that $\mu_{ki} = \mu_{k^{\prime}i}^N$.  As a
consequence, the associated eigenfunction $\Psi_{ki}$ satisfies both
(\ref{eq:Psi_def}) and (\ref{eq:Psi_def_N}) so that it must decay
faster than ${\mathcal O}\left(1/|\y|\right)$, and thus
$\Psi_{ki}(\infty) = 0$.  We conclude that $\mu_{ki} \in {\mathcal
P}_i^{0,N}$.  The opposite inclusion ${\mathcal P}_i^{0,N} \subset
{\mathcal P}_i^{0}$ in the first relation of (\ref{eq:Pset}) is proved
in the same way.

In turn, if $\mu_{ki} \in {\mathcal P}_i$ and $\mu_{k^{\prime}i} \in
{\mathcal P}_i^N$, the right-hand side of (\ref{eq:Psi_DN}) is
nonzero, implying that $\mu_{ki} \ne \mu_{k^{\prime}i}$ and that the
eigenfunctions $\Psi_{ki}$ and $\Psi_{k^{\prime}i}^N$ are not
orthogonal to each other.  This proves the second relation in
(\ref{eq:Pset}).
\end{proof}

Another practical consequence of the identity (\ref{eq:Psi_DN}) is the
possibility to re-expand an eigenfunction from one basis on the
eigenfunctions from the other basis.  More precisely, if $\mu_{ki}\in
{\mathcal P}_i \backslash {\mathcal P}_i^0$, then
\begin{equation}
  \Psi_{ki}(\y) = \sum\limits_{k^{\prime}=0}^\infty
  (\Psi_{ki}, \Psi_{k^{\prime}i}^N)_{L^2(\PT_i)} \Psi_{k^{\prime}i}^N(\y) 
  = \mu_{ki} d_{ki} \sum\limits_{k^{\prime}=0}^\infty
  \frac{\Psi_{k^{\prime}i}^N(\infty)}{\mu_{ki} - \mu_{k^{\prime}i}^N}
  \Psi_{k^{\prime}i}^N(\y) \,, \quad  \y\in\PT_i \,;
\end{equation}
similarly, if $\mu_{k^{\prime}i}^N\in {\mathcal P}_i^N \backslash
{\mathcal P}_i^{0}$, then
\begin{equation}  \label{eq:PsiN_expansion}
  \Psi_{k^{\prime}i}^N(\y) = \sum\limits_{k=0}^\infty (\Psi_{k^{\prime}i}^N,
  \Psi_{ki})_{L^2(\PT_i)} \Psi_{ki}(\y) 
  = \Psi_{k^{\prime}i}^N(\infty) \sum\limits_{k=0}^\infty
  \frac{\mu_{ki} d_{ki} }{\mu_{ki} - \mu_{k^{\prime}i}^N} \Psi_{ki}(\y) \,,
\quad  \y\in\PT_i \,.
\end{equation}

The complementary nature of the Steklov problems (\ref{eq:Psi_def})
and (\ref{eq:Psi_def_N}) suggest that the reactive capacitance can
actually be expressed in terms of the eigenpairs $\{\mu_{ki}^N,
\Psi_{ki}^N\}$, as shown in the following lemma.

\begin{lemma} 
For any $\kappa_i \notin {\mathcal P}_i^N$, one has
\begin{equation}  \label{eq:Cmu_N}
  \frac{1}{C_i(\kappa_i)} = \frac{1}{C_i(\infty)} +
  2\pi \sum\limits_{k=0}^\infty \frac{[\Psi_{ki}^N(\infty)]^2}
  {\mu_{ki}^N + \kappa_i} \,.
\end{equation}
\end{lemma}
\begin{proof}
To prove (\ref{eq:Cmu_N}), we consider an auxiliary function
\begin{equation}
  \tilde{w}_i(\y;\kappa) = \frac{w_i(\y;\infty)}{C_i(\infty)} -
  \frac{w_i(\y;\kappa)}{C_i(\kappa)} \,,
\end{equation}
which, by construction and via (\ref{mfpt:wc}), satisfies
\begin{subequations}  \label{eq:tilde_w}
\begin{align}  
\Delta_{\y} \tilde{w}_i & = 0\,, \quad \y \in \R_+^3 \,,\\
  \partial_n \tilde{w}_i + \kappa \tilde{w}_i & =
 \kappa \biggl(\frac{1}{C_i(\infty)} - \frac{1}{C_i(\kappa)}\biggr)
+ \frac{\partial_n w_i(\y;\infty)}{C_i(\infty)} \,,
\quad y_3=0 \,,\, (y_1,y_2)\in \PT_i\,,\\  
  \partial_n \tilde{w}_i & = 0 \,,
\quad y_3=0 \,,\, (y_1,y_2)\notin \PT_i\,,\\  
  \tilde{w}_i(\y) & \sim o(1/|\y|) \,,
                                    \quad \textrm{as}\quad |\y|\to \infty\,.
\end{align}
\end{subequations}
Since the decay of this function at infinity does not include the
monopole term $1/|\y|$, it can decomposed onto the Steklov
eigenfunctions $\{ \Psi_{ki}^N\}$, with the coefficients obtained from
the boundary condition:
\begin{equation}  \label{eq:tilde_w_PsiN}
  \tilde{w}_i(\y;\kappa) = \frac{1}{C_i(\infty)} - \frac{1}{C_i(\kappa)} +
  \frac{1}{C_i(\infty)} 
  \sum\limits_{k=0}^\infty \frac{\bigl(\partial_n w_i(\y;\infty),
    \Psi_{ki}^N\bigr)_{L^2(\PT_i)}}{\mu_{ki}^N + \kappa} \Psi_{ki}^N(\y) \,,
\end{equation}
where we used $\Psi_{0i}^N(\y) = 1/\sqrt{|\Gamma_i|}$ for the first two
terms.  Multiplying (\ref{eq:VkN_eq}) by $w_i(\y;\infty)$, multiplying
(\ref{mfpt:wc_1}) with $\kappa_i = \infty$ by $\Psi_{ki}^N(\y)$,
subtracting these equations, integrating them over a large hemisphere,
applying the Green's formula, the boundary conditions and the decay at
infinity, we get
\begin{equation}
  \int\limits_{\PT_i} \Psi_{ki}^N(\y) \, (\partial_n w_i(\y;\infty)) d\y =
  2\pi C_i(\infty) \Psi_{ki}^N(\infty) \,,
\end{equation}
that determines the scalar product in (\ref{eq:tilde_w_PsiN}).
Finally, in the limit $|\y|\to\infty$, the left-hand side of
(\ref{eq:tilde_w_PsiN}) vanishes, yielding the spectral expansion
(\ref{eq:Cmu_N}). 
\end{proof}

While the original representation (\ref{eq:Cmu_def0}) allowed us to
get the upper bound (\ref{eq:Cmu_upper}), the alternative expansion
(\ref{eq:Cmu_N}) gives access to lower bounds.  For instance, one has
\begin{equation} \label{eq:ineq_auxil}
  \frac{1}{C_i(\kappa_i)} \leq \frac{1}{C_i(\infty)} +
  \frac{2\pi}{\kappa_i |\PT_i|} 
+ 2\pi \sum\limits_{k=1}^\infty \frac{[\Psi_{ki}^N(\infty)]^2}{\mu_{ki}^N} \,.
\end{equation}
The last sum can computed by comparing the Taylor series of
$1/C_i(\kappa_i)$ as $\kappa_i\to 0$ with (\ref{eq:Cmu_Taylor}), from
which
\begin{equation}
  \frac{1}{C_i(\infty)} + 2\pi \sum\limits_{k=1}^\infty
  \frac{[\Psi_{ki}^N(\infty)]^2}{\mu_{ki}^N} =
  \frac{4\pi^2 a_i^3 c_{2i}}{|\PT_i|^2} \,,
\end{equation}
where the coefficient $c_{2i}$ is given by (\ref{eq:c1_def}).  We
conclude that
\begin{equation}  \label{eq:Cmu_bound}
  C_i(\kappa_i) \geq \frac{\kappa_i |\PT_i|}{2\pi}
  \biggl(1 + \frac{2\pi a_i^3 c_{2i}}{|\PT_i|} \kappa_i\biggr)^{-1} \,.
\end{equation}
We remark that if the last term in (\ref{eq:ineq_auxil}) is neglected,
we recover our sigmoidal approximation (\ref{eq:Cmu_approx}).  The
smallness of this last term as compared to $1/C_i(\infty)$ may explain
the high accuracy of this approximation (see \cite{Grebenkov26} for
further analysis).

From the representation (\ref{eq:Cmu_N}), it is clear that, when
$\kappa_i$ approaches $-\mu_{ki}^N$ such that $\Psi_{ki}^N(\infty)\ne
0$, the right-hand side of (\ref{eq:Cmu_N}) diverges, so that
$C_i(-\kappa_i)$ vanishes.  In other words, any element of ${\mathcal
P}_i^N$ is a zero of the function $C_i(-\kappa)$.  In the following
lemma, we prove that all zeros of $C_i(-\kappa)$ are in ${\mathcal
P}_i^N$.

\begin{lemma} \label{lem:Cmu}
Let $\mathcal Z_0$ be the set of zeros of the function $C_i(-\mu)$,
and $\mathcal P_i^{N}$ be the subset of eigenvalues $\mu_{ki}^N$ such that
$\Psi_{ki}^N(\infty) \ne 0$.  Then $\mathcal Z_0 = \mathcal P_i^{N}$.
\end{lemma}

\begin{proof} 
We first prove that $\mathcal P_i^{N} \subset \mathcal Z_0$.  
This inclusion follows directly from (\ref{eq:Cmu_N}) but we provide
an alternative argument here.

We recall from (\ref{mfpt:wc}) that $C_i(-\mu)$ is obtained from the
solution to
\bsub \label{appd:mfpt:wc}
\begin{align}
    \Delta_{\y} w_i &=0 \,, \quad \y \in \R_{+}^{3} \,, \label{appd:mfpt:wc_1}\\
   \partial_n w_i - \mu w_i &=-\mu \,, \quad y_3=0 \,,\,
    (y_1,y_2)\in \PT_i \,,  \label{appd:mfpt:wc_2}\\
    \partial_n w_i &=0 \,, \quad y_3=0 \,,\, (y_1,y_2)\notin \PT_i
    \,, \label{appd:mfpt:wc_3}\\
  w_i &\sim \frac{C_i(-\mu)}{|\y|} + {\mathcal O}(|\y|^{-2})  \,,
    \quad \mbox{as}\quad
    |\y|\to \infty \,. \label{appd:mfpt:wc_4}
\end{align}
\esub
By applying Green's second identity to $w_i$ and $\Psi_{ki}^{N}$,
which satisfies (\ref{eq:Psi_def_N}), over a large hemisphere of
radius $R$ in the upper half-plane we pass to the limit $R\to \infty$
to obtain
\begin{equation}\label{appd:green}
  \begin{split}
0 = \int_{\R_{+}^{3}} \left(w_i \, \Delta_{\y} \Psi_{ki}^N - \Psi_{ki}^N \,
 \Delta_{\y} w_i \right)\, d\y &=  \int_{\PT_i} \left(w_i \,\partial_n \Psi_{ki}^N
    -\Psi_{ki}^{N} \, \partial_n w_i \right) \, d\y  \\
  & \qquad + 2\pi \lim_{R\to \infty} R^2 \left(
    w_i \frac{\partial \Psi_{ki}^{N}}{\partial |\y|} -
    \Psi_{ki}^{N} \frac{\partial w_i}{\partial |\y|} \right)\Big{\vert}_{
    |\y|=R}\,.
  \end{split}
\end{equation}
Then, upon using the Steklov conditions (\ref{eq:VkN_stek}) and
(\ref{appd:mfpt:wc_2}) on $\PT_i$, together with the far-field
behaviors (\ref{eq:VkN_inf}) and (\ref{appd:mfpt:wc_4}) for
$\Psi_{ki}^{N}$ and $w_i$, we find that (\ref{appd:green}) reduces to
\begin{equation}\label{appd:scalarprod}
  \left(\mu_{ki}^{N} - \mu\right)\int_{\PT_i} w_i \Psi_{ki}^{N}  \,
  d\y + \mu \int_{\PT_i} \Psi_{ki}^{N} \, dy + 2\pi\Psi_{ki}^{N}(\infty)
  C_i(-\mu) =0 \,.
\end{equation}
Owing to the decay behavior (\ref{eq:VkN_inf}), we obtain from the
divergence theorem that for $k>0$, for which $\mu_{ki}^{N}>0$, we have
$\int_{\PT_i} \Psi_{ki}^{N}\, d\y=\left(\mu_{ki}^{N}\right)^{-1}
\int_{\PT_i} \partial_n \Psi_{ki}^{N}\, d\y=0$. As a result,
(\ref{appd:scalarprod}) simplifies to
\begin{equation}
  \left( \mu_{ki}^{N}-\mu\right) \int_{\PT} w_i \Psi_{ki}^{N}\, d\y =
  -2\pi \Psi_{ki}^{N}(\infty) C_i(-\mu) \,.
\end{equation}
We conclude that if an eigenfunction $\Psi_{ki}^N$ does not vanish at
infinity, then the function $C_i(-\mu)$ must vanish at $\mu =
\mu_{ki}^N$.  By its definition (\ref{eq:Cmu_def0}), $C_i(-\mu)$ also
vanishes at $\mu= \mu_{0i}^N = 0$.  This proves that $\mathcal P_i^{N}\subset
\mathcal Z_0$.

Let us now prove the opposite inclusion
$\mathcal Z_0 \subset \mathcal P_i^{N}$, i.e., there is no other zero of
$C_i(-\mu)$ than those determined by the eigenvalues $\mu_{ki}^N$.
Assume that there exists $\mu^{\prime} > 0$ such that
$-\mu^{\prime} \in \mathcal Z_0$ but
$-\mu^{\prime} \notin \mathcal P_i^{N}$.  Since the basis of
eigenfunctions $\Psi_{ki}^N$ is complete in $L^2(\PT_i)$, there is a
unique solution to the inhomogeneous problem
\begin{subequations}  \label{eq:auxil12}
\begin{align}  
\Delta_{\y} U & = 0\,, \quad \y \in \R_+^3 \,,\\
  \partial_n U - \mu^{\prime} U & = f \,,
\quad y_3=0 \,,\, (y_1,y_2)\in \PT_i\,,\\
\partial_n U & = 0 \,, \quad y_3=0 \,,\, (y_1,y_2)\notin \PT_i\,,\\
|\y|^2 \, |\nabla U(\y)| & \to 0\,, \quad \textrm{as}\quad |\y|\to \infty\,,
\end{align}
\end{subequations}
where we set $f = (\Psi_{0i})|_{\PT}$.  The divergence theorem implies
\begin{equation}\label{eq:app_ex_1}
0 = \int_{\PT_i} \partial_n U \, d\y= \mu^{\prime} \int_{\PT_i} U \,
d\y + d_{0i} \,,
\end{equation}
where $d_{0i}$ is given by (\ref{eq:dj}).
On one hand, upon multiplication of (\ref{eq:unity_Psi}) by $U$ and
integration over $\PT_i$, while using (\ref{eq:app_ex_1}), we obtain
\begin{equation}
- \frac{d_{0i}}{\mu^{\prime}} = \int_{\PT} U \, d\y = \sum\limits_{k=0}^\infty d_{ki} 
 \int_{\PT_i} U \Psi_{ki} \, d\y  \,.
\end{equation}
On the other hand, upon applying Green's second identity to $U$ and
$\Psi_{ki}$ over a large hemisphere in the upper half-plane and
passing to the limit we obtain
\begin{equation}\label{eq:app_ex_2}
  0 = \int_{\R_{+}^{3}} \left(\Psi_{ki} \Delta_{\y} U - U \Delta_{\y} \Psi_{ki}\right)
  \, d\y= 
\mu_{ki} d_{ki} U(\infty) + (\mu^{\prime} - \mu_{ki}) \int_{\PT_i}  U \Psi_{ki}\, d\y 
+ \delta_{0,k} \,,
\end{equation}
due to the orthogonality of $\Psi_{ki}$ to $\Psi_{0i}$.  Upon solving
(\ref{eq:app_ex_2}) for $\int_{\PT_i} U\Psi_{ki} \, d\y$, we obtain
from (\ref{eq:app_ex_1}) and the Steklov eigenfunction expansion of
$C_i(\mu)$ in (\ref{eq:Cmu}), that
\begin{equation}\label{eq:app_ex_3}
- d_{0i} = \mu^{\prime} \biggl[\frac{d_{0i}}{\mu_{0i} - \mu^{\prime}} +
U(\infty) \sum\limits_{k=0}^\infty
\frac{\mu_{ki} d_{ki}^2}{\mu_{ki} - \mu^{\prime}}\biggr]
= \frac{d_{0i} \mu^{\prime}}{\mu_{0i}-\mu^{\prime}} -
2\pi U(\infty) C_i(-\mu^{\prime})
\,.
\end{equation}
Since $\mu^{\prime}$ was assumed to be a zero of $C_i(-\mu)$, we
conclude that $-d_{0i} = d_{0i}\mu^{\prime}/(\mu_{0i}-\mu^{\prime})$.
Given that $d_{0i} \ne 0$, this yields that $\mu_{0i} = 0$, which
contradicts the strict positivity of the principal eigenvalue
$\mu_{0i}$.  We conclude that the second inclusion $\mathcal Z_0
\subset \mathcal P_i^{N}$ must also hold. Therefore, $\mathcal Z_0 =
\mathcal P_i^{N}$.
\end{proof}

\subsection{Relation to Dirichlet-to-Neumann operators}

The dual character of the Steklov problems (\ref{eq:Psi_def}) and
(\ref{eq:Psi_def_N}) can be further understood from the tight relation
between the associated Dirichlet-to-Neumann operators $\D_i$ and
$\D_i^N$.  The operator $\D_i$ associates to a function $f \in
H^{1/2}(\PT_i)$ on the patch $\PT_i$ another function $g = \D_i f =
(\partial_n u)|_{\PT_i} \in H^{-1/2}(\PT_i)$ on the same patch, where
$u$ is the unique solution of the BVP
\begin{subequations}  \label{eq:u_DtN}
\begin{align}
  \Delta u = 0 & \quad \textrm{in} ~ \R^3_+\,,
        \qquad u = f \quad \textrm{on}~ y_3= 0, ~ (y_1,y_2)\in \PT_i\,,   \\
\partial_n u = 0 &\quad \textrm{on}~ y_3 = 0, ~ (y_1,y_2)\notin \PT_i\,, \qquad
u \to 0 \quad \textrm{as} ~ |\y|\to \infty\,.
\end{align}
\end{subequations}
The operator $\D_i^N$ acts similarly, i.e., $g = \D_i^N f =
(\partial_n u^N)|_{\PT_i}$, where $u^N$ satisfies the same BVP, except
for the asymptotic decay $|\y|^2 |\nabla u^N| \to 0$ at infinity.  The
spectral properties of the operators $\D_i$ and $\D_i^N$ were
investigated in a more general setting in
\cite{Arendt15,Bundrock25}.  It is easy to check that $\mu_{ki}$ and
$\Psi_{ki}|_{\PT_i}$ are the eigenpairs of $\D_i$, whereas
$\mu_{ki}^N$ and $\Psi_{ki}^N|_{\PT_i}$ are the eigenpairs of
$\D_i^N$.  Moreover, Theorem 5.9 from \cite{Arendt15} states that
these two operators differ by a rank-one perturbation:
\begin{equation}
  \D_i f = \D_i^N f + \frac{1}{\beta_i} (f, \varphi_i)_{L^2(\PT_i)} \varphi_i  \,,
  \qquad \textrm{where} ~ \varphi_i = \D_i 1  
\quad \textrm{and} \quad \beta_i = (1, \varphi_i)_{L^2(\PT_i)} \,.
\end{equation}
Although this statement was proved for a slightly different
setting of exterior problems in the whole space, the arguments seem to
straightforwardly apply to our case (see further discussion in
\cite{Bundrock25}).
Since (\ref{eq:u_DtN}) with $f = 1$ is identical to (\ref{mfpt:wc})
with $\kappa_i = \infty$, one gets
\begin{equation}
  \varphi_i = \D_i 1 = (\partial_n w_i(\y;\infty))|_{\PT_i} \,,
  \qquad \beta_i = 2\pi C_i(\infty) \,,
\end{equation}
where the second relation follows from the divergence theorem.  

Many former relations can be rapidly recovered by using the operators
$\D_i$ and $\D_i^N$.  To illustrate this point, we first note that,
according to (\ref{eq:NGreen_Psi}), $1/(2\pi |\x-\y|)$ is the kernel
of the inverse of $\D_i$ so that the function $\omega_i$, defined in
(\ref{eq:omega_def}), can be formally written as $\omega_i =
\D_i^{-1} 1$ (the operator $\D_i$ is invertible because all its
eigenvalues are strictly positive).  The projections of this relation
onto a constant function or onto itself yield immediately that
\begin{subequations}
\begin{align}
  \int\limits_{\PT_i} \omega_i(\y) d\y &  = (1 , \D_i^{-1} 1)_{L^2(\PT_i)} =
         \sum\limits_{k=0}^\infty \frac{d_{ki}^2}{\mu_{ki}}  \,,\\
  \int\limits_{\PT_i} [\omega_i(\y)]^2 d\y &  = (\D_i^{-1} 1 , \D_i^{-1} 1)_{L^2(\PT_i)}
 = \sum\limits_{k=0}^\infty \frac{d_{ki}^2}{\mu_{ki}^2}  \,,
\end{align}
\end{subequations}
from which follows the representations (\ref{eq:c1_def}) for the
coefficients $c_{2i}$ and $c_{3i}$.

\section{Computation and Analysis of the Monopole Term $E_i$}
\label{sec:Ei}

We aim at computing numerically the coefficient $E_i$ given by
(\ref{mfpt:Ej_all}) for the circular patch $\PT_i$ of radius $a_i$.
Upon changing the integration variables, we represent it as
\begin{equation}\label{appE:new}  
  E_i(\kappa_i) = - \frac{C_i^2(\kappa_i)}{2} \log a_i +
  a_i^2 {\mathcal E}_i(\kappa_i a_i) \,,
\end{equation}
where
\begin{equation}  \label{eq:Ei_def}
  {\mathcal E}_i(\mu) = 2\int_{0}^{1} \frac{1}{r} \biggl(\int_{0}^{r}
   r^{\prime} a_i q_i(a_i r^{\prime};\mu/a_i) \, dr^{\prime}\biggr)^2 \, dr 
\end{equation}
now corresponds to the unit disk.  By recalling the limiting
asymptotics in (\ref{mfpt:Ej_asy}), we observe that
${\mathcal E}_i(\infty) = (3 - 4\log 2)/\pi^2$ and that
${\mathcal E}_i(\mu) \approx \mu^2/32$ as $\mu \to 0$.

In order to compute the integral in (\ref{eq:Ei_def}), one can employ
the spectral representation (\ref{eq:qi}) of the density
$q_i(\y;\kappa_i)$.  The presence of the factor $\mu_{ki}$ in the
numerator of (\ref{eq:qi}) deteriorates the numerical convergence of
the spectral expansion, thus requiring higher truncation
orders. Therefore, it is convenient to use the identity
(\ref{eq:unity_Psi}) to represent this function as
\begin{equation} \label{eq:qi2}
  q_i(\y;\kappa_i) = \frac{\kappa_i}{2}
  \biggl(1 - \kappa_i \sum\limits_{k=0}^\infty
  \frac{d_{ki} \, \Psi_{ki}(\y)}{\mu_{ki} + \kappa_i} \biggr)  \,,
  \quad \mbox{for}
  \quad \y\in \PT_i\,,
\end{equation}
which exhibits a faster convergence.  Repeating this trick, we get an
even faster converging representation for $\y\in \PT_i$ given by
\begin{equation} \label{eq:qi3}
  q_i(\y;\kappa_i) = \frac{\kappa_i}{2}
  \biggl(1 - \kappa_i \biggl[\frac{2a_i}{\pi} E(|\y|/a_i) 
  - \kappa_i \sum\limits_{k=0}^\infty
  \frac{d_{ki} \, \Psi_{ki}(\y)}{\mu_{ki}(\mu_{ki} + \kappa_i)}
\biggr]\biggr) \,,  \quad \mbox{for}   \quad \y\in \PT_i\,,
\end{equation}
where $E(z)$ is the complete elliptic integral of the second kind.  

Figure~\ref{fig:qi_mu_a} illustrates the behavior of the density $q_i$
versus $r$ for the unit disk ($a_i = 1$).  As expected, this density
approaches its limiting form $q_i(\y; \infty)$ as $\kappa_i
\to\infty$.  Curiously, for a fixed $r$, $q_i$ is not a monotone
increasing function of $\kappa_i$, as illustrated in
Fig.~\ref{fig:qi_mu_b}.  Finally, we highlight that even the use of
large truncation orders does not fully resolve the issue of the
numerical accuracy at large $\kappa_i$, especially for small $r$.

\begin{figure}
  \centering
     \begin{subfigure}[b]{0.49\textwidth}  
      \includegraphics[width =\textwidth]{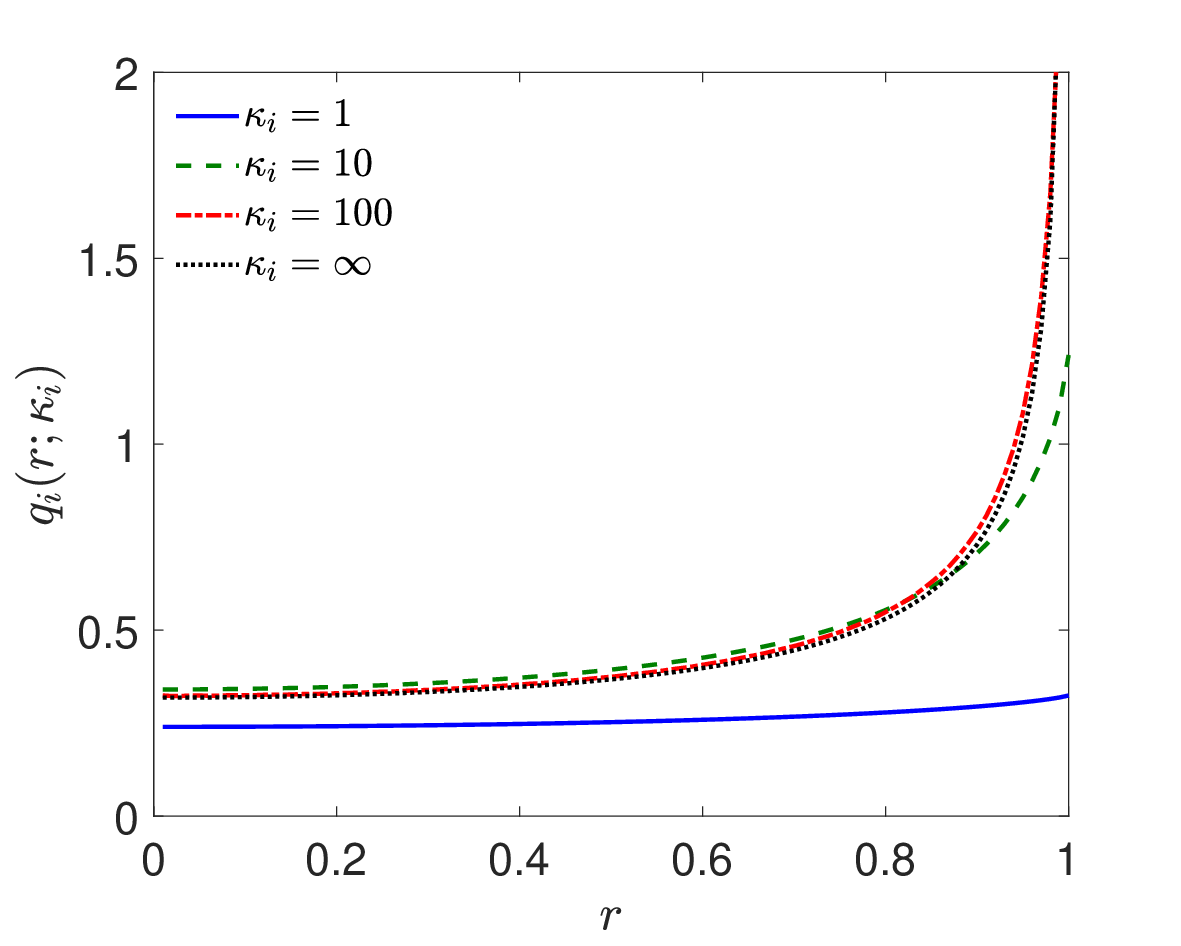} % {qi_mu2.eps} 
        \caption{$q_i(r;\kappa_i)$ versus $r$ for three $\kappa_i$}
        \label{fig:qi_mu_a}
    \end{subfigure}
    \begin{subfigure}[b]{0.49\textwidth}
      \includegraphics[width=\textwidth]{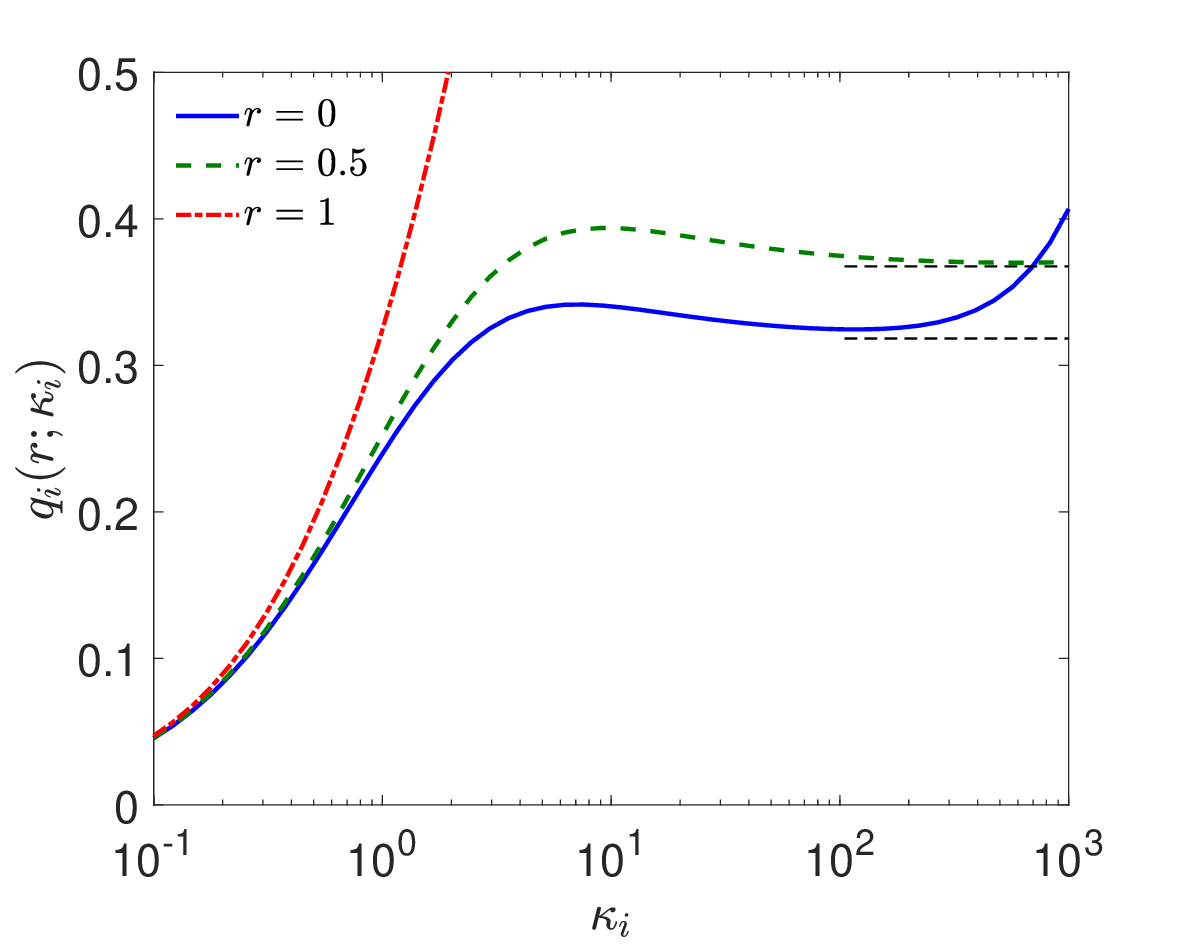} % {qi_r2.eps}
        \caption{$q_i(r;\kappa_i)$ versus $\kappa_i$ for three $r$}
        \label{fig:qi_mu_b}
    \end{subfigure}
\caption{ 
{\bf (a):} The density $q_i(r;\kappa_i)$ for the unit disk ($a_i = 1$)
as a function of $r$, for different values of $\kappa_i$, as indicated
in the legend.  The three curves were obtained by truncating the
series in (\ref{eq:qi2}) at $k \leq 500$.  The black dashed curve
presents the exact expression (\ref{mfpt:wc_q}). {\bf (b):} The
density $q_i(r;\kappa_i)$ as a function of $\kappa_i$, for three fixed
$r$.  Here we used the truncation up to $1000$ terms; nevertheless,
the obtained $q_i$ at $r = 0$ exhibits erroneous behavior at large
$\kappa_i$, since it does not approach its limit $1/(\pi
\sqrt{1-r^2})$, as indicated by dashed horizontal line. }
\label{fig:qi_mu}
% [q, r] = A_Ward_Steklov_3d_disk_q_fig;
% [q, r] = A_Ward_Steklov_3d_disk_q_fig3;
\end{figure}

In Fig.~\ref{fig:Ei}, we plot the numerically computed $E_i(\kappa_i)
= {\mathcal E}_i(\kappa_i)$ versus $\kappa_i$ for $\kappa_i<0$ for the
unit disk ($a_i = 1$).  In computing ${\mathcal E}_i$ from
(\ref{eq:Ei_def}), we numerically evaluated the integral in
(\ref{eq:Ei_def}) with the discretization step $\delta r = 10^{-4}$,
where $q_i$ was estimated from the series (\ref{eq:qi2}) truncated to
either $200$ terms (crosses) or to $1000$ terms (line).  An excellent
agreement between these two results confirms the accuracy of our
numerical computation.  Despite the vertical asymptotes (the poles of
$C_i(\kappa_i)$), the function $E_i(\kappa_i)$ remains always
positive.

In turn, Fig.~\ref{fig:Ekappa} shows the dependence of the coefficient
$E_i(\kappa_i)$ on $\kappa_i$ for $\kappa_i>0$.  In contrast to that
for $C_i(\kappa_i)$, this dependence is not monotonous.  We expect
that this is a result of the non-monotone approach of
$q_i(\y;\kappa_i)$ to $q_i(\y;\infty)$, as discussed earlier.  It is
also worth noting that the numerical computation for large $\kappa_i$
requires large truncation orders; in fact, the truncation order $100$
was not sufficient for $\kappa_i \geq 100$ (two curves start to
deviate from each other).  Even the large truncation order $500$
becomes insufficient for $\kappa_i \geq 1000$ (not shown).  To
overcome this difficulty, we propose below a simple empirical formula
(\ref{eq:Eapp}) for $E_i(\kappa_i)$ that closely predicts the
numerically computed value on the entire range $\kappa_i >0$.  The
remarkable accuracy of this simple approximation was illustrated in
Fig. \ref{fig:Ekappa}. 

\begin{figure}
    \centering
      \includegraphics[width = 85mm]{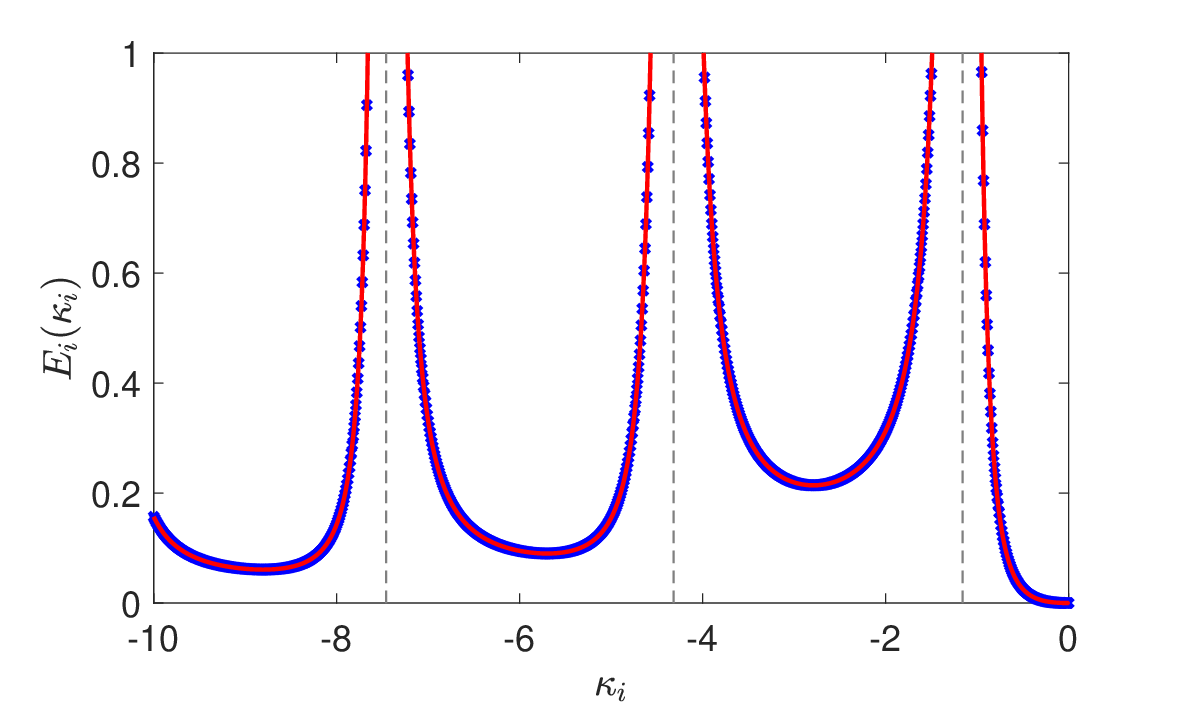} % {Ei_negative2.eps} 
\caption{ 
The numerically computed $E_i(\kappa_i)$ for $\kappa_i<0$ for a
circular patch $\PT_i$ with $a_i=1$.  Vertical dashed lines indicate
three poles $\{-\mu_{ki}\}$ of $C_i(\kappa_i)$, which are also the
poles of $q_i$ and thus of $E_i$.  Crosses correspond to the
truncation order $k_{\rm max} = 200$, whereas the line shows the result
for $k_{\rm max} = 1000$. }
\label{fig:Ei}
% [E1,E2,mu] = A_Ward_Steklov_3d_disk_E_fig1; 
%    [E1,E2,mu] = A_Ward_Steklov_3d_disk_E_fig; 
% [E,mu] = A_Ward_Steklov_3d_disk_E_fig2; 
\end{figure}

\subsection*{Empirical Approximation of $E$ for the Unit Disk}

The non-monotonic behavior of $E_i(\kappa_i)$ makes its approximation
trickier.  For this reason, we consider the ratio
$E_i(\kappa_i)/C_i^2(\kappa_i)$, which turns out to be a monotone
decreasing function of $\kappa_i$ that ranges from $1/8$ at $\kappa_i
= 0$ to $3/4 - \log 2$ as $\kappa_i\to \infty$.  This behavior
suggests to approximate it for the unit disk as
\begin{equation}
  \frac{E_i(\kappa)}{C_i^2(\kappa)} \approx \frac34 - \log{2} +
  \frac{1}{\frac{1}{\log{2} - 5/8} + f(\kappa)} \,,
\end{equation}
where $f(\kappa)$ is a suitable increasing function of $\kappa$ (such
that $f(0) = 0$ and $f(\infty) = \infty$).  Choosing $f(\kappa) =
3.04\, \kappa^{0.88}$, one can make this approximation accurate, but
it still requires the computation of $C_i(\kappa)$ via its spectral
representation (\ref{eq:Cmu}).  Applying the empirical approximation
$C^{\rm app}(\kappa)= 2\kappa/(\pi \kappa + 4)$ from
(\ref{mfpt:sigmoidal_2}) for $C_i(\kappa)$, we get a fully explicit
empirical approximation on $\kappa>0$ for the unit disk:
\begin{equation}  \label{eq:Eapp}
E^{\rm app}(\kappa) = \frac{4\kappa^2}{(\pi \kappa + 4)^2} 
\biggl(\frac34 - \log{2} + \frac{1}{\frac{1}{\log{2} - 5/8} + 5.17 \,
  \kappa^{0.81}} \biggr)\,,
\end{equation}
where we adjusted the function $f(\kappa)$ to ensure higher accuracy.
The closeness of this approximation was shown in
Fig.~\ref{fig:Ekappa}.

\section{Improved Numerics for the SN Problem}\label{appf:numer}

In this appendix, we devise an accurate numerical method to
treat the SN problem (\ref{sn:eig}) for the unit sphere with
either one circular patch, or two circular patches centered at the
north and south poles. To do so, we employ a general expansion of an
axially symmetric harmonic function $u$ in spherical coordinates in
terms of Legendre polynomials $P_n(z)$, given by
\begin{equation}  \label{eq:V_expansion_Pn}
u(r,\theta) = \sum\limits_{n=0}^\infty c_n r^n P_n(\cos\theta) \,.
\end{equation}
The unknown coefficients $c_n$ are determined from the mixed boundary
condition:
\begin{equation}
\partial_n u = \mu I_{\pa_a} u \,,\quad r = 1,~~ 0\leq \theta \leq \pi,
\end{equation}
where $I_{\pa_a}(\theta)$ is the indicator function of the Steklov
patches $\pa_a$, and $\mu = \sigma/\eps$.  The substitution of the
expansion (\ref{eq:V_expansion_Pn}) yields
\begin{equation}
  \sum\limits_{n=0}^\infty c_n n P_n(\cos\theta) = \mu
  \sum\limits_{n=0}^\infty c_n P_n(\cos\theta) I_{\pa_a}(\theta) \,.
\end{equation}
Multiplying this equation by $P_m(\cos\theta)\sin\theta$ and
integrating over $\theta \in (0,\pi)$, we get a system of linear
equations
\begin{equation}  \label{eq:AppF_auxil1}
m c_m = \mu \sum\limits_{n=0}^\infty K_{m,n} c_n \,,   \quad m = 0,1,2,\ldots \,,
\end{equation}
where
\begin{equation} \label{eq:AppF_Kdef}
K_{m,n} = (m+1/2)\int_0^\pi \sin\theta \, P_m(\cos\theta)
P_n(\cos\theta) I_{\pa_a}(\theta) d\theta \,.
\end{equation}
As the left-hand side of (\ref{eq:AppF_auxil1}) vanishes at $m = 0$,
it is convenient to isolate the coefficient $c_0$ from this relation
at $m = 0$ as
\begin{equation}  \label{eq:AppF_c0}
c_0 = - \frac{1}{K_{0,0}} \sum\limits_{n=1}^\infty K_{0,n} c_n \,.
\end{equation}
We substitute it into the above system to get
\begin{equation}
  m c_m = \mu \biggl(-\frac{K_{m,0}}{K_{0,0}} \sum\limits_{n=1}^\infty K_{0,n} c_n
  + \sum\limits_{n=1}^\infty K_{m,n} c_n \biggr)  \quad
 \mbox{for} \quad m = 1,2,\ldots \,.
\end{equation}
This system of linear equations can be written in a matrix form as
\begin{equation}  \label{eq:AppF_eigenM}
{\bf M} {\bf C} = \frac{1}{\mu} {\bf C} \,,
\end{equation}
where ${\bf C}$ is the (infinite-dimensional) vector of coefficients
$c_1,c_2,\ldots$, and
\begin{equation}  \label{eq:AppF_defM}
  {\bf M}_{m,n} = \frac{1}{m} \biggl(K_{m,n} -\frac{K_{m,0}K_{0,n}}{K_{0,0}} \biggr)
  \quad \mbox{for} \quad m,n = 1,2,\ldots \,.
\end{equation}
A truncation of the matrix ${\bf M}$ to a finite size $\nmax
\times \nmax$ and its numerical diagonalization allow one to
approximate the eigenvalues $\sigma = \mu \eps$ of the SN problem
(except for the trivial eigenvalue $\sigma^{(0)} = 0$); the accuracy of
this approximation can be improved by increasing the truncation order
$\nmax$.  Moreover, an eigenvector ${\bf C} =
(c_1,c_2,\ldots,c_{\nmax})^T$ of the truncated matrix ${\bf M}$
determines the coefficients in the (truncated) representation
(\ref{eq:V_expansion_Pn}) of the associated SN eigenfunction $u$,
whereas the coefficient $c_0$ is given by (\ref{eq:AppF_c0}).

Finally, the standard $L^2(\pa_a)$ normalization of the eigenfunction
$u$ can be ensured via the following relation:
\begin{align}  \nonumber
\int_{\pa_a} u^2({\bf s})\, d{\bf s} & = \int_0^\pi \sin\theta
\biggl[\sum\limits_{n=0}^\infty c_n P_n(\cos\theta)\biggr]^2 
I_{\pa_a}(\theta) \, d\theta 
= \sum\limits_{m=0}^\infty \sum\limits_{n=0}^\infty c_m c_n \frac{K_{m,n}}{m+1/2}
\\  \nonumber
& = 2c_0^2 K_{0,0} + 2c_0 \sum\limits_{n=1}^\infty c_n K_{0,n} +
c_0 \sum\limits_{m=1}^\infty c_m \frac{K_{m,0}}{m+1/2} \\ \nonumber
& \qquad + \sum\limits_{m=1}^\infty \sum\limits_{n=1}^\infty c_m c_n
\frac{1}{m+1/2}
\biggl(m {\bf M}_{m,n} + \frac{K_{m,0} K_{0,n}}{K_{0,0}}\biggr) \\  \label{eq:AppF_Vnorm}
& = \sum\limits_{m=1}^\infty \frac{m \, c_m}{m+1/2}
\sum\limits_{n=1}^\infty  {\bf M}_{m,n} c_n 
= \frac{1}{\mu}  \sum\limits_{m=1}^\infty \frac{m}{m+1/2} \, c_m^2 ,  
\end{align}
where we used (\ref{eq:AppF_Kdef}), (\ref{eq:AppF_c0}),
(\ref{eq:AppF_eigenM}), and (\ref{eq:AppF_defM}).  Note also that, in
the first equality, we omitted the factor $2\pi$, which would come
from the integral over the angle $\phi$, since we focus here
exclusively on axially symmetric eigenpairs; as a result, we ignore the
angular coordinate $\phi$ and the related normalization factor $2\pi$.
To fix the normalization of $u$, we have to set the left-hand side of
(\ref{eq:AppF_Vnorm}) to $1$, thus implying an equivalent condition on
the coefficients $c_m$ of the eigenvector ${\bf C}$.  In other words,
the normalization can be imposed directly through the coefficients
$c_m$.

For a single patch of angle $\epsilon_1$ at the north pole, an
explicit representation of the elements $K_{m,n}$ was given in
Appendix D.3 of \cite{Grebenkov19b} as
\begin{equation}
K_{m,n}^{(1)}(\epsilon_1) = \sum\limits_{k=0}^{\min\{m,n\}} B_{mn}^k 
\frac{P_{m+n-2k-1}(\cos\epsilon_1) - P_{m+n-2k+1}(\cos\epsilon_1)}{2(m+n-2k)+1} \,,
\end{equation}
where
\begin{equation}
  B_{mn}^k = \frac{A_k A_{m-k} A_{n-k}}{A_{m+n-k}} \, \frac{2m+2n-4k+1}{2m+2n-2k+1}
  \,,  \quad
A_k = \frac{\Gamma(k+1/2)}{\sqrt{\pi} \, \Gamma(k+1)}\,,  \quad A_0 = 1\,.
\end{equation}  
For a single patch of angle $\epsilon_2$ at the south pole, one can
use the symmetry of Legendre polynomials, $P_n(1-x) = (-1)^n P_n(x)$,
to get
\begin{equation}
K_{m,n}^{(2)}(\epsilon_2) = (-1)^{m+n} K_{m,n}^{(1)}(\epsilon_2) \,.
\end{equation}
When there are two patches, the matrix element $K_{m,n}$ is the
sum of these two contributions.

We emphasize that $u(r,\theta)$ in (\ref{eq:V_expansion_Pn}) is
constructed to be axially symmetric (i.e., independent of the angle
$\phi$).  In other words, this numerical procedure gives access
exclusively to axially symmetric eigenfunctions of the SN problem.  In
a similar way, one can construct non-axially-symmetric eigenfunctions
by using a representation in the form $e^{im\phi} r^n 
P_n^m(\cos\theta)$ with associated Legendre polynomials $P_n^m(z)$,
see similar constructions in \cite{Grebenkov24,Grebenkov25}.

\subsection*{Numerical Results for One Patch}

For validation purpose, we consider the case of a single patch of
radius $\eps$ at the north pole (with the angle $\epsilon =
\sin^{-1}(\eps)$).  We compute the first two SN eigenvalue by
truncating the matrix ${\bf M}$ to the size $\nmax \times \nmax$.  Our
numerical results for these eigenvalues, labeled by $\sigma_{\rm
num}^{(1)}$ and $\sigma_{\rm num}^{(2)}$, for different $\eps$ and
three truncation orders $\nmax$, are shown in
Table~\ref{tab:appf:one}.  From this table we observe that increasing
$\nmax$ by a factor of two does not significantly change our numerical
estimate for the first two SN eigenvalues, which confirms the high
accuracy of our numerical method. The favorable comparison between the
first two numerically computed SN eigenvalues and their asymptotic
predictions in (\ref{sn:example_1}) is shown in
Fig.~\ref{fig:SN_patch1}.

\begin{table}
\begin{center} 
\begin{tabular}{| c| c| c| c| c| c| c |}  \hline
$\nmax$& $\epsilon$ & 0.1 & 0.15 & 0.2 & 0.25 & 0.30 \\ \hline 
$1000$ & $\sigma_{\rm num}^{(1)}$ & 4.0578 & 4.0211 & 3.9814 & 3.9388 & 3.8933 \\
$2000$ & $\sigma_{\rm num}^{(1)}$ & 4.0576 & 4.0210 & 3.9813 & 3.9387 & 3.8933 \\
$4000$ & $\sigma_{\rm num}^{(1)}$ & 4.0576 & 4.0209 & 3.9813 & 3.9387 & 3.8933 \\  \hline
$1000$ & $\sigma_{\rm num}^{(2)}$ & 7.2758 & 7.2329 & 7.1845 & 7.1305 & 7.0710 \\
$2000$ & $\sigma_{\rm num}^{(2)}$ & 7.2751 & 7.2326 & 7.1843 & 7.1304 & 7.0710 \\
$4000$ & $\sigma_{\rm num}^{(2)}$ & 7.2750 & 7.2325 & 7.1843 & 7.1304 & 7.0709 \\  \hline
\end{tabular}
\end{center}
\caption{
Numerical approximation for the first two SN eigenvalues $\sigma_{\rm
num}^{(1)}$ and $\sigma_{\rm num}^{(2)}$ for a single circular patch
of angle $\epsilon$ and radius $\eps = \sin(\epsilon)$ on the unit
sphere for various values of $\epsilon$ and three different
truncations of the matrix ${\bf M}$.}
\label{tab:appf:one}
% load('SN_patch1_mu.mat');
% Mu4(:,1)'.*sin(eps1) ...  see Fig. 5
\end{table}

\bibliographystyle{plain}
\bibliography{references}

\end{document}